%% file: Draft_proof.tex
\documentclass[1.5pt, a4paper]{article}

\usepackage[utf8]{inputenc}
\usepackage[T1]{fontenc}
\usepackage{lmodern}
\usepackage{microtype}
\usepackage[pdfborder={0 0 0}]{hyperref}

\usepackage{color}
\usepackage{paralist}
\usepackage{graphicx}
\usepackage[all]{xy}
\xyoption{rotate}
\usepackage{verbatim}
\usepackage{tikz,pgfplots}
\usepackage{xcolor}

\usepackage[tbtags]{amsmath} 
\usepackage{amssymb}
\usepackage{amsthm}
\usepackage{mathtools}

\allowdisplaybreaks[2]

\usepackage{marginnote}

\usepackage[text={16.5cm, 24.2cm}, centering]{geometry}
\newcommand{\C}{\mathbb{C}}

\newcommand{\mac}{\mathcal C}
\newcommand{\mad}{\mathcal D}
\newcommand{\mam}{\mathcal M}
 \newcommand{\man}{\mathcal N}
 \newcommand{\map}{\mathcal P}
  \newcommand{\maq}{\mathcal Q}

\newcommand{\Hom}{\mathrm{Hom}}
\newcommand{\End}{\mathrm{End}}
\newcommand{\Ob}{\mathrm{Ob}}
\newcommand{\Fun}{\mathrm{Fun}}
\newcommand{\Bimod}{\mathrm{Bimod}}

\newcommand{\ev}{\mathrm{ev}}
\newcommand{\coev}{\mathrm{coev}}

\newcommand{\inv}[0]{{-1}}
\newcommand{\oo}[0]{\otimes}
\newcommand{\id}[0]{\mathrm{id}}
\newcommand{\tr}[0]{\mathrm{tr}}

\newcommand{\arrowIn}{
\tikz \draw[-stealth, line width=1pt] (-1pt,0) -- (1pt,0);
}

\newcommand{\arrowOut}{
\tikz \draw[-stealth, line width=.5pt] (-1pt,0) -- (1pt,0);
}

\newtheorem{theorem}{Theorem}[section]
\newtheorem*{theorem*}{Theorem}
\newtheorem{example}[theorem]{Example}
\newtheorem{lemma}[theorem]{Lemma}
\newtheorem{proposition}[theorem]{Proposition}
\newtheorem{corollary}[theorem]{Corollary}
\newtheorem{definition}[theorem]{Definition}
\newtheorem{remark}[theorem]{Remark}

\newcommand*\stdthebibliography{}
\let\stdthebibliography\thebibliography
\renewcommand{\thebibliography}[1]{%
\stdthebibliography{#1}\setlength{\itemsep}{-2pt}}

\def\mytitle{State sum models with defects\\ based on spherical fusion categories}
\def\myauthors{C.~Meusburger}
\hypersetup{pdftitle=\mytitle, pdfauthor=\myauthors}

\begin{document}

\begin{center}

  {\huge \mytitle}

  \vspace{.5em}

   {\large
   Catherine Meusburger\\[+1ex]
 }
 Department Mathematik \\
  Friedrich-Alexander-Universit\"at Erlangen-N\"urnberg \\
  Cauerstra\ss e 11,  91058 Erlangen, Germany\\
  Email: catherine.meusburger@math.uni-erlangen.de\\[+4ex]
  
{May 13, 2022}

  \begin{abstract}
We define a Turaev-Viro-Barrett-Westbury state sum model of  triangulated 3-manifolds with surface, line and point defects. Surface defects are oriented embedded 2d PL submanifolds and are labeled with bimodule categories  over spherical fusion categories with bimodule traces. Line and point defects form directed graphs on these surfaces and labeled with bimodule functors and bimodule natural transformations. 
The state sum is based on generalised 6j symbols  that encode the coherence isomorphisms of the defect data. We prove the triangulation independence of the state sum and show that it can be computed in terms of polygon diagrams that satisfy the cutting and gluing identities for polygon presentations of oriented surfaces. By  computing  state sums  with defect surfaces, we show that they detect the genus of a defect surface and are sensitive to its embedding. We show that  defect lines on defect surfaces  with trivial defect data define ribbon invariants for the centre of the underlying spherical fusion category. 
  \end{abstract}
\end{center}

\section*{Introduction}

Turaev-Viro-Barrett-Westbury state sums were introduced in \cite{TV, BW} and  define invariants of 3-manifolds and topological quantum field theories (TQFTs).  They are also of interest from the physics perspective, since they arise in the quantisation of  Chern-Simons and BF theories and have applications in topological quantum computing.  In particular, it was established  by Balsam and Kirillov in  \cite{BK, Kr} that they are closely related to  two  condensed matter physics models for topological quantum computers, namely  Kitaev's quantum double model \cite{Kit} and Levin-Wen models \cite{LW}.

Introducing defects in these models is  natural  in this physics context. Early work on specific line and point defects in  Turaev-Viro-Barrett-Westbury state sums by Karowski and Schrader \cite{KS} and by Barrett et al.~\cite{BMG} was motivated by the wish to introduce observables.  
In models from condensed matter physics it is also natural to study codimension one defects as domain walls. 
Introducing defects in these models amounts to labeling distinguished submanifolds  with higher categorical data that relates  the categorical 
data outside the defects. 
 The general algebraic data for defects of all codimensions in Levin-Wen models was identified  by Kitaev and Kong   \cite{KK} and coincides largely with the defect  data in this article.
 
 From the mathematics perspective the introduction of defects in Turaev-Viro-Barrett-Westbury state sums  is of interest  as a concrete model for defect TQFTs, in which the usual cobordism category is replaced with a decorated cobordism category that involves distinguished submanifolds of various codimensions. 
A general formalism and concrete examples of defect TQFTs were developed by Carqueville et al.~\cite{CRS,CRS2,CMS}. However, these  are not constructed as state sum models, but from a cobordism category. 

State sum models that extend  Turaev-Viro TQFTs and involve defect lines decorated by the centre of a spherical fusion category  were defined by Turaev and Virelizier  \cite{TVr,TVbook} and by Balsam and Kirillov in \cite{KB,B}. 
More recently, Fuchs et al.~\cite{FSS}  gave a state sum construction for a modular functor with bimodule categories over finite tensor categories and  Lee and Yetter \cite{LY} defined a triangulation independent state sum with defect surfaces. However,  the categorical defect data in \cite{LY} 
differs from  the data in the other publications and should be thought of as additional structure, as pointed out in \cite{LY}.

Thus, a Turaev-Viro-Barrett-Westbury style state sum model with defects of all codimensions labeled by bimodule categories, functors and natural transformations is still missing. 
In this article we define  such a state sum model  and prove its triangulation independence.  The model  generalises the usual Turaev-Viro-Barrett-Westbury state sum, which is obtained by choosing trivial defect data.

\medskip
{\bf State sum model with defects}

We consider three-dimensional PL manifolds $M$, possibly with boundary.  Defect surfaces  are  embedded oriented PL surfaces with compact boundary  contained in $\partial M$. Defect lines and points form directed graphs on these surfaces, whose edges may end on $\partial M$.  

Regions outside the defects are assigned spherical fusion categories.  Oriented defect surfaces are assigned  finite semisimple bimodule categories with bimodule traces, oriented defect lines bimodule functors and  defect vertices bimodule natural transformations. 

 Bimodule traces were introduced  by Schaumann in 
 \cite{G12, G14}, where it was shown that for fixed spherical fusion categories, bimodule categories with bimodule traces, bimodule functors and bimodule natural transformations form a pivotal 2-category.
The restriction to finite semisimple bimodule categories with bimodule traces is motivated by the fact that semisimplicity, summation over simple objects and the existence of traces are essential ingredients in the usual Turaev-Viro-Barrett-Westbury state sum. Although the latter can be generalised to pivotal fusion categories, see \cite{TVr,TVbook}, we stick with the original formulation in terms of spherical fusion categories  and  6j symbols and hence require bimodule traces.

The  state sum is based on a triangulation of $M$, which is required at first to intersect defect surfaces, defect lines and defect vertices transversally and generically. This means each tetrahedron  can intersect at most one defect surface,  either in a triangle, whose vertices lie on three edges incident at a common vertex, or in a quadrilateral, whose vertices lie on two pairs of opposite edges. Defect vertices must be in the interior of tetrahedra and defect edges must intersect their faces transversally in the interior. 

To define the state sum, one  labels the edges of the triangulation with simple objects, either in a spherical fusion category, if the edge does not intersect any defect surface,  or  in the bimodule category at its intersection point with a defect surface.  This assigns to each labeled tetrahedron a generalised 6j symbol that is  the 
 6j symbol for a spherical fusion category for a tetrahedron without defects. 
The state sum is then obtained by taking the product of the generalised 6j symbols of all tetrahedra  and the dimensions of the simple objects at the internal edges, summing over all  assignments of simple objects  and over bases of the morphism spaces for its triangles and rescaling by  the dimensions of the spherical fusion categories at the vertices. 

\medskip
{\bf Generalised 6j symbols}

The generalised 6j symbol for a tetrahedron depends on the simple objects assigned to its edges and the bimodule functors and natural transformation in its interior. By combining the latter with the morphisms assigned to its boundary triangles one obtains an endomorphism of a simple object in a bimodule category. The generalised 6j symbol is the trace of this endomorphism. It generalises the usual 6j symbols or F-matrices for spherical fusion categories and has a similar algebraic and geometric interpretation. 

Just as the  6j symbols for a spherical fusion category encode its associator on its simple objects, the 6j symbols for a bimodule category encode the natural isomorphisms that describe its bimodule category structure. If a tetrahedron also contains defect lines or defect vertices, these coherence isomorphisms are combined with the coherence isomorphisms for the bimodule functors at the defect lines and with the bimodule natural transformations at the defect points. 

The 6j symbols for a spherical fusion category relate the geometry of a tetrahedron to the diagrammatic calculus for a spherical fusion category. Their diagrams are 
obtained by taking the dual graph on the surface of the tetrahedron and labeling its edges and vertices with objects and morphisms in a spherical fusion category. The properties of a spherical fusion category imply that it has the symmetries of a tetrahedron. 

Similarly, generalised 6j symbols for defect tetrahedra relate the geometry of a tetrahedron with defects to a diagrammatic calculus for bimodule categories, functors and natural transformations. Their diagrams are obtained by projecting the duals of those edges and triangles that do not intersect the defect surface  on the defect surface. 
This yields a polygon with lines and vertices in its interior labeled by bimodule functors and natural transformations as well as objects and morphisms in spherical fusion categories. The diagrams also involve crossings labeled by coherence isomorphisms of bimodule categories and  bimodule functors.

We develop a diagrammatic calculus for such polygon diagrams and define their evaluation. For a diagram obtained from a tetrahedron with defects, this is  the generalised  6j symbol of the tetrahedron.
As the presence of defects reduces the symmetries of a tetrahedron, generalised 6j symbols exhibit fewer symmetries than the 6j symbols of a spherical fusion category. Nevertheless, the associated polygon diagrams can be cut and glued by summing over simple objects and bases of the morphism spaces at their boundary. 

{\bf Theorem 1: (Theorem \ref{th:polyids})}
\emph{
Evaluations of polygon diagrams are invariant under rotations of the polygon. They satisfy the usual cutting and gluing identities for polygon presentations of oriented surfaces and
an additional identity that allows one to cut and glue diagrams for spherical fusion categories in their interior.
}

\medskip
{\bf Triangulation independence}

With the cutting and gluing identities for polygon diagrams we prove that the state sum of a 3-ball bisected by a defect disc is the evaluation of the polygon diagram obtained by projecting the dual graph of the boundary triangulation on the defect disc.  This is sufficient to establish invariance under bistellar moves that involve only transversal and generic tetrahedra.

To establish  triangulation independence, we then refine the triangulations via stellar subdivisions, finite sequences of bistellar moves that involve only generic and transversal tetrahedra. The state sums can then be computed separately for neighbourhoods of the defect surfaces and for the regions between them. The latter are usual Turaev-Viro-Barrett-Westbury state sums.  The former are given by polygon diagrams, but  with additional summations over  boundary labels  that glue their sides pairwise. 
Combining them yields

{\bf Theorem 2: (Theorem \ref{th:fulltopinv})}
\emph{
If two transversal and generic triangulations of a 3-manifold with defect surfaces, defect lines and defect points agree at the boundary, their state sums are equal.}

This allows one to extend the definition of the state sum to triangulations that are non-generic or non-transversal in the interior of a 3-manifold via generic transversal subdivisions and establishes full triangulation independence for manifolds without boundaries. 

\medskip
{\bf Examples} 

To illustrate the formalism and the properties of the state sum we treat a number of examples.  We compute the state sum of a defect sphere  in a 3-ball for general categorical data. We also compute state sums with defect surfaces  of genus $g\geq 1$, but only for the simple categorical data underlying  Dijkgraaf-Witten models \cite{DW}. Here, we consider spherical fusion categories $\mathrm{Vec}_G$ and $\mathrm{Vec}_{G'}$ of graded vector spaces over finite groups $G, G'$, equipped with trivial cocycles, and bimodule categories defined by finite transitive $G\times G'^{op}$-sets, also equipped with the trivial cocycle. 
For this data, it is simple to 
compute the state sums of a 3-manifold $M=[0,1]\times\Sigma$ and of a 3-ball with a defect surface $\Sigma$ of general genus in the interior. This shows that the state sum detects properties of the bimodule categories labeling the defect surface as well as its  genus. 

We also compute the state sum for a 3-sphere with an embedded torus labeled by a trivial $G$-set and show that it gives the number of conjugacy classes of group homomorphisms from the fundamental group of its complement into $G$. This shows that the state sum is sensitive to the embedding of the surface. 

Our last  examples are framed links or ribbon links. They are  realised as defect graphs on trivial defect surfaces that are  labeled by a spherical fusion category $\mac$ as a bimodule category over itself. In this case,  bimodule functors labeling  defect lines and bimodule natural transformations labeling defect vertices correspond to objects and morphisms of the categorical centre $\mathcal Z(\mac)$. 

 In this setting, ribbons or framed links 
 can be realised in two ways. The first is a ribbon diagram on a defect surface, in which line segments are labeled by objects of $\mathcal Z(\mac)$ and crossings are viewed as defect vertices labeled by braidings. The second is an embedded ribbon link  labeled  by objects of $\mathcal Z(\mac)$ without  vertices. 
 
 We show that in both cases, the state sum is the product of the evaluation of the ribbon link with the usual Turaev-Viro-Barrett-Westbury state sum for the manifold without  defects.
This reproduces the results by Turaev and Virelizier \cite{TVr, TVbook} on  ribbon links in graph TQFTs.

 \medskip
 {\bf Structure of the article}
 
In Section \ref{sec:catbackground} we introduce the categorical data for the state sum model and summarise the relevant background from the literature. 

Section \ref{sec:diagrams} develops the diagrammatic calculus underlying generalised 6j symbols. After recalling diagrams for spherical fusion categories and  pivotal 2-categories, we introduce mixed diagrams, in which diagrams for spherical fusion categories overlap with diagrams 
labeled by bimodule functors and  natural transformations. 
By labeling   segments and line endpoints at their boundaries with objects and morphisms in bimodule categories we then obtain the polygon diagrams that are the building blocks of the state sum. 

In Section \ref{sec:projincl} we show that these polygon diagrams satisfy the usual  cutting and gluing identities for polygon presentations of surfaces.

The state sum of a triangulated 3-manifold with defect surfaces, defect lines and defect points is introduced in Section \ref{sec:statesum}.  After discussing the labeling of defects and  of the triangulation with algebraic data, we define the generalised 6j symbol of a tetrahedron with defects. 
We then introduce  the state sum, which  reduces to the usual Turaev-Viro Barrett-Westbury state sum  in the absence of defects, and show that it can be computed by gluing 3-manifolds with defects. 

In Section \ref{sec:topinvariance} we establish the triangulation independence of the state sum. We first summarise some background from PL topology and  introduce certain neighbourhoods of defect surfaces required for our constructions. We then prove that the state sum of a 3-ball bisected by a defect disc is the evaluation of the polygon diagram obtained by projecting the dual graph of the boundary triangulation on the disc.  This implies  that the state sum is invariant under bistellar moves that involve only generic transversal tetrahedra. We then prove triangulation independence by refining the triangulation  and decomposing it into neighbourhoods of defect surfaces and  3-manifolds without defects. 

In Section \ref{sec:examples} we compute examples and discuss the results.

\section{Categorical data}
\label{sec:catbackground}

\subsection{Spherical fusion categories}
We work over $\C$ and follow the conventions of Etingof et al.~in \cite{EGNO}. 
A {\bf multitensor category} is a  locally finite, $\C$-linear abelian rigid monoidal category  $\mac$ with a tensor product $\oo:\mac\times\mac\to \mac$  that is bilinear on the morphisms.  A {\bf tensor category} is an  indecomposable multitensor category
 with $\End_\mac(e)\cong \C$, where $e$ is the tensor unit.
A {\bf (multi)fusion category} is a finite semisimple (multi)tensor category.  

For all categories $\mac$ we denote by $\mac^{op}$ the opposite category and for all monoidal categories $\mac$ by $\mac^{rev}$ the category $\mac$ with the opposite monoidal structure. We define a {\bf left rigid} monoidal category as a category in which every object $x$ has a left dual $x^*$  and evaluation and coevaluation morphisms
\begin{align}\label{eq:coevl}
&\ev^L_x: x^*\oo x\to e, & &\coev^L_x: e\to x\oo x^*
\end{align}
that satisfy the identities 
\begin{align}\label{eq:snake}
(1_x\oo\ev^L_x)\circ (\coev^L_x\oo 1_x)=1_x\qquad (\ev^L_x\oo 1_{x^*})\circ (1_{x^*}\oo \coev^L_x)=1_{x^*},
\end{align}
 up to associators and unit constraints. 
We denote by $*: \mac\to\mac^{op,rev}$ the induced monoidal functor, see the  end of Section 2.10 in \cite{EGNO}. 
A {\bf pivotal} monoidal category $\mac$ is a left rigid monoidal category $\mac$ with a monoidal isomorphism  $\omega:**\Rightarrow\id_\mac$. 
Every pivotal category is right rigid with  
\begin{align}\label{eq:coevr}
&\ev^R_x=\ev^L_{x^*}\circ (\omega^\inv_x\oo 1_{x*}): x\oo x^*\to e, & &\coev^R_x=(1_{x^*}\oo\omega)\circ \coev^L_{x^*}: e\to x^*\oo x,
\end{align}
and the left and right trace of a morphism $\alpha\in \End_\mac(x)$ are given by
\begin{align}
&\tr^L(\alpha)=\ev^L_x\circ (1_{x^*}\oo \alpha)\circ \coev^R_x, & &\tr^R(\alpha)=\ev^R_x\circ (\alpha\oo 1_{x^*})\circ \coev^L_x.
\end{align}
A pivotal category $\mac$ is {\bf spherical} if $\tr^L(\alpha)=\tr^R(\alpha)$ for all endomorphisms $\alpha$ in $\mac$. In this case we  write $\tr(\alpha)=\tr^L(\alpha)=\tr^R(\alpha)$ and $\dim(x)=\tr(1_x)=\tr(1_{x^*})=\dim(x^*)$ for all $\alpha\in \End_\mac(x)$ and $x\in\Ob\mac$. 
The {\bf dimension} of a spherical fusion category $\mac$ is 
\begin{align}
\label{eq:catdimension}
\dim(\mac)=\sum_{i\in I} \dim(i)^2,
\end{align}
where $I$ is a set of representatives of the isomorphism classes of simple objects. 

\begin{example}\label{ex:vectgomega}\cite[Ex.~2.3.6]{EGNO}\\ Let $G$ be a finite group and $\omega: G\times G\times G\to \C^\times$ a normalised  3-cocycle. 
 The fusion category
$\mathrm{Vec}_G^\omega$ has
\begin{compactitem}
\item  as objects finite-dimensional $G$-graded vector spaces $V=\oplus_{g\in G} V_g$,
\item  as morphisms linear maps $f: V\to W$ with $f(V_g)\subset W_g$ for all $g\in G$,
\item  simple objects  $\delta^g$ for $g\in G$ with $(\delta^g)_g=\C$ and $(\delta^g)_h=0$ for $g\neq h$, 
\item  the  tensor product $V\oo W=\oplus_{g\in G}(\oplus_{hk=g} V_h\oo W_k)$, 
\item  the associator  given by $a_{\delta^g,\delta^h,\delta^k}=\omega(g,h,k) \id_{\delta^{ghk}}$ on the simple objects,  

\item dual objects $V^*=\oplus_{g\in G} V_{g^\inv}^*$ with the  usual (co)evaluation for finite-dimensional vector spaces,
\item pivotal structures that  are in bijection with characters $\kappa: G\to\C^\times$. 
 \end{compactitem}
A pivotal structure is spherical iff $\kappa(g)\in \{1,-1\}$ for all $g\in G$.
In particular, the trivial character $\kappa\equiv 1$ determines a spherical structure on  $\mathrm{Vec}_G^\omega$, its {\bf standard spherical structure}.
\end{example}

\subsection{(Bi)module categories,  functors and natural transformations}

In this section, we summarise background on (bi)module categories, (bi)module functors and (bi)module natural transformations from \cite[Ch.~7]{EGNO}.

\begin{definition}\label{def:modulecat}
Let $\mac, \mad$ be multitensor categories with associators  $a:\oo(\oo\times\id)\Rightarrow\oo(\id\times\oo)$ 
and left and right unit constraints  $l: e\oo \id\Rightarrow \id$,  $r:\id \oo e\Rightarrow \id $.
\begin{compactenum}
\item  A {\bf $\mac$-(left) module category} is a locally finite $\C$-linear abelian category $\mam$ together with
\begin{compactitem}
\item a functor $\rhd:\mac\times\mam\to\mam$ that is $\C$-bilinear on the morphisms and exact in the first variable, 
\item  natural isomorphisms 
$c: \rhd(\oo\times\id_\mam)\Rightarrow \rhd(\id_\mac\times\rhd)$ and $\gamma: e\rhd -\Rightarrow \id_\mam$
\end{compactitem}
such that  the following diagrams commute for all  $x,y,z\in\Ob \mac$ and $m\in\Ob \mam$:
\begin{align}\label{eq:pentagoncdef}
&\xymatrix{
((x\oo y)\oo z)\rhd m \ar[r]^{c_{x\oo y, z,m}} \ar[d]_{a_{x,y,z}\rhd 1_m}& (x\oo y)\rhd (z\rhd m) \ar[r]^{c_{x,y,z\rhd m}} & x\rhd(y\rhd(z\rhd m))\\
(x\oo(y\oo z)) \rhd m \ar[r]_{c_{x,y\oo z, m}}& x\rhd ((y\oo z)\rhd m)\ar[ru]_{\quad1_x\rhd c_{y,z,m}}
}\\
&\xymatrix{ (x\oo e)\rhd m \ar[rr]^{c_{x,e,m}} \ar[rd]_{r_X\rhd 1_m}  & & x\rhd (e\rhd m) \ar[ld]^{1_x\rhd \gamma_m}\\
&x\rhd m.
}\label{eq:trianglec}
\end{align}

\item A {\bf $\mad$-right module category} is  a  locally finite $\C$-linear abelian category $\mam$ together with 
\begin{compactitem}
\item a functor $\lhd:\mam\times \mad\to\mam$ that is $\C$-bilinear on the morphisms and exact in the second variable, 
\item  natural isomorphisms 
$d: \lhd(\lhd \times\id_\mad)\Rightarrow \lhd(\id_\mam\times\oo)$ and $\delta: -\lhd e\Rightarrow \id_\mam$
\end{compactitem}
that make the following diagrams commute for all  $x,y,z\in\Ob \mad$ and $m\in\Ob \mam$
\begin{align}\label{eq:pentagonddef}
&\xymatrix{
((m\lhd x)\lhd y)\lhd z \ar[r]^{d_{m\lhd x, y,z}} \ar[d]_{d_{m,x,y}\lhd 1_z}& (m\lhd x)\lhd (y\oo z) \ar[r]^{d_{m,x,y\oo z}} & m\lhd(x\oo(y\oo z))\\
(m\lhd(x\oo y)) \lhd z \ar[r]_{d_{m,x\oo y,z}}& m\lhd ((x\oo y)\oo z)\ar[ru]_{\quad 1_m\lhd a_{x,y,z}}
}\\
&\xymatrix{ m\lhd (e\oo x) \ar[rr]^{d_{m,e,x}} \ar[rd]_{1_m\lhd l_X}  & & (m\lhd e)\lhd x \ar[ld]^{\delta_m\lhd 1_x}\\
&m\lhd x.
}\label{eq:triangled}
\end{align}

\item A  {\bf $(\mac,\mad)$-bimodule category}  is a   locally finite $\C$-linear abelian  category $\mam$  with
\begin{compactitem}
\item a $\mac$-left module category structure $(\rhd,c,\gamma)$, 
\item a $\mad$-right module category structure $(\lhd, d,\delta)$, 
\item a natural isomorphism
 $b: \lhd(\rhd\times\id_\mad)\Rightarrow \rhd(\id_\mac\times\lhd)$
 \end{compactitem}
such that  the following diagrams commute for all  $x,y\in\Ob \mac$, $u,v\in\Ob \mad$ and $m\in\Ob \mam$
\begin{align}\label{eq:pentagonbdef}
&\xymatrix{
((x\oo y)\rhd m)\lhd u \ar[r]^{b_{x\oo y, m,u}} \ar[d]_{c_{x,y,m}\lhd 1_u}& (x\oo y)\rhd (m\lhd u) \ar[r]^{c_{x,y,m\lhd u}} & x\rhd(y\rhd(m\lhd u))\\
 (x\rhd(y\rhd m))\lhd u \ar[r]_{b_{x, y\rhd m, u}}& x\rhd ((y\rhd m)\lhd u)\ar[ru]_{\quad 1_x\rhd b_{y,m,u}}
}\\
&\xymatrix{
((x\rhd m)\lhd u)\lhd v \ar[r]^{d_{x\rhd m,u,v}} \ar[d]_{b_{x,m,u}\lhd 1_v}& (x\rhd m)\lhd (u\oo v) \ar[r]^{b_{x,m,u\oo v}} & x\rhd(m\lhd(u\oo v))\\
 (x\rhd(m\lhd u))\lhd v \ar[r]_{b_{x, m\lhd u, v}}& x\rhd ((m\lhd u)\lhd v).\ar[ru]_{\quad 1_x\rhd d_{y,m,u}}
}
\end{align}
\end{compactenum}
\end{definition}

A $\mad$-right module category can be defined more succinctly as a $\mad^{rev}$-module category and a $(\mac,\mad)$-bimodule category as a $\mac\boxtimes \mad^{rev}$-module category, where $\boxtimes$ is the Deligne product.  We expanded the definitions to set up  notation and  introduce the relevant diagrams.

\begin{remark}\label{rem:macmodule} \cite[Prop.~7.1.3]{EGNO}\\
 $\mac$-module category structures on $\mam$ correspond bijectively to 
monoidal functors $F:\mac\to \End(\mam)$. The monoidal functor defined by a $\mac$-module structure assigns to an object $x$ in $\mac$ the endofunctor $x\rhd -:\mam\to\mam$ and to a morphism $\alpha:x\to y$ in $\mac$ the natural transformation $\alpha\rhd-:x\rhd-\Rightarrow y\rhd -$.  The monoidal structure of $F$ is given by the natural isomorphisms $\gamma: e\rhd-\Rightarrow\id_\mam$ and $c_{x,y,-}: (x\oo y)\rhd -\Rightarrow x\rhd (y\rhd -)$. 
 \end{remark}

\begin{example}\label{ex:bimodulecategory} 
 Every  fusion category $\mac$ is a $(\mac,\mac)$-bimodule category with $\rhd=\lhd=\oo:\mac\times\mac\to\mac$ and  coherence isomorphisms $c_{x,y,z}=d_{x,y,z}=b_{x,y,z}=a_{x,y,z}: (x\oo y)\oo z\to x\oo (y\oo z)$, $\gamma_x=l_x: e\oo x\to x$ and $\delta_x=r_x: x\oo e\to x$.
 \end{example}

\begin{example}\label{ex:vecmodule}
 Any locally finite $\C$-linear abelian category $\mam$ is a bimodule category over the fusion category $\mathrm{Vect}_\C$ of finite-dimensional complex vector spaces.  This  $\mathrm{Vect}_\C$-module category structure is unique up to bimodule equivalence (cf.~Definition \ref{def:modulefunc}).
\end{example}

\begin{example}\label{ex:dualcat}
For  any $(\mac,\mad)$-bimodule category $\mam$ over pivotal fusion categories $\mac,\mad$ the category $\mam^{op}$ is a $(\mad,\mac)$-bimodule category with 
\begin{align}\label{eq:opcat}
&\lhd^{op}=\rhd (*\times  \id_\mam)\tau: \mam\times\mac\to\mam & &\rhd^{op}=\lhd (\id_\mam\times *)\tau: \mad\times \mam\to\mam,
\end{align}
where $*: \mac\to\mac^{op,rev}$ is the duality functor  and  $\tau$ the flip functor that exchanges the two factors in a cartesian product. 
The coherence isomorphisms are given by $c^{op}_{x,y,m}=d_{m,y^*,x^*}$, $d^{op}_{x,y,m}=c_{y^*,x^*, m}$, $b^{op}_{x,m,y}=b_{y^*,m,x^*}$.
We call it the {\bf opposite bimodule category}  $\mam^\#$.
\end{example}

\begin{example} \label{ex:modvectgomega} \cite[Ex.~7.4.10]{EGNO}\\ Indecomposable semisimple module categories $\mam$  over  the fusion category $\mathrm{Vec}_G:=\mathrm{Vec}_G^{\omega\equiv 1}$, 
up to equivalence, correspond to pairs $(X, \psi)$ of a  finite transitive $G$-set $X\cong G/L$ and an element $\psi\in H^2(L, \C^{\times})$.
\begin{compactitem}
\item The objects of $\mam$ are $X$-graded vector spaces.
\item The simple objects of $\mam$ are of the form $\delta^x$ for $x\in X$ . 

\item The action functor $\rhd: \mathrm{Vec}_G\times\mam\to\mam$ is given by $\delta^g\rhd \delta^x=\delta^{g\rhd x}$ on the simple objects. 

\item The natural isomorphism $c:\rhd(\oo\times \id)\Rightarrow \rhd(\id\times\rhd)$ is given by $\psi\in H^2(L,\C^\times)$.
\end{compactitem}
\end{example}

\begin{definition}\label{def:modulefunc}
Let $\mac,\mad$ be multitensor categories.  
\begin{compactenum}
\item  A {\bf $\mac$-module functor} $F:\mam\to\man$ between $\mac$-module categories $\mam,\man$ is a $\C$-linear right exact functor $F:\mam\to\man$ together with a natural isomorphism
$s: F\rhd \Rightarrow \rhd(\id_\mac\times F)$ that satisfies the pentagon and the triangle axiom: for all $x,y\in\Ob\mac$ and $m\in\Ob\mam$ the following diagrams commute
\begin{align}\label{eq:pentagoncfunc}
&\xymatrix{ F((x\oo y)\rhd m) \ar[d]_{F(c_{x,y,m})}\ar[r]^{s_{x\oo y, m}} & (x\oo y)\rhd F(m) \ar[r]^{c_{x,y, F(m)}} & x\rhd (y\rhd F(m))\\
F(x\rhd(y\rhd m)) \ar[r]_{s_{x. y\rhd m}} & x\rhd F(y\rhd m)\ar[ru]_{\;\; 1_x\rhd s_{y,m}}
}\\
&\xymatrix{ F(e\rhd m)\ar[rd]_{F(\gamma_m)} \ar[rr]^{s_{e,m}} & & e\rhd F(m)\ar[ld]^{\gamma_{F(m)}}\\
& F(m).
}\label{eq:trianglecfunc}
\end{align}

\item A {\bf $\mad$-right module functor} $F:\mam\to\man$ between $\mad$-right module categories $\mam,\man$ is a $\C$-linear right exact functor $F:\mam\to\man$ together with a natural isomorphism
$t: \lhd(F\times\id_\mad)\Rightarrow F\lhd$ that satisfies the pentagon and the triangle axiom: for all $x,y\in\Ob\mad$ and $m\in\Ob\mam$ the following diagrams commute
\begin{align}\label{eq:pentagondfunc}
&\xymatrix{ (F(m)\lhd x)\lhd y \ar[d]_{t_{m,x}\lhd 1_y}\ar[r]^{d_{F(m),x, y}} &  F(m)\lhd(x\oo y) \ar[r]^{t_{m, x\oo y}} & F(m\lhd (x\oo y))\\
F(m\lhd x)\lhd y \ar[r]_{t_{m\lhd x, y}} & F((m\lhd x)\lhd y) \ar[ru]_{\;\; F(d_{m,x,y})}
}\\
&\xymatrix{ F(m)\lhd e \ar[rd]_{\delta_{F(m)}} \ar[rr]^{t_{m,e}} & & F(m\lhd e)\ar[ld]^{F(\delta_m)}\\
& F(m).
}\label{eq:triangledfunc}
\end{align}

\item A  {\bf $(\mac,\mad)$-bimodule functor}  between $(\mac,\mad)$-bimodule categories $\mam,\man$ is a $\C$-linear right exact functor $F:\mam\to\man$ with a $\mac$-module functor structure $s$ and a $\mad$-right module functor structure $t$ that satisfy the hexagon axiom: for all $x\in\Ob\mac$, $y\in\Ob\mad$ and $m\in\Ob\mam$ the following diagram commutes
\begin{align}
\label{eq:hexa}
\xymatrix{
F(x\rhd m)\lhd y\ar[d]_{s_{x,m}\lhd 1_y} \ar[r]^{t_{x\rhd m, y}} & F((x\rhd m)\lhd y) \ar[r]^{F(b_{x,m,y})}& F(x\rhd (m\lhd y))\ar[d]^{s_{x, m\lhd y}}\\
(x\rhd F(m))\lhd y \ar[r]_{\;b_{x, F(m),y}} & x\rhd (F(m)\lhd y)  \ar[r]_{1_x\rhd t_{m,y}} & x\rhd F(m\lhd y).
}
\end{align}
\end{compactenum}
An {\bf equivalence of (bi)module categories} is a (bi)module functor that is an equivalence of categories.
\end{definition}

\begin{definition}\label{def:modnat} Let $\mac,\mad$ be multitensor categories.
\begin{compactenum}
\item A $\mac$-module natural transformation $\nu: F\Rightarrow G$ for $\mac$-module functors $F,G: \mam\to\man$ is a natural transformation  such that the
following diagram commutes for all $x\in\Ob\mac$ and $m\in\Ob\mam$
\begin{align}
\label{eq:cnat}
\xymatrix{ F(x\rhd m)\ar[d]_{s^F_{x,m}} \ar[r]^{\nu_{x\rhd m}} & G(x\rhd m)\ar[d]^{s^G_{x,m}}\\
x\rhd F(m) \ar[r]_{1_x\rhd \nu_m} & x\rhd G(m).
}
\end{align}
\item A $\mad$-right module natural transformation $\nu: F\Rightarrow G$ for $\mad$-right module functors $F,G: \mam\to\man$ is a natural transformation  such that the
following diagram commutes for all $x\in\Ob\mad$ and $m\in\Ob\mam$
\begin{align}
\label{eq:dnat}
\xymatrix{ F(m)\lhd x \ar[d]_{t^F_{m,x}} \ar[r]^{\nu_{m}\lhd 1_x} & G( m)\lhd x\ar[d]^{t^G_{m,x}}\\
F(m\lhd x) \ar[r]_{\nu_{m\lhd x}} & G(m\lhd x).
}
\end{align}
\item A $(\mac,\mad)$-bimodule natural transformation $\nu: F\rightarrow G$ between $(\mac,\mad)$-bimodule functors $F,G: \mam\to\man$ is a natural transformation $\nu: F\Rightarrow G$  that is a $\mac$-module and a $\mad$-right module natural transformation.
\end{compactenum}
\end{definition}

For $\mac$-module categories $\mam,\man$ we denote by $\Fun_\mac(\mam,\man)$ the category of $\mac$-module functors $F:\mam\to\man$ and $\mac$-module natural transformations between them and write $\End_\mac(\mam)=\Fun_\mac(\mam,\mam)$. Analogously, for $(\mac,\mad)$-bimodule categories $\mam,\man$, we write
$\Fun_{(\mac,\mad)}(\mam,\man)$ for the category of $(\mac,\mad)$-bimodule functors $F:\mam\to\man$ and bimodule natural transformations between them and $\End_{(\mac,\mad)}(\mam)=\Fun_{(\mac,\mad)}(\mam,\mam)$.

Composites of (bi)module functors inherit a  (bi)module functor  structure.  If $G:\mam\to\man$ and $F:\man\to\map$ are $(\mac,\mad)$-bimodule functors, then the  $(\mac,\mad)$-bimodule functor structure for $FG:\mam\to\map$ is given by 
\begin{align}\label{eq:bimodcomp}
&s^{FG}=s^F (\id_\mac\times G)\circ Fs^G: FG\rhd \Rightarrow \rhd(\id_\mac\times FG)\\
& t^{FG}=Ft^G\circ t^F (G\times\id_\mad): \lhd(FG\times\id_\mad)\Rightarrow FG\lhd.\nonumber
\end{align}
For all  bimodule functors  $F,F':\man\to\map$, $G,G':\mam\to\man$   and bimodule natural transformations $\nu: F\Rightarrow F'$ and $\mu: G\Rightarrow G'$  the natural transformations $\nu G: FG\Rightarrow F'G$ and $F\mu: FG\Rightarrow FG'$ are bimodule natural transformations with respect to this bimodule functor structure. As composites of bimodule natural transformations are  also bimodule natural transformations, it follows that  $(\mac,\mad)$-bimodule categories, $(\mac,\mad)$-bimodule functors and $(\mac,\mad)$-bimodule natural transformations form a 2-category.

The following examples show how the data of a (bi)module category can be recovered as coherence data of  (bi)module functors.
They are the categorical counterparts of standard constructions for modules over rings.

\begin{example}\label{ex:modulefuncs} Let $\mac,\mad$ be  multitensor categories with (right) centres $\mathcal Z(\mac)$, $\mathcal Z(\mad)$ and $\mam,\man$ $(\mac,\mad)$-bimodule categories. 
\begin{enumerate}
\item  For any $x\in\Ob \mad$ the functor
$F=-\lhd x: \mam\to\mam$  is a $\mac$-module functor with coherence isomorphisms $s_{c,m}=b_{c,m,x}: (c\rhd m)\lhd x\to c\rhd (m\lhd x)$. Any morphism  $\alpha\in \Hom_\mad(x,y)$ defines a $\mac$-module natural transformation $\nu=-\lhd \alpha: -\lhd x\Rightarrow -\lhd y$. This defines a functor $\mad\to \End_\mac(\mam)$.

\item If $x\in \Ob \mathcal Z(\mad)$ with half-braidings $\sigma_{d,x}: d\oo x\to x\oo d$,  the functor $F$ from 1.~becomes  a $(\mac,\mad)$-bimodule functor with $t_{m,d}= d_{m,d,x}^\inv\circ (1_m\lhd \sigma_{x,d}^\inv)\circ d_{m,x,d}: (m\lhd x)\lhd d\to (m\lhd d)\lhd x$.
For $\alpha\in \Hom_{\mathcal Z(\mad)}(x,y)$ the morphisms $\nu_m=1\lhd \alpha: m\lhd x\to m\lhd y$ define a $(\mac,\mad)$-bimodule natural transformation $\nu: -\lhd x\Rightarrow -\lhd y$. 
This defines a functor $\mathcal Z(\mad)\to \End_{(\mac,\mad)}(\mam)$.

\item Analogously, for  $x\in \Ob\mac$  $F=x\rhd -: \mam\to\mam$    with $t_{m,d}=b_{x,m,d}: (x\rhd m)\lhd d\to x\rhd (m\lhd d)$ is a $\mad$-right module functor. Any morphism  $\alpha\in \Hom_\mac(x,y)$  defines a $\mad$-right module natural transformation $\nu=\alpha\rhd -: x\rhd -\Rightarrow y\rhd -$. This defines a functor $\mac\to \End_{\mad^{rev}}(\mam)$.

\item If  $x\in\Ob \mathcal Z(\mac)$ with half-braidings $\sigma_{c,x}: c\oo x\to x\oo c$, the functor $F$ from 3.~becomes a  $(\mac,\mad)$-bimodule functor with $s_{c,m}=c_{c,x,m}\circ (\sigma^\inv_{c,x}\rhd 1_m)\circ c_{x,c,m}^\inv: x\rhd (c\rhd m)\to c\rhd (x\rhd m)$. 
For any morphism $\alpha\in \Hom_{\mathcal Z(\mac)}(x,y)$ the morphisms $\nu_m=\alpha\rhd 1_m: x\rhd m\to y\rhd m$ define a $\mac$-module natural transformation $\nu: x\rhd-\Rightarrow y\rhd -$. This defines a functor $\mathcal Z(\mac)\to \mathrm{End}_{(\mac,\mad)}(\mam)$. 

\item For any $\mac$-module category $\mam$ and $m\in\Ob \mam$, the functor $F=-\rhd m:\mac\to\mam$ is a $\mac$-module functor with $s_{x,y}=c_{x,y,m}: (x\oo y)\rhd m\to x\rhd (y\rhd m)$. Any morphism $\alpha\in \Hom_\mam(m,m')$ defines a $\mac$-module natural transformation $\nu=-\rhd \alpha: -\rhd m\Rightarrow -\rhd m'$. 

If $\mam$ is a $(\mac,\mad)$-bimodule category, this defines a $(\mac,\mad)$-bimodule equivalence $\mam\to \Fun_\mac(\mac,\mam)$ with respect to the $(\mac,\mad)$-bimodule structure on $\Fun_\mac(\mac,\mam)$  given by $(c\rhd F)(x)=F(x\oo c)$ and $(F\lhd d)(x)=F(x)\lhd d$ for $c,x\in\Ob\mac$, $d\in\Ob\mad$.

\item For any $\mad$-right module category $\mam$ and $m\in\Ob \mam$, the functor $F=m\lhd -:\mad\to\mam$ is a $\mad$-right module functor with $t_{x,y}=d_{x,y,m}: (m\lhd x)\lhd y\to m\lhd (x\oo y)$. Any morphism $\alpha\in \Hom_\mam(m,m')$ defines a $\mac$-module natural transformation $\nu= \alpha\lhd -:  m\lhd - \Rightarrow  m'\lhd -$. 

If $\mam$ is a $(\mac,\mad)$-bimodule category, this defines a $(\mac,\mad)$-bimodule equivalence $\mam\to \Fun_{\mad^{rev}}(\mad,\mam)$ with respect to the $(\mac,\mad)$-bimodule structure on $\Fun_{\mad^{rev}}(\mad,\mam)$  given by $(c\rhd F)(x)=c\rhd F(x)$ and $(F\lhd d)(x)=F(d\oo x)$ for $c\in\Ob\mac$, $d,x\in\Ob\mad$.

\item If $\mathcal C,\mathcal D$ are pivotal fusion categories,  any $(\mac,\mad)$-bimodule functor $F: \mam\to\man$ defines a $(\mad,\mac)$-bimodule functor $F^\#:\mam^\#\to\man^\#$ with coherence isomorphisms  $s^{F^\#}_{d,m}=t^F_{m,d^*}$ and $t^{F^\#}_{m,c}=s^F_{c^*,m}$. 
 Any $(\mac,\mad)$-bimodule natural transformation $\nu: F\Rightarrow G$ induces a $(\mad,\mac)$-bimodule natural transformation $\nu^\#: G^\#\Rightarrow F^\#$ with  $\nu^\#_m=\nu_m$.

\item For any $(\mac,\mad)$-bimodule functor $F:\mam\to\man$, the morphisms $s_{c,m}: F(c\rhd m)\to c\rhd F(m)$ define a $\mad$-right module natural isomorphism $s_c: F(c\rhd -)\Rightarrow c\rhd F(-)$. Analogously, the morphisms $t_{m,d}: F(m)\lhd d\to F(m\lhd d)$ define a $\mac$-module natural transformation $t_d: F(-)\lhd d\Rightarrow F(-\lhd d)$. This follows  from the  hexagon identity \eqref{eq:hexa} for a bimodule functor and the  module functor structure  \eqref{eq:bimodcomp} for a composite functor.

\end{enumerate}
\end{example}

If $\mac$ is a finite multitensor category as a bimodule category over itself, then  all bimodule endofunctors of $\mac$ are naturally isomorphic to functors of the  form $F=c\oo -:\mac\to\mac$ with an object $c\in \mathcal Z(\mac)$, as in Example \ref{ex:modulefuncs}, 4.~and all bimodule natural transformations between them are of the form $\nu=f\oo -: c\oo -\Rightarrow c'\oo -$ with a morphism $f: c\to c'$ in $\mathcal Z(\mac)$. This identifies $\End_\mac(\mac)$ with the  centre $\mathcal Z(\mac)$.

\begin{example}\label{ex:center} \cite[Prop.~7.13.8.]{EGNO} Let $\mac$ be a finite multitensor category, viewed as a $(\mac,\mac)$-bimodule category. 
Then  the monoidal category $\End_\mac(\mac)$ is canonically monoidally isomorphic to the centre $\mathcal Z(\mac)$.
\end{example}

\subsection{Finite semisimple (bi)module categories over pivotal fusion categories}
\label{subsec:semisimple}
We focus on {\em finite semisimple} $(\mac,\mad)$-bimodule categories  over {\em fusion categories} $\mac$, $\mad$ and on $(\mac,\mad)$-bimodule functors and  natural transformations between them. In this case
all categories $\Fun_{(\mac,\mad)}(\mam,\man)$ are finite $\C$-linear abelian categories  \cite[Prop.~7.11.6]{EGNO} and  semisimple, see \cite[Th.~2.16]{ENO} by Etingof et al. 

Moreover, the  2-category  $\Bimod(\mac,\mad)$  of finite semisimple $(\mac,\mad)$-bimodule categories $(\mac,\mad)$-bimodule functors and $(\mac,\mad)$-bimodule natural transformations  is equipped with duals. Any $(\mac,\mad)$-bimodule functor $F:\mam\to\man$ between finite semisimple $(\mac,\mad)$-bimodule categories
has a left adjoint $F^l:\man\to\mam$ and a right adjoint $F^r:\man\to\mam$, and  these adjoints are also $(\mac,\mad)$-bimodule functors.  The left and right  duals of $F$  are  given by its left and right adjoint and  the units and counits of the adjunction.

This follows, because module functors between finite semisimple module categories are always exact, as finite semisimple module categories are exact and module functors between exact module categories are exact \cite[Prop.~7.6.9]{EGNO}. Any finite abelian $\C$-linear  category is equivalent to the category  of modules over some finite-dimensional semisimple $\C$-algebra \cite[Def.~1.8.5]{EGNO}. Any right exact $\C$-linear functor between such categories  is $\otimes$-representable \cite[Prop.~1.8.10]{EGNO}  and hence has a left adjoint, namely the corresponding Hom-functor. Analogously,  left exactness implies the existence of a right adjoint.

 It follows that category $\End_\mac(\mam)$ is a finite multifusion category. It is a fusion category if and only if $\mam$ is {\em indecomposable}, that is, not the direct sum of non-trivial module categories,  see \cite[Sec.~7.3.6]{EGNO} and  the remark after \cite[Def.~7.2.12]{EGNO}. 
 The category $\mam$ is an $\End_\mac(\mam)$-module category with the action functor that applies endofunctors and natural transformations to objects and morphisms in $\mam$.

We also require  that the fusion categories $\mac$ and $\mad$ are {\em pivotal}. This has implications for the opposite bimodule category from Example  \ref{ex:dualcat}. 
 A $(\mad,\mac)$-bimodule structure on $\mam^{op}$ can be defined analogously to Example  \ref{ex:dualcat} for any fusion category $\mam$. However,  the category $\mam^{\#\#}$ is  in general not equivalent to the bimodule category $\mam$, and one needs to distinguish the $(\mad,\mac)$-bimodule category structures on $\mam^{op}$ that are induced by left and right duals, see the results by Douglas et al.~in \cite[Sec.~2.4]{DSS}. As shown in \cite{G14}, the pivotal structures of $\mac$ and $\mad$ relate their left and right duals and  eliminate this ambiguity. This justifies the name {\em opposite bimodule category}.

\begin{proposition}\label{ex:oppositebimod} \cite{G14, DSS}

 Let $\mac,\mad$ be pivotal fusion categories  and $\mam,\man$ finite semisimple $(\mac,\mad)$-bimodule categories.
\begin{compactenum}

\item  There are equivalences  $\Omega: \mam^{\#\#}\to \mam$ of $(\mac,\mad)$-bimodule categories with $F\Omega=\Omega F^{\#\#}$ for all $(\mac,\mad)$-bimodule functors $F:\mam\to\man$.

\item There are equivalences of  $(\mac,\mad)$-bimodule categories $\mam^\#\to \Fun_\mac(\mam,\mac)$ and $\mam^\#\to \Fun_{\mad^{rev}}(\mam,\mad)$.
\end{compactenum}
\end{proposition}

\begin{proof}
Claim 1.~is shown in \cite{G14}. As a functor $\Omega=\id_\mam$,  and its coherence isomorphisms are given by the pivots $\omega:**\Rightarrow \id$   as $s_{c,m}=\omega_c\rhd 1_m: c^{**}\rhd m \to c\rhd m$ and $t_{m,d}=1_m\lhd \omega^{\inv}_d: m\lhd d\to m\lhd d^{**}$.  The identity  $F\Omega=\Omega F^{\#\#}$ follows directly from the definition of $F^\#$ in Example \ref{ex:modulefuncs}, 7.~and formula \eqref{eq:bimodcomp} for the coherence data of a composite bimodule functor. 

Claim 2.~is essentially \cite[Prop.~2.4.9]{DSS} for  module categories over pivotal fusion categories. It follows by combining \cite[Prop.~2.4.9]{DSS} with 1.  The  $(\mad,\mac)$-bimodule structure on $\Fun_\mac(\mam,\mac)$  is given by  $(d\rhd F)(m)=F(m\lhd d)$ and $(F\lhd c)(m)=F(m)\oo c$ and  the one on $\Fun_{\mad^{rev}}(\mam,\mad)$ by   $(d\rhd F)(m)=d\oo F(m)$ and $(F\lhd c)(m)=F(c\rhd m)$ for all $c\in\Ob\mac$, $m\in\Ob\mam$ and $d\in\Ob\mad$. The equivalences are given by the left adjoints of the functors $-\rhd m: \mac\to\mam$  and $m\lhd-:\mad\to\mam$ from Example \ref{ex:modulefuncs}, 5.~and 6. 
\end{proof}

\subsection{Bimodule categories with bimodule traces}

(Bi)module traces on (bi)module categories over fusion categories were first  introduced by Schaumann in \cite{G12} and investigated further in  \cite{G14}.
We summarise (bi)module traces for finite semisimple (bi)module categories over pivotal fusion categories from \cite{G12,G14}. 
To keep notation simple, we focus on left module categories and identify $(\mac,\mad)$-bimodule categories with $\mac\boxtimes\mad^{rev}$-left module categories.

\begin{definition}\label{def:moduletrace} \cite[Def.~3.7]{G12}

Let $\mac$ be a pivotal multifusion category and $\mam$  a finite semisimple $\mac$-module category.\\
A {\bf trace} on $\mam$ is a collection of morphisms $\theta_m: \End_\mam(m)\to \C$ indexed by objects $m\in\Ob\mam$ that satisfy
\begin{compactenum}
\item {\bf cyclicity:} $\theta_m(\beta\circ \alpha)=\theta_{m'}(\alpha\circ \beta)$ for all $\alpha\in \Hom_\mam(m,m')$ and $\beta\in \Hom_\mam(m',m)$.
\item {\bf non-degeneracy:} for all $m,m'\in\Ob\mam$ the following bilinear maps are non-degenerate.
$$
\Hom_\mam(m',m)\times \Hom_\mam(m,m')\to \C,\quad (\beta,\alpha)\mapsto \theta_m(\beta\circ \alpha)
$$
\end{compactenum}
A trace  on $\mam$ is called a {\bf $\mac$-module trace} if $\theta_{c\rhd m}(\alpha)=\theta_{m}(\tr^\mac(\alpha))$ for all $\alpha\in \End_\mam(c\rhd m)$, where $\tr^\mac(\alpha)$  is the partial trace of $\alpha$ with respect to $\mac$
$$
\tr^\mac(\alpha)=(\ev^L_c\rhd 1_m)\circ c^\inv_{c^*,c,m}\circ (1_{c^*}\rhd \alpha)\circ c_{c^*,c,m}\circ (\coev^R_{c}\rhd 1_m): m\to m.
$$
\end{definition}

Analogously, a $\mad$-right module trace on a $\mad$-right module category $\mam$ is defined as a $\mad^{rev}$-module trace and a
$(\mac,\mad)$-bimodule trace on a $(\mac,\mad)$-bimodule category $\mam$ as a $\mac\boxtimes \mad^{rev}$-module trace. 

As for spherical categories,  the {\bf dimension} of an object $m$ in a (bi)module category $\mam$ with a (bi)module trace is $\dim(m)=\theta_m(1_m)$ and  for any set $I$ of representatives of the isomorphism classes of simple objects 
\begin{align}
\dim(\mam)=\sum_{m\in I} \dim(m)^2.
\end{align}
If $\mathcal C$ is spherical, the condition that $\theta$ is a $\mac$-module trace  implies for all $c\in \Ob\mac$  and $m\in\Ob\mam$
\begin{align}\label{eq:dimensionsmodule}
\dim(c\rhd m)=\dim(c)\dim(m).
\end{align}
It is shown in \cite[Prop.~4.4]{G12} that a $\mac$-module trace on an {\em indecomposable} finite semisimple module category $\mam$ is unique up to rescaling $\theta_m\to z\theta_m$ with $z\in\C^\times$. The existence of a $\mac$-module trace on a $\mac$-module category $\mam$ is a condition on  the dimensions of simple objects of $\mac$ and $\mam$, see \cite[Sec.~5]{G12}.

\begin{example} \label{ex:sphtriv}
 A pivotal fusion category $\mac$ as a $(\mac,\mac)$-bimodule category over itself has a bimodule trace if and only if $\mac$ is spherical. 
\end{example}

\begin{example}\label{ex:opptrace}
Each  $(\mac,\mad)$-bimodule trace on a finite semisimple $(\mac,\mad)$-bimodule category $\mam$ defines a $(\mad,\mac)$-bimodule trace on the opposite bimodule category $\mam^\#$ from Example \ref{ex:dualcat}.
\end{example}

\begin{example}\cite[Ex.~3.13]{G12} \label{ex:tracevect}An indecomposable module category over the pivotal fusion category $(\mathrm{Vec}_G, \kappa)$,  given by a pair $(G/L,\psi)$ as in Example \ref{ex:modvectgomega}, admits a module trace iff $\kappa\vert_{L}\equiv 1$.
Any indecomposable semisimple module category over $\mathrm{Vec}_G$ with the standard spherical structure admits a module trace.
\end{example}

(Bi)module traces on (bi)module categories equip (bi)module functors between them with additional structure. It is shown in \cite[Th.~4.5]{G12}  that for module categories $\mam$ and $\man$  over a pivotal fusion category $\mac$  with  module traces, any module functor $F:\mam\to\man$ is naturally isomorphic to its double left adjoint $F^{ll}: \mam\to\man$. This  defines a pivotal structure in the sense of \cite[Def.~A.12]{G14}, namely  a natural 2-isomorphism $\omega: \id\Rightarrow **$, where $*$ is the contravariant 2-functor that assigns each module category to itself,  each functor to its left dual and each module natural transformation $\nu: F\Rightarrow G$ to the induced natural transformation $\nu^*: G^l\Rightarrow F^l$.

\begin{theorem}\label{th:pivot}\cite[Th.~4.5]{G12}, \cite[Th.~5.8]{G14} 
\begin{enumerate}
\item For all pivotal fusion categories $\mac$ the 2-category $\text{Mod}^{\,\theta}(\mac)$ of $\mac$-module categories with $\mac$-module traces, $\mac$-module functors between them  and $\mac$-module natural transformations is pivotal.

\item For all  pivotal fusion categories $\mac,\mad$ the 2-category  $\Bimod^\theta(\mac,\mad)$ of  $(\mac,\mad)$-bimodule categories with bimodule traces, $(\mac,\mad)$-bimodule functors between them and $(\mac,\mad)$-bimodule transformations is pivotal. 
\end{enumerate}
\end{theorem}

In particular, this yields  pivotal structures  on the finite multifusion  categories  $\End_\mac(\mam)$. If the underlying fusion category $\mac$ is spherical, these pivotal structures become spherical.  If $\mam$ is also indecomposable, then  $\End_\mac(\mam)$ is a spherical fusion category. An earlier proof of this was given by Mueger in \cite[Th.~5.16]{Mue}.

\begin{corollary} \label{cor:endmpiv} \cite[Prop.~5.10]{G12}  For any spherical fusion category $\mac$ and any  $\mac$-module category $\mam$ with a $\mac$-module trace,
$\End_\mac(\mam)$ is  spherical. If $\mam$ is indecomposable, $\End_\mac(\mam)$ is a spherical fusion category.
\end{corollary}

For any module category $\mam$ over a fusion category $\mac$ with a $\mac$-module trace, the multifusion category $\End_\mac(\mam)$ acts on $\mam$. This  action functor is given by the evaluation of endofunctors  and natural transformations between them on objects and morphisms of $\mam$. If $\mac$ is pivotal, then $\End_\mac(\mam)$ is pivotal and the module trace of $\mam$ also satisfies the module trace condition for the $\End_\mac(\mam)$-module category structure.

\begin{corollary}\label{lem:macmodule}\cite[Cor.~4.6]{G12} Let $\mac$ be a pivotal fusion category and  $\mam$ be a $\mac$-module category with a $\mac$-module trace $\theta$. Then $\theta$ is also an $\End_\mac(\mam)$-module trace.  
\end{corollary}

More explicitly, the pivotal structure from Theorem \ref{th:pivot} and Corollaries \ref{cor:endmpiv} and \ref{lem:macmodule} is given as follows. If we denote by $F^l: \mam\to\mam$ the left adjoint of a $\mac$-module functor $F:\mam\to\mam$, by $\eta^F: \id_\man\Rightarrow FF^l$ the unit and  by $\epsilon^F:F^lF\Rightarrow \id_\mam$ the counit of this adjunction,
 then the component morphisms
 $\omega^F: F^{ll}\Rightarrow F$ of the pivotal structure are characterised by the condition
\begin{align}
\label{eq:pivcond}
\theta^\mam_{F^l(n)}(\beta\circ\epsilon^F_m\circ F^l(\alpha))=\theta^\man_n (\epsilon^{F^l}_n\circ F^{ll}(\beta)\circ\omega^{F\inv}_m \circ \alpha)
\end{align}
for all $\alpha\in \Hom_\man(n, F(m))$ and $\beta\in \Hom_\mam(m, F^l(n))$. This condition encodes the chain of natural isomorphisms that defines $\omega^F$  in equation (4.14) in \cite{G12}.

Condition \eqref{eq:pivcond} implies an identity that  generalises Corollary \ref{lem:macmodule} to $\mac$-module functors that are not endofunctors. 
It is obtained from the (co)evaluations for the right duals  induced by  $\omega^F: F^{ll}\Rightarrow F$
\begin{align}\label{eq	:etaprime}\eta'^{F}=F^l\omega^F\circ \eta^{F^l}:\id_\mam\Rightarrow F^lF\qquad\qquad \epsilon'^F= \epsilon^{F^l}\circ \omega^{F\inv}F^l:FF^l \Rightarrow \id_\man.
\end{align}

\begin{corollary}\label{cor:functortrace} Let $\mac$ be a pivotal fusion category and $\mam$, $\man$ $\mac$-module categories with $\mac$-module traces. Then any $\mac$-module functor $F:\mam\to\man$ satisfies
\begin{align}\label{eq:traceconds}
&\theta^\mam_m(\epsilon^{F}_m\circ F^l(\alpha)\circ \eta'^F_m)=\theta^\man_{F(m)}(\alpha)
\end{align}
for all $m\in\Ob\mam$, $n\in\Ob\man$ and morphisms $\alpha: F(m)\to F(m)$.
\end{corollary}

\begin{proof}
This follows by setting $n=F(m)$ and $\beta=\eta'^F_m=F^l(\omega^F_m)\circ \eta^{F^l}_m: m\to F^lF(m)$ in \eqref{eq:pivcond}. By cyclicity of $\theta^\mam$, the left-hand side of  \eqref{eq:pivcond} is  equal to the left-hand side of \eqref{eq:traceconds}. The right-hand side of \eqref{eq:pivcond} can  be simplified using the naturality of  $\omega^F$ and the for the defining identities for the unit and counit of an adjunction and yields the right-hand side of \eqref{eq:traceconds}. 
\end{proof}

Note that for an endofunctor $F:\mam\to\mam$, condition \eqref{eq:traceconds} is just the partial trace condition from Definition \ref{def:moduletrace} for the  $\End_\mac(\mam)$-module structure on $\mam$ and  Corollary \ref{cor:functortrace} reduces to Corollary \ref{lem:macmodule}.

\section{Diagrammatic calculus}
\label{sec:diagrams}

\subsection{Diagrams for spherical fusion categories}
We use the usual graphical calculus for monoidal categories, suppressing the tensor unit and  coherence data. Morphisms in diagrams are composed from top to bottom and tensor products  from the right to the left.  

We also use the usual diagrammatic notation for the duals. If the categories are pivotal, we draw oriented lines for the objects, where a line that is directed downwards and labeled $x^*$ is replaced by a line that is directed upwards and labeled $x$. Lines without arrows are assumed to be  oriented downwards. The diagrams for the (co)evaluations from \eqref{eq:coevl} and \eqref{eq:coevr} are
 \begin{align}\label{pic:unitcounitsphere}
&\begin{tikzpicture}[scale=.3]
\draw[line width=.5pt, ] (0,0)  .. controls (0,3) and (-3,3) .. (-3,0) node[sloped, pos=0.15, allow upside down]{\arrowIn} node[anchor=north]{$x$};
\end{tikzpicture}
&
&\begin{tikzpicture}[scale=.3]
\draw[line width=.5pt,] (0,3)  .. controls (0,0) and (-3,0) .. (-3,3) (0,3)node[anchor=south]{$x$} node[sloped, pos=0.9, allow upside down]{\arrowIn} ;
\end{tikzpicture}
&
&\begin{tikzpicture}[scale=.3]
\draw[line width=.5pt,] (0,0) .. controls (0,3) and (3,3) .. (3,0) node[anchor=north]{$x$} node[sloped, pos=0.15, allow upside down]{\arrowIn};
\end{tikzpicture}
&
&\begin{tikzpicture}[scale=.3]
\draw[line width=.5pt, ] (0,3)node[anchor=south]{$x$}   .. controls (0,0) and (3,0) .. (3,3) node[sloped, pos=0.9, allow upside down]{\arrowIn} ;
\end{tikzpicture}\\
&\coev^L_x: e\to x\oo x^* & &\ev^L_x: x^*\oo x\to e & &\coev^R_x: e\to x^*\oo x & &\ev^R_x: x\oo x^*\to e, \nonumber
\end{align}
and the defining conditions \eqref{eq:snake} on the left duals read 
\begin{align}\label{pic:snake}
&\begin{tikzpicture}[scale=.3]
\begin{scope}[shift={(-8,0)}]
\draw[line width=.5pt,stealth-] (0,0)  .. controls (0,3) and (3,3) .. (3,0) ;
\draw[line width=.5pt,stealth-] (3,0) .. controls (3,-3) and (6,-3) .. (6,0) ;
\draw[line width=.5pt] (0,0)--(0,-3);
\draw[line width=.5pt,stealth-] (6,0)--(6,3) node[anchor=south]{$x$};
\end{scope}
\node at (0,0){$=$};
\begin{scope}[shift={(2,0)}]
\draw[line width=.5pt,stealth-] (0,0)--(0,3) node[anchor=south]{$x$};
\draw[line width=.5pt] (0,-3)--(0,0);
\end{scope}
\end{tikzpicture}
&
&\begin{tikzpicture}[scale=.3]
\begin{scope}[shift={(-5,0)}]
\draw[line width=.5pt,stealth-] (0,0)  .. controls (0,3) and (3,3) .. (3,0) ;
\draw[line width=.5pt,stealth-] (-3,0) .. controls (-3,-3) and (0,-3) .. (0,0) ;
\draw[line width=.5pt] (-3,0)--(-3,3) node[anchor=south]{$x$};
\draw[line width=.5pt,stealth-] (3,0)--(3,-3) ;
\end{scope}
\node at (0,0){$=$};
\begin{scope}[shift={(2,0)}]
\draw[line width=.5pt,-stealth] (0,-3)--(0,0);
\draw[line width=.5pt] (0,0)--(0,3) node[anchor=south]{$x$};
\end{scope}
\end{tikzpicture}
\end{align}
Pivotality  is expressed in the diagram
\begin{align}\label{pic:pivotalfus}
\begin{tikzpicture}[scale=.25]
\begin{scope}[shift={(-8,0)}]
\draw[line width=.5pt] (0,0) .. controls (0,2) and (-2,2).. (-2,0);
\draw[line width=.5pt, -stealth] (0,0) .. controls (0,-2) and (2,-2).. (2,0);
\draw[line width=.5pt,stealth-] (-2,0) .. controls (-2,-6) and (4,-6).. (4,0);
\draw[line width=.5pt,-stealth] (2,0) .. controls (2,6) and (-4,6).. (-4,0);
\draw[line width=.5pt,-stealth] (4,4)node[anchor=south]{$x$}--(4,0);
\draw[line width=.5pt] (-4,0)--(-4,-4) node[anchor=north]{$y$};
\draw[fill=black] (0,0) circle(.25) node[anchor=west]{$\,\nu$};
\end{scope}
\node(0,0){$=$};
\draw[line width=.5pt] (4,4) node[anchor=south]{$x$}--(4,-4) node[anchor=north]{$y$};
\draw[fill=black] (4,0) circle(.25) node[anchor=west]{$\,\nu$};
\end{tikzpicture}
\end{align}
for all morphisms $\nu:x\to y$, and  a pivotal category is spherical iff for all $\alpha: x\to x$
\begin{align}\label{eq:tracesph}
\begin{tikzpicture}[scale=.3]
\begin{scope}[shift={(-5,0)}]
\draw[line width=.5pt, ->,>=stealth] (-2,2) node[anchor=east]{$x$}--(-2,-2) node[anchor=east]{$x$};
\draw[line width=.5pt, <-,>=stealth] (2,2) --(2,-2) node[midway,anchor=west]{$x$};
\draw[fill=black] (-2,0)  circle (.2) node[anchor=east]{$\alpha$};
\draw[line width=.5pt, ->,>=stealth] (-2,-2) .. controls (-2,-4) and (2,-4)..(2,-2);
\draw[line width=.5pt, ->,>=stealth] (2,2) .. controls (2,4) and (-2,4)..(-2,2);
\node at (0,-4)[anchor=north]{$\tr^R(\alpha)$};
\end{scope}
\node at (0,0){$=$};
\begin{scope}[shift={(5,0)}]
\draw[line width=.5pt, <-,>=stealth] (-2,2)--(-2,-2) node[midway,anchor=east]{$x$}  ;
\draw[line width=.5pt, ->,>=stealth] (2,2)  node[anchor=west]{$x$} --(2,-2) node[anchor=west]{$x$} ;
\draw[fill=black] (2,0) node[anchor=west]{$\alpha$}  circle (.2);
\draw[line width=.5pt, <-,>=stealth] (-2,-2) .. controls (-2,-4) and (2,-4)..(2,-2);
\draw[line width=.5pt, <-,>=stealth] (2,2) .. controls (2,4) and (-2,4)..(-2,2);
\node at (0,-4)[anchor=north]{$\tr^L(\alpha)$};
\end{scope}
\end{tikzpicture}
\end{align}

\subsection{Diagrams for (bi)module categories, functors and natural transformations}

In the following, we  use  diagrammatic calculus for 2-categories and apply it  to the 2-categories Cat of (small) categories, functors and natural transformations and to the 2-categories
$\Bimod^\theta(\mac,\mad)$ and $\mathrm{Mod}^\theta(\mac)$ from Theorem \ref{th:pivot}. 
The diagrammatic calculus goes back to Joyal and Street \cite{JS}  in the smooth setting, and was also studied by Barrett et al.~\cite{BMS} in the PL setting. We summarise the description from \cite{BMS}.

\subsubsection{Diagrams for 2-categories}

A {\bf 2-category diagram} for a 2-category $\mac$  is a direct generalisation of a diagram for a monoidal category. It is a PL stratification $\emptyset\subset X^0\subset X^1\subset X^2=[0,1]^2$ of  the unit square. Elements of $X^0$ are called {\bf vertices}, connected components of $X^1\setminus X^0$ {\bf lines} and connected components of $X^2\setminus X^1$ {\bf regions} of the diagram.  Lines may end at the top and bottom of the square only and all vertices are  in the interior. A diagram is called  {\bf progressive}, if its horizontal projection is regular, i.e.~its restriction to each line is a PL isomorphism onto the image. It is called {\bf generic}, if no two vertices are at the same height.  In pictures, we often omit the  unit square that borders  the diagram for better legibility.

 Regions in the diagram are labeled by objects of $\mac$, lines by 1-morphisms and vertices by 2-morphisms. Diagrams are read from top to bottom, and can be composed horizontally if the top of one diagram matches the bottom of the other. This corresponds to the vertical composition of 2-morphisms. Horizontal composition is from the right to the left. This is analogous to the usual calculus for monoidal categories, up to the fact that regions of the diagrams are labeled with objects, lines with 1-morphisms and vertices with 2-morphisms. 
 
A  diagram whose right-hand side is labeled by an object $\mathcal A$, whose left-hand side by an object $\mathcal B$ and with   lines at its upper and lower boundary labeled with 1-morphisms $F_1,...,F_n$  and  $G_1,...,G_m$,  from the right to the left, describes a 2-morphism $\nu: F_n\cdots F_1\Rightarrow G_m\cdots G_1$ between 1-morphisms $F_n\cdots F_1:\mathcal A\to\mathcal B$ and $G_m\cdots G_1:\mathcal A\to\mathcal B$.
This 2-morphism is obtained by first  horizontally composing the 2-morphisms at the vertices in the diagram with the 1-morphisms on the lines to the left and right and then composing the resulting 2-morphisms  in their vertical order. 

\begin{example}\label{ex:2catdiagram} The diagram 
\begin{align}\label{pic:diagram}
&\begin{tikzpicture}[scale=.3, baseline=(current bounding box.center)]
\draw[line width=1pt, color=black, style=dashed, ] (0,3)node[anchor=south]{$F$}--(0,-1); 
\draw[line width=1pt, color=black, style=dashed,] (0,-1)--(-2,-3) node[anchor=north]{$G$}; 
\draw[line width=1pt, color=black, style=dashed, ] (0,-1)--(2,-3) node[anchor=north]{$H$};
 \draw[line width=1pt, color=black, style=dashed, ] (4,3) node[anchor=south]{$K$} --(4,1.2) ;
 \draw[line width=1pt, color=black, style=dashed, ] (4,1)--(4,-3) node[anchor=north]{$L$}; 
 \draw[fill=black] (0,-1) circle (.2) node[anchor=east]{$\rho$};
  \draw[fill=black] (4,1) circle (.2) node[anchor=east]{$\sigma$};
  \node at (6,0)[anchor=west, color=blue]{$\mathcal M$};
    \node at (2.5,0)[anchor=east, color=violet]{$\mathcal P$};
      \node at (0,-2.5)[color=cyan ]{$\mathcal Q$};
        \node at (-2,0)[anchor=east, color=red]{$\mathcal N$};
\end{tikzpicture}
\end{align}
for the 2-category Cat with functors  $K,L:\mathcal M\to\mathcal P$, $F:\mathcal P\to\mathcal N$, $H:\mathcal P\to\mathcal Q$ and $G:\mathcal Q\to\mathcal N$ and natural transformations $\sigma: K\Rightarrow L$ and $\rho: F\Rightarrow GH$ describes the natural transformation $(\rho  L)\circ (F\sigma): FK\Rightarrow GHL$ with component morphisms
$\rho_{L(m)}  \circ  F(\sigma_m): FK(m)\to FL(m)\to GHL(m)$. 
\end{example}

It is shown in \cite[Th.~1.12]{JS} in the smooth context, see also \cite[Th.~2.12]{BMS} for the PL setting, that 2-category diagrams represent the same 2-morphisms if they are related by PL isotopies that start at the identity and define a one-parameter family of isomorphisms of diagrams. An isomorphism of diagrams is a PL homeomorphism  that is the identity on the boundary of the diagram  and such that the homeomorphism and its inverse preserve all skeleta of the diagram. 
This invariance is  a consequence of the identity
\begin{align}
\begin{tikzpicture}[scale=.3, baseline=(current bounding box.center)]
\begin{scope}[shift={(-9,0)}]
\draw[line width=1pt, color=black, style=dashed] (0,3)node[anchor=south]{$F$}--(0,-3) node[anchor=north]{$G$};
\draw[line width=1pt, color=black, style=dashed] (3,3)node[anchor=south]{$H$}--(3,-3) node[anchor=north]{$K$};
\draw[fill=black] (0,1) circle (.2) node[anchor=east]{$\nu\;$};
\draw[fill=black] (3,-1) circle (.2) node[anchor=west]{$\;\mu$};
\node  at (6,0) [color=blue]{$\mathcal A$};
\node  at (1.5,0) [color=violet]{$\mathcal B$};
\node  at (-3,0) [color=red]{$\mathcal C$};
\end{scope}
\node at (0,0){$=$};
\begin{scope}[shift={(6,0)}]
\draw[line width=1pt, color=black, style=dashed] (0,3)node[anchor=south]{$F$}--(0,-3) node[anchor=north]{$G$};
\draw[line width=1pt, color=black, style=dashed] (3,3)node[anchor=south]{$H$}--(3,-3) node[anchor=north]{$K$};
\draw[fill=black] (0,-1) circle (.2) node[anchor=east]{$\nu\;$};
\draw[fill=black] (3,1) circle (.2) node[anchor=west]{$\;\mu$};
\node  at (6,0) [color=blue]{$\mathcal A$};
\node  at (1.5,0) [color=violet]{$\mathcal B$};
\node  at (-3,0) [color=red]{$\mathcal C$};
\end{scope}
\node at (0,-6) {$(G\mu)\circ (\nu H)=(\nu K)\circ(F\mu)$.};
\end{tikzpicture}
\end{align}
for all 1-morphisms $F,G:\mathcal B\to\mathcal C$ and $H,K:\mathcal A\to\mathcal B$ and 2-morphisms  $\nu: F\Rightarrow G$ and $\mu: H\Rightarrow K$.
It allows one to  associate 2-morphisms to non-generic diagrams by perturbing them to generic ones. 
In the 2-category Cat this identity expresses the naturality of natural transformations. 

\subsubsection{Diagrams for pivotal 2-categories}

If the 2-category under consideration 
 is a pivotal 2-category in the sense \cite[Def.~A.12]{G14} or, more strictly, a planar 2-category in the sense of \cite[Def.~3.2.1]{BMS}, we use the diagrams for planar 2-categories introduced in \cite[Def.~3.5]{BMS}. The main examples are  $\mathrm{Mod}^\theta(\mac)$ and $\Bimod^\theta(\mac,\mad)$ from Theorem \ref{th:pivot}. 
 
 The  diagrams for pivotal 2-categories are essentially the diagrams  for a pivotal monoidal category.
The only difference is that regions of the diagram are labeled with objects, lines with 1-morphisms and vertices with 2-morphisms. 
As for a pivotal category, the pivotal structure implies that left duals are also right duals and that double duals of a 1-morphism $F$ are 2-isomorphic to $F$.
We thus denote the duals of a 1-morphism $F$ by a dashed line  labeled $F$ with an arrow that points upwards. Lines without arrows are assumed to be oriented downwards.
The (co)evaluations of the duals, which  generalise \eqref{eq:coevl} and \eqref{eq:coevr}, are then given by the diagrams
\begin{align}\label{pic:unitcounit}
&\begin{tikzpicture}[scale=.3]
\draw[line width=1pt, style=dashed] (0,0)  .. controls (0,3) and (-3,3) .. (-3,0) node[sloped, pos=0.15, allow upside down]{\arrowIn} node[anchor=north]{$F$};
\node  at (-1.5,0) [color=blue, anchor=south]{$\mathcal A$};
\node  at (1,0) [color=red, anchor=south]{$\mathcal B$};
\end{tikzpicture}
&
&\begin{tikzpicture}[scale=.3]
\draw[line width=1pt, style=dashed] (0,3)  .. controls (0,0) and (-3,0) .. (-3,3) (0,3)node[anchor=south]{$F$} node[sloped, pos=0.9, allow upside down]{\arrowIn} ;
\node  at (1,3) [color=blue, anchor=north]{$\mathcal A$};
\node  at (-1.5,3) [color=red, anchor=north]{$\mathcal B$};
\end{tikzpicture}
&
&\begin{tikzpicture}[scale=.3]
\draw[line width=1pt, style=dashed] (0,0) .. controls (0,3) and (3,3) .. (3,0) node[anchor=north]{$F$} node[sloped, pos=0.15, allow upside down]{\arrowIn};
\node  at (4,0) [color=blue, anchor=south]{$\mathcal A$};
\node  at (1.5,0) [color=red, anchor=south]{$\mathcal B$};
\end{tikzpicture}
&
&\begin{tikzpicture}[scale=.3]
\draw[line width=1pt, style=dashed] (0,3)node[anchor=south]{$F$}   .. controls (0,0) and (3,0) .. (3,3) node[sloped, pos=0.9, allow upside down]{\arrowIn} ;
\node  at (1.5,3) [color=blue, anchor=north]{$\mathcal A$};
\node  at (4,3) [color=red, anchor=north]{$\mathcal B$};
\end{tikzpicture}
\\
&\eta^F: \id_{\mathcal B}\Rightarrow FF^l &
&\epsilon^F: F^lF\Rightarrow \id_{\mathcal A}
&
&\eta'^F:  \id_{\mathcal A} \Rightarrow F^lF &
&\epsilon'^F:FF^l\Rightarrow \id_{\mathcal B}\nonumber 
\end{align}
that generalise \eqref{pic:unitcounitsphere}. The defining conditions on the left (co)evaluations, which  generalise \eqref{eq:snake}, read
\begin{align}\label{pic:adjunction}
&\begin{tikzpicture}[scale=.3]
\begin{scope}[shift={(-8,0)}]
\draw[line width=1pt, style=dashed,] (0,0)  .. controls (0,3) and (3,3) .. (3,0);
\draw[line width=1pt, style=dashed,-stealth] (6,0) .. controls (6,-3) and (3,-3) .. (3,0) ;
\draw[line width=1pt, style=dashed] (0,0)--(0,-3)  node[sloped, pos=0.5, allow upside down]{\arrowIn};
\draw[line width=1pt, style=dashed] (6,3) node[anchor=south]{$F$}--(6,0) node[sloped, pos=0.5, allow upside down]{\arrowIn};
\node  at (.5,-2) [color=blue, anchor=west]{$\mathcal A$};
\node  at (-.5,-2) [color=red, anchor=east]{$\mathcal B$};
\end{scope}
\node at (0,0){$=$};
\begin{scope}[shift={(2,0)}]
\draw[line width=1pt, style=dashed] (0,3) node[anchor=south]{$F$}--(0,-3) node[sloped, pos=0.5, allow upside down]{\arrowIn};
\node  at (.5,-2) [color=blue, anchor=west]{$\mathcal A$};
\node  at (-.5,-2) [color=red, anchor=east]{$\mathcal B$};
\end{scope}
\end{tikzpicture}
&
&\begin{tikzpicture}[scale=.3]
\begin{scope}[shift={(-5,0)}]
\draw[line width=1pt, style=dashed, stealth-] (0,0)  .. controls (0,3) and (3,3) .. (3,0) ;
\draw[line width=1pt, style=dashed] (-3,0) .. controls (-3,-3) and (0,-3) .. (0,0) ;
\draw[line width=1pt, style=dashed] (-3,0)--(-3,3) node[anchor=south]{$F$} node[sloped, pos=0.5, allow upside down]{\arrowIn};
\draw[line width=1pt, style=dashed] (3,-3)--(3,0) node[sloped, pos=0.5, allow upside down]{\arrowIn};
\node  at (-3.5,2) [color=blue, anchor=east]{$\mathcal A$};
\node  at (-2.5,2) [color=red, anchor=west]{$\mathcal B$};
\end{scope}
\node at (0,0){$=$};
\begin{scope}[shift={(2,0)}]
\draw[line width=1pt, style=dashed] (0,-3)--(0,3) node[anchor=south]{$F$}  node[sloped, pos=0.5, allow upside down]{\arrowIn};
\node  at (-.5,-2) [color=blue, anchor=east]{$\mathcal A$};
\node  at (.5,-2) [color=red, anchor=west]{$\mathcal B$};
\end{scope}
\end{tikzpicture}
\end{align}
 and the pivotality condition that generalises \eqref{pic:pivotalfus} is expressed in the diagrammatic identity
\begin{align}\label{pic:pivotal}
\begin{tikzpicture}[scale=.25]
\begin{scope}[shift={(-8,0)}]
\draw[line width=1pt, dashed] (0,0) .. controls (0,2) and (-2,2).. (-2,0);
\draw[line width=1pt, dashed, -stealth] (0,0) .. controls (0,-2) and (2,-2).. (2,0);
\draw[line width=1pt, dashed,stealth-] (-2,0) .. controls (-2,-6) and (4,-6).. (4,0);
\draw[line width=1pt, dashed,-stealth] (2,0) .. controls (2,6) and (-4,6).. (-4,0);
\draw[line width=1pt, dashed,-stealth] (4,4)node[anchor=south]{$F$}--(4,0);
\draw[line width=1pt, dashed] (-4,0)--(-4,-4) node[anchor=north]{$G$};
\draw[fill=black] (0,0) circle(.25) node[anchor=west]{$\,\nu$};
\node at (5,3)[anchor=west, color=blue]{$\mathcal A$};
\node at (-5,-3)[anchor=east, color=red]{$\mathcal B$};
\end{scope}
\node(0,0){$=$};
\draw[line width=1pt, dashed] (4,4) node[anchor=south]{$F$}--(4,-4) node[anchor=north]{$G$};
\draw[fill=black] (4,0) circle(.25) node[anchor=west]{$\,\nu$};
\node at (3,3)[anchor=east, color=red]{$\mathcal B$};
\node at (5,3)[anchor=west, color=blue]{$\mathcal A$};
\end{tikzpicture}
\end{align}

In the diagrammatic calculus the pivotal structure allows one to drop the requirement that the diagrams are {\em progressive}, i.e.~that the horizontal projection of the diagram is regular. 
Instead, singular points of the projection are labeled with (co)evaluations.  Although we draw them as smooth maxima and minima for diagrammatic convenience, they represent   vertices in a PL diagram. While a diagram with maxima and minima is never progressive itself, it stands for the progressive diagram obtained by replacing maxima and minima with vertices labeled by (co)evaluations, see \cite[Sec.~3.2]{BMS}.

This enlarges the class of diagrams that represent the same 2-morphism. 
It is shown in  \cite[Th.~3.9]{BMS} that  planar 2-category diagrams related by orientation preserving PL homeomorphisms that preserve the unit square  represent the same 2-morphism if their labels are transformed accordingly. This means that the objects labeling the regions and the 1-morphisms labeling lines of the diagrams must coincide, while the 2-morphisms at the vertices are composed with the (co)evaluations of the 1-morphisms at the incident lines. 

 Composing a
2-morphism $\nu: F_{n}\cdots F_1\Rightarrow G_m\cdots G_1$ with the (co)evaluations for $F_n$ and $G_1$ from \eqref{pic:unitcounit} transforms it  into  2-morphisms $\nu': F_{n-1}\cdots F_1 \Rightarrow F_{n}^l G_m\cdots G_1$ 
 and $\nu'': F_{n}\cdots F_1 G^l_1\Rightarrow G_m\cdots G_2$.
\begin{align}
&\begin{tikzpicture}[scale=.25]
\begin{scope}[shift={(-9,0)}]
\draw[line width=1pt, dashed] (0,0)--(-5,5) node[anchor=south]{$F_n$};
\draw[line width=1pt, dashed] (0,0)--(-1,5) node[anchor=south]{$F_{n-1}$};
\node at (1,4){$\ldots$};
\draw[line width=1pt, dashed] (0,0)--(5,5) node[anchor=south]{$F_1$};
\draw[line width=1pt, dashed] (0,0)--(-5,-5) node[anchor=north]{$G_m$};
\node at (-1,-4){$\ldots$};
\draw[line width=1pt, dashed] (0,0)--(1,-5) node[anchor=north]{$G_{2}$};
\draw[line width=1pt, dashed] (0,0)--(5,-5) node[anchor=north]{$G_1$};
\draw[fill=black] (0,0) circle (.25) node[anchor=east]{$\nu\;$};
\node at (3,0) [color=blue, anchor= west]{$\mathcal A$};
\node at (-3.5,4.5) [color=violet, anchor= west]{$\mathcal B$};
\node at (-3,0) [color=red, anchor=east]{$\mathcal C$};
\node at (3.5,-4.5) [color=cyan, anchor=east]{$\mathcal D$};
\end{scope}
\begin{scope}[shift={(11,0)}]
\draw[line width=1pt, dashed] (0,0) .. controls (-1,3) and (-2,2) .. (-4,0);
\draw[line width=1pt, dashed] (-9,-5) node[anchor=north]{$F_n$}--(-4,0) node[sloped, pos=0.5, allow upside down]{\arrowIn};
\draw[line width=1pt, dashed] (0,0)--(-1,5) node[anchor=south]{$F_{n-1}$};
\node at (1,4){$\ldots$};
\draw[line width=1pt, dashed] (0,0)--(5,5) node[anchor=south]{$F_1$};
\draw[line width=1pt, dashed] (0,0)--(-5,-5) node[anchor=north]{$G_m$};
\node at (-1,-4){$\ldots$};
\draw[line width=1pt, dashed] (0,0)--(1,-5) node[anchor=north]{$G_{2}$};
\draw[line width=1pt, dashed] (0,0)--(5,-5) node[anchor=north]{$G_1$};
\draw[fill=black] (0,0) circle (.25) node[anchor=east]{$\nu\;$};
\node at (3,0) [color=blue, anchor= west]{$\mathcal A$};
\node at (-7,0) [color=violet, anchor= west]{$\mathcal B$};
\node at (-6.5,-5) [color=red, anchor=south]{$\;\;\mathcal C$};
\node at (3.5,-4.5) [color=cyan, anchor=east]{$\mathcal D$};
\end{scope}
\draw[line width=1pt, -stealth] (-1,0)--(1,0);
\end{tikzpicture}\nonumber\\
\intertext{}
&\begin{tikzpicture}[scale=.25]
\begin{scope}[shift={(-9,0)}]
\draw[line width=1pt, dashed] (0,0)--(-5,5) node[anchor=south]{$F_n$};
\draw[line width=1pt, dashed] (0,0)--(-1,5) node[anchor=south]{$F_{n-1}$};
\node at (1,4){$\ldots$};
\draw[line width=1pt, dashed] (0,0)--(5,5) node[anchor=south]{$F_1$};
\draw[line width=1pt, dashed] (0,0)--(-5,-5) node[anchor=north]{$G_m$};
\node at (-1,-4){$\ldots$};
\draw[line width=1pt, dashed] (0,0)--(1,-5) node[anchor=north]{$G_{2}$};
\draw[line width=1pt, dashed] (0,0)--(5,-5) node[anchor=north]{$G_1$};
\draw[fill=black] (0,0) circle (.25) node[anchor=east]{$\nu\;$};
\node at (3,0) [color=blue, anchor= west]{$\mathcal A$};
\node at (-3.5,4.5) [color=violet, anchor= west]{$\mathcal B$};
\node at (-3,0) [color=red, anchor=east]{$\mathcal C$};
\node at (3.5,-4.5) [color=cyan, anchor=east]{$\mathcal D$};
\end{scope}
\draw[line width=1pt, -stealth] (-1,0)--(1,0);
\begin{scope}[shift={(11,0)}]
\draw[line width=1pt, dashed] (0,0)--(-5,5) node[anchor=south]{$F_n$};
\draw[line width=1pt, dashed] (0,0)--(-1,5) node[anchor=south]{$F_{n-1}$};
\node at (1,4){$\ldots$};
\draw[line width=1pt, dashed] (0,0)--(5,5) node[anchor=south]{$F_1$};
\draw[line width=1pt, dashed] (0,0)--(-5,-5) node[anchor=north]{$G_m$};
\node at (-1,-4){$\ldots$};
\draw[line width=1pt, dashed] (0,0)--(1,-5) node[anchor=north]{$G_{2}$};
\draw[line width=1pt, dashed] (0,0) .. controls (1,-3) and (2,-2) .. (4,0);
\draw[line width=1pt, dashed] (4,0)--(9,5) node[anchor=south]{$G_1$}  node[sloped, pos=0.5, allow upside down]{\arrowIn};
\draw[fill=black] (0,0) circle (.25) node[anchor=east]{$\nu\;$};
\node at (6,5) [color=blue, anchor= north]{$\mathcal A\;\;$};
\node at (-3.5,4.5) [color=violet, anchor= west]{$\mathcal B$};
\node at (-3,0) [color=red, anchor=east]{$\mathcal C$};
\node at (6,0) [color=cyan, anchor=west]{$\mathcal D$};
\end{scope}
\end{tikzpicture}
\label{fig:updown}
\end{align}

The pivotal 2-category structure ensures that  
 all diagrams that are related by the two moves  in \eqref{fig:updown}  and have the same edge labels  represent the same 2-morphism.

\begin{definition}\label{def:cyclictrans} Let $\mac$ be a 2-category. 
 A {\bf cyclic transformation} of a 2-morphism $\nu: F_{n}\cdots F_1\Rightarrow G_m\cdots G_1$ is a 2-morphism obtained by applying a finite sequence of the moves in \eqref{fig:updown}. 
Two  2-morphisms $\nu$ and $\nu'$ are called {\bf cyclically equivalent} if they are related by a cyclic transformation. 
\end{definition}

In pictures, we sometimes suppress the change in the labeling and the units and counits of the adjunctions and use the label $\nu$ for a cyclic equivalence class of 2-morphisms.

\subsubsection{Mixed diagrams}
\label{subsec:bimod}

The diagrams for pivotal 2-categories in the previous section can be used to describe morphisms in the pivotal 2-categories
$\Bimod^\theta(\mac,\mad)$ and $\mathrm{Mod}^\theta(\mac)$ from Theorem \ref{th:pivot}.  In this case, regions are labeled with (bi)module categories with (bi)module traces, lines with (bi)module functors and vertices with (bi)module natural transformations. The duals are given by adjoint functors and the pivot is the natural isomorphism between a bimodule functor and its double adjoint. 

To describe the interaction of the data in $\Bimod^\theta(\mac,\mad)$ with data from $\mac$ and $\mad$, we generalise pivotal 2-category diagrams. We consider  diagrams with additional lines and vertices labeled with data from $\mac$ and $\mad$. These diagrams may involve  crossings of  lines labeled by $\mac$ and $\mad$ as well as  crossings of such lines with lines labeled by $(\mac,\mad)$-bimodule functors.

\begin{definition} \label{def:mixeddiag}Let $\mac,\mad$ be spherical fusion categories. A {\bf mixed diagram} for $\Bimod^\theta(\mac,\mad)$ is a  diagram  obtained by superimposing a diagram $D$ for $\Bimod^\theta(\mac,\mad)$ with diagrams $D_\mac$ for $\mac$ and $D_\mad$ for $\mad^{rev}$. It is called {\bf generic} if
\begin{compactitem}
\item no vertex of any diagram is on a line or vertex of any of the others,
\item each point is on at most two lines from distinct diagrams,
\item  each point  in the interior that is on lines of two distinct diagrams is a transversal crossing.
\end{compactitem}
\end{definition}

While each diagram for $\Bimod^\theta(\mac,\mad)$ describes a 2-morphism in $\Bimod^\theta(\mac,\mad)$, a mixed diagram for $\Bimod^\theta(\mac,\mad)$ can be viewed as a 2-category diagram for Cat.

For this recall from Remark \ref{rem:macmodule} that each $\mac$-module category $\mam$ defines a monoidal functor $F:\mac \to \End(\mam)$ that assigns to  $c\in \Ob\mac $  the   functor $c\rhd-:\mam\to \mam$ and to a morphism $\alpha: c\to c'$ the  natural transformation $\alpha\rhd-:c\rhd-\Rightarrow c'\rhd-$. Its monoidal structure  is given by the 
coherence isomorphisms  from Definition \ref{def:modulecat}. 

The monoidal functor $F:\mac\to \End(\mam)$  preserves duals.
The adjoint of  $c\rhd -:\mam\to\mam$ is  $c^*\rhd-:\mam\to \mam$  with the (co)units  of the adjunction given by the  (co)evaluations  from \eqref{eq:coevl} and \eqref{eq:coevr}, up to coherence data:
\begin{align*}
&\eta^{c\rhd -}=c_{c,c^*,-}\circ (\coev^L_c\rhd -)
& &\epsilon^{c\rhd -}=(\ev^L_c\rhd -)\circ c^\inv_{c^*,c,-}
\\
&\eta'^{c\rhd -}= c_{c^*,c,-} \circ (\coev^R_c\rhd -)
& &\epsilon'^{c\rhd-}=(\mathrm{ev}^R_{c}\rhd -)\circ c^\inv_{c,c^*,-}.
\end{align*}
The  pentagon relation \eqref{eq:pentagoncdef}   and triangle relation \eqref{eq:trianglec} allow one to suppress bracketings and action of the unit object in expressions involving multiple action functors. They lead to  strictification and coherence theorems for module categories that generalise the ones for monoidal categories, see  \cite[Rem.~7.2.4]{EGNO}.

One can therefore interpret diagrams for $\mac$  as  2-category diagrams for Cat, whose regions are labeled by a $\mac$-module category $\mam$ and whose lines by action functors $c\rhd -$ for $c\in \Ob\mac$. Vertices are labeled by natural transformations that are composites of natural transformations  $\alpha\rhd -$ for morphisms $\alpha$ in $\mac$   with the coherence isomorphisms $c$ and $\gamma$.
The result is a natural transformation between the endofunctors given by the objects at the top and bottom of the diagrams. It is  unique up to coherence data.

Analogous statements hold for $\mad$-right module categories $\mam$. The only difference is that the action is a right action and one has an associated monoidal functor $F:\mad^{rev}\to \End(\mam)$, whose coherence data is given by the natural isomorphisms $d_{-,i,j}: (-\lhd i)\lhd j\Rightarrow -\lhd(i\oo j)$ and  $\delta: -\lhd  e \Rightarrow \id_\mam$  from Definition \ref{def:modulecat}.

In this way, one can  interpret any mixed diagram for $\Bimod^\theta(\mac,\mad)$ \emph{without crossings} as a 2-category diagram  for Cat in which the objects and morphisms from $\mac$ and $\mad$ are replaced by the associated action functors and natural transformations for the bimodule category labeling their region in the diagram. 

Crossings of  lines labeled by   $c\in \Ob \mac$ and $d\in \Ob \mad$ and crossings of such lines with lines labeled by (bi)module functors $F$ also have a direct interpretation. The former correspond to the natural isomorphisms $b_{c,-,d}$   from Definition \ref{def:modulecat} and the latter  to the natural isomorphisms $s^F$ and $t^F$ from Definition \ref{def:modulefunc}. We represent them as over- and undercrossings.

\begin{definition} Let $D'$ be a generic mixed diagram for $\Bimod^\theta(\mac,\mad)$ that superimposes  diagrams $D$, $D_\mac$ and  $D_\mad$. The 2-category diagram associated to $D'$  is the diagram obtained by 
\begin{itemize}
\item replacing  labels $c\in\Ob\mac$, $d\in \Ob\mad$ in a region labeled   $\mam$ with  $c\rhd-:\mam\to\mam$, $-\lhd d:\mam\to \mam$,

\item replacing morphism labels $\alpha$ and $\beta$ from $\mac$ and $\mad$ by  natural transformations  composed of   $\alpha\rhd -$ and $-\lhd\beta$ and the coherence isomorphisms $c,\gamma$ and $d,\delta$ from Definition \ref{def:modulecat},

\item transforming crossings of a line from 
 $D_\mac$ with a line from $D$ or  $D_\mad$  into overcrossings of the line from $D_\mac$ and crossings of a line from $D$ with a line from $D_\mad$ into overcrossings of the line from $D$, 
 
 \item labeling the crossing points  with natural transformations as follows:
\begin{compactitem}
\item crossings of  a line  labeled by $c\in \Ob\mac$ with  a line labeled by $d\in\Ob\mad$ by the natural isomorphism 
$b_{c,-,d}: (c\rhd -)\lhd d\Rightarrow c\rhd(-\lhd d)$ and its inverse  from Definition \ref{def:modulecat} 
\begin{align}\label{pic:bimodule}
&\begin{tikzpicture}[scale=.3, baseline=(current bounding box.center)]
\draw[line width=.5pt, color=gray,] (-2,2) node[anchor=south]{$d$}--(2,-2);
\draw[color=white, fill=white] (0,0) circle (.5);
\draw[line width=.5pt,color=black,] (2,2) node[anchor=south]{$c$}--(-2,-2);
\end{tikzpicture}
&
&\begin{tikzpicture}[scale=.3,baseline=(current bounding box.center)]
\draw[line width=.5pt, color=gray] (2,2) node[anchor=south]{$d$}--(-2,-2);
\draw[color=white, fill=white] (0,0) circle (.5);
\draw[line width=.5pt,color=black,] (-2,2) node[anchor=south]{$c$}--(2,-2);
\end{tikzpicture}
\\
&b_{c,-,d}: (c\rhd -)\lhd d\Rightarrow c\rhd(-\lhd d)
&
&b^\inv_{c,-,d}:  c\rhd(-\lhd d)\Rightarrow (c\rhd -)\lhd d,
 \nonumber
\end{align}

\item crossings of a line labeled by $c\in\Ob\mac$ with  a line from $D$ labeled $F$ by the natural isomorphism 
$s^F: F(c\rhd -)\Rightarrow c\rhd F(-)$ and its inverse from Definition \ref{def:modulefunc} 
\begin{align}\label{pic:modulefunc}
&\begin{tikzpicture}[scale=.3, baseline=(current bounding box.center)]
\draw[line width=1pt, style=dashed, color=black,] (-2,2) node[anchor=south]{$F$}--(2,-2);
\draw[color=white, fill=white] (0,0) circle (.5);
\draw[line width=.5pt, color=black,] (2,2) node[anchor=south]{$c$}--(-2,-2);
\end{tikzpicture}
&
&\begin{tikzpicture}[scale=.3, baseline=(current bounding box.center)]
\draw[line width=1pt, style=dashed, color=black,] (2,2) node[anchor=south]{$F$}--(-2,-2);
\draw[color=white, fill=white] (0,0) circle (.5);
\draw[line width=.5pt, color=black,] (-2,2) node[anchor=south]{$c$}--(2,-2);
\end{tikzpicture}
\\
 &s^F_{c,-}: F(c\rhd -)\Rightarrow c\rhd F(-)
& &s^{F\,\inv}_{c,-}:  c\rhd F(-)\Rightarrow F(c\rhd -),\nonumber
\end{align}
\item crossings of a line  labeled by $d\in\Ob\mad$ with  a line from $D$ labeled $F$ by the natural isomorphism 
$t^F: F( -)\lhd d\Rightarrow F(-\lhd d)$ and its inverse from Definition \ref{def:modulefunc} 
\begin{align}\label{pic:modulefuncinv}
&\begin{tikzpicture}[scale=.3, baseline=(current bounding box.center)]
\draw[line width=.5pt, color=gray,] (-2,2) node[anchor=south]{$d$}--(2,-2);
\draw[color=white, fill=white] (0,0) circle (.5);
\draw[line width=1pt, style=dashed, color=black,] (2,2) node[anchor=south]{$F$}--(-2,-2);
\end{tikzpicture}
&
&\begin{tikzpicture}[scale=.3, baseline=(current bounding box.center)]
\draw[line width=.5pt, color=gray,] (2,2) node[anchor=south]{$d$}--(-2,-2);
\draw[color=white, fill=white] (0,0) circle (.5);
\draw[line width=1pt, style=dashed, color=black,] (-2,2) node[anchor=south]{$F$}--(2,-2);
\end{tikzpicture}\\
 &t^F_{-,d}: F(-)\lhd d\Rightarrow F(-\lhd d) 
& &t^{F\;\inv}_{-,d}:  F(-\lhd d)\Rightarrow F(-)\lhd d.\nonumber
\end{align}
\end{compactitem}
\end{itemize}
\end{definition}
In this way, every generic mixed diagram defines a 
2-category diagram for Cat that represents a unique natural transformation between the composites of the functors at the  top and  bottom of the diagram. 

\begin{definition} The {\bf evaluation} of a generic mixed diagram for $\Bimod^\theta(\mac,\mad)$ is the natural transformation represented by the associated 2-category diagram.
\end{definition}

The properties of the coherence isomorphisms associated to crossing points allow one to define the evaluation of non-generic  mixed diagrams. 
The invertibility of the isomorphisms $b$, $s^F$ and $t^F$ translates into diagrammatic identities that are an analogue of the Reidemeister 2 move for a braid diagram
\begin{align}\label{pic:rm2}
&\begin{tikzpicture}[scale=.15]
\begin{scope}[shift={(-4,0)}]
\draw[line width=.5pt, color=gray] (-2,2) node[anchor=south]{$d$}--(2,-2)--(-2,-6);
\draw[color=white, fill=white] (0,0) circle (.5);
\draw[color=white, fill=white] (0,-4) circle (.5);
\draw[line width=.5pt, color=black] (2,2) node[anchor=south]{$c$}--(-2,-2)--(2,-6);
\end{scope}
\node at (0,-2){$=$};
\begin{scope}[shift={(4,0)}]
\draw[line width=.5pt, color=gray] (-2,2) node[anchor=south]{$d$}--(-2,-6);
\draw[line width=.5pt, color=black] (2,2) node[anchor=south]{$c$}--(2,-6);
\end{scope}
\end{tikzpicture}\quad
%%%%%
\begin{tikzpicture}[scale=.15]
\begin{scope}[shift={(-4,0)}]
\draw[line width=.5pt, color=gray] (2,2) node[anchor=south]{$d$}--(-2,-2)--(2,-6);
\draw[color=white, fill=white] (0,0) circle (.5);
\draw[color=white, fill=white] (0,-4) circle (.5);
\draw[line width=.5pt, color=black] (-2,2) node[anchor=south]{$c$}--(2,-2)--(-2,-6);
\end{scope}
\node at (0,-2){$=$};
\begin{scope}[shift={(4,0)}]
\draw[line width=.5pt, color=gray] (2,2) node[anchor=south]{$d$}--(2,-6);
\draw[line width=.5pt, color=black] (-2,2) node[anchor=south]{$c$}--(-2,-6);
\end{scope}
\end{tikzpicture}\quad
%%%%
\begin{tikzpicture}[scale=.15]
\begin{scope}[shift={(-4,0)}]
\draw[line width=1pt, dashed] (-2,2) node[anchor=south]{$F$}--(2,-2)--(-2,-6);
\draw[color=white, fill=white] (0,0) circle (.5);
\draw[color=white, fill=white] (0,-4) circle (.5);
\draw[line width=.5pt, color=black] (2,2) node[anchor=south]{$c$}--(-2,-2)--(2,-6);
\end{scope}
\node at (0,-2){$=$};
\begin{scope}[shift={(4,0)}]
\draw[line width=1pt, dashed] (-2,2) node[anchor=south]{$F$}--(-2,-6);
\draw[line width=.5pt, color=black] (2,2) node[anchor=south]{$c$}--(2,-6);
\end{scope}
\end{tikzpicture}\quad
%%%%%
\begin{tikzpicture}[scale=.15]
\begin{scope}[shift={(-4,0)}]
\draw[line width=1pt, dashed] (2,2) node[anchor=south]{$F$}--(-2,-2)--(2,-6);
\draw[color=white, fill=white] (0,0) circle (.5);
\draw[color=white, fill=white] (0,-4) circle (.5);
\draw[line width=.5pt, color=black] (-2,2) node[anchor=south]{$c$}--(2,-2)--(-2,-6);
\end{scope}
\node at (0,-2){$=$};
\begin{scope}[shift={(4,0)}]
\draw[line width=1pt, dashed] (2,2) node[anchor=south]{$F$}--(2,-6);
\draw[line width=.5pt, color=black] (-2,2) node[anchor=south]{$c$}--(-2,-6);
\end{scope}
\end{tikzpicture}
%%%%%%
\quad \begin{tikzpicture}[scale=.15]
\begin{scope}[shift={(-4,0)}]
\draw[line width=.5pt, color=gray] (-2,2) node[anchor=south]{$d$}--(2,-2)--(-2,-6);
\draw[color=white, fill=white] (0,0) circle (.5);
\draw[color=white, fill=white] (0,-4) circle (.5);
\draw[line width=1pt, dashed] (2,2) node[anchor=south]{$F$}--(-2,-2)--(2,-6);
\end{scope}
\node at (0,-2){$=$};
\begin{scope}[shift={(4,0)}]
\draw[line width=.5pt, color=gray] (-2,2) node[anchor=south]{$d$}--(-2,-6);
\draw[line width=1pt, dashed] (2,2) node[anchor=south]{$F$}--(2,-6);
\end{scope}
\end{tikzpicture}
\quad
%%%%%%
\begin{tikzpicture}[scale=.15]
\begin{scope}[shift={(-4,0)}]
\draw[line width=.5pt, color=gray] (2,2) node[anchor=south]{$d$}--(-2,-2)--(2,-6);
\draw[color=white, fill=white] (0,0) circle (.5);
\draw[color=white, fill=white] (0,-4) circle (.5);
\draw[line width=1pt, dashed] (-2,2) node[anchor=south]{$F$}--(2,-2)--(-2,-6);
\end{scope}
\node at (0,-2){$=$};
\begin{scope}[shift={(4,0)}]
\draw[line width=.5pt, color=gray] (2,2) node[anchor=south]{$d$}--(2,-6);
\draw[line width=1pt, dashed] (-2,2) node[anchor=south]{$F$}--(-2,-6);
\end{scope}
\end{tikzpicture}
\end{align}
Similarly, the hexagon relation \eqref{eq:hexa} for a $(\mac,\mad)$-bimodule functor  translates into a diagrammatic identity that is an analogue of the Reidemeister 3 move for a braid diagram
 \begin{align}\label{pic:hexagon}
 &\begin{tikzpicture}[scale=.3, baseline=(current bounding box.center)]
 \begin{scope}[shift={(-4,0)}]
  \draw[line width=.5pt,  color=gray,] (0,2) node[anchor=south]{$d$}--(-2,-4);
    \draw[color=white, fill=white] (-1,0) circle (.7);
     \draw[line width=1pt, style=dashed, color=black] (-2,2) node[anchor=south]{$F$}--(2,-4);
         \draw[color=white, fill=white] (.9,-1.7) circle (.7);
    \draw[line width=.5pt,  color=black] (2,2) node[anchor=south]{$c$}--(0,-4);
    \end{scope}
    \node at (0,-1)[anchor=west]{$=$};
     \begin{scope}[shift={(6,0)}]
  \draw[line width=.5pt,  color=gray,] (0,2) node[anchor=south]{$d$}--(2,-1);
      \draw[line width=.5pt,  color=gray] (2,-1) --(-2,-4);
         \draw[color=white, fill=white] (1,1) circle (.5);
         \draw[color=white, fill=white] (-1,-3) circle (.5);
              \draw[color=white, fill=white] (.5,-2) circle (.7);
   \draw[line width=1pt, style=dashed, color=black] (-2,2) node[anchor=south]{$F$}--(2,-4);
     \draw[color=white, fill=white] (-1,0) circle (.7);
    \draw[line width=.5pt,  color=black,] (2,2) node[anchor=south]{$c$}--(-2,-1);
        \draw[line width=.5pt,  color=black] (-2,-1) --(0,-4);
    \end{scope}
 \end{tikzpicture}
\end{align}

The naturality conditions from Definitions \ref{def:modulecat} and \ref{def:modulefunc} and the conditions on a bimodule natural transformation from  \eqref{eq:cnat} and \eqref{eq:dnat} 
 allow one to slide vertices under or over crossings. 
The naturality of the isomorphisms $b_{c,d,m}: (c\rhd m)\lhd d\to c\rhd (m\lhd d)$ in $c$ and $d$ is expressed as
\begin{align}\label{pic:natb}
&\begin{tikzpicture}[scale=.3, baseline=(current bounding box.center)]
\begin{scope}[shift={(-4,0)}]
\draw[line width=.5pt, color=gray,] (-2,2) node[anchor=south]{$d$}--(2,-2);
\draw[color=white, fill=white] (0,0) circle (.5);
\draw[line width=.5pt,color=black,] (2,2) node[anchor=south]{$c$}--(-2,-2) node[anchor=north]{$c'$};
\draw[fill=black] (1,1) circle (.15) node[anchor=west]{$\;\alpha$};
\end{scope}
\node at (0,0){$=$};
\begin{scope}[shift={(4,0)}]
\draw[line width=.5pt, color=gray,] (-2,2) node[anchor=south]{$d$}--(2,-2);
\draw[color=white, fill=white] (0,0) circle (.5);
\draw[line width=.5pt,color=black,] (2,2) node[anchor=south]{$c$}--(-2,-2) node[anchor=north]{$c'$};
\draw[fill=black] (-1,-1) circle (.15) node[anchor=east]{$\alpha\;$};
\end{scope}
\end{tikzpicture}
&
&\begin{tikzpicture}[scale=.3, baseline=(current bounding box.center)]
\begin{scope}[shift={(-4,0)}]
\draw[line width=.5pt, color=gray,] (-2,2) node[anchor=south]{$d$}--(2,-2) node[anchor=north]{$d'$};
\draw[color=white, fill=white] (0,0) circle (.5);
\draw[line width=.5pt,color=black,] (2,2) node[anchor=south]{$c$}--(-2,-2) ;
\draw[fill=gray, color=gray] (-1,1) circle (.15) node[anchor=east]{$\beta\;$};
\end{scope}
\node at (0,0){$=$};
\begin{scope}[shift={(4,0)}]
\draw[line width=.5pt, color=gray,] (-2,2) node[anchor=south]{$d$}--(2,-2) node[anchor=north]{$d'$};
\draw[color=white, fill=white] (0,0) circle (.5);
\draw[line width=.5pt,color=black,] (2,2) node[anchor=south]{$c$}--(-2,-2) ;
\draw[fill=gray, color=gray] (1,-1) circle (.15) node[anchor=west]{$\;\beta$};
\end{scope}
\end{tikzpicture}
\end{align}
The naturality of  $s^F_{c,m}: F(c\rhd m)\to c\rhd F(m)$ and  $t^F_{m,d}: F(m)\lhd d\to  F(m\lhd d)$ in $c$ and $d$ implies
\begin{align}\label{pic:natfuncconst}
&\begin{tikzpicture}[scale=.3, baseline=(current bounding box.center)]
\begin{scope}[shift={(-4,0)}]
\draw[line width=1pt, style=dashed] (-2,2) node[anchor=south]{$F$}--(2,-2);
\draw[color=white, fill=white] (0,0) circle (.5);
\draw[line width=.5pt,color=black,] (2,2) node[anchor=south]{$c$}--(-2,-2) node[anchor=north]{$c'$};
\draw[fill=black] (1,1) circle (.15) node[anchor=west]{$\;\alpha$};
\end{scope}
\node at (0,0){$=$};
\begin{scope}[shift={(4,0)}]
\draw[line width=1pt, style=dashed] (-2,2) node[anchor=south]{$F$}--(2,-2);
\draw[color=white, fill=white] (0,0) circle (.5);
\draw[line width=.5pt,color=black,] (2,2) node[anchor=south]{$c$}--(-2,-2) node[anchor=north]{$c'$};
\draw[fill=black] (-1,-1) circle (.15) node[anchor=east]{$\alpha\;$};
\end{scope}
\end{tikzpicture}
&
&\begin{tikzpicture}[scale=.3, baseline=(current bounding box.center)]
\begin{scope}[shift={(-4,0)}]
\draw[line width=.5pt, color=gray,] (-2,2) node[anchor=south]{$d$}--(2,-2) node[anchor=north]{$d'$};
\draw[color=white, fill=white] (0,0) circle (.5);
\draw[line width=1pt,style=dashed] (2,2) node[anchor=south]{$F$}--(-2,-2) ;
\draw[fill=gray, color=gray] (-1,1) circle (.15) node[anchor=east]{$\beta\;$};
\end{scope}
\node at (0,0){$=$};
\begin{scope}[shift={(4,0)}]
\draw[line width=.5pt, color=gray,] (-2,2) node[anchor=south]{$d$}--(2,-2) node[anchor=north]{$d'$};
\draw[color=white, fill=white] (0,0) circle (.5);
\draw[line width=1pt,style=dashed] (2,2) node[anchor=south]{$F$}--(-2,-2) ;
\draw[fill=gray, color=gray] (1,-1) circle (.15) node[anchor=west]{$\;\beta$};
\end{scope}
\end{tikzpicture}
\end{align}
The defining conditions \eqref{eq:cnat} and \eqref{eq:dnat} on a $\mac$-module and $\mad$-right module natural transformation $\nu: F\Rightarrow G$ 
give the diagrammatic identities
 \begin{align}\label{pic:bimodulenat}
 &\begin{tikzpicture}[scale=.3, baseline=(current bounding box.center)]
\begin{scope}[shift={(-4,0)}]
\draw[line width=1pt, style=dashed] (-2,2) node[anchor=south]{$F$}--(2,-2) node[anchor=north]{$G$};
\draw[color=white, fill=white] (0,0) circle (.5);
\draw[line width=.5pt,] (2,2) node[anchor=south]{$c$}--(-2,-2) ;
\draw[fill=black, color=black] (-1,1) circle (.2) node[anchor=east]{$\nu\;$};
\end{scope}
\node at (0,0){$=$};
\begin{scope}[shift={(4,0)}]
\draw[line width=1pt, style=dashed] (-2,2) node[anchor=south]{$F$}--(2,-2) node[anchor=north]{$G$};
\draw[color=white, fill=white] (0,0) circle (.5);
\draw[line width=.5pt] (2,2) node[anchor=south]{$c$}--(-2,-2) ;
\draw[fill=black, color=black] (1,-1) circle (.2) node[anchor=west]{$\;\nu$};
\end{scope}
\end{tikzpicture}
&
&\begin{tikzpicture}[scale=.3, baseline=(current bounding box.center)]
\begin{scope}[shift={(-4,0)}]
\draw[line width=.5pt, color=gray] (-2,2) node[anchor=south]{$d$}--(2,-2);
\draw[color=white, fill=white] (0,0) circle (.5);
\draw[line width=1pt, style=dashed] (2,2) node[anchor=south]{$F$}--(-2,-2) node[anchor=north]{$G$};
\draw[fill=black] (1,1) circle (.2) node[anchor=west]{$\;\nu$};
\end{scope}
\node at (0,0){$=$};
\begin{scope}[shift={(4,0)}]
\draw[line width=.5pt, color=gray] (-2,2) node[anchor=south]{$d$}--(2,-2);
\draw[color=white, fill=white] (0,0) circle (.5);
\draw[line width=1pt, style=dashed] (2,2) node[anchor=south]{$F$}--(-2,-2) node[anchor=north]{$G$};
\draw[fill=black] (-1,-1) circle (.2) node[anchor=east]{$\nu\;$};
\end{scope}
\end{tikzpicture}
 \end{align}
In particular, identities \eqref{pic:natb} and \eqref{pic:natfuncconst} can be applied to the (co)evaluations for the duals in $\mac$ and $\mad$. This  allows one to slide maxima and minima of lines labeled with objects in $\mac$ and $\mad$ over or under lines labeled with bimodule functors. Similarly,  the (co)evaluations for  1-morphisms in $\Bimod^\theta(\mac,\mad)$ are
 $(\mac,\mad)$-bimodule natural transformations and hence satisfy identity \eqref{pic:bimodulenat}. This allows one to slide maxima and minima of lines labeled with bimodule functors over or under lines labeled with objects  from $\mac$ or $\mad$. 
 
Note  that the diagrammatic identities involving only lines  and vertices labeled by $\mac$ and $\mad$ are special cases of the ones involving data from $\Bimod^\theta(\mac,\mad)$. 
The diagrams \eqref{pic:bimodule} for the natural isomorphism $b_{c,-,d}: (c\rhd -)\lhd d\Rightarrow c\rhd(-\lhd d)$ are just the diagrams \eqref{pic:modulefunc} and \eqref{pic:modulefuncinv} for the $\mac$-module functor $-\lhd d:\mam\to\mam$ and the $\mad$-right module functor $c\rhd -:\mam\to\mam$. The identities \eqref{pic:natb}
are a special case of \eqref{pic:natfuncconst} and  \eqref{pic:bimodulenat}. 

The diagrammatic identities \eqref{pic:rm2} to \eqref{pic:bimodulenat} allow one to define the evaluation of non-generic  mixed diagrams.  
Any point of a mixed diagram that lies on more than two lines from distinct diagrams is on three lines, one from each diagram $D, D_\mac$ and $D_\mad$. By slightly displacing the lines  one can  transform such a triple point into multiple crossings of two lines. The diagrammatic identities for the crossings and their inverses and the identity \eqref{pic:hexagon} ensures that all diagrams obtained in this way have the same evaluation. 

Similarly, if there is a point on two lines from different diagrams that is not a transversal crossing, then by slightly displacing the lines one can either create a double crossing or ensure that the lines do not meet. Identities \eqref{pic:rm2}  ensure that all resulting diagrams have the same evaluations. 
Any  vertex  on  two lines from different diagrams can be displaced slightly from the crossing. 
Identities \eqref{pic:natb}, \eqref{pic:natfuncconst} and \eqref{pic:bimodulenat} ensure that all ways of doing so yield  diagrams with the same evaluation.

\begin{definition} The evaluation of a  mixed diagram for $\Bimod^\theta(\mac,\mad)$ is defined as the evaluation of any generic  mixed diagram obtained by slightly displacing the lines and vertices of the diagram.
\end{definition}

\subsection{Polygon diagrams}

\subsubsection{Action on category diagrams}

We also use the diagrammatic calculus to describe   evaluation functors $\mathrm{ev}: \Fun(\mam,\man)\times \mam\to \man$ that send a pair $(F,m)$ of a functor $F:\mam\to\man$ and an object $m\in\Ob\mam$ to the object $F(m)$ and a pair $(\nu,\alpha)$ of a natural transformation $\nu: F\Rightarrow G$ and a morphism $\alpha:m\to m'$ to the morphism
$\nu_{m'}\circ F(\alpha)=G(\alpha)\circ \nu_m: F(m)\to G(m')$. 

For this, we describe objects  in $\mam$ by a thick  vertical line and morphisms between them by  vertices on this line. 
Composition of morphisms is from the top to the bottom, and identity morphisms are omitted from the diagrams. Thus, a line whose top end is labeled by an object $m$ in $\mam$ and whose bottom end by an object $m'$ in $\mam$ describes a morphism $\alpha:m\to m'$ in $\mam$. 

The diagram for the evaluation $\mathrm{ev}(\nu,\alpha): F(m)\to G(m')$ of a natural transformation $\nu: F\Rightarrow G$ in $\Fun(\mam,\man)$ on 
a  morphism $\alpha:m\to m'$ in $\mam$  is obtained by placing the  diagram for $\alpha$  
to the right of  a 2-category diagram  for  $\nu$, such that no vertex of the diagram for $\nu$ is on the same height as a vertex on the line for $\alpha$.
 
The morphism $\mathrm{ev}(\nu,\alpha)$  is obtained by horizontally projecting the diagram for $\nu$ on the line for $\alpha$.  All functors labeling dashed lines to the left of a line segment or vertex on the latter  are applied to the associated objects and morphisms in $\mam$.  Natural transformations labeling the vertices  in the diagram for $\nu$ are composed with all functors labeling lines to their left and right  and then evaluated  on the object of $\mam$ to their right. The resulting morphisms in $\man$ are then composed in order of their height.

\begin{example} The diagram 
\begin{align}\label{pic:evaluation}
&\begin{tikzpicture}[scale=.4, baseline=(current bounding box.center)]
\draw[line width=1pt, color=black, style=dashed, ] (0,3)node[anchor=south]{$F$}--(0,-1);
\draw[line width=1pt, color=black, style=dashed,] (0,-1)--(-2,-3) node[anchor=north]{$G$}; 
\draw[line width=1pt, color=black, style=dashed, ] (0,-1)--(2,-3) node[anchor=north]{$H$};
 \draw[line width=1pt, color=black, style=dashed, ] (4,3) node[anchor=south]{$K$} --(4,1.2) ;
 \draw[line width=1pt, color=black, style=dashed, ] (4,1)--(4,-3) node[anchor=north]{$L$}; 
 \draw[fill=black] (0,-1) circle (.2) node[anchor=east]{$\rho$};
  \draw[fill=black] (4,1) circle (.2) node[anchor=east]{$\sigma$};
 \node at (6,0)[anchor=west, color=blue]{$\mathcal M$};
\node at (2.5,0)[anchor=east, color=violet]{$\mathcal P$};
\node at (0,-2.5)[color=cyan ]{$\mathcal Q$};
 \node at (-2,0)[anchor=east, color=red]{$\mathcal N$};
\draw[line width=1.5pt, color=blue] (8,3)node[anchor=south]{$m$}--(8,-3) node[anchor=north]{$m''$} node[midway, anchor=west]{$m'$};
\draw[color=blue, fill=blue] (8,2) circle (.2) node[anchor=west]{$\phi$};
\draw[color=blue, fill=blue] (8,-2) circle (.2) node[anchor=west]{$\psi$};        
\end{tikzpicture}
\end{align}
describes the evaluation $\mathrm{ev}(\nu, \alpha): FK(m)\to GHL(m'')$ of  $\nu=(\rho L)\circ (F\sigma): FK\Rightarrow GHL$ from diagram \eqref{pic:diagram} on the morphism $\alpha=\psi\circ\phi: m\to m'\to m''$ in $\mam$. 
It is  the  composite 
\begin{align*}
&\mathrm{ev}(\nu,\alpha): FK(m)\xrightarrow{FK(\phi)}FK(m')\xrightarrow{F(\sigma_{m'})} FL(m')\xrightarrow{\rho_{L(m')}} GHL(m')\xrightarrow{GHL(\psi)} GHL(m'').
\end{align*}
\end{example}

Note that the evaluation depends only on the natural transformation represented by the diagram on the left  and on the morphism represented by the line on the right, but not on the relative position of vertices in the two diagrams. This follows from  naturality and corresponds to the  identity
\begin{align}\label{pic:naturality}
&\begin{tikzpicture}[scale=.3]
\begin{scope}[shift={(-7,0)}]
\draw[line width=1pt, color=black, style=dashed] (0,3)node[anchor=south]{$K$}--(0,-3) node[anchor=north]{$L$};
\draw[line width=1.5pt, color=blue] (3,3)node[anchor=south]{$m$}--(3,-3) node[anchor=north]{$m'$};
\draw[fill=black] (0,1) circle (.2) node[anchor=east]{$\sigma\;$};
\draw[fill=blue, color=blue] (3,-1) circle (.2) node[anchor=west]{$\;\phi$};
\node  at (1.5,0) [color=blue]{$\mathcal M$};
\node  at (-3,0) [color=violet]{$\mathcal P$};
\end{scope}
\node at (0,0){$=$};
\begin{scope}[shift={(6,0)}]
\draw[line width=1pt, color=black, style=dashed] (0,3)node[anchor=south]{$K$}--(0,-3) node[anchor=north]{$L$};
\draw[line width=1.5pt, color=blue] (3,3)node[anchor=south]{$m$}--(3,-3) node[anchor=north]{$m'$};
\draw[fill=black] (0,-1) circle (.2) node[anchor=east]{$\sigma\;$};
\draw[fill=blue, color=blue] (3,1) circle (.2) node[anchor=west]{$\;\phi$};
\node  at (1.5,0) [color=blue]{$\mathcal M$};
\node  at (-3,0) [color=violet]{$\mathcal P$};
\end{scope}
\node at (0,-7){$L(\phi)\circ \sigma_m=\sigma_{m'}\circ K(\phi)$};
\end{tikzpicture}
\end{align}
which allows one to move vertices in the diagram for  $\Fun(\mam,\man)$ above or below vertices on the line for $\mam$. One can thus drop the requirement that  the  former are at different heights from the latter. Any diagram that violates this condition can be deformed into one that satisfies it by slightly  displacing the vertices on the line on the right. As all ways of doing so yield the same morphism, its evaluation is defined by these deformations.

We also use diagrammatic calculus to describe the composites of morphisms of the form $\mathrm{ev}(\nu,\alpha)$  with general morphisms in $\man$. The resulting diagrams consist of a multicoloured thick line on the right and a diagram for $\Fun(\mam,\man)$ to the left whose lines may end either at the top or bottom of the diagram for $\Fun(\mam,\man)$ or on vertices on the thick line on the right. For instance, the following diagrams describe morphisms $\tau: J(m)\to n$ and $\omega: n\to J(m)$ in $\man$ for a functor $J:\mam\to\man$
\begin{align}\label{pic:mixedmorphism}
&\begin{tikzpicture}[scale=.3]
\draw[line width=1pt, style=dashed] (-4,4)node[anchor=south]{$J$}--(0,0);
\draw[line width=1.5pt, color=blue] (0,4) node[anchor=south]{$m$}--(0,0);
\draw[line width=1.5pt, color=red] (0,0)--(0,-4) node[anchor=north]{$n$};
\draw[color=red, fill=red] (0,0) circle (.2) node[anchor=west]{$\;\tau$};
\node at (-1.5,3.5) [color=blue]{$\mathcal M$};
\node at (-3,1) [color=red]{$\mathcal N$};
\end{tikzpicture}
&
&\begin{tikzpicture}[scale=.3]
\draw[line width=1pt, style=dashed] (-4,-4)node[anchor=north]{$J$}--(0,0);
\draw[line width=1.5pt, color=blue] (0,-4) node[anchor=north]{$m$}--(0,0);
\draw[line width=1.5pt, color=red] (0,0)--(0,4) node[anchor=south]{$n$};
\draw[color=red, fill=red] (0,0) circle (.2) node[anchor=west]{$\;\omega$};
\node at (-1.5,-3.5) [color=blue]{$\mathcal M$};
\node at (-3,-1) [color=red]{$\mathcal N$};
\end{tikzpicture}
\end{align}
In this diagram, the dashed line may also stand for a composite  functor and  may be replaced by several dashed lines that  end in a common vertex of the line on the right. 
For $\mac$-module or $\mad$-right module categories $\mam$ and $\man$, any dashed line  may be replaced by a thin black or grey line labeled with an object of $\mac$ or $\mad$ that stands for an endofunctor $c\rhd -$ or $-\lhd d$. 

\subsubsection{Bordered diagrams}

Using the diagrammatic calculus from the previous sections, we  define  a generalisation of 2-category diagrams for Cat, in which  boundary segments and endpoints of lines at the boundary are also labeled with categorical data. We also allow multiple lines  to end in a common point on a boundary.

\begin{definition}\label{def:bordereddiag} Let  $D$ be a 2-category diagram for Cat. 
 A {\bf bordered diagram} $D'$ for $D$ is a diagram obtained by  labeling  boundary segments and endpoints of lines at the boundary of $D$ as follows:
\begin{compactitem}
\item each boundary segment  with an object of the adjacent category in $D$,
\item an endpoint of lines labeled with functors $F_1,...,F_n$, from right to left,  between a
 segment labeled by  $a\in \Ob\mathcal A$ on the right and a segment labeled by $b\in\Ob\mathcal B$ on the left  with a morphism $\alpha: b\to F_n\cdots F_1(a)$ if it is on the top and with a morphism $\alpha: F_n\cdots F_1(a)\to b$ if it is on the  bottom of  $D$.
\end{compactitem}
\begin{align*}
&\begin{tikzpicture}[scale=.3]
\draw[line width=1.5pt, color=red] (-3,0)node[anchor=east]{$b$}--(0,0);
\draw[line width=1.5pt, color=blue] (3,0) node[anchor=west]{$a$}--(0,0);
\draw[line width=1pt, dashed] (0,0)--(-2,-4) node[anchor=north]{$F_3$};
\draw[line width=1pt, dashed] (0,0)--(0,-4) node[anchor=north]{$F_2$};
\draw[line width=1pt, dashed] (0,0)--(2,-4) node[anchor=north]{$F_1$};
\node at (3,-2)[color=blue, anchor=west]{$\mathcal A$};
\node at (-3,-2)[color=red, anchor=east]{$\mathcal B$};
\draw[color=red, fill=red] (0,0) circle(.2) node[anchor=south]{$\alpha$};
\node at (0,3){$\alpha: b\to F_3F_2F_1(a)$};
\end{tikzpicture}
&
&\begin{tikzpicture}[scale=.3]
\draw[line width=1.5pt, color=red] (-3,0)node[anchor=east]{$b$}--(0,0);
\draw[line width=1.5pt, color=blue] (3,0) node[anchor=west]{$a$}--(0,0);
\draw[line width=1pt, dashed] (0,0)--(-2,4) node[anchor=south]{$F_3$};
\draw[line width=1pt, dashed] (0,0)--(0,4) node[anchor=south]{$F_2$};
\draw[line width=1pt, dashed] (0,0)--(2,4) node[anchor=south]{$F_1$};
\node at (3,2)[color=blue, anchor=west]{$\mathcal A$};
\node at (-3,2)[color=red, anchor=east]{$\mathcal B$};
\draw[color=red, fill=red] (0,0) circle(.2) node[anchor=north]{$\alpha$};
\node at (0,-3){$\alpha: F_3F_2F_1(a)\to b$};
\end{tikzpicture}
\end{align*}
\end{definition}

\begin{definition} Let $D$ be a 2-category diagram that represents a natural transformation $\nu: F\Rightarrow G$ and $D'$ a bordered diagram for $D$ labeled with objects $n\in \Ob\man$ on the left and $m\in\Ob\mam$ on the right.

The {\bf evaluation} of $D'$  is the morphism 
$\ev(D')=b\circ \nu_m\circ t:n\to n$, where $t: n\to F(m)$ and $b: G(m)\to n$ are morphisms  assigned to the top and bottom of $D'$ as follows:
\begin{compactitem}
\item take the image of each morphism at the top or bottom of $D$ under the composite of all functors labeling endpoints of lines to its left,
\item compose the resulting morphisms from the left to the right at the top and from the right to the left at the bottom of $D$. 
\end{compactitem}
\end{definition}

The evaluation of   a bordered diagram  $D'$ for a 2-category diagram  $D$ is represented diagrammatically by cutting its boundary square on the left-hand side and straightening its boundary to a vertical line. The naturality condition \eqref{pic:naturality} allows one to move vertices of $D$ freely with respect to boundary vertices.

\begin{example}
A bordered diagram $D'$ for  the  diagram  $D$ in \eqref{pic:diagram} and its  evaluation are given by
\begin{align}\label{pic:circeval}
&\begin{tikzpicture}[scale=.5, baseline=(current bounding box.center)]
\draw[line width=1pt] (7,3)--(-3,3)--(-3,-3)--(7,-3)--(7,3);
\draw[line width=1.5pt, color=blue] (4,3)--(7,3)--(7,-3)--(4,-3);
\node at (7,0)[color=blue, anchor=west]{$m$};
\draw[line width=1.5pt, color=violet] (0,3)--(4,3);
\node at (2,3)[color=violet, anchor=south]{$p$};
\draw[line width=1.5pt, color=violet] (2,-3)--(4,-3);
\node at (3,-3)[color=violet, anchor=north]{$p'$};
\draw[line width=1.5pt, color=red] (0,3)--(-3,3)--(-3,-3)--(-2,-3);
\node at (-3,0)[color=red, anchor=east]{$n$};
\draw[line width=1.5pt, color=cyan ] (-2,-3)--(2,-3);
\node at (0,-3)[color=cyan , anchor=north]{$q$};
\draw[line width=1pt, color=black, style=dashed, ] (0,3)--(0,-1);
\node at (0,3) [color=black, anchor=north west]{$F$};
\draw[line width=1pt, color=black, style=dashed,] (0,-1)--(-2,-3);
\node at (-2,-2.5)[color=black, anchor=south]{$G$}; 
\draw[line width=1pt, color=black, style=dashed, ] (0,-1)--(2,-3);
\node at (2,-2.5) [color=black, anchor=south]{$H$};
 \draw[line width=1pt, color=black, style=dashed, ] (4,3) --(4,1.2) ;
 \node at (4,3) [color=black, anchor=north west] {$K$};
 \draw[line width=1pt, color=black, style=dashed, ] (4,1)--(4,-3);
 \node at (4,-3)  [color=black, anchor=south west]{$L$}; 
 \draw[fill=black] (0,-1) circle (.15) node[anchor=east]{$\rho$};
  \draw[fill=black] (4,1) circle (.15) node[anchor=east]{$\sigma$};
  \node at (5,0)[anchor=west, color=blue]{$\mathcal M$};
  \node at (2.5,0)[anchor=east, color=violet]{$\mathcal P$};
  \node at (0,-2.5)[color=cyan ]{$\mathcal Q$};
  \node at (-1,0)[anchor=east, color=red]{$\mathcal N$};
  \draw[fill=red, color=red] (0,3) circle (.15);
  \node at (0,3.2)[color=red,anchor=south]{$\alpha$};
  \draw[fill=violet, color=violet] (4,3) circle (.15) ;
    \node at (4,3.2)[color=violet, anchor=south]{$\beta$};
        \draw[fill=violet, color=violet] (4,-3) circle (.15);
          \node at (4,-3.2)[color=violet, anchor=north]{$\gamma$};
         \draw[fill=cyan , color=cyan ] (2,-3) circle (.15) ;
           \node at (2,-3.2)[color=cyan ,anchor=north]{$\delta$};
          \draw[fill=red, color=red] (-2,-3) circle (.15) ;
            \node at (-2,-3.2)[anchor=north, color=red]{$\epsilon$};
\end{tikzpicture}
&
&\begin{tikzpicture}[scale=.5, baseline=(current bounding box.center)]
\draw[line width=1.5pt, color=red] (0,5)--(0,3) node[midway, anchor=west]{$n$};
\draw[line width=1pt, style=dashed] (0,3) node[anchor=west, color=red]{$\alpha$}--(-5,-1.5)node[anchor=east]{$\rho$} node[midway, anchor=south east]{$F$};
\draw[line width=1.5pt, color=violet] (0,3)--(0,1) node[midway, anchor=west]{$p$};
\draw[line width=1.5pt, color=blue] (0,1)--(0,-1) node[midway, anchor=west]{$m$};
\draw[line width=1pt, style=dashed] (0,1)node[anchor=west, color=violet]{$\beta$}--(-1,0) node[anchor=east]{$\sigma$} node[midway, anchor=south]{$K\;$};
\draw[line width=1pt, style=dashed] (0,-3)node[anchor=west, color=cyan ]{$\delta$}--(-5,-1.5) node[midway, anchor=south]{$H$};
\draw[line width=1pt, style=dashed] (0,-1)node[anchor=west, color=violet]{$\gamma$} --(-1,0) node[midway, anchor=north]{$L\;$};
\draw[line width=1.5pt, color=violet] (0,-1)--(0,-3) node[midway, anchor=west]{$p'$};
\draw[line width=1pt, style=dashed] (0,-5) node[anchor=west, color=red]{$\epsilon$}--(-5,-1.5) node[midway, anchor=north east]{$G$};
\draw[line width=1.5pt, color=cyan ] (0,-3)--(0,-5) node[midway, anchor=west]{$q$};
\draw[line width=1.5pt, color=red] (0,-5)--(0,-7) node[midway, anchor=west]{$n$};
\draw[fill=black] (-5,-1.5) circle (.15);
\draw[fill=black] (-1,0) circle (.15);
\draw[fill=red, color=red] (0,3) circle (.15);
\draw[fill=violet, color=violet] (0,1) circle (.15);
\draw[fill=violet, color=violet] (0,-1) circle (.15);
\draw[fill=cyan, color=cyan] (0,-3) circle (.15);
\draw[fill=red, color=red] (0,-5) circle (.15);
\end{tikzpicture}
\end{align}
with morphisms  
$\alpha: n\to F(p)$, $\beta: p\to K(m)$, $\gamma:L(m)\to p'$, $\delta: H(p')\to q$ and $\epsilon: G(q)\to n$. 

 Its  evaluation is  the morphism $b\circ \nu_m\circ t: n\to n$, where $\nu=(\rho L)\circ (F\sigma):FK\Rightarrow GHL$ is the natural transformation represented by $D$, $t=F(\beta)\circ \alpha: n \to  FK(m)$ and $b=\epsilon\circ (G\delta)\circ (GH\gamma): GHL(m)\to n$. By \eqref{pic:naturality}, this coincides with the morphism represented by the diagram on the right, which reads
\begin{align*}
n\xrightarrow{\alpha} F(p)\xrightarrow{F(\beta)} FK(m )\xrightarrow{F(\sigma_m)} FL(m)\xrightarrow{\gamma} F(p')\xrightarrow{\rho_{p'}} GH(p') \xrightarrow{G(\delta)} G(q)\xrightarrow{\epsilon} n.
\end{align*}
\end{example}

\subsubsection{Polygon diagrams}
\label{subsec:polygondiags}

By definition, the evaluation of a bordered diagram is an endomorphism of the object labeling the left-hand side of the diagram.  
For bordered diagrams that involve data from the pivotal 2-category  $\Bimod^\theta(\mac,\mad)$ one can take the bimodule trace of this morphism to obtain a complex number.

\begin{definition} Let $D$ be a bordered diagram labeled with data from $\Bimod^\theta(\mac,\mad)$. The {\bf cyclic evaluation} of $D$ is the trace of its evaluation. 
\end{definition}

The cyclic evaluation of a bordered diagram has better invariance properties than the evaluation of a bordered diagram. This is most easily seen by using a diagrammatic calculus for bimodule traces, which is essentially the calculus from \cite{G12}. 
If a finite semisimple category $\mam$ is equipped with a trace $\theta$ as in Definition \ref{def:moduletrace}, we denote the associated morphisms $\theta_m: \End_\mam(m)\to\C$ by horizontal lines at the upper and lower ends of the lines representing  endomorphisms of $m$. 
The cyclicity condition  from Definition \ref{def:moduletrace} then takes the form
\begin{align}\label{pic:cyclic trace}
\begin{tikzpicture}[scale=.3]
\begin{scope}[shift={(-4,0)}]
\draw[line width=1.5pt, color=blue] (0,3) --(0,-3);
\draw[line width=1.5pt, color=blue] (-1,3)--(1,3) node[anchor=west]{$m$};
\draw[line width=1.5pt, color=blue] (-1,-3)--(1,-3)node[anchor=west]{$m$};
\draw[fill=blue, color=blue] (0,1.5) circle (.2) node[anchor=east]{$\alpha\;$};
\draw[fill=blue, color=blue] (0,-1.5) circle (.2) node[anchor=east]{$\beta\;$};
\node (0,0)[anchor=west, color=blue]{$m'$};
\end{scope}
\node at (0,0){$=$};
\begin{scope}[shift={(4,0)}]
\draw[line width=1.5pt, color=blue] (0,3) --(0,-3);
\draw[line width=1.5pt, color=blue] (-1,3)--(1,3) node[anchor=west]{$m'$};
\draw[line width=1.5pt, color=blue] (-1,-3)--(1,-3)node[anchor=west]{$m'$};
\draw[fill=blue, color=blue] (0,1.5) circle (.2) node[anchor=east]{$\beta\;$};
\draw[fill=blue, color=blue] (0,-1.5) circle (.2) node[anchor=east]{$\alpha\;$};
\node (0,0)[anchor=west, color=blue]{$m$};
\end{scope}
\end{tikzpicture}
\end{align}
and the compatibility conditions between the trace and the $\mac$-module or $\mad$-right module structure read
\begin{align}\label{pic:cmoduletrace}
&\begin{tikzpicture}[scale=.35]
\begin{scope}[shift={(-3,0)}]
\draw[line width=1.5pt, color=blue] (0,3) node[anchor=north west]{$m$} --(0,-3) node[anchor=south west]{$m$};
\draw[line width=1.5pt, color=blue] (-2.5,3)--(1,3);
\draw[line width=1.5pt, color=blue] (-2.5,-3)--(1,-3);
\draw[line width=.5pt, color=black] (-2,3)node[anchor=north east]{$c$} --(0,0);
\draw[line width=.5pt, color=black] (-2,-3) node[anchor=south east]{$c$} --(0,0);
\draw[fill=blue, color=blue] (0,0) circle (.2) node[anchor=west]{$\;\alpha$};
\end{scope}
\node at (0,0) {$=$};
\begin{scope}[shift={(5,0)}]
\draw[line width=1.5pt, color=blue] (0,3) node[anchor=north west]{$m$} --(0,-3) node[anchor=south west]{$m$};
\draw[line width=1.5pt, color=blue] (-1,3)--(1,3);
\draw[line width=1.5pt, color=blue] (-1,-3)--(1,-3);
\draw[line width=.5pt] plot [smooth, tension=0.6] coordinates 
      {(0,0)(-1,2)(-2,1.5) (-2.2,0)(-2,-1.5) (-1,-2)(0,0)};
\draw[line width=.5pt, ->,>=stealth]  (-2.2,-.1)--(-2.2,.1) node[anchor=east]{$c$};    
\draw[fill=blue, color=blue] (0,0) circle (.2) node[anchor=west]{$\;\alpha$};
\end{scope}
\end{tikzpicture}
&
&\begin{tikzpicture}[scale=.35]
\begin{scope}[shift={(-3,0)}]
\draw[line width=1.5pt, color=blue] (0,3) node[anchor=north west]{$m$} --(0,-3) node[anchor=south west]{$m$};
\draw[line width=1.5pt, color=blue] (-2.5,3)--(1,3);
\draw[line width=1.5pt, color=blue] (-2.5,-3)--(1,-3);
\draw[line width=.5pt, color=gray] (-2,3)node[anchor=north east]{$d$} --(0,0);
\draw[line width=.5pt, color=gray] (-2,-3) node[anchor=south east]{$d$} --(0,0);
\draw[fill=blue, color=blue] (0,0) circle (.2) node[anchor=west]{$\;\alpha$};
\end{scope}
\node at (0,0) {$=$};
\begin{scope}[shift={(5,0)}]
\draw[line width=1.5pt, color=blue] (0,3) node[anchor=north west]{$m$} --(0,-3) node[anchor=south west]{$m$};
\draw[line width=1.5pt, color=blue] (-1,3)--(1,3);
\draw[line width=1.5pt, color=blue] (-1,-3)--(1,-3);
\draw[line width=.5pt, color=gray] plot [smooth, tension=0.6] coordinates 
      {(0,0)(-1,2)(-2,1.5) (-2.2,0)(-2,-1.5) (-1,-2)(0,0)};
\draw[line width=.5pt, ->,>=stealth, color=gray]  (-2.2,-.1)--(-2.2,.1) node[anchor=east]{$d$};    
\draw[fill=blue, color=blue] (0,0) circle (.2) node[anchor=west]{$\;\alpha$};
\end{scope}
\end{tikzpicture}
\end{align}
For a spherical fusion category $\mathcal C$ as a module or right module category  over itself, we use both, the diagrams from \eqref{eq:tracesph}
and the diagrams
for module traces. In the latter we draw thin black or grey lines instead of thick coloured lines for edges labeled by $\mathcal C$.

By  Theorem \ref{th:pivot} and Corollaries \ref{cor:endmpiv} to \ref{cor:functortrace},  (bi)module traces  induce the pivotal 2-category structures on the 2-categories $\mathrm{Mod}^\theta(\mac)$ and $\Bimod^\theta(\mac,\mad)$.  The 
defining condition \eqref{eq:pivcond} on the
component morphisms $\omega^F: F^{ll}\Rightarrow F$ of the pivots  and identity  \eqref{eq:traceconds} are  expressed diagrammatically as

\begin{align}\label{eq:functortrpivot}
&\begin{tikzpicture}[scale=.35]
\begin{scope}[shift={(-3,0)}]
\draw[line width=1.5pt, color=blue] (1,5)--(-3,5);
\draw[line width=1.5pt, color=blue] (1,-5)--(-3,-5);
\draw[line width=1.5pt, color=red] (0,5)--(0,2) node[midway, anchor=west]{$n$};
\draw[line width=1.5pt, color=red] (0,-5)--(0,-2) node[midway, anchor=west]{$n$};
\draw[line width=1.5pt, color=blue] (0,-2)--(0,2) node[midway,anchor=west] {$m$};
\draw[line width=1pt, style=dashed,-stealth] (0,2)..controls (-.5,0) and (-2,0).. (-2,2) node[anchor=east]{$F$};
\draw[fill=red, color=red] (0,2) circle (.2) node[anchor=west]{$\;\alpha$};
\draw[fill=blue, color=blue] (0,-2) circle (.2) node[anchor=west]{$\;\beta$};
\draw[line width=1pt, style=dashed] (-2,2)--(-2,5) node[midway, anchor=west, color=red]{$\man$} node[midway, anchor=east, color=blue]{$\mam$}; 
\draw[line width=1pt, style=dashed] (-2,-5) node[anchor=south east]{$F$} --(0,-2)  node[sloped, pos=0.5, allow upside down]{\arrowIn}; 
\end{scope}
\node at (0,0) {$\stackrel{\eqref{eq:pivcond}}=$};
\begin{scope}[shift={(5,0)}]
\draw[line width=1.5pt, color=red] (1,5)--(-3,5);
\draw[line width=1.5pt, color=red] (1,-5)--(-3,-5);
\draw[line width=1.5pt, color=red] (0,5)--(0,2) node[midway, anchor=west]{$n$};
\draw[line width=1.5pt, color=red] (0,-5)--(0,-2) node[midway, anchor=west]{$n$};
\draw[line width=1.5pt, color=blue] (0,-2)--(0,2) node[midway,anchor=west] {$m$};
\draw[fill=red, color=red] (0,2) circle (.2) node[anchor=west]{$\;\alpha$};
\draw[fill=blue, color=blue] (0,-2) circle (.2) node[anchor=west]{$\;\beta$};
\draw[line width=1pt, style=dashed,-stealth] plot [smooth, tension=0.6] coordinates 
      {(0,2)(-1,0)(-2,-2.3) (-1.3,-3.5)(-.2,-2.2)};
      \node at (-1,0)[anchor=east]{$F$};
      \node at (-1.2,-2)[color=blue]{$\mam$};
         \node at (-3,-2)[color=red]{$\man$};
\end{scope}
\end{tikzpicture}\qquad
&
&\qquad \begin{tikzpicture}[scale=.5]
\begin{scope}[shift={(-2.5,0)}]
\draw[line width=1.5pt, color=blue] (0,3) node[anchor=north west]{$m$} --(0,-3) node[anchor=south west]{$m$};
\draw[line width=1.5pt, color=red] (-2.5,3)--(1,3);
\draw[line width=1.5pt, color=red] (-2.5,-3)--(1,-3);
\draw[line width=1pt, style=dashed] (-2,3)node[anchor=north east]{$F$} --(0,0) ;
\draw[line width=1pt, style=dashed] (-2,-3) node[anchor=south east]{$F$} --(0,0);
\draw[fill=blue, color=red] (0,0) circle (.15) node[anchor=west]{$\;\alpha$};
\node at (-1.4,2.8)[anchor=north west, color=blue]{$\mam$};
\node at (-1,0)[anchor=east, color=red]{$\man$};
\end{scope}
\node at (0,0) {$\stackrel{\eqref{eq:traceconds}}=$};
\begin{scope}[shift={(4.5,0)}]
\draw[line width=1.5pt, color=blue] (0,3) node[anchor=north west]{$m$} --(0,-3) node[anchor=south west]{$m$};
\draw[line width=1.5pt, color=blue] (-1,3)--(1,3);
\draw[line width=1.5pt, color=blue] (-1,-3)--(1,-3);
\draw[line width=1pt, style=dashed] plot [smooth, tension=0.6] coordinates 
      {(0,0)(-1,2)(-2,1.5) (-2.2,0)(-2,-1.5) (-1,-2)(0,0)};
\draw[line width=1pt, ->,>=stealth]  (-2.2,-.1)--(-2.2,.1) node[anchor=east]{$F$};    
\draw[fill=red, color=red] (0,0) circle (.15) node[anchor=west]{$\;\alpha$};
\node at (-.5,0)[anchor=east, color=red]{$\man$};
\node at (-3,3)[anchor=north west, color=blue]{$\mam$};
\end{scope}
\end{tikzpicture}
\end{align}

By definition, the evaluation of a bordered diagram $D'$ for a 2-category diagram $D$ depends on the natural transformation represented by $D$ and on the data labeling the boundary of $D'$. In contrast, the \emph{cyclic evaluation} of a bordered diagram for $\Bimod^\theta(\mac,\mad)$ depends only on the cyclic equivalence class of the morphism represented by $D$, if the data at the boundary of $D'$ is transformed accordingly.

Cyclic transformations of diagrams induce cyclic transformations of bordered diagrams. These  modify the underlying natural transformation, move the object labels at the boundary and modify the morphisms at the boundary by moving them  and composing them with the (co)evaluations from \eqref{pic:unitcounit}:
\begin{align}
\label{eq:movemorph1}
&\begin{tikzpicture}[scale=.35, baseline=(current bounding box.center)]
\draw[line width=1.5pt, color=red] (0,3)--(-5,3)--(-5,-3)--(0,-3);
\draw[line width=1.5pt, color=blue] (0,3)--(3,3) node[anchor=west]{$m$};
\draw[line width=1.5pt, color=violet] (0,-3)--(3,-3) node[anchor=west]{$n$};
\draw[line width=1pt, style=dashed] (0,3)--(0,0) node[sloped, pos=0.5, allow upside down]{\arrowIn} node[pos=.5, anchor=east]{$F$};
\draw[line width=1pt, style=dashed] (0,0)--(0,-3) node[sloped, pos=0.5, allow upside down]{\arrowIn} node[pos=.5, anchor=east]{$H$};
\draw[line width=1pt, style=dashed] (0,0)--(1,1);
\draw[line width=1pt, style=dashed] (0,0)--(2,1);
\draw[line width=1pt, style=dashed] (0,0)--(2,-1);
\draw[line width=1pt, style=dashed] (0,0)--(1,-1);
\draw[color=red, fill=red](0,3) circle (.17) node[anchor=south]{$\alpha$};
\draw[color=red, fill=red](0,-3) circle (.17) node[anchor=north]{$\beta$};
\draw[color=black, fill=black](0,0) circle (.17) node[anchor=east]{$\nu$};
\node at (-5,0)[anchor=east, color=red]{$p$};
\end{tikzpicture}
&
&
\longrightarrow
&
&
\begin{tikzpicture}[scale=.35, baseline=(current bounding box.center)]
\draw[line width=1.5pt, color=red] (-3,-3)--(0,-3);
\draw[line width=1.5pt, color=blue] (-3,-3)--(-5,-3)--(-5,3)--(3,3) node[anchor=west]{$m$};
\draw[line width=1.5pt, color=violet] (0,-3)--(3,-3) node[anchor=west]{$n$};
\draw[line width=1pt, style=dashed]  (0,0) .. controls (0,2) and (-3,2)..(-3,0);
\draw[line width=1pt, style=dashed] (-3,-3)--(-3,0) node[sloped, pos=0.5, allow upside down]{\arrowIn} node[pos=.5, anchor=west]{$F$};
\draw[line width=1pt, style=dashed] (0,0)--(0,-3) node[sloped, pos=0.5, allow upside down]{\arrowIn} node[pos=.5, anchor=east]{$H$};
\draw[line width=1pt, style=dashed] (0,0)--(1,1);
\draw[line width=1pt, style=dashed] (0,0)--(2,1);
\draw[line width=1pt, style=dashed] (0,0)--(2,-1);
\draw[line width=1pt, style=dashed] (0,0)--(1,-1);
\draw[color=blue, fill=blue](-3,-3) circle (.17) node[anchor=north]{$\alpha^d$};
\draw[color=red, fill=red](0,-3) circle (.17) node[anchor=north]{$\beta$};
\draw[color=black, fill=black](0,0) circle (.17) node[anchor=east]{$\nu$};
\node at (-5,0)[anchor=east, color=blue]{$m$};
\node at (-1.5,-3)[anchor=north, color=red]{$p$};
\end{tikzpicture}
\\
&\begin{tikzpicture}[scale=.3]
\draw[line width=1.5pt, color=red] (-3,0) node[anchor=south east]{$p$}--(0,0);
\draw[line width=1.5pt, color=blue] (3,0) node[anchor=south west]{$m$}--(0,0);
\draw[line width=1pt, style=dashed] (0,0) --(0,-4)  node[sloped, pos=0.5, allow upside down]{\arrowIn} node[midway, anchor=west]{$F$}; 
\draw[color=red, fill=red] (0,0) circle (.2) node[anchor=south]{$\alpha$};
%%%%%
\begin{scope}[shift={(9,0)}]
\draw[line width=1.5pt, color=red] (0,3) node[anchor=south]{$p$}--(0,0);
\draw[line width=1.5pt, color=blue] (0,-3) node[anchor=north]{$m$}--(0,0);
\draw[line width=1pt, style=dashed] (0,0) --(-3,-3)  node[anchor=north]{$F$}; 
\draw[color=red, fill=red] (0,0) circle (.2) node[anchor=west]{$\alpha$};
\end{scope}
\end{tikzpicture}
&
&
&
&\begin{tikzpicture}[scale=.3]
\draw[line width=1.5pt, color=red] (-3,0) node[anchor=north east]{$p$}--(0,0);
\draw[line width=1.5pt, color=blue] (3,0) node[anchor=north west]{$m$}--(0,0);
\draw[line width=1pt, style=dashed] (0,0) --(0,4)  node[sloped, pos=0.5, allow upside down]{\arrowIn} node[midway, anchor=west]{$F$}; 
\draw[color=blue, fill=blue] (0,0) circle (.2) node[anchor=north]{$\alpha^d$};
%%%%%
\begin{scope}[shift={(10,0)}]
\draw[line width=1.5pt, color=red] (0,3) node[anchor=south]{$p$}--(0,0);
\draw[line width=1.5pt, color=blue] (0,-3) node[anchor=north]{$m$}--(0,0);
\draw[line width=1pt, style=dashed] (-3,0)  --(-3,3) node[anchor=south]{$F$}  node[sloped, pos=0.5, allow upside down]{\arrowIn} ;
\draw[line width=1pt, style=dashed] (-3,0) ..controls (-2.5,-2) and (-.5,-2).. (0,0);
\draw[color=red, fill=red] (0,0) circle (.2) node[anchor=west]{$\alpha$};
\end{scope}
\end{tikzpicture}\nonumber\\
&\alpha: p\to F(m)
&
&
&
&\alpha^d=\epsilon^F_m\circ F^l(\alpha): F^l(p)\to m\nonumber
\end{align}

\begin{align}
\label{eq:movemorph2}
&\begin{tikzpicture}[scale=.35, baseline=(current bounding box.center)]
\draw[line width=1.5pt, color=violet] (0,3)--(5,3)--(5,-3)--(0,-3);
\draw[line width=1.5pt, color=blue] (0,3)--(-3,3) node[anchor=east]{$m$};
\draw[line width=1.5pt, color=red] (0,-3)--(-3,-3) node[anchor=east]{$p$};
\draw[line width=1pt, style=dashed] (0,3)--(0,0) node[sloped, pos=0.5, allow upside down]{\arrowIn} node[pos=.5, anchor=west]{$K$};
\draw[line width=1pt, style=dashed] (0,0)--(0,-3) node[sloped, pos=0.5, allow upside down]{\arrowIn} node[pos=.5, anchor=west]{$G$};
\draw[line width=1pt, style=dashed] (0,0)--(-1,1);
\draw[line width=1pt, style=dashed] (0,0)--(-2,1);
\draw[line width=1pt, style=dashed] (0,0)--(-2,-1);
\draw[line width=1pt, style=dashed] (0,0)--(-1,-1);
\draw[color=blue, fill=blue](0,3) circle (.17) node[anchor=south]{$\delta$};
\draw[color=red, fill=red](0,-3) circle (.17) node[anchor=north]{$\gamma$};
\draw[color=black, fill=black](0,0) circle (.17) node[anchor=west]{$\nu$};
\node at (5,0)[anchor=west, color=violet]{$n$};
\end{tikzpicture}
&
&\longrightarrow
&
&
\begin{tikzpicture}[scale=.35, baseline=(current bounding box.center)]
\draw[line width=1.5pt, color=violet] (0,3)--(3,3);
\draw[line width=1.5pt, color=blue] (0,3)--(-3,3) node[anchor=east]{$m$};
\draw[line width=1.5pt, color=red] (-3,-3) node[anchor=east]{$p$}--(5,-3)--(5,3)--(3,3);
\draw[line width=1pt, style=dashed] (0,3)--(0,0) node[sloped, pos=0.5, allow upside down]{\arrowIn} node[pos=.5, anchor=west]{$K$};
\draw[line width=1pt, style=dashed] (0,0) .. controls (0,-2) and (3,-2).. (3,0);
\draw[line width=1pt, style=dashed] (3,0)--(3,3) node[sloped, pos=0.5, allow upside down]{\arrowIn} node[pos=.5, anchor=west]{$G$};
\draw[line width=1pt, style=dashed] (0,0)--(-1,1);
\draw[line width=1pt, style=dashed] (0,0)--(-2,1);
\draw[line width=1pt, style=dashed] (0,0)--(-2,-1);
\draw[line width=1pt, style=dashed] (0,0)--(-1,-1);
\draw[color=blue, fill=blue](0,3) circle (.17) node[anchor=south]{$\delta$};
\draw[color=violet, fill=violet](3,3) circle (.17) node[anchor=south]{$\gamma^u$};
\draw[color=black, fill=black](0,0) circle (.17) node[anchor=west]{$\nu$};
\node at (5,0)[anchor=west, color=red]{$p$};
\node at (1.5,3)[anchor=south, color=violet]{$n$};
\end{tikzpicture}
\\
&\begin{tikzpicture}[scale=.3]
\draw[line width=1.5pt, color=red] (-3,0) node[anchor=south east]{$p$}--(0,0);
\draw[line width=1.5pt, color=violet] (3,0) node[anchor=south west]{$n$}--(0,0);
\draw[line width=1pt, style=dashed] (0,4) --(0,0)  node[sloped, pos=0.5, allow upside down]{\arrowIn} node[midway, anchor=west]{$G$}; 
\draw[color=red, fill=red] (0,0) circle (.2) node[anchor=north]{$\gamma$};
%%%%%
\begin{scope}[shift={(9,0)}]
\draw[line width=1.5pt, color=violet] (0,3) node[anchor=south]{$n$}--(0,0);
\draw[line width=1.5pt, color=red] (0,-3) node[anchor=north]{$p$}--(0,0);
\draw[line width=1pt, style=dashed] (0,0) --(-3,3)  node[anchor=south]{$G$}; 
\draw[color=red, fill=red] (0,0) circle (.2) node[anchor=west]{$\gamma$};
\end{scope}
\end{tikzpicture}
&
&
&
&\begin{tikzpicture}[scale=.3]
\draw[line width=1.5pt, color=violet] (-3,0) node[anchor=south east]{$n$}--(0,0);
\draw[line width=1.5pt, color=red] (3,0) node[anchor=south west]{$p$}--(0,0);
\draw[line width=1pt, style=dashed] (0,-4) --(0,0)  node[sloped, pos=0.5, allow upside down]{\arrowIn} node[midway, anchor=west]{$G$}; 
\draw[color=violet, fill=violet] (0,0) circle (.2) node[anchor=south]{$\gamma^u$};
%%%%%
\begin{scope}[shift={(10,0)}]
\draw[line width=1.5pt, color=violet] (0,3) node[anchor=south]{$n$}--(0,0);
\draw[line width=1.5pt, color=red] (0,-3) node[anchor=north]{$p$}--(0,0);
\draw[line width=1pt, style=dashed] (-3,-3) node[anchor=north]{$G$} --(-3,0)  node[sloped, pos=0.5, allow upside down]{\arrowIn} ;
\draw[line width=1pt, style=dashed] (-3,0) ..controls (-2.5,2) and (-.5,2).. (0,0);
\draw[color=red, fill=red] (0,0) circle (.2) node[anchor=west]{$\gamma$};
\end{scope}
\end{tikzpicture}\nonumber\\
&\gamma: G(n)\to p
&
&
&
&\gamma^u= G^l(\gamma)\circ \eta'^G_n: n\to G^l(p).\nonumber\\
\nonumber
\end{align}
Identity \eqref{pic:adjunction}  ensures that  moving an endpoint from the upper (lower) to the lower (upper) side of a diagram and then up (down) again  yields its  original morphism label: $(\alpha^u)^d=\alpha$ and $(\gamma^u)^d=\gamma$. The pivotal structure depicted in  \eqref{pic:pivotal} ensures that this does not change the bimodule natural transformation labeling a vertex. Transforming the boundary morphisms in this way extends  cyclic transformations of diagrams to cyclic transformations of bordered diagrams, as shown in Figure \ref{fig:cyclicperm}.

 \begin{figure}
\begin{tikzpicture}[scale=.4, baseline=(current bounding box.center)]
\draw[line width=1.5pt, color=red] (0,3)--(-2.75,3) node[anchor=south]{$p$}--(-5.5,3) --(-5.5,-5)--(-3,-5) ;
\draw[line width=1.5pt, color=blue] (0,3)--(2.75, 3) node[anchor=south]{$m$}  --(5.5,3)--(5.5,-5)--(3,-5) ;
\draw[line width=1.5pt, color=violet] (-3,-5)--(3,-5) node[midway,anchor=north]{$n$};
\draw[line width=1pt, style=dashed] (0,3) node[anchor=north west]{$F$}--(0,-1) node[sloped, pos=0.5, allow upside down]{\arrowIn};
\draw[line width=1pt, style=dashed] (0,-1)--(-3,-5)  node[midway,anchor=south east]{$G$} node[sloped, pos=0.5, allow upside down]{\arrowIn};
\draw[line width=1pt, style=dashed] (0,-1) --(3,-5) node[midway,anchor=south west]{$H$} node[sloped, pos=0.5, allow upside down]{\arrowIn};
\draw[fill=black] (0,-1) circle (.15) node[anchor=east]{$\nu\;$};
\draw[fill=red, color=red] (0,3) circle (.15);
\node at (0,3.4)[anchor=south, color=red]{$\alpha$};
\draw[fill=violet, color=violet] (3,-5) circle (.15);
\node at (3,-5.2)[anchor=north, color=violet]{$\beta$};
\draw[fill=red, color=red] (-3,-5) circle (.15);
\node at (-3,-5.2)[anchor=north, color=red]{$\gamma$};
\node at (-3,1)[anchor=west, color=red]{$\mathcal P$};
\node at (3,1)[anchor=east, color=blue]{$\mathcal M$};
\node at (0,-3)[anchor=north, color=violet]{$\mathcal N$};
\end{tikzpicture}\quad
%%%%%%%%
\begin{tikzpicture}[scale=.4, baseline=(current bounding box.center)]
\node at (0,5.4) [anchor=south, color=blue]{$m$};
\draw[line width=1.5pt, color=blue] (0,5)--(-6,5)--(-6,1)  --(-6,-3) --(-4,-3) ;
\draw[line width=1.5pt, color=blue] (0,5)--(5,5)--(5,1)  --(5,-3)--(3.5,-3)--(2,-3) ;
\draw[line width=1.5pt, color=red] (-2,-3)--(-4,-3) node[midway,anchor=north]{$p$};
\draw[line width=1.5pt, color=violet] (-2,-3)--(2,-3) node[midway,anchor=north]{$n$};
\draw[line width=1pt, style=dashed]  (0,0).. controls (0,3) and (-4,3).. (-4,0);
\draw[line width=1pt, style=dashed] (-4,-3)--(-4,0) node[anchor=east]{$F$} node[sloped, pos=0.9, allow upside down]{\arrowIn};
\draw[line width=1pt, style=dashed] (0,0)--(-2,-3)  node[midway,anchor=south east]{$G$} node[sloped, pos=0.5, allow upside down]{\arrowIn};
\draw[line width=1pt, style=dashed] (0,0) --(2,-3) node[midway,anchor=south west]{$H$} node[sloped, pos=0.5, allow upside down]{\arrowIn};
\draw[fill=black] (0,0) circle (.15) node[anchor=east]{$\nu\;$};
\draw[fill=red, color=red] (-4,-3) circle (.15);
\node at (-4,-3.2)[anchor=north, color=red]{$\alpha^d$};
\draw[fill=violet, color=violet] (2,-3) circle (.15);
\node at (2,-3.2)[anchor=north, color=violet]{$\beta$};
\draw[fill=red, color=red] (-2,-3) circle (.15);
\node at (-2,-3.2)[anchor=north, color=red]{$\gamma$};
\node at (-3,1)[anchor=west, color=red]{$\mathcal P$};
\node at (0,3)[anchor=east, color=blue]{$\mathcal M$};
\node at (0,-2.5)[anchor=south, color=violet]{$\mathcal N$};
\end{tikzpicture}\quad
%%%%%%%%
%%%%%%%%
\begin{tikzpicture}[scale=.4, baseline=(current bounding box.center)]
\node at (0,5.2) [anchor=south, color=blue]{$m$};
\node at (4.5,5.2) [anchor=south, color=violet]{$n$};
\draw[line width=1.5pt, color=blue] (2,5)--(-6,5)--(-6,1)  --(-6,-3) --(-4,-3) ;
\draw[line width=1.5pt, color=red] (-2,-3)--(-4,-3) node[midway,anchor=north]{$p$};
\draw[line width=1.5pt, color=violet] (2,5)-- (4,5) --(5,5)--(5,-3)-- (-2,-3)--(2,-3) ;
\draw[line width=1pt, style=dashed]  (0,0).. controls (0,3) and (-4,3).. (-4,0);
\draw[line width=1pt, style=dashed] (-4,-3)--(-4,0) node[anchor=east]{$F$} node[sloped, pos=0.9, allow upside down]{\arrowIn};
\draw[line width=1pt, style=dashed] (0,0)--(-2,-3)  node[midway,anchor=south east]{$G$} node[sloped, pos=0.5, allow upside down]{\arrowIn};
\draw[line width=1pt, style=dashed] (0,0).. controls (0,-2) and (2,-2)..(2,0);
\draw[line width=1pt, style=dashed] (2,0)--(2,5) node[anchor=north west]{$H$} node[sloped, pos=0.5, allow upside down]{\arrowIn};
\draw[fill=black] (0,0) circle (.15) node[anchor=east]{$\nu\;$};
\draw[fill=red, color=red] (-4,-3) circle (.15);
\node at (-4,-3.2)[anchor=north, color=red]{$\alpha^d$};
\draw[fill=violet, color=violet] (2,5) circle (.15);
\node at (2,5)[anchor=south, color=violet]{$\beta^u$};
\draw[fill=red, color=red] (-2,-3) circle (.15);
\node at (-2,-3.2)[anchor=north, color=red]{$\gamma$};
\node at (-3,1)[anchor=west, color=red]{$\mathcal P$};
\node at (0,3)[anchor=east, color=blue]{$\mathcal M$};
\node at (5,0)[anchor=east, color=violet]{$\mathcal N$};
\end{tikzpicture}\quad
%%%%%%%%
%%%%%%%%%%
\begin{tikzpicture}[scale=.4, baseline=(current bounding box.center)]
\node at (0,5.2) [anchor=south, color=violet]{$n$};
\draw[line width=1.5pt, color=blue]  (-6,-3)--(-4,-3) node[midway, anchor=north]{$m$} ;
\draw[line width=1.5pt, color=violet] (-6,-3)--(-8,-3)--(-8,5)--(3,5)--(3,-3)--(2,-3) ;
\draw[line width=1.5pt, color=red] (-2,-3)--(-4,-3) node[midway,anchor=north]{$p$};
\draw[line width=1.5pt, color=violet] (-2,-3)--(2,-3);
\draw[line width=1pt, style=dashed]  (0,0).. controls (0,3) and (-4,3).. (-4,0);
\draw[line width=1pt, style=dashed] (-4,-3)--(-4,0) node[anchor=east]{$F$} node[sloped, pos=0.9, allow upside down]{\arrowIn};
\draw[line width=1pt, style=dashed] (0,0)--(-2,-3)  node[midway,anchor=south east]{$G$} node[sloped, pos=0.5, allow upside down]{\arrowIn};
\draw[line width=1pt, style=dashed] (0,0) .. controls (0,-2) and (2,-2)..(2,0) node[sloped, pos=0.95, allow upside down]{\arrowIn};
\draw[line width=1pt, style=dashed] (2,0).. controls (2,5) and (-6,5) ..(-6,0);
\draw[line width=1pt, style=dashed] (-6,0)--(-6,-3) node[midway, anchor=east]{$H$};
\draw[fill=black] (0,0) circle (.15) node[anchor=east]{$\nu\;$};
\draw[fill=red, color=red] (-4,-3) circle (.15);
\node at (-4,-3.2)[anchor=north, color=red]{$\alpha^d$};
\draw[fill=violet, color=violet] (-6,-3) circle (.15);
\node at (-6,-3.2)[anchor=north, color=violet]{$\beta$};
\draw[fill=red, color=red] (-2,-3) circle (.15);
\node at (-2,-3.2)[anchor=north, color=red]{$\gamma$};
\node at (-3,1)[anchor=west, color=red]{$\mathcal P$};
\node at (-5,-2)[anchor=south, color=blue]{$\mathcal M$};
\node at (-6.5,4)[anchor=north, color=violet]{$\mathcal N$};
\end{tikzpicture}\quad
%%%%%%%%
\begin{tikzpicture}[scale=.4, baseline=(current bounding box.center)]
\node at (1,-3) [anchor=north, color=red]{$p$};
\draw[line width=1.5pt, color=blue]  (-6,-3)--(-4,-3) node[midway, anchor=north]{$m$} ;
\draw[line width=1.5pt, color=red] (2,5)--(3,5)--(3,-3)--(-4,-3);
\draw[line width=1.5pt, color=violet] (-6,-3)--(-8,-3)  -- (-8,5)--(2,5);
\draw[line width=1pt, style=dashed]  (-2,0).. controls (-2,2) and (-4,2).. (-4,0);
\draw[line width=1pt, style=dashed] (-4,-3)--(-4,0) node[anchor=east]{$F$} node[sloped, pos=0.9, allow upside down]{\arrowIn};
\draw[line width=1pt, style=dashed] (-2,0).. controls (-3,-3) and (2,-3) .. (2,0);
\draw[line width=1pt, style=dashed] (2,0)--(2,5) node[anchor=north east] {$G$} node[sloped, pos=0.9, allow upside down]{\arrowIn};
\draw[line width=1pt, style=dashed] (-2,0) .. controls (-2,-2) and (0,-2)..(0,0) node[sloped, pos=0.95, allow upside down]{\arrowIn};
\draw[line width=1pt, style=dashed] (0,0).. controls (0,4) and (-6,4) ..(-6,0);
\draw[line width=1pt, style=dashed] (-6,0)--(-6,-3) node[midway, anchor=east]{$H$};
\node at (0,5)[anchor=south, color=violet]{$n$};
\draw[fill=black] (-2,0) circle (.15) node[anchor=east]{$\nu\;$};
\draw[fill=red, color=red] (-4,-3) circle (.15);
\node at (-4,-3.2)[anchor=north, color=red]{$\alpha^d$};
\draw[fill=violet, color=violet] (-6,-3) circle (.15);
\node at (-6,-3.2)[anchor=north, color=violet]{$\beta$};
\draw[fill=red, color=red] (2,5) circle (.15);
\node at (2, 5.2)[anchor=south, color=red]{$\gamma^u$};
\node at (-3,-2)[color=red]{$\mathcal P$};
\node at (-5,-2)[anchor=south, color=blue]{$\mathcal M$};
\node at (-6.2,0)[anchor=east, color=violet]{$\mathcal N$};
\end{tikzpicture}\quad
%%%%%%%
\begin{tikzpicture}[scale=.4, baseline=(current bounding box.center)]
\node at (1,-3) [anchor=north, color=red]{$p$};
\draw[line width=1.5pt, color=blue]  (-6,-3)--(-4,-3) node[midway, anchor=north]{$m$} ;
\draw[line width=1.5pt, color=red] (-6,-3)--(-10,-3)--(-10,5)--(3,5)--(3,-3)--(2,-3) --(-4,-3);
\draw[line width=1.5pt, color=violet] (-6,-3)--(-8,-3) node[midway,anchor=north]{$n$};
\draw[line width=1pt, style=dashed]  (-2,0).. controls (-2,2) and (-4,2).. (-4,0);
\draw[line width=1pt, style=dashed] (-4,-3)--(-4,0) node[anchor=east]{$F$} node[sloped, pos=0.9, allow upside down]{\arrowIn};
\draw[line width=1pt, style=dashed, -stealth] (-2,0).. controls (-3,-3) and (2,-3) .. (2,0);
\draw[line width=1pt, style=dashed, -stealth] (2,0).. controls (2,6) and (-8,6) .. (-8,0);
\draw[line width=1pt, style=dashed] (-8,0)--(-8,-3) node[anchor=south east] {$G$};
\draw[line width=1pt, style=dashed] (-2,0) .. controls (-2,-2) and (0,-2)..(0,0) node[sloped, pos=0.95, allow upside down]{\arrowIn};
\draw[line width=1pt, style=dashed] (0,0).. controls (0,4) and (-6,4) ..(-6,0);
\draw[line width=1pt, style=dashed] (-6,0)--(-6,-3) node[midway, anchor=east]{$H$};
\draw[fill=black] (-2,0) circle (.15) node[anchor=east]{$\nu\;$};
\draw[fill=red, color=red] (-4,-3) circle (.15);
\node at (-4,-3.2)[anchor=north, color=red]{$\alpha^d$};
\draw[fill=violet, color=violet] (-6,-3) circle (.15);
\node at (-6,-3.2)[anchor=north, color=violet]{$\beta$};
\draw[fill=red, color=red] (-8,-3) circle (.15);
\node at (-8,-3.2)[anchor=north, color=red]{$\gamma$};
\node at (-3,-2)[color=red]{$\mathcal P$};
\node at (-5,-2)[anchor=south, color=blue]{$\mathcal M$};
\node at (-6.2,0)[anchor=east, color=violet]{$\mathcal N$};
\end{tikzpicture}
\caption{Some cyclic transformations of a bordered diagram.}
\label{fig:cyclicperm}
\end{figure}
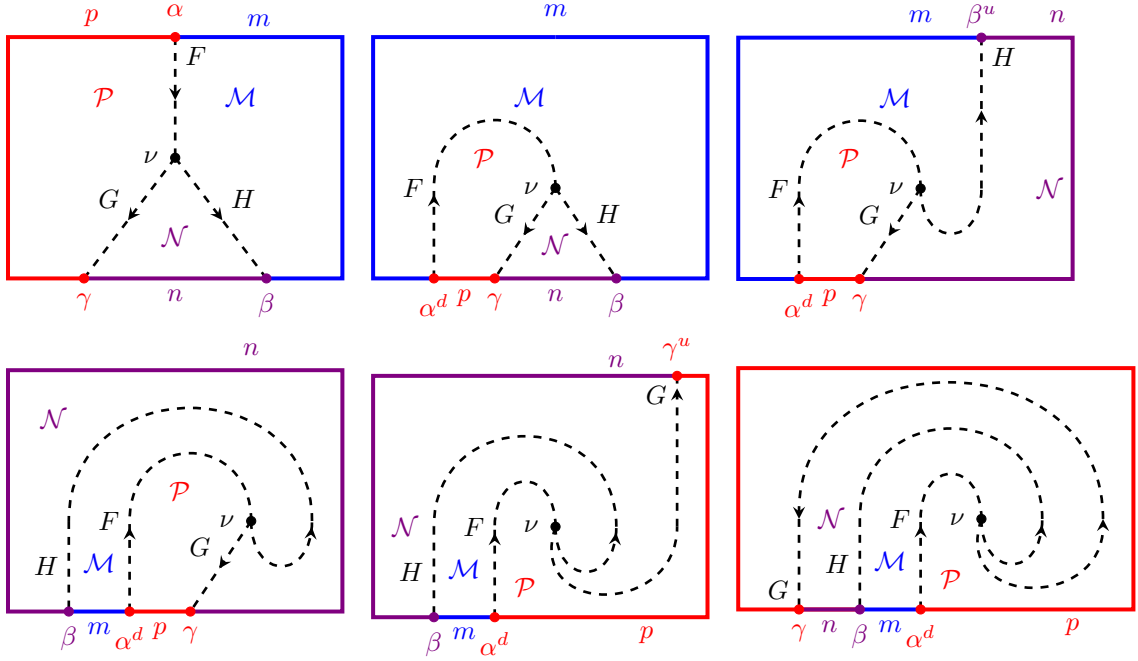

\begin{lemma}\label{lem:cyclicinv} Let $\mac,\mad$ be spherical fusion categories and $D$, $D'$  bordered diagrams for  $\Bimod^\theta(\mac,\mad)$ that are related by cyclic transformations.  Then the cyclic evaluations of  $D$ and $D'$ are equal.
\end{lemma}

\begin{proof}
For the cyclic transformation in \eqref{eq:movemorph1} this follows from the cyclic invariance of the  trace in \eqref{pic:cyclic trace} and its compatibility conditions  with the pivotal 2-category structure of  $\Bimod^\theta(\mac,\mad)$ from  \eqref{eq:pivcond}, \eqref{eq:traceconds}  and \eqref{eq:functortrpivot}:

\begin{align*}
\begin{tikzpicture}[scale=.4]
\draw[line width=1.5pt, color=red] (0,3)--(0,5)node[midway,anchor=west]{$p$};
\draw[line width=1.5pt, color=red] (0,-3)--(0,-5) node[midway,anchor=west]{$p$};
\draw[line width=1.5pt, color=red] (-.5,5)--(.5,5) ;
\draw[line width=1.5pt, color=red] (-.5,-5)--(.5,-5) ;
\draw[line width=1.5pt, color=blue] (0,3)--(0,1) node[midway, anchor=west]{$m$};
\draw[line width=1.5pt, color=violet] (0,-3)--(0,-1) node[midway, anchor=west]{$n$};
\draw[line width=1pt, style=dashed] (0,3)--(-3,0) node[sloped, pos=0.5, allow upside down]{\arrowIn} node[pos=.5, anchor=south east]{$F$};
\draw[line width=1pt, style=dashed] (-3,0)--(0,-3) node[sloped, pos=0.5, allow upside down]{\arrowIn} node[pos=.5, anchor=north east]{$H$};
\draw[line width=1pt, style=dashed] (-3,0)--(-1,1);
\draw[line width=1pt, style=dashed] (-3,0)--(-1,.5);
\draw[line width=1pt, style=dashed] (-3,0)--(-1,-.5);
\draw[line width=1pt, style=dashed] (-3,0)--(-1,-1);
%%%%%%%%%
\node at (0,0) {$\ldots$};
\draw[color=red, fill=red] (0,3) circle (.15) node[anchor=west]{$\alpha$};
\draw[color=red, fill=red] (0,-3) circle (.15) node[anchor=west]{$\beta$};
\draw[color=black, fill=black] (-3,0) circle (.15) node[anchor=east]{$\nu$};
\node at (2,0) {$\stackrel{\eqref{pic:cyclic trace}}=$};
\begin{scope}[shift={(7,0)}]
\draw[line width=1.5pt, color=blue] (0,3)--(0,5) node[midway, anchor=west]{$m$};
\draw[line width=1.5pt, color=blue] (0,-3)--(0,-5) node[midway, anchor=west]{$m$};
\draw[line width=1.5pt, color=red] (-2.5,5)--(.5,5);
\draw[line width=1.5pt, color=red] (-2.5,-5)--(.5,-5);
\draw[line width=1.5pt, color=violet] (0,1)--(0,-1) node[midway, anchor=west]{$n$};
\draw[line width=1.5pt, color=red] (0,-3)--(0,-1) node[midway, anchor=west]{$p$};
\draw[line width=1pt, style=dashed] (-2,5)--(-3,2) node[sloped, pos=0.5, allow upside down]{\arrowIn} node[pos=.5, anchor=east]{$F$};
\draw[line width=1pt, style=dashed] (0,-3)--(-2,-5) node[sloped, pos=0.5, allow upside down]{\arrowIn} node[pos=.5, anchor=east]{$F$};
\draw[line width=1pt, style=dashed] (-3,2)--(0,-1) node[sloped, pos=0.5, allow upside down]{\arrowIn} node[pos=.5, anchor=north east]{$H$};
\draw[line width=1pt, style=dashed] (-3,2)--(-1,3);
\draw[line width=1pt, style=dashed] (-3,2)--(-1,2.5);
\draw[line width=1pt, style=dashed] (-3,2)--(-1,1.5);
\draw[line width=1pt, style=dashed] (-3,2)--(-1,1);
%%%%%%%%%
\node at (0,2) {$\ldots$};
\draw[color=red, fill=red] (0,-1) circle (.15) node[anchor=west]{$\beta$};
\draw[color=red, fill=red] (0,-3) circle (.15) node[anchor=west]{$\alpha$};
\draw[color=black, fill=black] (-3,2) circle (.15) node[anchor=east]{$\nu$};
\node at (2,0) {$\stackrel{\eqref{eq:functortrpivot}}=$};
\end{scope}
%%%%%%%%%%%%%%%%%%%
\begin{scope}[shift={(15,0)}]
\draw[line width=1.5pt, color=blue] (0,3)--(0,5) node[midway, anchor=west]{$m$};
\draw[line width=1.5pt, color=blue] (0,-3)--(0,-5) node[midway, anchor=west]{$m$};
\draw[line width=1.5pt, color=blue] (-.5,5)--(.5,5);
\draw[line width=1.5pt, color=blue] (-.5,-5)--(.5,-5);
\draw[line width=1.5pt, color=violet] (0,1)--(0,-1) node[midway, anchor=west]{$n$};
\draw[line width=1.5pt, color=red] (0,-3)--(0,-1) node[midway, anchor=west]{$p$};
%%%%%%
\draw[line width=1pt, style=dashed] (-2,2).. controls (-2,4) and (-4,4)..(-4,2);
\draw[line width=1pt, style=dashed] (0,-3).. controls (-1,-4) and (-4,-4)..(-4,-2);
\draw[line width=1pt, style=dashed] (-4,-2)--(-4,2) node[sloped, pos=0.5, allow upside down]{\arrowIn} node[pos=.5, anchor=north east]{$F$};
\draw[line width=1pt, style=dashed] (-2,2)--(0,-1) node[sloped, pos=0.5, allow upside down]{\arrowIn} node[pos=.5, anchor=north east]{$H$};
\draw[line width=1pt, style=dashed] (-2,2)--(-1,3);
\draw[line width=1pt, style=dashed] (-2,2)--(-1,2.5);
\draw[line width=1pt, style=dashed] (-2,2)--(-1,1.5);
\draw[line width=1pt, style=dashed] (-2,2)--(-1,1);
%%%%%%%%%
\node at (0,2) {$\ldots$};
\draw[color=red, fill=red] (0,-1) circle (.15) node[anchor=west]{$\beta$};
\draw[color=red, fill=red] (0,-3) circle (.15) node[anchor=west]{$\alpha$};
\draw[color=black, fill=black] (-2,2) circle (.15) node[anchor=east]{$\nu$};
\end{scope}
\end{tikzpicture}
\end{align*}
For the cyclic transformation in \eqref{eq:movemorph2} it is a direct consequence of identity \eqref{pic:adjunction}:
\begin{align*}
\begin{tikzpicture}[scale=.4]
\draw[color=blue, line width=1.5pt] (0,2)--(0,5) node[midway, anchor=west]{$m$};
\draw[color=violet, line width=1.5pt] (0,2)--(0,-2) node[midway, anchor=west]{$n$};
\draw[color=red, line width=1.5pt] (0,-2)--(0,-5) node[midway, anchor=west]{$p$};
\draw[line width=1pt, style=dashed] (0,2)--(-3,0) node[sloped, pos=0.5, allow upside down]{\arrowIn} node[pos=.5, anchor=south]{$K\;\;$};
\draw[line width=1pt, style=dashed] (-3,0)-- (0,-2) node[sloped, pos=0.5, allow upside down]{\arrowIn} node[pos=.5, anchor=north]{$G$};
\draw[line width=1pt, style=dashed] (-3,0)--(-2.2,1.5);
\draw[line width=1pt, style=dashed] (-3,0)--(-2.2,2.5);
\draw[line width=1pt, style=dashed] (-3,0)--(-2.2,-1.5);
\draw[line width=1pt, style=dashed] (-3,0)--(-2.2,-2.5);
\draw[color=blue, fill=blue] (0,2) circle (.15) node[anchor=west]{$\delta$};
\draw[color=red, fill=red] (0,-2) circle (.15) node[anchor=west]{$\gamma$};
\draw[color=black, fill=black] (-3,0) circle (.15) node[anchor=east]{$\nu$};
%%%%%%%
\node at (3,0){$\stackrel{\eqref{pic:adjunction}}=$};
%%%%%%%
\begin{scope}[shift={(9,0)}]
\draw[color=blue, line width=1.5pt] (0,2)--(0,5) node[midway, anchor=west]{$m$};
\draw[color=violet, line width=1.5pt] (0,2)--(0,-2) node[midway, anchor=west]{$n$};
\draw[color=red, line width=1.5pt] (0,-2)--(0,-5) node[midway, anchor=west]{$p$};
\draw[line width=1pt, style=dashed] (0,2)--(-3,0) node[sloped, pos=0.5, allow upside down]{\arrowIn} node[pos=.5, anchor=south]{$K\;\;$};
\draw[line width=1pt, style=dashed] (0,-2) .. controls (-1, 4) and (-1,-4)..(-3,0) node[sloped, pos=0.5, allow upside down]{\arrowIn} node[pos=.5, anchor=east]{$G$};
\draw[line width=1pt, style=dashed] (-3,0)--(-2.2,1.5);
\draw[line width=1pt, style=dashed] (-3,0)--(-2.2,2.5);
\draw[line width=1pt, style=dashed] (-3,0)--(-2.2,-1.5);
\draw[line width=1pt, style=dashed] (-3,0)--(-2.2,-2.5);
\draw[color=blue, fill=blue] (0,2) circle (.15) node[anchor=west]{$\delta$};
\draw[color=red, fill=red] (0,-2) circle (.15) node[anchor=west]{$\gamma$};
\draw[color=black, fill=black] (-3,0) circle (.15) node[anchor=east]{$\nu$};
\end{scope}
\end{tikzpicture}
\end{align*}
\end{proof}

Lemma \ref{lem:cyclicinv} shows that the cyclic evaluation of a bordered diagram depends only on its cyclic equivalence class. We  therefore represent cyclic equivalence classes of bordered diagrams by discs with diagrams for $\Bimod^\theta(\mac,\mad)$  in their interior and endpoints of  lines  at the boundary. 
Interior lines are  oriented, and the object labels at the boundary are assigned to the segments between 
their endpoints. Vertices in the interior are labeled with cyclic equivalence classes of bimodule natural transformations. To the line endpoints at the boundary we
assign elements of the corresponding morphism spaces at the top of a bordered diagram, if the line is incoming at the boundary,  and the one for  the bottom of a bordered diagram, if it is outgoing.
In the following, we often represent these diagrams by polygons and hence call them polygon diagrams. 

\begin{definition} \label{def:polygon} $\quad$
\begin{enumerate}
\item A {\bf polygon diagram} for $\Bimod^\theta(\mac,\mad)$ is a cyclic equivalence class of bordered diagrams, represented by a polygon with a diagram for $\Bimod^\theta(\mac,\mad)$ in the interior such that all diagram lines are oriented and end in internal vertices or at the sides of the polygon.
\item The {\bf evaluation} $\mathrm{ev}(D)$ of a polygon diagram $D$ is the cyclic evaluation of the associated equivalence class of bordered diagrams.
\end{enumerate}
\end{definition}

 It is  clear from Lemma \ref{lem:cyclicinv} that the labeling and the evaluation of a polygon diagram is invariant under rotations of the polygon. Also, one may add additional sides without line endpoints and move endpoints of  lines  from a side  to an adjacent side, if the labels are adjusted accordingly.  
 For instance, some polygon diagrams for the diagram from Figure \ref{fig:cyclicperm} are  given by
\begin{align*}
&\begin{tikzpicture}[scale=.3, baseline=(current bounding box.center)]
\draw[line width=1pt, style=dashed] (0,3)--(0,0) node[midway, anchor=east]{$F$} node[sloped, pos=0.5, allow upside down]{\arrowIn};
\draw[line width=1pt, style=dashed] (0,0)--(-2,-2) node[midway, anchor=south east]{$G$} node[sloped, pos=0.5, allow upside down]{\arrowIn};
\draw[line width=1pt, style=dashed] (0,0)--(2,-2) node[midway, anchor=south west]{$H$} node[sloped, pos=0.5, allow upside down]{\arrowIn};
\draw[line width=1.5pt, color=violet] (0,-4.5)--(-2,-2); 
\draw[line width=1.5pt, color=violet] (0,-4.5) node[anchor=north]{$n$}--(2,-2); 
\draw[line width=1.5pt, color=red] (-2,-2)--(-6,3) node[anchor=south east] {$p$}; 
\draw[line width=1.5pt, color=red] (0,3)--(-6,3); 
\draw[line width=1.5pt, color=blue] (2,-2)--(6,3)node[anchor=south west]{$m$}; 
\draw[line width=1.5pt, color=blue] (0,3)--(6,3); 
\draw[fill=black] (0,0) circle (.2) node[anchor=east]{$\nu$};
\node at (-4,2)[anchor=west, color=red]{$\mathcal P$};
\node at (4,2)[anchor=east, color=blue]{$\mathcal M$};
\node at (0,-2)[anchor=north, color=violet]{$\mathcal N$};
\draw[fill=red, color=red] (0,3) circle (.2);
\node at (0,3.2) [anchor=south, color=red]{$\alpha$};
\draw[fill=violet, color=violet] (2,-2) circle (.2) node[anchor=north west]{$\beta$};
\draw[fill=red, color=red] (-2,-2) circle (.2) node[anchor=north east]{$\gamma$};
\end{tikzpicture}
=
\begin{tikzpicture}[scale=.3, baseline=(current bounding box.center)]
\draw[line width=1pt, style=dashed] (0,0)--(0,3) node[midway, anchor=east]{$G$} node[sloped, pos=0.5, allow upside down]{\arrowIn};
\draw[line width=1pt, style=dashed] (0,0)--(-2,-2) node[midway, anchor=south east]{$H$} node[sloped, pos=0.5, allow upside down]{\arrowIn};
\draw[line width=1pt, style=dashed] (2,-2)--(0,0) node[midway, anchor=south west]{$F$} node[sloped, pos=0.5, allow upside down]{\arrowIn};
\draw[line width=1.5pt, color=blue] (0,-4.5)--(-2,-2); 
\draw[line width=1.5pt, color=blue] (0,-4.5) node[anchor=north]{$m$}--(2,-2); 
\draw[line width=1.5pt, color=violet] (-2,-2)--(-6,3) node[anchor=south east] {$n$}; 
\draw[line width=1.5pt, color=violet] (0,3)--(-6,3); 
\draw[line width=1.5pt, color=red] (2,-2)--(6,3)node[anchor=south west]{$p$}; 
\draw[line width=1.5pt, color=red] (0,3)--(6,3); 
\draw[fill=black] (0,0) circle (.2) node[anchor=east]{$\nu$};
\node at (-4,2)[anchor=west, color=violet]{$\mathcal N$};
\node at (4,2)[anchor=east, color=red]{$\mathcal P$};
\node at (0,-2)[anchor=north, color=blue]{$\mathcal M$};
\draw[fill=red, color=red] (0,3) circle (.2);
\node at (0,3.2) [anchor=south, color=red]{$\gamma$};
\draw[fill=violet, color=red] (2,-2) circle (.2) node[anchor=north west]{$\alpha$};
\draw[fill=red, color=violet] (-2,-2) circle (.2) node[anchor=north east]{$\beta$};
\end{tikzpicture}
=
\begin{tikzpicture}[scale=.3, baseline=(current bounding box.center)]
\draw[line width=1pt, style=dashed] (0,0)--(0,3) node[midway, anchor=east]{$H$} node[sloped, pos=0.5, allow upside down]{\arrowIn};
\draw[line width=1pt, style=dashed] (-2,-2)--(0,0) node[midway, anchor=south east]{$F$} node[sloped, pos=0.5, allow upside down]{\arrowIn};
\draw[line width=1pt, style=dashed] (0,0)--(2,-2) node[midway, anchor=south west]{$G$} node[sloped, pos=0.5, allow upside down]{\arrowIn};
\draw[line width=1.5pt, color=red] (0,-4.5)--(-2,-2); 
\draw[line width=1.5pt, color=red] (0,-4.5) node[anchor=north]{$p$}--(2,-2); 
\draw[line width=1.5pt, color=blue] (-2,-2)--(-6,3) node[anchor=south east] {$m$}; 
\draw[line width=1.5pt, color=blue] (0,3)--(-6,3); 
\draw[line width=1.5pt, color=violet] (2,-2)--(6,3)node[anchor=south west]{$n$}; 
\draw[line width=1.5pt, color=violet] (0,3)--(6,3); 
\draw[fill=black] (0,0) circle (.2) node[anchor=east]{$\nu$};
\node at (-4,2)[anchor=west, color=blue]{$\mathcal M$};
\node at (4,2)[anchor=east, color=violet]{$\mathcal N$};
\node at (0,-2)[anchor=north, color=red]{$\mathcal P$};
\draw[fill=red, color=violet] (0,3) circle (.2);
\node at (0,3.2) [anchor=south, color=violet]{$\beta$};
\draw[fill=violet, color=red] (2,-2) circle (.2) node[anchor=north west]{$\gamma$};
\draw[fill=red, color=red] (-2,-2) circle (.2) node[anchor=north east]{$\alpha$};
\end{tikzpicture}
\end{align*}
and also by
\begin{align*}
&\begin{tikzpicture}[scale=.3, baseline=(current bounding box.center)]
\draw[line width=1pt, style=dashed] (0,-.75)--(0,3) node[midway, anchor=east]{$H$} node[sloped, pos=0.5, allow upside down]{\arrowIn};
\draw[line width=1pt, style=dashed] (0,-.75)--(6,-.75) node[midway, anchor=north]{$G$} node[sloped, pos=0.5, allow upside down]{\arrowIn};
\draw[line width=1pt, style=dashed] (-6,-.75)--(0,-.75) node[midway, anchor=north]{$F$} node[sloped, pos=0.5, allow upside down]{\arrowIn};
\draw[fill=black] (0,-.75) circle (.2) node[anchor=north]{$\nu$};
\draw[line width=1.5pt, color=blue] (-6,-.75)--(-6,3) node[anchor=south east]{$m$}--(0,3); 
\draw[line width=1.5pt, color=violet] (0,3)--(6,3)node[anchor=south west]{$n$}--(6,-.75); 
\draw[line width=1.5pt, color=red] (6,-.75)--(6,-4.5)node[anchor=north west]{$p$}--(-6,-4.5)node[anchor=north east]{$p$}--(-6,-.75); 
\node at (-3.5,2) [anchor=north, color=blue]{$\mam$};
\node at (3.5,2) [anchor=north, color=violet]{$\man$};
\node at (0,-3.5) [anchor=south, color=red]{$\map$};
\draw[fill=red, color=red] (-6,-.75) circle (.2) node[anchor=east]{$\alpha\;$};
\draw[fill=red, color=red] (6,-.75) circle (.2) node[anchor=west]{$\;\gamma$};
\draw[fill=red, color=violet] (0,3) circle (.2) node[anchor=south]{$\beta$};
\end{tikzpicture}
=
\begin{tikzpicture}[scale=.3, baseline=(current bounding box.center)]
\draw[line width=1pt, style=dashed] (0,-.75)--(0,3) node[midway, anchor=east]{$H$} node[sloped, pos=0.5, allow upside down]{\arrowIn};
\draw[line width=1pt, style=dashed] (0,-.75)--(0,-4.5) node[midway, anchor=east]{$G$} node[sloped, pos=0.5, allow upside down]{\arrowIn};
\draw[line width=1pt, style=dashed] (-6,-.75)--(0,-.75) node[midway, anchor=north]{$F$} node[sloped, pos=0.5, allow upside down]{\arrowIn};
\draw[fill=black] (0,-.75) circle (.2) node[anchor=west]{$\nu$};
\draw[line width=1.5pt, color=blue] (-6,-.75)--(-6,3) node[anchor=south east]{$m$}--(0,3); 
\draw[line width=1.5pt, color=violet] (0,3)--(6,3)node[anchor=south west]{$n$}--(6,-4.5) node[anchor=north west]{$n$}--(0,-4.5); 
\draw[line width=1.5pt, color=red] (0,-4.5)--(-6,-4.5)node[anchor=north east]{$p$}--(-6,-.75); 
\node at (-3.5,2) [anchor=north, color=blue]{$\mam$};
\node at (3.5,-.75) [anchor=west, color=violet]{$\man$};
\node at (-3.5,-2.5) [anchor=north, color=red]{$\map$};
\draw[fill=red, color=red] (-6,-.75) circle (.2) node[anchor=east]{$\alpha\;$};
\draw[fill=red, color=red] (0,-4.5) circle (.2) node[anchor=north]{$\gamma$};
\draw[fill=red, color=violet] (0,3) circle (.2) node[anchor=south]{$\beta$};
\end{tikzpicture}
=
\begin{tikzpicture}[scale=.3, baseline=(current bounding box.center)]
\draw[line width=1pt, style=dashed] (0,-.75)--(0,3) node[midway, anchor=east]{$H$} node[sloped, pos=0.5, allow upside down]{\arrowIn};
\draw[line width=1pt, style=dashed] (0,-.75)--(6,-.75) node[midway, anchor=north]{$G$} node[sloped, pos=0.5, allow upside down]{\arrowIn};
\draw[line width=1pt, style=dashed] (0,-4.5)--(0,-.75) node[midway, anchor=east]{$F$} node[sloped, pos=0.5, allow upside down]{\arrowIn};
\draw[fill=black] (0,-.75) circle (.2) node[anchor=east]{$\nu$};
\draw[line width=1.5pt, color=blue] (0,-4.5)--(-6,-4.5)node[anchor=north east]{$m$}--(-6,3) node[anchor=south east]{$m$}--(0,3); 
\draw[line width=1.5pt, color=violet] (0,3)--(6,3)node[anchor=south west]{$n$}--(6,-.75); 
\draw[line width=1.5pt, color=red] (6,-.75)--(6,-4.5)node[anchor=north west]{$p$}--(0,-4.5); 
\node at (-3.5,-.75) [anchor=east, color=blue]{$\mam$};
\node at (3.5,2) [anchor=north, color=violet]{$\man$};
\node at (3.5,-3.8) [anchor=south, color=red]{$\map$};
\draw[fill=red, color=red] (0,-4.5) circle (.2) node[anchor=north]{$\alpha$};
\draw[fill=red, color=red] (6,-.75) circle (.2) node[anchor=west]{$\;\gamma$};
\draw[fill=red, color=violet] (0,3) circle (.2) node[anchor=south]{$\beta$};
\end{tikzpicture}
\end{align*}

Note  that any diagram for a spherical fusion category $\mac$ that describes the trace of an endomorphism $\phi: c\to c$ is a polygon diagram for $\Bimod^{\theta}(\mac,\mac)$ with  cyclic evaluation $\tr(\phi)$. In this case, all regions of the diagram are labeled with $\mac$ and all lines and vertices  by data from $\mac$.

We also consider  mirror images of  polygon diagrams. If the orientation of the lines is kept fixed, there is no canonical data labeling such a mirror image. This would reverse the direction of a bimodule natural transformation labeling a vertex, which is feasible only if it is an isomorphism. However, taking a mirror image of a polygon diagram $D$ for $\Bimod^{\theta}(\mac,\mad)$ and reversing the orientation of all lines yields a polygon diagram $D^{\#}$ for $\Bimod^{\theta}(\mad,\mac)$.  This is obtained by replacing bimodule categories, functors and natural transformations with their opposites from  Examples \ref{ex:modvectgomega} and  \ref{ex:modulefuncs}, 7.  Objects and morphisms at the boundary of $D$ are kept fixed, but  a morphism $\alpha: a\to b$ in $\mam$ is  interpreted as a morphism $\alpha: b\to a$ in $\mam^\#$.

\begin{definition}\label{def:oppoly} Let $D$ be a polygon diagram  labeled by $\Bimod^{\theta}(\mac,\mad)$. 
 The {\bf opposite polygon diagram} $D^\#$ is the  polygon diagram by $\Bimod^{\theta}(\mad,\mac)$ obtained as follows:
 \begin{compactitem}
\item  taking the mirror image of $D$, 
\item reversing the orientation of each line in $D$, 
\item replacing each  $(\mac,\mad)$-bimodule category $\mam$ in $D$  by the 
$(\mad,\mac)$-bimodule category $\mam^\#$, 
\item replacing each bimodule functor $F:\mam\to\man$ in $D$ by  $F^\#:\mam^\#\to\man^\#$,
\item replacing  each bimodule natural transformation $\nu: F\Rightarrow G$ in $D$ by  $\nu^\#: G^\#\Rightarrow F^\#$,
\item  keeping the boundary data fixed.
\end{compactitem}
\end{definition}

Geometrically, $D^\#$ is the polygon diagram obtained by looking at $D$ from the back, behind the plane of the drawing. The orientation reversal for lines  is  necessary, because 
a functor labeling a line goes from the category on the left of the line to the category on the right, viewed in the direction of its orientation. Looking at the diagram from the back exchanges left and right and hence reverses the orientation of the line.

\begin{corollary}\label{rem:polyback}
For any polygon diagram $D$ one has $\mathrm{ev}(D)=\mathrm{ev}(D^\#)$.
\end{corollary}
\begin{proof}
Taking the mirror image, reversing each line and replacing  line labels $F:\mam\to\man$ by $F^\#:\mam^\#\to \man^\#$ replaces a functor $F_1\cdots F_n: \mam\to\man$   by the  functor $(F_n^\#)^l\cdots (F_1^\#)^l=(F_1^\#\cdots F_n^\#)^l: \man^\#\to\mam^\#$. A bimodule natural transformation $\nu: F_1\cdots F_n\Rightarrow G_1\cdots G_m$ in $D$ is replaced by  a natural transformation from $(F_1^\#\cdots F_n^\#)^l$ to $(G_1^\#\cdots G_n^\#)^l$ in $D'$, which is a cyclic transform of $\nu^\#: G_1^\#\cdots G_m^\#\Rightarrow F_1^\#\cdots F_n^\#$. Each morphism $\alpha\in \Hom_\man(n,F(m))$ in $D$ is replaced by a  morphism  $\alpha\in \Hom_{\mam^\#}(m, (F^\#)^l(n))=\Hom_\mam((F^\#)^l(n), m)$ which is related to $\alpha$ via the (co)unit of the adjunction $(F^\#)^l\dashv F^\#$. Identities \eqref{pic:adjunction} and \eqref{pic:pivotal} and the induced bimodule trace on the opposite category (cf.~Example \ref{ex:opptrace}) then ensure that the evaluations of $D$ and $D^\#$ are equal.
\end{proof}

\section{Gluing polygon diagrams}
\label{sec:projincl}

In this section we describe how polygon diagrams for the pivotal 2-category $\Bimod^\theta(\mac,\mad)$  can be glued by summing over simple objects and bases of the morphism spaces at their boundary. Taking the product of the evaluations of  two diagrams and performing these summations  yields the evaluation of the glued diagram.

\subsection{Projections and inclusions for simple objects}
Let $\mam$ a finite semisimple  $\C$-linear category, equipped with a trace $\theta$ in the sense of Definition \ref{def:moduletrace} and let $I$ be a set of representatives of the isomorphism classes of its simple objects. 
Then  every object $x$ in $\mam$ can be expressed as a direct sum $x\cong \oplus_{m\in I} m^{\oplus n_{xm}}$ with multiplicities $n_{xm}=\dim_\C \Hom_\mam(m,x)\in\mathbb N_0$. 
One can choose bases
 $\{p^\alpha_{xm}\}$ of $\Hom_\mathcal M(x,m)$ and  $\{j^\alpha_{xm}\}$ of $\Hom_\mam(m,x)$ with $m\in I$ and $\alpha\in\{1,..., n_{xm}\}$
such that 
\begin{align}\label{eq:projincl}
p^\beta_{xm'} \circ j^\alpha_{xm}=\frac{\delta_{\alpha\beta} \delta_{mm'}}{\dim(m)}1_m \qquad\qquad 1_x=\sum_{m\in I}\sum_{\alpha=1}^{n_{xm}} \dim(m) \;\; j^\alpha_{xm}\circ p^\alpha_{xm}
\end{align}
for all $m,m'\in I$, where $\dim(m)=\theta_m(1_m)$. 
Their normalisation is fixed by \eqref{eq:projincl}  such that  
\begin{align}
\theta_m(p^\beta_{xm}\circ j^\alpha_{xm})=\delta_{\alpha\beta},
\end{align}
and this implies with the cyclicity of the trace 
\begin{align}\label{eq:dimformula}
\sum_{m\in I} \dim(m) \dim_{\C}\Hom_\mam(m,x)=\dim (x).
\end{align}
The identities \eqref{eq:projincl}  imply for all morphisms $\phi: m\to x$ and $\psi:x\to m$
\begin{align}\label{eq:traceidnd}
\phi=\sum_{\alpha=1}^{n_{xm}} \theta_m(p^\alpha_{xm}\circ \phi)\; j^\alpha_{xm}\qquad  \psi=\sum_{\alpha=1}^{n_{ym}} \theta_m(\psi\circ j^\alpha_{xm})\; p^\alpha_{xm}
\end{align}
and for all morphisms $\chi: x\to y$
 \begin{align}\label{eq:traceidnd2}
\chi=\sum_{m\in I}\sum_{\alpha=1}^{n_{xm}}\sum_{\beta=1}^{n_{ym}} \dim(m)\,\theta_m(p^\beta_{ym}\circ \chi\circ j^\alpha_{xm})\;  j^\beta_{ym} \circ p^\alpha_{xm}.
\end{align}
These identities are represented  by the diagrams
\begin{align}\label{pic:traceidnd}
&\begin{tikzpicture}[scale=.3]
\node at (-3,0) [anchor=west] {$\sum_\alpha$}; 
\draw[line width=1.5pt,color=blue] (0,4)--(0,-4) node[midway, anchor=west]{$x$};
\draw[line width=1.5pt,color=blue] (-.5,4)--(.5,4) node[anchor=west]{$m$};
\draw[line width=1.5pt,color=blue] (-.5,-4)--(.5,-4)node[anchor=west]{$m$}; 
\draw[color=blue, fill=blue] (0,2) circle (.2) node[anchor=west]{$\phi$};
\draw[color=blue, fill=blue] (0,-2) circle (.2) node[anchor=west]{$\alpha$};
\draw[color=blue, line width=1.5pt] (3,4) node[anchor=west]{$m$}--(3,-4) node[anchor=west]{$x$};
\draw[color=blue, fill=blue] (3,0) circle (.2) node[anchor=west]{$\alpha$};
\node at (6,0){$\stackrel{\eqref{eq:traceidnd}}=$};
\draw[color=blue, line width=1.5pt] (8,4) node[anchor=west]{$m$}--(8,-4) node[anchor=west]{$x$};
\draw[color=blue, fill=blue] (8,0) circle (.2) node[anchor=west]{$\phi$};
\end{tikzpicture}
&
&\begin{tikzpicture}[scale=.3]
\node at (-3,0) [anchor=west] {$\sum_\alpha$}; 
\draw[line width=1.5pt,color=blue] (0,4)--(0,-4) node[midway, anchor=west]{$x$};
\draw[line width=1.5pt,color=blue] (-.5,4)--(.5,4) node[anchor=west]{$m$};
\draw[line width=1.5pt,color=blue] (-.5,-4)--(.5,-4)node[anchor=west]{$m$}; 
\draw[color=blue, fill=blue] (0,2) circle (.2) node[anchor=west]{$\alpha$};
\draw[color=blue, fill=blue] (0,-2) circle (.2) node[anchor=west]{$\psi$};
\draw[color=blue, line width=1.5pt] (3,4) node[anchor=west]{$x$}--(3,-4) node[anchor=west]{$m$};
\draw[color=blue, fill=blue] (3,0) circle (.2) node[anchor=west]{$\alpha$};
\node at (6,0){$\stackrel{\eqref{eq:traceidnd}}=$};
\draw[color=blue, line width=1.5pt] (8,4) node[anchor=west]{$x$}--(8,-4) node[anchor=west]{$m$};
\draw[color=blue, fill=blue] (8,0) circle (.2) node[anchor=west]{$\psi$};
\end{tikzpicture}
&
&\begin{tikzpicture}[scale=.3]
\node at (-11,0) [anchor=west] {$\sum_m\sum_{\alpha,\beta}\dim(m)$}; 
\draw[line width=1.5pt,color=blue] (0,4)--(0,-4);
\node at (0, 1.5)[anchor=west, color=blue]{$x$};
\node at (0, -1.5)[anchor=west, color=blue]{$y$};
\draw[line width=1.5pt,color=blue] (-.5,4)--(.5,4) node[anchor=west]{$m$};
\draw[line width=1.5pt,color=blue] (-.5,-4)--(.5,-4)node[anchor=west]{$m$}; 
\draw[color=blue, fill=blue] (0,3) circle (.2) node[anchor=west]{$\alpha$};
\draw[color=blue, fill=blue] (0,0) circle (.2) node[anchor=west]{$\chi$};
\draw[color=blue, fill=blue] (0,-3) circle (.2) node[anchor=west]{$\beta$};
\draw[color=blue, line width=1.5pt] (3,4) node[anchor=west]{$x$}--(3,-4) node[anchor=west]{$y$} node[midway, anchor=west]{$m$};
\draw[color=blue, fill=blue] (3,2) circle (.2) node[anchor=west]{$\alpha$};
\draw[color=blue, fill=blue] (3,-2) circle (.2) node[anchor=west]{$\beta$};
\node at (6,0){$\stackrel{\eqref{eq:traceidnd2}}=$};
\draw[color=blue, line width=1.5pt] (8,4) node[anchor=west]{$x$}--(8,-4) node[anchor=west]{$y$};
\draw[color=blue, fill=blue] (8,0) circle (.2) node[anchor=west]{$\chi$};
\end{tikzpicture}
\end{align}

If $\mam$ and $\man$ are bimodule categories with bimodule traces over spherical  fusion categories $\mac$ and $\mad$ and $F:\mam\to\man$ is a $\C$-linear functor, then one can characterise $F$ by decomposing the images of simple objects $m\in I_\mam$ into simple objects $n\in I_\man$. We write $p^\alpha_{Fmn}=p^\alpha_{F(m)n}: F(m)\to n$
and $j^\alpha_{Fmn}=j^\alpha_{F(m)n}: n\to F(m)$ for the associated projection and inclusion morphisms satisfying \eqref{eq:projincl} and denote them by
\begin{align}\label{pic:projincl}
&\begin{tikzpicture}[scale=.3]
\draw[line width=1.5pt, color=blue] (0,2) node[anchor=south]{$m$}--(0,0);
\draw[line width=1.5pt, color=violet] (0,0)-- (0,-2) node[anchor=north]{$n$};
\draw[line width=1pt, style=dashed] (-2,2) node[anchor=south]{$F$}--(0,0);
\draw[color=violet, fill=violet] (0,0) circle(.2) node[anchor=west]{$\;\alpha$};
\end{tikzpicture}
&\begin{tikzpicture}[scale=.3]
\draw[line width=1.5pt, color=blue] (0,-2) node[anchor=north]{$m$}--(0,0);
\draw[line width=1.5pt, color=violet] (0,0)-- (0,2) node[anchor=south]{$n$};
\draw[line width=1pt, style=dashed] (-2,-2) node[anchor=north]{$F$}--(0,0);
\draw[color=violet, fill=violet] (0,0) circle(.2) node[anchor=west]{$\;\alpha$};
\end{tikzpicture}
\end{align}
such that the  identities \eqref{eq:projincl} are given by
\begin{align}\label{pic:semisimpleids}
&\begin{tikzpicture}[scale=.3]
\begin{scope}[shift={(-3,0)}]
\draw[line width=1.5pt, color=violet] (0,3) node[anchor=south]{$n$}--(0,2);
\draw[line width=1pt, style=dashed] (0,2)..controls (-2,0).. (0,-2) node[midway,anchor=east]{$F$};
\draw[line width=1.5pt, color=blue] (0,2)--(0,-2) node[midway,anchor=west]{$m$};
\draw[line width=1.5pt, color=violet] (0,-3) node[anchor=north]{$n'$}--(0,-2);
\draw[fill=violet, color=violet] (0,2) circle (.2) node[anchor=west]{$\;\alpha$};
\draw[fill=violet, color=violet] (0,-2) circle (.2) node[anchor=west]{$\;\beta$};
\end{scope}
\node at (0,0) [anchor=west] {$=\frac{\delta_{\alpha\beta}\delta_{nn'}}{\dim(n)}$};
\draw[line width=1.5pt, color=violet] (6,3) node[anchor=south]{$n$}--(6,-3);
\end{tikzpicture}
&
&\begin{tikzpicture}[scale=.3]
\begin{scope}[shift={(-3,0)}]
\node at (-9,0)[anchor=west] {$\sum_{n,\alpha}\dim(n)$};
\draw[line width=1pt, style=dashed] (-2,3) node[anchor=south]{$F$}--(0,1);
\draw[line width=1pt, style=dashed] (-2,-3) node[anchor=north]{$F$}--(0,-1);
\draw[line width=1.5pt, color=blue] (0,3) node[anchor=south]{$m$}--(0,1);
\draw[line width=1.5pt, color=blue] (0,-3) node[anchor=north]{$m$}--(0,-1);
\draw[line width=1.5pt, color=violet] (0,1) --(0,-1) node[midway,anchor=west]{$n$};
\draw[fill=violet, color=violet] (0,1) circle (.2) node[anchor=south west]{$\;\alpha$};
\draw[fill=violet, color=violet] (0,-1) circle (.2) node[anchor=north west]{$\;\alpha$};
\end{scope}
\node at (0,0){$=$};
\begin{scope}[shift={(4,0)}]
\draw[line width=1pt, style=dashed] (-2,3) node[anchor=south]{$F$}--(-2,-3);
\draw[line width=1.5pt, color=blue] (0,3) node[anchor=south]{$m$}--(0,-3);
\end{scope}
\end{tikzpicture}
\end{align}
For functors that are 1-morphisms in the pivotal 2-categories $\Bimod^\theta(\mac,\mad)$ and $\mathrm{Mod}^\theta(\mac)$ from Theorem \ref{th:pivot} and hence have left and right adjoints, the morphisms $p^\alpha_{Fmn}$ and $j^\alpha_{Fmn}$ from \eqref{pic:projincl} can also be modified by composing them with the units and counits of the adjunctions. 
By composing them with the component morphisms of  $\eta'^{F}:\id_\mam\Rightarrow F^lF$ and $\epsilon^F:F^lF \Rightarrow \id_\man$,  introduced before Corollary \ref{cor:functortrace},  one obtains 
\begin{align}
&\begin{tikzpicture}[scale=.3]
\draw[line width=1.5pt, color=violet] (0,0)--(0,3) node[anchor=south]{$n$}; 
\draw[line width=1.5pt, color=blue] (0,0)--(0,-3) node[anchor=north]{$m$}; 
\draw[line width=1pt, dashed] (0,0).. controls (0,-2) and (-2,-2)..(-2,0); 
\draw[line width=1pt, dashed] (-2,0)--(-2,3) node[sloped, pos=.5, allow upside down]{\arrowIn} node[pos=.5,anchor=east]{$F$}; 
\draw[color=violet, fill=violet] (0,0) circle (.2) node[anchor=west]{$\alpha$};
\end{tikzpicture}
&
&\begin{tikzpicture}[scale=.3]
\draw[line width=1.5pt, color=blue] (0,0)--(0,3) node[anchor=south]{$m$}; 
\draw[line width=1.5pt, color=violet] (0,0)--(0,-3) node[anchor=north]{$n$}; 
\draw[line width=1pt, dashed] (0,0).. controls (0,2) and (-2,2)..(-2,0); 
\draw[line width=1pt, dashed] (-2,-3)--(-2,0) node[sloped, pos=.5, allow upside down]{\arrowIn} node[pos=.5,anchor=east]{$F$}; 
\draw[color=violet, fill=violet] (0,0) circle (.2) node[anchor=west]{$\alpha$};
\end{tikzpicture}\\
&
\pi^\alpha_{F^lnm}: F^l(n)\to m
&
&\iota^\alpha_{F^lnm}: m\to F^l(n)\nonumber
\end{align}
These morphisms satisfy identities analogous to \eqref{pic:semisimpleids}. For all $m,m'\in I_\mam$ and $n\in I_\man$  one has
\begin{align}\label{eq:genorthcheck}
&\begin{tikzpicture}[scale=.3]
\draw[line width=1.5pt, color=blue] (0,1)--(0,4) node[anchor=south]{$m$};
\draw[line width=1.5pt, color=blue] (0,-1)--(0,-4) node[anchor=north]{$m'$};
\draw[line width=1.5pt, color=violet] (0,1)--(0,-1) node[pos=.5, anchor=west]{$n$};
\draw[line width=1pt, dashed] (0,1).. controls (0,3) and (-2,3).. (-2,1);
\draw[line width=1pt, dashed] (-2,-1)--(-2,1) node[sloped, pos=.5, allow upside down]{\arrowIn} node[pos=.5, anchor=east]{$F$};
\draw[line width=1pt, dashed] (0,-1).. controls (0,-3) and (-2,-3).. (-2,-1) ;
\draw[color=violet, fill=violet] (0,1) circle (.2) node[anchor=south west]{$\,\alpha$};
\draw[color=violet, fill=violet] (0,-1) circle (.2) node[anchor=north west]{$\,\beta$};
\node at (2,0)[anchor=west]{$=\frac{\delta_{\alpha\beta}\delta_{mm'}}{\dim(m)}$};
\draw[line width=1.5pt, color=blue] (9,-4)--(9,4) node[anchor=south]{$m$};
\end{tikzpicture}
&
&\begin{tikzpicture}[scale=.3]
\node at (-2,0)[anchor=east]{$\sum_{m,\alpha}\dim(m)$};
\draw[line width=1.5pt, color=violet] (0,2)--(0,4) node[anchor=south]{$n$};
\draw[line width=1.5pt, color=violet] (0,-2)--(0,-4) node[anchor=north]{$n$};
\draw[line width=1.5pt, color=blue] (0,2)--(0,-2) node[pos=.5, anchor=west]{$m$};
\draw[line width=1pt, dashed] (0,2).. controls (0,0) and (-2,0).. (-2,2);
\draw[line width=1pt, dashed] (-2,2)--(-2,4) node[sloped, pos=.5, allow upside down]{\arrowIn} node[anchor=south]{$F$};
\draw[line width=1pt, dashed] (0,-2).. controls (0,0) and (-2,0).. (-2,-2) ;
\draw[line width=1pt, dashed] (-2,-4)node[anchor=north]{$F$} --(-2,-2) node[sloped, pos=.5, allow upside down]{\arrowIn} ;
\draw[color=violet, fill=violet] (0,2) circle (.2) node[anchor=west]{$\,\alpha$};
\draw[color=violet, fill=violet] (0,-2) circle (.2) node[anchor=west]{$\,\alpha$};
\node at (3,0){$=$};
\draw[line width=1pt, dashed] (5,-4)--(5,4) node[sloped, pos=.5, allow upside down]{\arrowIn} node[anchor=south]{$F$};
\draw[line width=1.5pt, color=violet] (7,-4)--(7,4) node[anchor=south]{$n$};
\end{tikzpicture}
\end{align}
The first identity  follows directly from the fact that $m,m'\in I_\mam$ are simple, by taking the trace and then using its cyclicity, the second identity in \eqref{eq:functortrpivot} and the first identity in \eqref{eq:projincl} and \eqref{pic:semisimpleids}. For the second identity note that  the morphisms $\pi^\alpha_{F^lnm}$ and $\iota^\alpha_{F^lnm}$ can be expressed as linear combinations
\begin{align}\label{eq:lincombi}
\pi^{\alpha}_{F^lnm}=\sum_{\beta} M^{\alpha\beta} p^\beta_{F^lmn}\qquad \iota^{\alpha}_{F^lnm}=\sum_{\beta} N^{\alpha\beta} j^\beta_{F^lmn}
\end{align}
with invertible complex matrices $M=(M^{\alpha\beta})$ and $N=(N^{\alpha\beta})$.
 The first identity in \eqref{eq:genorthcheck} implies $M^T=N^\inv$.  Inserting \eqref{eq:lincombi} into the left-hand side of the second equation in \eqref{eq:genorthcheck}, together with the identity $M^T=N^\inv$ and the second identity in \eqref{eq:projincl} and \eqref{pic:semisimpleids} then yields the second identity in \eqref{eq:genorthcheck}.
This implies for all $n\in I_\man$
\begin{align}\label{eq:traceout}
\begin{tikzpicture}[scale=.3]
\node at (-6,0)[anchor=east]{$\sum_{m,\alpha}\dim(m)$};
\draw[line width=1.5pt, color=violet] (0,2)--(0,4) node[anchor=south]{$n$};
\draw[line width=1.5pt, color=violet] (0,-2)--(0,-4) node[anchor=north]{$n$};
\draw[line width=1.5pt, color=blue] (0,2)--(0,-2) node[pos=.5, anchor=west]{$m$};
\draw[line width=1pt, dashed] (0,2)--(-4,0) node[pos=.5, anchor=south]{$F$};
\draw[line width=1pt, dashed] (0,-2)--(-4,0) node[pos=.5, anchor=north]{$F$};
\draw[line width=1pt, dashed] (-4,4)node[anchor=south]{$G$} --(-4,-4) node[anchor=north]{$H$};
\draw[fill=black] (-4,0) circle (.2) node[anchor=east]{$\nu$};
\draw[fill=violet, color=violet] (0,2) circle (.2) node[anchor=west]{$\,\alpha$};
\draw[fill=violet, color=violet] (0,-2) circle (.2) node[anchor=west]{$\,\alpha$};
\node at (2,0)[anchor=west]{$=$};
%%%%%%
\begin{scope}[shift={(18,0)}]
\node at (-6,0)[anchor=east]{$\sum_{m,\alpha}\dim(m)$};
\draw[line width=1.5pt, color=violet] (0,2)--(0,4) node[anchor=south]{$n$};
\draw[line width=1.5pt, color=violet] (0,-2)--(0,-4) node[anchor=north]{$n$};
\draw[line width=1.5pt, color=blue] (0,-2)--(0,2) node[midway, anchor=west]{$m$};
\draw[line width=1pt, dashed] (-4,4)node[anchor=south]{$G$} --(-4,-4) node[anchor=north]{$H$};
\draw[line width=1pt, dashed] (0,2).. controls (-1,-2) and (-3,4)..(-4,0) node[sloped, pos=.5, allow upside down]{\arrowIn} node[pos=.5,anchor=south]{$F$};
\draw[line width=1pt, dashed] (0,-2).. controls (-1,2) and (-3,-4)..(-4,0) node[sloped, pos=.5, allow upside down]{\arrowIn} node[pos=.5,anchor=north]{$F$};
\draw[fill=black] (-4,0) circle (.2) node[anchor=east]{$\nu$};
\draw[fill=violet, color=violet] (0,2) circle (.2) node[anchor=west]{$\,\alpha$};
\draw[fill=violet, color=violet] (0,-2) circle (.2) node[anchor=west]{$\,\alpha$};
\end{scope}
\node at (22,0){$\stackrel{\eqref{eq:genorthcheck}}=$};
%%%%%
\begin{scope}[shift={(30,0)}]
\draw[line width=1.5pt, color=violet] (0,-4)--(0,4) node[anchor=south]{$n$};
\draw[line width=1pt, dashed] (-4,4)node[anchor=south]{$G$} --(-4,-4) node[anchor=north]{$H$};
\draw[line width=1pt, dashed] (-4,0).. controls (-4,2) and (-2,2)..(-2,0);
\draw[line width=1pt, dashed, -stealth] (-4,0).. controls (-4,-2) and (-2,-2)..(-2,0) node[anchor=west]{$F$};
\draw[fill=black] (-4,0) circle (.2) node[anchor=east]{$\nu$};
\end{scope}
\end{tikzpicture}
\end{align}

In particular, identities \eqref{pic:semisimpleids}, \eqref{eq:genorthcheck} and \eqref{eq:traceout} hold for endofunctors  
 $F=c\rhd -: \mam\to \mam$ or $F=-\lhd d:\mam\to\mam$ for a $(\mac,\mad)$-bimodule category $\mam$. In this case, we use thin black or grey lines labeled $c$ or $d$ instead of dashed lines,  and the upward arrows denote the duals in $\mac$ and $\mad$.  Identities \eqref{pic:semisimpleids}, together with \eqref{eq:catdimension},  \eqref{eq:dimensionsmodule} and  \eqref{eq:dimformula},  then imply for all simple objects $n\in I_\mam$
\begin{align}\label{eq:2gonformula}
\begin{tikzpicture}[scale=.25]
\begin{scope}[shift={(-3,0)}]
\node at (-10,0) {$\sum_{i,m, \alpha}\dim(i)\dim(m)$};
\draw[line width=1.5pt, color=blue] (0,4) node[anchor=south]{$n$}--(0,-4) node[anchor=north]{$n$} node[midway, anchor=west]{$m$};
\draw[line width=.5pt, color=black] (0,2).. controls (-2,0)..(0,-2) node[midway, anchor=east]{$i$};
\draw[color=blue, fill=blue] (0,2) circle (.2) node[anchor=west]{$\alpha$}; 
\draw[color=blue, fill=blue] (0,-2) circle (.2) node[anchor=west]{$\alpha$}; 
\end{scope}
\node at (0,0){$=$};
\begin{scope}[shift={(7,0)}]
\node at (-3,0) {$\dim(\mac)$};
\draw[line width=1.5pt, color=blue] (0,4) node[anchor=south]{$n$}--(0,-4) node[anchor=north]{$n$}; 
\end{scope}
\end{tikzpicture}
\end{align}

\subsection{Gluing identities for polygon diagrams}

We consider the pivotal 2-category $\Bimod^\theta(\mac,\mad)$ for spherical fusion categories $\mac,\mad$.
We fix sets  $I_\mac$, $I_\mad$  and $I_\mam$ of representatives of the isomorphism classes of simple objects in $\mac$, $\mad$ and in each bimodule category  
$\mam$.  For each object $x$ in $\mac$, $\mad$ or $\mam$ and each simple object $i$ in the set of representatives, we choose projection morphisms $p^\alpha_{xi}: x\to i$ and inclusion morphisms $j^\alpha_{xi}:i\to x$ that satisfy  \eqref{eq:projincl}. 

Then we can glue and cut  polygon diagrams by summing  over the simple objects at their boundary segments and over bases of the morphism spaces at their boundary vertices.

\begin{theorem} \label{th:polyids}
The cyclic evaluations of the polygon diagrams for $\Bimod^\theta(\mac,\mad)$ satisfy:

\medskip
\begin{compactenum}
\item {\bf Gluing sides:} for all objects $m\in I_\mam$, $n\in I_\man$
\begin{align}\label{pic:gluepoly}
&\begin{tikzpicture}[scale=.3]
\node at (-22,0) {$\sum_{\alpha}$};
\begin{scope}[shift={(-15,0)}]
\draw[draw=none, fill=gray, fill opacity=.2] (-30:5)--(30:5)--(90:5)--(150:5)--(210:5)--(270:5)--(330:5);
\draw[line width=1.5pt, color=blue] (0:4.35)--(30:5) node[anchor=south]{$\;m$};
\draw[line width=1.5pt, color=blue] (30:5)--(60:4.35);
\draw[line width=1.5pt, color=red] (-30:5)--(-60:4.35);
\draw[line width=1.5pt, color=red] (0:4.35)--(-30:5) node[anchor=north]{$\;n$};
\draw[line width=1pt, dashed] (1,0) node[anchor=east]{$F$}--(0:4.35) node[sloped, pos=0.5, allow upside down]{\arrowIn};
\draw[color=red, fill=red] (0:4.35) circle (.2) node[anchor=west]{$\alpha$};
\node at (3,.5)[anchor=south, color=blue]{$\mam$};
\node at (3,-.5)[anchor=north, color=red]{$\man$};
\draw[draw=none, fill=gray, fill opacity=.2] (8,-2.5)--(8,2.5)--(13,2.5)--(13,-2.5)--(8,-2.5);
\draw[color=blue, line width=1.5pt] (8,0)--(8,2.5) node[anchor=south]{$m$}--(10.5,2.5);
\draw[color=red, line width=1.5pt] (8,0)--(8,-2.5) node[anchor=north]{$n$}--(10.5,-2.5);
\draw[line width=1pt, dashed] (8,0) --(10.5,0) node[anchor=west]{$F$} node[sloped, pos=0.7, allow upside down]{\arrowIn};
\draw[color=red, fill=red] (8,0) circle (.2) node[anchor=east]{$\alpha$};
\node at (9.5,.5)[anchor=south, color=blue]{$\mam$};
\node at (9.5,-.5)[anchor=north, color=red]{$\man$};
\end{scope}
\node at (0,0){$=$};
\begin{scope}[shift={(7,0)}]
\draw[draw=none, fill=gray, fill opacity=.2] (-30:5)--(30:5)--(90:5)--(150:5)--(210:5)--(270:5)--(330:5);
\draw[draw=none, fill=gray, fill opacity=.2] (4.33,-2.5)--(4.33,2.5)--(9.35,2.5)--(9.35,-2.5)--(9.35,-2.5);
\draw[line width=1.5pt, color=blue] (30:5) node[anchor=south]{$\;m$} --(60:4.35);
\draw[line width=1.5pt, color=blue] (4.34,2.5)--(6.84,2.5);
\draw[line width=1.5pt, color=red] (-30:5) node[anchor=north]{$\;n$} --(-60:4.35);
\draw[line width=1.5pt, color=red] (4.34,-2.5)--(6.84,-2.5);
\draw[line width=1pt, dashed] (1,0) node[anchor=east]{$F$}--(6,0) node[sloped, pos=0.5, allow upside down]{\arrowIn};
\node at (3.5,.5)[anchor=south, color=blue]{$\mam$};
\node at (3.5,-.5)[anchor=north, color=red]{$\man$};
\end{scope}
\end{tikzpicture}
\end{align}

\bigskip
\item {\bf Gluing around a vertex:} for all objects $m\in I_\mam$
\begin{align}
\label{pic:closepoly}
&\begin{tikzpicture}[scale=.3]
\node at (-17,0) {$\sum_{n,\alpha} \dim(n)$};
\begin{scope}[shift={(-7,0)}]
\draw[draw=none, fill=gray, fill opacity=.2] (0,0)--(30:5)--(90:5)--(150:5)--(210:5)--(270:5)--(330:5)--(0,0);
\draw[line width=1.5pt, color=red] (30:2.5)--(0,0) node[anchor=west]{$\;n$}--(-30:2.5);
\draw[line width=1.5pt, color=blue] (30:2.5)--(30:5) node[anchor=west]{$m$}--(60:4.35);
\draw[line width=1.5pt, color=blue] (-30:2.5)--(-30:5)node[anchor=west]{$m$} --(-60:4.35);
\draw[line width=1pt, dashed] (0,2.5) node[anchor=east]{$F$}--(30:2.5) node[sloped, pos=0.5, allow upside down]{\arrowIn};
\draw[line width=1pt, dashed] (-30:2.5)--(0,-2.5) node[anchor=east]{$F$} node[sloped, pos=0.5, allow upside down]{\arrowIn};
\draw[fill=blue, color=red] (30:2.5) circle (.2) node[anchor=west]{$\;\alpha$};
\draw[fill=blue, color=red] (-30:2.5) circle (.2) node[anchor=west]{$\;\alpha$};
\node at (2,1.8)[anchor=south, color=blue]{$\mam$};
\node at (2,-1.8)[anchor=north, color=blue]{$\mam$};
\node at (-.5,0)[anchor=east, color=red]{$\man$};
\end{scope}
\node at (0,0){$=$};
\begin{scope}[shift={(7,0)}]
\draw[draw=none, fill=gray, fill opacity=.2] (0:5)--(72:5)--(144:5)--(216:5)--(288:5)--(0:5);
\draw[line width=1.5pt, color=blue] (-36:4)--(0:5) node[anchor=west]{$m$}--(36:4);
\draw[line width=1pt, dashed] (50:2)  .. controls (0:2)..(-50:2) node[anchor=north east]{$F$} node[sloped, pos=1, allow upside down]{\arrowIn};
\node at (2,0)[anchor=west, color=blue]{$\mam$};
\node at (1,0)[anchor=east, color=red]{$\man$};
\end{scope}
\end{tikzpicture}
\end{align}

\bigskip
\item {\bf Gluing a 2-gon:} for all objects $m,n\in I_\mam$, $x\in \Ob\mac$ or $x\in \Ob\mad$ and all morphisms $\phi: x\to x$ 
\begin{align}\label{pic:2gon} 
\begin{tikzpicture}[scale=.2]
\node at (-14,0) {$\sum_{m,n,\alpha}\dim(m)\dim(n)$};
\draw [blue,line width=1.5pt,domain=0:90] plot ({0.1*\x}, {5*sin(\x)});
\draw [blue,line width=1.5pt,domain=0:90] plot ({0.1*\x}, {-5*sin(\x)});
\draw [blue,line width=1.5pt,domain=0:90] plot ({0.1*\x+9}, {5*sin(\x+90)});
\draw [blue,line width=1.5pt,domain=0:90] plot ({0.1*\x+9}, {-5*sin(\x+90)});
\draw[line width=.5pt, color=black] (9,5)--(9,0) node[midway, anchor=west]{$x$} node[sloped, pos=.5, allow upside down]{\arrowOut} ;
\draw[line width=.5pt, color=black] (9,0)--(9,-5) node[midway, anchor=west]{$x$} node[sloped, pos=.5, allow upside down]{\arrowOut};
\draw[fill=blue, color=blue] (9,5) circle (.3) node[anchor=south]{$\alpha$};
\draw[fill=blue, color=blue] (9,-5) circle (.3) node[anchor=north]{$\alpha$};
\draw[fill=black] (9,0) circle (.3) node[anchor=west]{$\phi$};
\node at (3,0)[anchor=west, color=blue]{$\mam$};
\node at (0,0)[anchor=east, color=blue]{$m$};
\node at (18,0)[anchor=west, color=blue]{$n$};
\node at (22,0)[anchor=west]{$=\dim(\mam)$};
\draw[line width=.5pt] (37,3)--(37,-3) node[sloped, pos=.1, allow upside down]{\arrowOut} node[sloped, pos=1, allow upside down]{\arrowOut};
\draw[line width=.5pt] (33,-3)--(33,3) node[sloped, pos=.5, allow upside down]{\arrowOut};
\draw[line width=.5pt] (33,3).. controls (33,5) and (37,5).. (37,3);
\draw[line width=.5pt] (33,-3).. controls (33,-5) and (37,-5).. (37,-3);
\draw[fill=black] (37,0) circle (.3) node[anchor=west]{$\phi$};
\node at (37,3)[anchor=west]{$x$};
\node at (37,-3)[anchor=west]{$x$};
\end{tikzpicture}
\end{align}

\pagebreak

\item {\bf Inserting a diagram for $\mac$ or $\mad$:} \\ for all objects $i\in I_\mac$, $x\in \Ob\mac$ or $i\in I_\mad$, $x\in \Ob\mad$ and morphisms $\phi: i\to x$ and $\psi: x\to i$
\begin{align}\label{pic:insert}
&\begin{tikzpicture}[scale=.26]
\node at (-10,0){$\sum_{\alpha}$};
\begin{scope}[shift={(-4,0)}]
\draw[draw=none, fill=gray, fill opacity=.2] (-30:5)--(30:5)--(90:5)--(150:5)--(210:5)--(270:5)--(330:5);
\draw[color=black, fill=white] (0,0) circle (3);
\draw[line width=.5pt, color=black] (0,3)--(0,0) node[midway, anchor=west]{$i$} node[sloped, pos=.5, allow upside down]{\arrowOut};
\draw[line width=.5pt, color=black] (0,0)--(0,-3) node[midway, anchor=west]{$x$}  node[sloped, pos=.5, allow upside down]{\arrowOut};
\draw[color=black, fill=black] (0,0) circle (.2) node[anchor=east]{$\alpha$};
\end{scope}
\begin{scope}[shift={(4,0)}]
\draw[line width=.5pt, stealth-] (0,3) node[anchor=south east]{$i$}.. controls (0,5) and (3,5) .. (3,3);
\draw[line width=.5pt] (0,-3)node[anchor=north east]{$i$}.. controls (0,-5) and (3,-5) .. (3,-3);
\draw[line width=.5pt, -stealth] (3,-3)--(3,0);
\draw[line width=.5pt,] (3,3)--(3,0);
\draw[line width=.5pt,-stealth] (0,3)--(0,0) node[anchor=east]{$x$};
\draw[line width=.5pt,-stealth] (0,0)--(0,-3);
\draw[color=black, fill=black] (0,2) circle (.2) node[anchor=east]{$\phi$};
\draw[color=black, fill=black] (0,-2) circle (.2) node[anchor=east]{$\alpha$};
\end{scope}
\node at (9,0){$=$};
\begin{scope}[shift={(16,0)}]
\draw[draw=none, fill=gray, fill opacity=.2] (-30:5)--(30:5)--(90:5)--(150:5)--(210:5)--(270:5)--(330:5);
\draw[color=black, fill=white] (0,0) circle (3);
\draw[line width=.5pt, color=black] (0,3)--(0,0) node[midway, anchor=west]{$i$} node[sloped, pos=.5, allow upside down]{\arrowOut};
\draw[line width=.5pt, color=black] (0,0)--(0,-3) node[midway, anchor=west]{$x$}  node[sloped, pos=.5, allow upside down]{\arrowOut};
\draw[color=black, fill=black] (0,0) circle (.2) node[anchor=east]{$\phi$};
\end{scope}
\end{tikzpicture}\nonumber\\
&\begin{tikzpicture}[scale=.26]
\node at (-10,0){$\sum_{\alpha}$};
\begin{scope}[shift={(-4,0)}]
\draw[draw=none, fill=gray, fill opacity=.2] (-30:5)--(30:5)--(90:5)--(150:5)--(210:5)--(270:5)--(330:5);
\draw[color=black, fill=white] (0,0) circle (3);
\draw[line width=.5pt, color=black] (0,3)--(0,0) node[midway, anchor=west]{$x$} node[sloped, pos=.5, allow upside down]{\arrowOut};
\draw[line width=.5pt, color=black] (0,0)--(0,-3) node[midway, anchor=west]{$i$}  node[sloped, pos=.5, allow upside down]{\arrowOut};
\draw[color=black, fill=black] (0,0) circle (.2) node[anchor=east]{$\alpha$};
\end{scope}
\begin{scope}[shift={(4,0)}]
\draw[line width=.5pt, stealth-] (0,3) node[anchor=south east]{$i$}.. controls (0,5) and (3,5) .. (3,3);
\draw[line width=.5pt] (0,-3)node[anchor=north east]{$i$}.. controls (0,-5) and (3,-5) .. (3,-3);
\draw[line width=.5pt, -stealth] (3,-3)--(3,0);
\draw[line width=.5pt,] (3,3)--(3,0);
\draw[line width=.5pt,-stealth] (0,3)--(0,0) node[anchor=east]{$x$};
\draw[line width=.5pt,-stealth] (0,0)--(0,-3);
\draw[color=black, fill=black] (0,2) circle (.2) node[anchor=east]{$\alpha$};
\draw[color=black, fill=black] (0,-2) circle (.2) node[anchor=east]{$\psi$};
\end{scope}
\node at (9,0){$=$};
\begin{scope}[shift={(16,0)}]
\draw[draw=none, fill=gray, fill opacity=.2] (-30:5)--(30:5)--(90:5)--(150:5)--(210:5)--(270:5)--(330:5);
\draw[color=black, fill=white] (0,0) circle (3);
\draw[line width=.5pt, color=black] (0,3)--(0,0) node[midway, anchor=west]{$x$} node[sloped, pos=.5, allow upside down]{\arrowOut};
\draw[line width=.5pt, color=black] (0,0)--(0,-3) node[midway, anchor=west]{$i$}  node[sloped, pos=.5, allow upside down]{\arrowOut};
\draw[color=black, fill=black] (0,0) circle (.2) node[anchor=east]{$\psi$};
\end{scope}
\end{tikzpicture}
\end{align}
\end{compactenum}
\end{theorem}
Here, the polygon diagrams on each side of a equation stand for the product of their cyclic evaluations.  It is assumed that all parts of the diagrams  that are not drawn coincide on the left-hand side and right-hand side    and that the object and morphism labels occur nowhere else in the diagrams. The vertex label $\alpha$ stands for the chosen projection and inclusion morphisms. The summations are over simple objects in the sets of representatives and over bases of the morphism spaces.  The functor $F$ in \eqref{pic:gluepoly} and \eqref{pic:closepoly} may  be a composite functor consisting of several lines or  an action functor $c\rhd -$ or $-\lhd d$.

\begin{proof} Identity  \eqref{pic:gluepoly} follows from \eqref{eq:traceidnd}. 
The data at the boundaries and  in the interiors of the first and second diagram in \eqref{pic:gluepoly} that is not drawn combines into morphisms $\mu:n\to F(m)$ and $\nu: F(m)\to n$, respectively. The  diagram on the right-hand side of \eqref{pic:gluepoly}  is then given by the trace of $\nu\circ\mu$. Equation \eqref{eq:traceidnd}
 implies that the cyclic evaluations on the left-hand side and right-hand side  of \eqref{pic:gluepoly} are given by
\begin{align*}
\begin{tikzpicture}[scale=.3]
\node at (-8,0) {$\sum_{\alpha}$};
\begin{scope}[shift={(-3,0)}]
\draw[line width=1.5pt, color=red] (-.7,4)--(.7,4) node[anchor=west]{$n$};
\draw[line width=1.5pt, color=red] (-.7,-4)--(.7,-4) node[anchor=west]{$n$};
\draw[line width=1.5pt, color=red] (0,4)--(0,2);
\draw[line width=1pt, dashed] (0,2).. controls (-2,0).. (0,-2) node[midway, anchor=east]{$F$};
\draw[line width=1.5pt, color=blue] (0,2)--(0,-2) node[midway, anchor=west]{$m$};
\draw[line width=1.5pt, color=red] (0,-4)--(0,-2);
\draw[color=red, fill=red] (0,2) circle (.2) node[anchor=west]{$\mu$};
\draw[color=red, fill=red] (0,-2) circle (.2) node[anchor=west]{$\alpha$};
\end{scope}
\begin{scope}[shift={(3,0)}]
\draw[line width=1.5pt, color=red] (-.7,4)--(.7,4) node[anchor=west]{$n$};
\draw[line width=1.5pt, color=red] (-.7,-4)--(.7,-4) node[anchor=west]{$n$};
\draw[line width=1.5pt, color=red] (0,4)--(0,2);
\draw[line width=1pt, dashed] (0,2).. controls (-2,0).. (0,-2) node[midway, anchor=east]{$F$};
\draw[line width=1.5pt, color=blue] (0,2)--(0,-2) node[midway, anchor=west]{$m$};
\draw[line width=1.5pt, color=red] (0,-4)--(0,-2);
\draw[color=red, fill=red] (0,2) circle (.2) node[anchor=west]{$\alpha$};
\draw[color=red, fill=red] (0,-2) circle (.2) node[anchor=west]{$\nu$};
\end{scope}
\node at (7,0) {$\stackrel{\eqref{pic:semisimpleids}}=$};
\begin{scope}[shift={(12,0)}]
\draw[line width=1.5pt, color=red] (-.7,4)--(.7,4) node[anchor=west]{$n$};
\draw[line width=1.5pt, color=red] (-.7,-4)--(.7,-4) node[anchor=west]{$n$};
\draw[line width=1.5pt, color=red] (0,4)--(0,2);
\draw[line width=1pt, dashed] (0,2).. controls (-2,0).. (0,-2) node[midway, anchor=east]{$F$};
\draw[line width=1.5pt, color=blue] (0,2)--(0,-2) node[midway, anchor=west]{$m$};
\draw[line width=1.5pt, color=red] (0,-4)--(0,-2);
\draw[color=red, fill=red] (0,2) circle (.2) node[anchor=west]{$\mu$};
\draw[color=red, fill=red] (0,-2) circle (.2) node[anchor=west]{$\nu$};
\end{scope}
\end{tikzpicture}
\end{align*}
The proof of \eqref{pic:insert} is analogous. The only difference is that it uses the traces of $\mac$ or $\mad$.
Identity \eqref{pic:closepoly}  is a direct consequence of the second identity in \eqref{eq:projincl}, depicted  in \eqref{pic:semisimpleids}. In this case, 
 the contributions of the diagram in the interior  and  the boundary morphisms that are not drawn  combine into a morphism  $\rho: F(m)\to F(m)$ in both diagrams. The cyclic evaluations of the left-hand side and right-hand side of \eqref{pic:closepoly}  are then given by
\begin{align*}
\begin{tikzpicture}[scale=.3]
\node at (-10,-1) {$\sum_{n,\alpha} \dim(n)$};
\begin{scope}[shift={(-3,0)}]
\draw[line width=1.5pt, color=red] (-2.5,4.4)--(.5,4.5) ;
\draw[line width=1.5pt, color=red] (-2.5,-6.5)--(.5,-6.5) ;
\draw[line width=1.5pt, color=blue] (0,4.5)node[ anchor=north west]{$m$} --(0,2) ;
\draw[line width=1.5pt, color=blue] (0,2)--(0,-1)node[midway, anchor=west]{$m$};
\draw[line width=1.5pt, color=red] (0,-4)--(0,-1) node[midway, anchor=west]{$n$};
\draw[line width=1.5pt, color=blue] (0,-4)--(0,-6.5) node[anchor=south west]{$m$};
\draw[line width=1pt, dashed] (-1.5,4.5)node[anchor=north east]{$F$}--(0,2);
\draw[line width=1pt, dashed]  (0,2).. controls (-2,.5).. (0,-1) node[midway, anchor=east]{$F$};
\draw[line width=1pt, dashed] (-1.5,-6.5) node[anchor=south east]{$F$}--(0,-4);
\draw[color=red, fill=red] (0,2) circle (.2) node[anchor=west]{$\rho$};
\draw[color=red, fill=red] (0,-4) circle (.2) node[anchor=west]{$\alpha$};
\draw[color=red, fill=red] (0,-1) circle (.2) node[anchor=west]{$\alpha$};
\end{scope}
\node at (1,0) {$\stackrel{\eqref{pic:semisimpleids}}=$};
\begin{scope}[shift={(6,0)}]
\draw[line width=1.5pt, color=red] (-2.5,2)--(.5,2) ;
\draw[line width=1.5pt, color=red] (-2.5,-4)--(.5,-4) ;
\draw[line width=1.5pt, color=blue] (0,2)node[ anchor=north west]{$m$} --(0,-4) node[anchor=south west]{$m$};
\draw[line width=1pt, dashed] (-1.5,2) node[anchor=north east]{$F$}--(0,-1);
\draw[line width=1pt, dashed] (-1.5,-4) node[anchor=south east]{$F$}--(0,-1);
\draw[color=red, fill=red] (0,-1) circle (.2) node[anchor=west]{$\rho$};
\end{scope}
\end{tikzpicture}
\end{align*}
Finally, to prove identity \eqref{pic:2gon}, we compute with the definition of  $\dim(n)$ and $\dim\mam$
\begin{align*}
\begin{tikzpicture}[scale=.3]
\begin{scope}[shift={(-4,0)}]
\node at (-10,0){$\sum_{m,n,\alpha}$};
\node at (-6,0){$\dim(m)$};
\node at (-6,-1.5){$\dim(n)$};
\draw[line width=1.5pt, color=blue] (0,3.5)--(0,-3.5) node[midway, anchor=west]{$n$};
\draw[line width=1.5pt, color=blue] (-.5,3.5)--(.5,3.5) node[anchor=west]{$m$};
\draw[line width=1.5pt, color=blue] (-.5,-3.5)--(.5,-3.5) node[anchor=west]{$m$};
\draw[line width=.5pt] (0,2)--(-2,0) node[sloped, pos=.5, allow upside down]{\arrowOut} node[midway, anchor=south east]{$x$};
\draw[line width=.5pt] (-2,0)--(0,-2) node[sloped, pos=.5, allow upside down]{\arrowOut} node[midway, anchor=north east]{$x$};
\draw[fill=black] (-2,0) circle (.2) node[anchor=east]{$\phi$};
\draw[blue,fill=blue] (0,2) circle (.2) node[anchor=west]{$\alpha$};
\draw[blue,fill=blue] (0,-2) circle (.2) node[anchor=west]{$\alpha$};
\end{scope}
\node at (-1.5,0){$\stackrel{\eqref{pic:cyclic trace}}=$};
\begin{scope}[shift={(11,0)}]
\node at (-9,0){$\sum_{m,n,\alpha}$};
\node at (-5,0){$\dim(m)$};
\node at (-5,-1.5){$\dim(n)$};
\draw[line width=1.5pt, color=blue] (0,3.5)--(0,-3.5) node[midway, anchor=west]{$m$};
\draw[line width=1.5pt, color=blue] (-2.5,3.5)--(.5,3.5);
\draw[line width=1.5pt, color=blue] (-2.5,-3.5)--(.5,-3.5);
\draw[line width=.5pt] (-2,3.5)--(-2,2) node[sloped, pos=.5, allow upside down]{\arrowOut} node[midway, anchor=west]{$x$};
\draw[line width=.5pt] (-2,2)--(0,1) node[sloped, pos=.5, allow upside down]{\arrowOut} node[midway, anchor=north east]{$x$};
\draw[line width=.5pt] (0,-1)--(-2,-3.5) node[sloped, pos=.5, allow upside down]{\arrowOut} node[midway, anchor=east]{$x$};
\draw[fill=black] (-2,2) circle (.2) node[anchor=east]{$\phi$};
\draw[blue,fill=blue] (0,1) circle (.2) node[anchor=west]{$\alpha$};
\draw[blue,fill=blue] (0,-1) circle (.2) node[anchor=west]{$\alpha$};
\node at (0,3)[anchor=west, blue]{$n$};
\node at (0,-3)[anchor=west, blue]{$n$};
\end{scope}
\node at (14,0){$\stackrel{\eqref{pic:semisimpleids}}=$};
\begin{scope}[shift={(25,0)}]
\node at (-7,0){$\sum_{n}\dim(n)$};
\draw[line width=1.5pt, color=blue] (0,3.5)--(0,-3.5) node[midway, anchor=west]{$n$};
\draw[line width=1.5pt, color=blue] (-2.5,3.5)--(.5,3.5);
\draw[line width=1.5pt, color=blue] (-2.5,-3.5)--(.5,-3.5);
\draw[line width=.5pt] (-2,3.5)--(-2,0) node[sloped, pos=.5, allow upside down]{\arrowOut} node[midway, anchor=east]{$x$};
\draw[line width=.5pt] (-2,0)--(-2,-3.5) node[sloped, pos=.5, allow upside down]{\arrowOut} node[midway, anchor=east]{$x$};
\draw[fill=black] (-2,0) circle (.2) node[anchor=east]{$\phi$};
\end{scope}
\node at (28,0){$\stackrel{\eqref{pic:cmoduletrace}}=$};
\begin{scope}[shift={(38,0)}]
\node at (-7,0){$\dim\mam$};
\draw[line width=.5pt] (-2,2)node[anchor=west]{$x$}--(-2,-2) node[anchor=west]{$x$} node[sloped, pos=.2, allow upside down]{\arrowOut} node[sloped, pos=.8, allow upside down]{\arrowOut};
\draw[line width=.5pt] (-4,-2)--(-4,2) node[sloped, pos=.5, allow upside down]{\arrowOut};
\draw[line width=.5pt] (-2,2)..controls (-2,3) and (-4,3)..(-4,2);
\draw[line width=.5pt] (-2,-2)..controls (-2,-3) and (-4,-3)..(-4,-2);
\draw[fill=black] (-2,0) circle (.2) node[anchor=west]{$\phi$};
\end{scope}
\end{tikzpicture}\\[-6ex]
\end{align*}
\end{proof}

Note that the first three identities in Theorem \ref{th:polyids} are precisely the elementary  transformations for polygon presentations of surface groups, see for instance  the books by Lee \cite[Ch.~6]{Lee} and by Seifert and Threlfall \cite[Ch.~6.38 and 6.40]{ST}.  Identity \eqref{pic:gluepoly} corresponds to the cutting and pasting operation,  identity \eqref{pic:closepoly} to the folding and unfolding operation and identity \eqref{pic:2gon} to the special case of the folding operation for $S^2$. 
 Identity \eqref{pic:insert} describes the insertion of data from spherical categories into polygon diagrams.

\section{State sum models with defects}
\label{sec:statesum}

\subsection{Triangulated 3-manifolds with defects}
\label{subsec:triangsec}

We consider a compact oriented 3d PL manifold  $M$ with (possibly empty) boundary $\partial M$. Defect structures are assigned to a compact  oriented embedded 2d PL submanifold $D\subset M$ with boundary  $\partial D\subset \partial M$ that is a finite sum of circles. The submanifold $D$ is equipped with a (finite, possibly empty) embedded PL graph $D^1$ with vertex set $D^0$  such that $D^1\cap \partial D\subset D^0$ contains only univalent vertices. 
Connected components of $M\setminus D$ are called {\bf regions} of $M$, 
connected  components of $D$  {\bf defect surfaces}, connected components of $D\setminus D^1$ {\bf defect areas}, connected components of $D^1\setminus D^0$ {\bf defect lines},  elements of $D^0\setminus \partial D$ {\bf defect vertices}.

We equip each defect line and defect surface with an orientation. In pictures we indicate the orientation of a defect surface by a positive normal vector to the surface  and the orientation of a defect line by an arrow on the line.
Regions and  defect areas,   lines and  vertices are labeled with categorical data as follows, see Figure \ref{fig:dectet} for an illustration:
\begin{itemize}
\item Regions of $M$ are assigned spherical fusion categories.
\item Each oriented defect area is assigned a pair $(\mam,\theta)$ of a  $(\mac,\mad)$-bimodule category $\mam$ and a $(\mac,\mad)$-bimodule trace $\theta$ on $\mam$, where $\mac$ and $\mad$ are the spherical fusion categories for  the regions at the tip and the tail of the surface's normal vector.

Reversing the orientation of a defect area amounts to replacing  $\mam$ by the opposite bimodule category $\mam^\#$ from Example \ref{ex:dualcat}.

\item  Each oriented defect line is assigned a $(\mac,\mad)$-bimodule functor $F:\mam\to\man$, where $\mam$ and $\man$ are the $(\mac,\mad)$-bimodule categories 
to the left and to the right of the defect line, viewed in the direction of its orientation and from $\mac$.

Reversing the orientation of a defect line corresponds to replacing  $F:\mam\to\man$ by its left adjoint $F^l:\man\to\mam$. Reversing the orientation of the adjacent defect areas and the orientation of the line corresponds to replacing $\mam,\man$ by  $\mam^\#, \man^\#$ and  $F$ by  $F^\#:\mam^\#\to\man^\#$ from Example \ref{ex:modulefuncs}, 7.

\item A defect vertex whose incident defect lines are labeled by  $(\mac,\mad)$-bimodule functors is assigned a cyclic equivalence class of $(\mac,\mad)$-bimodule natural transformations between them, whose sources and targets are determined by the orientation and the cyclic ordering of the incident defect lines. 

\end{itemize}

Defect areas between two regions labeled both with  $\mac$ may be labeled with  $\mac$ as a bimodule category over itself  and with its trace as a bimodule trace. Such defect areas are called {\bf trivial defect areas}. Similarly, a defect line between two defect areas labeled  with  $(\mam,\theta)$ may be labeled with the identity functor  $F=\id_\mam:\mam\to\mam$ as a $(\mac,\mad)$-bimodule functor. Such defect lines are called {\bf trivial defect lines}. A defect vertex labeled with  $\id_F$ or with a counit $\epsilon^F$, $\epsilon'^F$ or unit $\eta^F, \eta'^F$ of the adjunction $F^l\vdash F$ is called a {\bf trivial defect vertex}. 

A compact oriented 3d PL manifold $M$, possibly with boundary, together with  defects  $D^0\subset D^1\subset D$ is called a {\bf defect 3-manifold}. A defect 3-manifold together with a labeling of the defects with
categorical data is called a {\bf 3-manifold with defect data}.

To define a state sum model for a 3-manifold with defect data, we require a triangulation  of $M$. It must be sufficiently refined to resolve all defect surfaces and must intersect the defect surfaces generically.

\begin{definition} \label{def:gentrans} Let $M$ be a defect 3-manifold and $T$ a triangulation of $M$. 
\begin{compactenum}
\item A triangle in $T$ is called {\bf transversal}, if
\begin{compactitem}
\item each of its sides contains at most one point on a defect surface,
\item its intersection with each defect surface is empty or a line with endpoints on its sides,
\end{compactitem}

and {\bf generic} if
\begin{compactitem}
\item none of its vertices is on a defect surface,
\item none of its sides intersects a defect line, 
\item it does not contain a defect vertex.\\[-2ex]
\end{compactitem}
\item A tetrahedron in $T$ is transversal, if all  of its faces  are transversal and 
 its intersection with each defect surface is either empty or a polygon with vertices on its edges. It is  generic, if all  its faces are generic. \\[-2ex]

\item The  triangulation $T$ is transversal  or  generic if all of its tetrahedra are transversal or generic.
\end{compactenum}
\end{definition}

It follows directly that a generic transversal tetrahedron either  intersects a single defect surface in three edges incident at a common vertex, as in Figure \ref{fig:dectet} (a), 
 intersects a defect surface in two pairs of opposite edges, as  in Figure \ref{fig:dectet} (b), or 
does not  intersect any defect surface at all.

Any transversal triangulation can be deformed into a generic one by slightly perturbing the defect surfaces and 
the defect graphs. By taking  derived subdivisions, see Rourke and Sanderson \cite[Ch.~2]{RS}, and perturbing  defect surfaces and defect graphs one can transform any triangulation into one that is both transversal and generic.  In particular,  one can achieve that the defect manifold is  a subcomplex of  a first derived subdivision.

We orient the edges of a generic transversal triangulation in such a way that 
\begin{compactenum}[(i)]
\item   edges in each triangle do not form a cycle,
\item  in each tetrahedron, the edges that intersect a defect surface are  oriented parallel to its positive normal vector.
\end{compactenum}
This does not restrict generality.  Condition (i) can be achieved by numbering the vertices of the triangulation and orienting each edge in such a way that it points from the vertex with the smaller number to the vertex with the bigger number. This is possible because any PL triangulation is determined by an underlying combinatorial simplicial complex, see for instance 
\cite[Th.~2.11]{RS}.

Given a generic transversal  triangulation that satisfies (i),  reverse the orientation of each edge that intersects a defect surface and is oriented against its positive normal vector.  Each triangle that intersects a defect surface intersects it in exactly two edges.  After this orientation reversal these two edges are  either both incoming or  outgoing at a common vertex and hence cannot be part of a cycle.  The resulting triangulation satisfies (i) and (ii).

A defect 3-manifold  together with an oriented triangulation is called a {\bf  triangulated defect 3-manifold}. A 3-manifold with defect data together with an oriented triangulation is called a {\bf triangulated 3-manifold with defect data}.
They are called {\bf transversal} or {\bf generic} if the triangulation is transversal or generic.

\begin{figure} 
\centering
\def\svgwidth{.75\columnwidth}
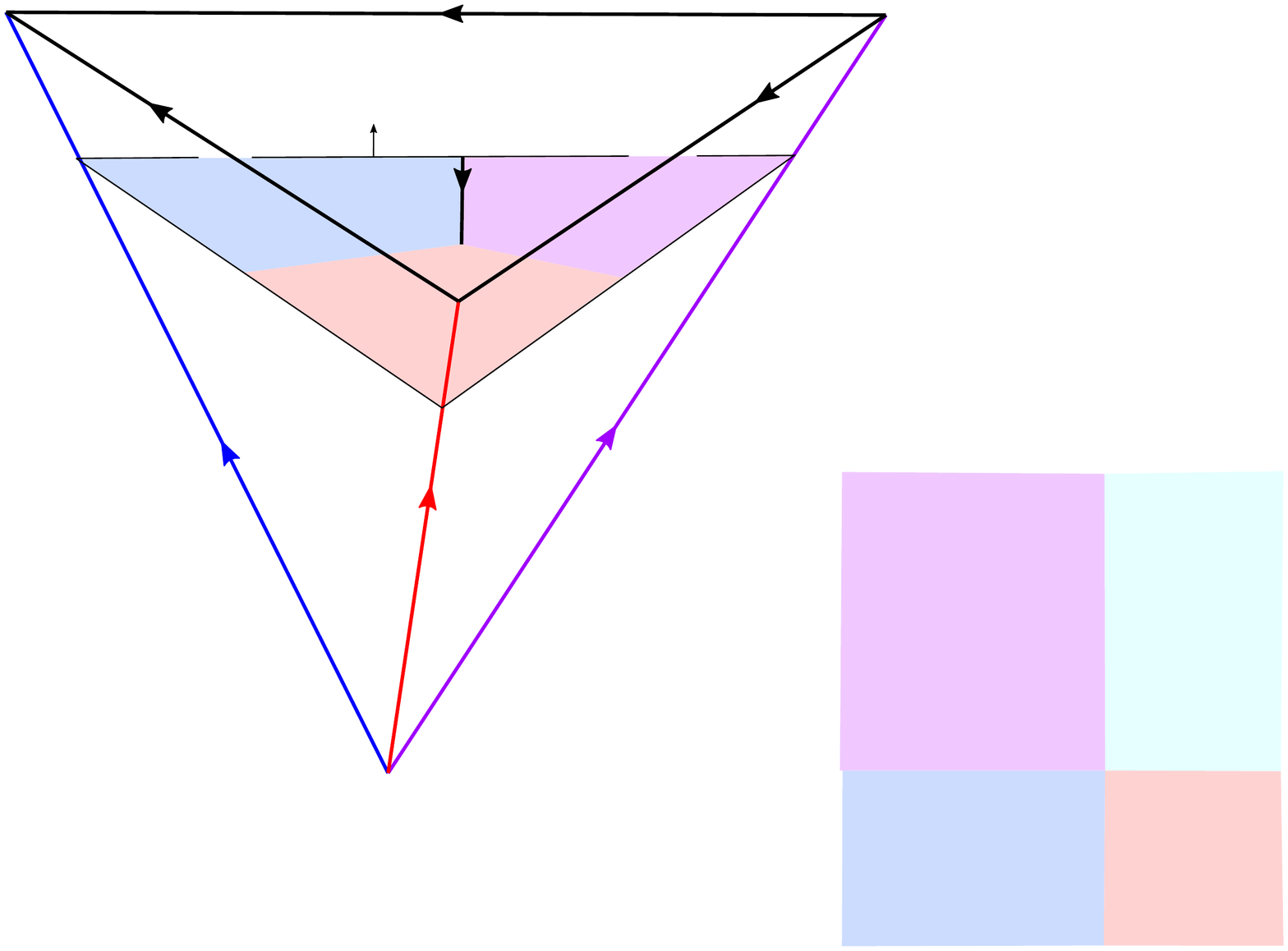
\caption{Generic transversal tetrahedra that intersect a defect surface. \newline
(a) labeling with  $F:\mathcal P\to\mam$, $G:\man\to\mam$, $H:\map\to\man$ and  $\nu:F\Rightarrow GH$,\newline
(b) labeling with  $F:\maq\to \map$, $G:\man\to\maq$, $H:\mam\to\map$, $K:\man\to\mam$ and  $\nu: FG\Rightarrow HK$.
}
\label{fig:dectet}
\end{figure}

\subsection{Labeling of the triangulation}

To define the state sum of a generic  transversal    triangulated 3-manifold with defect data, we assign   objects and morphisms of the associated categories to the edges and triangles of the triangulation. For this,  we fix sets  $I$ of representatives of the isomorphism classes of simple objects and associated projection and inclusion morphisms $p^\alpha_{xi}: x\to i$ and $j^\alpha_{xi}: i\to x$ satisfying \eqref{eq:projincl} 
 for all spherical fusion categories at the regions of $M$ and for all bimodule categories at the defect areas.

\begin{definition} \label{def:label} Let $M$ be generic transversal   triangulated 3-manifold with defect data. A {\bf labeling} of $M$  is an assignment of the following data to the edges of the triangulation:
\begin{compactenum}[(i)]
\item to each 
edge $e$ contained in a region labeled by  $\mac$, a simple object $l(e)\in I_{\mac}$,
\item to each  
edge $e$ that intersects a defect area labeled $\mam$, a  simple object $l(e)\in I_{\mam}$.
\end{compactenum}
A  pair $(M,l)$ of a generic  transversal  triangulated 3-manifold $M$ with defect data  and a labeling $l$ is called a {\bf labeled defect 3-manifold}.
\end{definition}

Note that the orientation of the edges  from Section \ref{subsec:triangsec}  is arbitrary in (i), but fixed by the orientations of the defect  surfaces  in (ii). In (i) we would assign to the edge with the opposite orientation the unique object in $I_\mac$ that represents the isomorphism class of the dual object $l(e)^*$. In (ii) orientation could only be reversed simultaneously for all edges of a tetrahedron  that intersect a  defect surface. This would amount to reversing the orientation of each defect area  and replacing its bimodule category $\mathcal M$ by $\mathcal M^\#$, as explained in Section \ref{subsec:triangsec}.
In this case, we would assign to the edge with the opposite orientation  the same object as an object of  $\mam^\#$.

A labeling of a triangulated manifold with defect data defines a labeling of each triangle in the triangulation. To each labeled triangle, we associate two morphism spaces,
 one for each orientation. 
 For a triangle labeled by objects in $(\mac,\mad)$-bimodule categories $\mam$, $\man$ that intersects a defect line labeled by a bimodule  functor  $F:\mam\to\man$ or  $G:\man\to\mam$, the associated morphisms spaces are 

\begin{align}\label{pic:homspace}
&\begin{tikzpicture}[scale=.4, baseline=(current bounding box.center)]
\draw[line width=1.5pt, color=red] (0,3)--(3,0) node[midway, anchor=south west]{$n\in I_\man$} node[sloped, pos=0.5, allow upside down]{\arrowIn};
\draw[line width=1.5pt, color=blue] (0,3)--(-3,0) node[midway, anchor=south east]{$m\in I_\mam$}node[sloped, pos=0.5, allow upside down]{\arrowIn};
\draw[line width=.5pt] (-3,0)--(3,0) node[midway, anchor=north]{$c\in I_\mac$} node[sloped, pos=0.5, allow upside down]{\arrowOut};
\draw[line width=1pt, stealth-] (0,.8)--(0,1.4) node[midway,anchor=west]{$F$};
\end{tikzpicture}
&
&\begin{tikzpicture}[scale=.4, baseline=(current bounding box.center)]
\draw[line width=1.5pt, color=blue] (0,3)--(3,0) node[midway, anchor=south west]{$m\in I_\mam$} node[sloped, pos=0.5, allow upside down]{\arrowIn};
\draw[line width=1.5pt, color=red] (0,3)--(-3,0) node[midway, anchor=south east]{$n\in I_\man$}node[sloped, pos=0.5, allow upside down]{\arrowIn};
\draw[line width=.5pt] (3,0)--(-3,0) node[midway, anchor=north]{$c\in I_\mac$} node[sloped, pos=0.5, allow upside down]{\arrowOut};
\draw[line width=1pt, -stealth] (0,.8)--(0,1.4) node[midway,anchor=west]{$F$};
\end{tikzpicture}
&
&\begin{array}{ll}
+: \Hom_\man(F(c\rhd m), n)\\
-: \Hom_\man(n, F(c\rhd m))
\end{array}
\\
&\begin{tikzpicture}[scale=.4, baseline=(current bounding box.center)]
\draw[line width=1.5pt, color=red] (0,3)--(3,0) node[midway, anchor=south west]{$n\in I_\man$} node[sloped, pos=0.5, allow upside down]{\arrowIn};
\draw[line width=1.5pt, color=blue] (0,3)--(-3,0) node[midway, anchor=south east]{$m\in I_\mam$}node[sloped, pos=0.5, allow upside down]{\arrowIn};
\draw[line width=.5pt] (-3,0)--(3,0) node[midway, anchor=north]{$c\in I_\mac$} node[sloped, pos=0.5, allow upside down]{\arrowOut};
\draw[line width=1pt, -stealth] (0,.8)--(0,1.4) node[midway,anchor=west]{$G$};
\end{tikzpicture}
&
&
\begin{tikzpicture}[scale=.4, baseline=(current bounding box.center)]
\draw[line width=1.5pt, color=blue] (0,3)--(3,0) node[midway, anchor=south west]{$m\in I_\mam$} node[sloped, pos=0.5, allow upside down]{\arrowIn};
\draw[line width=1.5pt, color=red] (0,3)--(-3,0) node[midway, anchor=south east]{$n\in I_\man$}node[sloped, pos=0.5, allow upside down]{\arrowIn};
\draw[line width=.5pt] (3,0)--(-3,0) node[midway, anchor=north]{$c\in I_\mac$} node[sloped, pos=0.5, allow upside down]{\arrowOut};
\draw[line width=1pt, stealth-] (0,.8)--(0,1.4) node[midway,anchor=west]{$G$};
\end{tikzpicture}
&
&\begin{array}{ll}
+: \Hom_\mam(c\rhd m, G(n))\\
-: \Hom_\mam(G(n), c\rhd m)
\end{array}\nonumber\\
&\begin{tikzpicture}[scale=.4, baseline=(current bounding box.center)]
\draw[line width=1.5pt, color=red] (3,0)--(0,-3) node[midway, anchor=north west]{$n\in I_\man$} node[sloped, pos=0.5, allow upside down]{\arrowIn};
\draw[line width=1.5pt, color=blue] (-3,0)--(0,-3) node[midway, anchor=north east]{$m\in I_\mam$}node[sloped, pos=0.5, allow upside down]{\arrowIn};
\draw[line width=.5pt, gray] (3,0)--(-3,0) node[midway, anchor=south]{$d\in I_\mad$} node[sloped, pos=0.5, allow upside down]{\arrowOut};
\draw[line width=1pt, -stealth] (0,-.8)--(0,-1.4) node[midway,anchor=west]{$F$};
\end{tikzpicture}
&
&\begin{tikzpicture}[scale=.4, baseline=(current bounding box.center)]
\draw[line width=1.5pt, color=blue] (3,0)--(0,-3) node[midway, anchor=north west]{$m\in I_\mam$} node[sloped, pos=0.5, allow upside down]{\arrowIn};
\draw[line width=1.5pt, color=red] (-3,0)--(0,-3) node[midway, anchor=north east]{$n\in I_\man$}node[sloped, pos=0.5, allow upside down]{\arrowIn};
\draw[line width=.5pt, gray] (-3,0)--(3,0) node[midway, anchor=south]{$d\in I_\mad$} node[sloped, pos=0.5, allow upside down]{\arrowOut};
\draw[line width=1pt, stealth-] (0,-.8)--(0,-1.4) node[midway,anchor=west]{$F$};
\end{tikzpicture}
&
&\begin{array}{ll}
+: \Hom_\man(F(m\lhd d),n)\\
-: \Hom_\man(n,F(m\lhd d))\\
\end{array}\nonumber
\\
&\begin{tikzpicture}[scale=.4, baseline=(current bounding box.center)]
\draw[line width=1.5pt, color=red] (3,0)--(0,-3) node[midway, anchor=north west]{$n\in I_\man$} node[sloped, pos=0.5, allow upside down]{\arrowIn};
\draw[line width=1.5pt, color=blue] (-3,0)--(0,-3) node[midway, anchor=north east]{$m\in I_\mam$}node[sloped, pos=0.5, allow upside down]{\arrowIn};
\draw[line width=.5pt, gray] (3,0)--(-3,0) node[midway, anchor=south]{$d\in I_\mad$} node[sloped, pos=0.5, allow upside down]{\arrowOut};
\draw[line width=1pt, stealth-] (0,-.8)--(0,-1.4) node[midway,anchor=west]{$G$};
\end{tikzpicture}
&
&\begin{tikzpicture}[scale=.4, baseline=(current bounding box.center)]
\draw[line width=1.5pt, color=blue] (3,0)--(0,-3) node[midway, anchor=north west]{$m\in I_\mam$} node[sloped, pos=0.5, allow upside down]{\arrowIn};
\draw[line width=1.5pt, color=red] (-3,0)--(0,-3) node[midway, anchor=north east]{$n\in I_\man$}node[sloped, pos=0.5, allow upside down]{\arrowIn};
\draw[line width=.5pt, gray] (-3,0)--(3,0) node[midway, anchor=south]{$d\in I_\mad$} node[sloped, pos=0.5, allow upside down]{\arrowOut};
\draw[line width=1pt, -stealth] (0,-.8)--(0,-1.4) node[midway,anchor=west]{$G$};
\end{tikzpicture}
&
&\begin{array}{ll}
+:\Hom_\mam(m\lhd d, G(n))\\
-:\Hom_\mam(G(n), m\lhd d).
\end{array}\nonumber
\end{align}
Here, the triangles are equipped with the orientation induced by their edge orientation and the embedding in the plane, via the right-hand rule. The first triangle in each line has positive (+) and the second negative (-) orientation. The arrows  indicate the orientation of the defect line relative to the orientation of the triangle. They agree, if the arrow points downwards, and are opposite, if it points upwards. As the orientation of the defect lines is fixed, reversing the orientation of the triangle reverses the  orientation of this arrow.

The morphism spaces are obtained from the triangles as follows: 
The objects in $\mac$ and $\mad$  act on the object of $\mam$ or $\man$ with whom they share a vertex with one incoming and one outgoing edge. Arrows labeled with $\mac$ have this vertex as their starting vertex, arrows labeled by $\mad$ as their target vertex. 
The morphisms go from the object on the left to the one on the right. The functor is applied to the object on the left, if its arrow points downwards, and to the one on the right, if it points upwards.

For $\mam=\man$ and $F=G=\id_\mam$,  the diagrams in the first and the second row and in the third and the fourth row in \eqref{pic:homspace} coincide. Setting $\mam=\man=\mac$, $\rhd=\oo=\lhd$ and $F=G=\id_\mac$  yields the usual assignment of morphisms to triangles for the Turaev-Viro-Barrett-Westbury invariants of a spherical fusion category.

\subsection{Generalised 6j symbols}

We now associate a generalised 6j symbol to each  labeled generic  transversal tetrahedron $t$ with a defect surface. Note that such a defect tetrahedron  is more than a combinatorial tetrahedron with categorical data assigned to each oriented edge. It consists of a combinatorial tetrahedron $t$, an oriented defect polygon $P$,  an assignment of certain edges of $t$ to the vertices  of $P$ and an assignment of the vertices of $t$ to the two regions labeled by $\mac$ and $\mad$. 
  In particular,  a labeled generic  transversal tetrahedron inherits an orientation, because the defect surface is oriented and assigning $\mac$ and $\mad$ to the two sides of the surface specifies a positive normal vector to the surface
  
  \begin{figure} 
\centering
\def\svgwidth{.85\columnwidth}
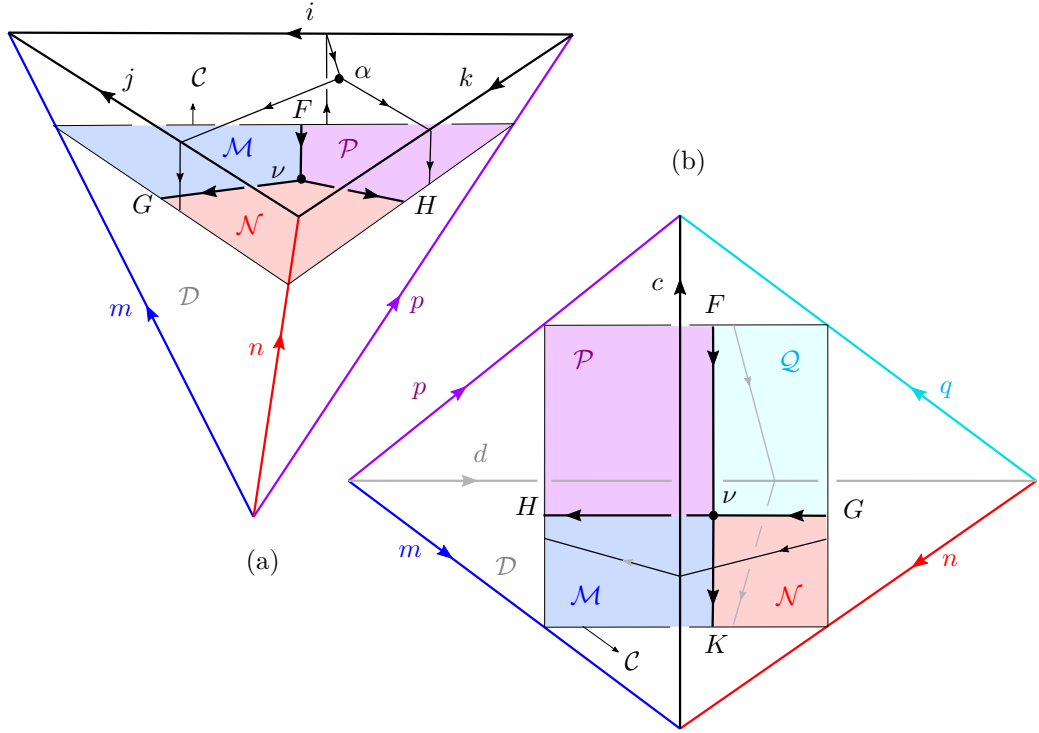
\caption{Projecting the data from spherical categories on the polygon.}
\label{fig:projecting}
\end{figure}
  
   The orientation of $t$ induces an orientation of all of its triangles. To each oriented triangle we assign a morphism in the associated morphism space from \eqref{pic:homspace}. 
To a tetrahedron with these assignments, we  associate a polygon diagram $D$ as follows:

\begin{enumerate}
\item The  defect surface in $t$ is a polygon $P$.  The defect data in $P$, viewed from the region labeled  $\mac$ towards the region labeled $\mad$, defines a cyclic equivalence class of  diagrams for $\Bimod^\theta(\mac,\mad)$. 

\item Pushing the endpoints of defect lines on each side of $P$ to the middle without creating crossings does not change the cyclic equivalence class of the diagram. We thus assume that each side of $P$ is the endpoint of exactly one defect line as  in Figure \ref{fig:dectet}, possibly labeled with  a composite of functors.  

\item Draw  $t$  such that  edges and triangles labeled by data from $\mac$ and $\mad$ lie on the corresponding sides of the polygon $P$ and that edges intersecting $P$ are oriented by the positive normal vector of $P$. 

\item For each edge $e$ of $t$ that does not intersect $P$, draw the dual edge $\bar e$, starting and ending
 on the boundary  $\partial P$ and with its endpoints are displaced slightly from the endpoints of defect lines. They are displaced  towards the starting and target end of $e$, if $e$ is contained in the region for $\mac$ and $\mad$, respectively.
Orient $\bar e$ by duality, as shown in Figure \ref{fig:projecting}, and project it on $P$.

\item Label the vertices of $P$ with the simple objects of the associated edges of $t$ and the projected edges  with the simple objects 
 of their duals in $t$. Label the endpoints of defect lines and the projected vertices in $P$ with the morphisms assigned to the corresponding triangles in $t$. The result is a polygon diagram $D$.  The diagrams for the tetrahedra in Figures \ref{fig:dectet} and \ref{fig:projecting} are  shown in \eqref{fig:diagramproj}.
\end{enumerate}

Note that the word \emph{project} in 3.~is used informally here. For the tetrahedron in Figure \ref{fig:projecting} (a) it means  drawing a trivalent vertex on $P$  whose incident edges are connected with the endpoints on $\partial P$. For the tetrahedron in Figure \ref{fig:projecting} (b) it amounts to drawing two edges that connect  endpoints on opposite sides of $P$. The properties of diagrams for spherical fusion categories and identities \eqref{pic:rm2} to \eqref{pic:bimodulenat} ensure that all ways of doing so yield polygon diagrams with the same evaluation.

Note also that  tetrahedra that do not intersect defect surfaces can be viewed as special cases of Figure \ref{fig:dectet} (a) and (b). 
Each such tetrahedron has one vertex with only outgoing  edges, which can be taken as the lower vertex in Figure \ref{fig:dectet} (a) or as the left vertex in Figure \ref{fig:dectet}  (b). 
Up to  orientation reversal  the tetrahedron then coincides with the one in  Figure \ref{fig:dectet} (a) or (b), if the defect surfaces  are labeled with the spherical fusion category $\mac$  as a $(\mac,\mac)$-bimodule category, with identity functors and with  identity natural transformations. 

Interpreting a combinatorial tetrahedron as a tetrahedron with trivial defects amounts to a choice of orientation: one needs to specify an orientation for the trivial defect surface and to specify its  normal vector. We therefore  treat \emph{oriented} tetrahedra labeled by a spherical fusion category $\mac$  as generic transversal tetrahedra with trivial defect data.  In the following, we assume that all tetrahedra are equipped with an orientation. 

\begin{definition}\label{def:gen6j}
Let $(t,l)$ be a labeled generic transversal tetrahedron  with an assignment $b$ of morphisms in the spaces \eqref{pic:homspace} to its triangles.
The {\bf generalised 6j symbol} $\mathrm{6j}(t,l,b)$  is the cyclic evaluation  of the associated polygon diagram $D$. This defines a linear map
$$
\mathrm{6j}(t): \bigotimes_{f\in\Delta} \Hom(f,l)\to \C, \quad b\mapsto \mathrm{6j}(t,l,b),
$$
where $\Delta$ denotes the set of triangles of $t$,  $\Hom(f,l)$ the  morphism space from \eqref{pic:homspace}  for a triangle $f\in \Delta$ and the tensor product stands for the unordered tensor product of these morphism spaces for all triangles.
\end{definition}

\begin{example}
The generalised 6j symbols for the tetrahedra in Figure \ref{fig:dectet}  are the cyclic evaluations of  
\begin{align}\label{fig:diagramproj}
\begin{tikzpicture}[scale=.35]
\draw[line width=1.5pt, color=blue] (-4,-1)--(0,-6) node[anchor=north]{$m$}--(4,-1);
\draw[line width=1.5pt, color=violet] (-4,-1)--(-8,4) node[anchor=east]{$p$}--(0,4);
\draw[line width=1.5pt, color=red] (0,4)--(8,4)node[anchor=west]{$n$}--(4,-1) ;
\draw[line width=1pt, dashed] (-1,0)--(0,4) node[sloped, pos=0.7, allow upside down]{\arrowIn} node[midway, anchor=south west]{$H$};
\draw[line width=1pt, dashed] (-4,-1)--(-1,0) node[sloped, pos=0.5, allow upside down]{\arrowIn} node[midway, anchor=north]{$F$};
\draw[line width=1pt, dashed] (-1,0)--(4,-1) node[sloped, pos=0.5, allow upside down]{\arrowIn} node[midway, anchor=north]{$G$};
\draw[fill=white, color=white] (-.7,1.4) circle(.3);
\draw[line width=.5pt] (-4,-1)--(-2,2) node[sloped, pos=0.5, allow upside down]{\arrowOut} node[midway, anchor=east]{$i\,$};
\draw[line width=.5pt] (-2,2)--(0,4) node[sloped, pos=0.5, allow upside down]{\arrowOut} node[midway, anchor=east]{$k\,$};
\draw[line width=.5pt] (-2,2)--(4,-1) node[sloped, pos=0.5, allow upside down]{\arrowOut} node[midway, anchor=south]{$j$};
\draw[fill=black, color=black] (-1,0) circle (.2) node[anchor=south east]{$\nu$};
\draw[fill=black, color=black] (-2,2) circle (.2) node[anchor=east]{$\alpha$};
\draw[fill=red, color=red] (0,4) circle (.2) node[anchor=south]{$\gamma$};
\draw[fill=blue, color=blue] (-4,-1) circle (.2) node[anchor=north east]{$\beta$};
\draw[fill=blue, color=blue] (4,-1) circle (.2) node[anchor=north west]{$\delta$};
\end{tikzpicture}
\qquad
\begin{tikzpicture}[scale=.35]
\draw[line width=1.5pt, color=violet] (-6,0)--(-6,6) node[anchor=south east]{$p$}--(0,6);
\draw[line width=1.5pt, color=cyan] (0,6)--(6,6) node[anchor=south west]{$q$}--(6,0);
\draw[line width=1.5pt, color=red] (6,0)--(6,-6) node[anchor=north west]{$n$}--(0,-6);
\draw[line width=1.5pt, color=blue] (0,-6)--(-6,-6) node[anchor=north east]{$m$}--(-6,0);
\draw[line width=.5pt, color=gray] (0,6).. controls (3,0).. (0,-6) node[sloped, pos=0.6, allow upside down]{\arrowOut} node[midway, anchor=north west]{$d$};
\draw[fill=white, color=white] (2.3,0) circle(.3);
\draw[fill=white, color=white] (2,-2) circle(.3);
\draw[line width=1pt, dashed] (0,6)--(0,0) node[sloped, pos=0.6, allow upside down]{\arrowIn} node[midway, anchor=west]{$F$};
\draw[line width=1pt, dashed] (6,0)--(0,0) node[sloped, pos=0.5, allow upside down]{\arrowIn} node[midway, anchor=south]{$G$};
\draw[line width=1pt, dashed] (0,0)--(-6,0) node[sloped, pos=0.5, allow upside down]{\arrowIn} node[midway, anchor=south]{$H$};
\draw[line width=1pt, dashed] (0,0)--(0,-6) node[sloped, pos=0.6, allow upside down]{\arrowIn} node[midway, anchor=west]{$K$};
\draw[fill=white, color=white] (0,-2.3) circle(.3);
\draw[line width=.5pt,] (6,0).. controls (0,-3).. (-6,0) node[sloped, pos=0.6, allow upside down]{\arrowOut} node[midway, anchor=north east]{$c\;\;\;$};
\draw[fill=black, color=black] (0,0) circle (.2) node[anchor=south east]{$\nu$};
\draw[fill=violet, color=violet] (0,6) circle (.2) node[anchor=south]{$\alpha$};
\draw[fill=cyan, color=cyan] (6,0) circle (.2) node[anchor=west]{$\beta$};
\draw[fill=blue, color=blue] (0,-6) circle (.2) node[anchor=north]{$\gamma$};
\draw[fill=violet, color=violet] (-6,0) circle (.2) node[anchor=east]{$\delta$};
\end{tikzpicture}
\end{align}
and given by the diagrams
\begin{align}\label{fig:6jgeneral}
\begin{tikzpicture}[scale=.5, baseline=(current bounding box.center)]
\draw[line width=1.5pt, color=blue] (-1,5)--(1,5) node[anchor=west]{$m$};
\draw[line width=1.5pt, color=blue] (-1,-5)--(1,-5) node[anchor=west]{$m$};
\draw[line width=1.5pt, color=blue] (0,5)--(0,4);
\draw[line width=1.5pt, color=violet] (0,4)--(0,-1)node[midway, anchor=west]{$p$};
\draw[line width=1.5pt, color=red] (0,-1)--(0,-4) node[midway, anchor=west]{$n$};
\draw[line width=1.5pt, color=blue] (0,-4)--(0,-5);
\draw[line width=1pt, color=black, style=dashed] (0,4)--(-5,1)  node[midway, anchor=south east]{$F$};
\draw[line width=.5pt, color=black] (0,4)--(-2,1)  node[midway, anchor=west]{$i$};
\draw[line width=1pt, color=black, style=dashed] (-5,1) --(0,-1)  node[midway, anchor=north]{$H$};
\draw[line width=.5pt, color=black] (-2,1) --(0,-1)  node[midway, anchor=south]{$k$};
\draw[color=white, fill=white] (-1.4,-.5) circle (.4);
\draw[line width=1pt, color=black, style=dashed] (-5,1) --(0,-4)  node[midway, anchor=north east]{$G$};
\draw[line width=.5pt, color=black] (-2,1) --(0,-4)  node[midway, anchor=west]{$j$};
\draw[color=blue, fill=blue] (0,4) circle (.15) node[anchor=west]{$\;\beta$};
\draw[color=red, fill=red] (0,-1) circle (.15) node[anchor=west]{$\;\gamma$};
\draw[color=blue, fill=blue] (0,-4) circle (.15) node[anchor=west]{$\;\delta$};
\draw[ fill=black] (-2,1) circle (.15) node[anchor=east]{$\alpha\;$};
\draw[ fill=black] (-5,1) circle (.15) node[anchor=east]{$\nu\;$};
\end{tikzpicture}
%%%%%%%%%%%%%%%
\qquad\qquad\qquad\qquad
\begin{tikzpicture}[scale=.5, baseline=(current bounding box.center)]
\draw[line width=1.5pt, color=violet] (-1,5)--(1,5) node[anchor=west]{$p$};
\draw[line width=1.5pt, color=violet] (-1,-5)--(1,-5) node[anchor=west]{$p$};
\draw[line width=1.5pt, color=cyan] (0,4)--(0,1) node[midway, anchor=west]{$q$};
\draw[line width=1.5pt, color=red] (0,1)--(0,-1) node[midway, anchor=west]{$n$};
\draw[line width=1.5pt, color=blue] (0,-4)--(0,-1)node[midway, anchor=west]{$m$};
\draw[line width=1.5pt, color=violet] (0,5)--(0,4);
\draw[line width=1.5pt, color=violet] (0,-4)--(0,-5);
\draw[line width=.5pt, color=gray] (0,4) .. controls (-2,2.5)  and (-2,.5).. (0,-1) node[midway, anchor=west]{$d$};
\draw[color=white, fill=white] (-1.1,.83) circle (.4);
\draw[color=white, fill=white] (-.8,0) circle (.4);
\draw[line width=1pt, color=black, style=dashed] (0,4) --(-5,0)  node[midway, anchor=south east]{$F$};
\draw[line width=1pt, color=black, style=dashed] (0,1) --(-5,0) node[midway, anchor=south]{$\;\;G$};
\draw[line width=1pt, color=black, style=dashed] (-5,0) --(0,-1)  node[midway, anchor=north]{$\;\;K$};
\draw[line width=1pt, color=black, style=dashed] (-5,0) --(0,-4) node[midway, anchor=north east]{$H$};
\draw[color=white, fill=white] (-1.2,-.83) circle (.4);
\draw[line width=.5pt, color=black] (0,-4) .. controls (-2,-2.5)  and (-2,-.5).. (0,1) node[midway, anchor=west]{$c$};
\draw[color=violet, fill=violet] (0,4) circle (.15) node[anchor=west]{$\;\alpha$};
\draw[color=cyan, fill=cyan] (0,1) circle (.15) node[anchor=west]{$\;\beta$};
\draw[color=blue, fill=blue] (0,-1) circle (.15) node[anchor=west]{$\;\gamma$};
\draw[color=violet, fill=violet] (0,-4) circle (.15) node[anchor=west]{$\;\delta$};
\draw[ fill=black] (-5,0) circle (.15) node[anchor=east]{$\nu\;$};
\end{tikzpicture}
\end{align}
\end{example}

Definition \ref{def:gen6j} also defines the usual 6j symbols of tetrahedra labeled by a spherical fusion category $\mac$  that occur in the Turaev-Viro-Barrett-Westbury state sum. 
However, they are not defined as 6j symbols of combinatorial tetrahedra, as in \cite{BW}, but as 6j symbols of \emph{oriented} combinatorial tetrahedra. 

The 6j symbols from Definition \ref{def:gen6j} 
encode the natural transformations on the defect polygon and the coherence isomorphisms for the defect data. If the former are trivial, we obtain the following cases:

\begin{itemize}
\item If  $\map=\man=\mam$, $F=G=H=\id_\mam$ and $\nu=\id_{\id_\mam}$ for the tetrahedron in Figure \ref{fig:dectet} (a) and the associated 6j symbol in  \eqref{fig:6jgeneral}, we obtain a number that characterises the isomorphism $c_{j,k,p}: (j\oo k)\rhd p\to j\rhd (k\rhd p)$ from Definition \ref{def:modulecat}. 
It follows from \eqref{eq:projincl} that it is given by
\begin{align}\label{eq:c6j}
\begin{tikzpicture}[scale=.38, baseline=(current bounding box.center)]
\node at (-11,0) {$c_{j,k,p}=\sum_{\substack{{i,m,n}\\{\alpha,\beta,\gamma,\delta}}}$};
\node at (-6,1) {$\dim(i)$};
\node at (-6,0) {$\dim(m)$};
\node at (-6,-1) {$\dim(n)$};
\draw[line width=1.5pt, color=blue] (-1,5)--(1,5) node[anchor=west]{$m$};
\draw[line width=1.5pt, color=blue] (-1,-5)--(1,-5) node[anchor=west]{$m$};
\draw[line width=1.5pt, color=blue] (0,5)--(0,-5);
\node at (0,1.5)[anchor=west, color=blue] {$p$};
\node at (0,-1.5)[anchor=west, color=blue] {$n$};
\draw[line width=.5pt, color=black] (0,3) --(-3,1)  node[midway, anchor=south east]{$i$};
\draw[line width=.5pt, color=black] (-3,1)--(0,0)  node[midway, anchor=south west]{$k$};
\draw[line width=.5pt, color=black] (-3,1)--(0,-3)  node[midway, anchor=north east]{$j$};
\draw[color=blue, fill=blue] (0,3) circle (.2) node[anchor=west]{$\;\beta$};
\draw[color=blue, fill=blue] (0,0) circle (.2) node[anchor=west]{$\;\gamma$};
\draw[color=blue, fill=blue] (0,-3) circle (.2) node[anchor=west]{$\;\delta$};
\draw[color=black, fill=black] (-3,1) circle (.2) node[anchor=east]{$\alpha\;$};
%%%%%%%%
\begin{scope}[shift={(7,0)}]
\draw[line width=.5pt] (-4,5)node[anchor=south]{$j$}--(-3,4);
\draw[line width=.5pt] (-2,5)node[anchor=south]{$k$}--(-3,4);
\draw[line width=.5pt] (-3,4)--(0,1) node[midway, anchor=north east]{$i$};
\draw[line width=.5pt] (0,-1)--(-4,-5) node[anchor=north]{$j$};
\draw[line width=.5pt] (0,-3)--(-2,-5) node[anchor=north]{$k$};
\draw[color=blue, line width=1.5pt] (0,5)node[anchor=south]{$p$}--(0,-5) node[anchor=north]{$p$} node[midway, anchor=west]{$m$};
\node at (0,-2) [color=blue, anchor=west]{$n$};
\draw[fill=black] (-3,4) circle (.2) node[anchor=east]{$\alpha$};
\draw[fill=blue, color=blue] (0,1) circle (.2) node[anchor=west]{$\;\beta$};
\draw[fill=blue, color=blue] (0,-1) circle (.2) node[anchor=west]{$\;\delta$};
\draw[fill=blue, color=blue] (0,-3) circle (.2) node[anchor=west]{$\;\gamma$};
\end{scope}
\end{tikzpicture}
\end{align}
where the sum runs over  simple objects $i\in I_\mac$, $m,p\in I_\mam$ and  bases of the associated  morphism spaces. 

There is an analogous generalised 6j symbol for $\mad$-right module categories, which arises from a tetrahedron analogous to Figure \ref{fig:dectet} (a), but with a vertex in the region labeled by $\mad$ and with the orientations of the coloured edges reversed. 

\item In particular, if  $\mam$ is a spherical fusion category $\mac$ as a left module category over itself, we obtain the usual 6j symbol for the spherical fusion category $\mac$ that encodes the associator on simple objects:
\begin{align}\label{eq:a6j}
\begin{tikzpicture}[scale=.38, baseline=(current bounding box.center)]
\node at (-11,0) {$a_{j,k,p}=\sum_{\substack{{i,m,n}\\{\alpha,\beta,\gamma,\delta}}}$};
\node at (-6,1) {$\dim(i)$};
\node at (-6,0) {$\dim(m)$};
\node at (-6,-1) {$\dim(n)$};
\draw[line width=.5pt] (-1,5)--(1,5) node[anchor=west]{$m$};
\draw[line width=.5pt, ] (-1,-5)--(1,-5) node[anchor=west]{$m$};
\draw[line width=.5pt,] (0,5)--(0,-5);
\node at (0,1.5)[anchor=west, ] {$p$};
\node at (0,-1.5)[anchor=west, ] {$n$};
\draw[line width=.5pt, color=black] (0,3) --(-3,1)  node[midway, anchor=south east]{$i$};
\draw[line width=.5pt, color=black] (-3,1)--(0,0)  node[midway, anchor=south west]{$k$};
\draw[line width=.5pt, color=black] (-3,1)--(0,-3)  node[midway, anchor=north east]{$j$};
\draw[color=black, fill=black] (0,3) circle (.2) node[anchor=west]{$\;\beta$};
\draw[color=black, fill=black] (0,0) circle (.2) node[anchor=west]{$\;\gamma$};
\draw[color=black, fill=black] (0,-3) circle (.2) node[anchor=west]{$\;\delta$};
\draw[color=black, fill=black] (-3,1) circle (.2) node[anchor=east]{$\alpha\;$};
%%%%%%%%
\begin{scope}[shift={(7,0)}]
\draw[line width=.5pt] (-4,5)node[anchor=south]{$j$}--(-3,4);
\draw[line width=.5pt] (-2,5)node[anchor=south]{$k$}--(-3,4);
\draw[line width=.5pt] (-3,4)--(0,1) node[midway, anchor=north east]{$i$};
\draw[line width=.5pt] (0,-1)--(-4,-5) node[anchor=north]{$j$};
\draw[line width=.5pt] (0,-3)--(-2,-5) node[anchor=north]{$k$};
\draw[color=black, line width=.5pt] (0,5)node[anchor=south]{$p$}--(0,-5) node[anchor=north]{$p$} node[midway, anchor=west]{$m$};
\node at (0,-2) [color=black, anchor=west]{$n$};
\draw[fill=black] (-3,4) circle (.2) node[anchor=east]{$\alpha$};
\draw[fill=black, color=black] (0,1) circle (.2) node[anchor=west]{$\;\beta$};
\draw[fill=black, color=black] (0,-1) circle (.2) node[anchor=west]{$\;\delta$};
\draw[fill=black, color=black] (0,-3) circle (.2) node[anchor=west]{$\;\gamma$};
\end{scope}
\end{tikzpicture}
\end{align}

\item Setting $\maq=\map=\man=\mam$,  $F=G=H=K=\id_\mam$ and $\nu=\id_{\id_\mam}$   in Figure \ref{fig:dectet} (b) and  \eqref{fig:6jgeneral}  yields the 6j symbol that characterises the isomorphism $b_{c,n,d}: (c\rhd n)\lhd d\to c\rhd (n\lhd d)$ from Definition \ref{def:modulecat}: 
\begin{align}\label{eq:6jb}
\begin{tikzpicture}[scale=.38, baseline=(current bounding box.center)]
\node at (-11,0) {$b_{c,n,d}=\sum_{\substack{{m,p,q}\\{\alpha,\beta,\gamma,\delta}}}$};
\node at (-6,1) {$\dim(m)$};
\node at (-6,0) {$\dim(p)$};
\node at (-6,-1) {$\dim(q)$};
\draw[line width=1.5pt, color=blue] (-1,5)--(1,5) node[anchor=west]{$p$};
\draw[line width=1.5pt, color=blue] (-1,-5)--(1,-5) node[anchor=west]{$p$};
\draw[line width=1.5pt, color=blue] (0,5)--(0,-5);
\draw[line width=.5pt, color=gray] (0,3)  .. controls (-3,2) and (-3,0) .. (0,-1)  node[midway, anchor=south east]{$d$};
\draw[color=white, fill=white] (-2,0) circle (.4);
\draw[line width=.5pt, color=black] (0,1)  .. controls (-3,0) and (-3,-2) .. (0,-3)  node[midway, anchor=south east]{$c$};
\node at (0,2)[anchor=west, color=blue] {$q$};
\node at (0,0)[anchor=west, color=blue] {$n$};
\node at (0,-2)[anchor=west, color=blue] {$m$};
\draw[color=blue, fill=blue] (0,3) circle (.2) node[anchor=west]{$\;\alpha$};
\draw[color=blue, fill=blue] (0,1) circle (.2) node[anchor=west]{$\;\beta$};
\draw[color=blue, fill=blue] (0,-1) circle (.2) node[anchor=west]{$\;\gamma$};
\draw[color=blue, fill=blue] (0,-3) circle (.2) node[anchor=west]{$\;\delta$};
%%%%%%%
\begin{scope}[shift={(7,0)}]
\draw[line width=.5pt, color=gray] (-4,5) node[anchor=south]{$d$}--(0,1);
\draw[line width=.5pt, color=black] (-2,5) node[anchor=south]{$c$}--(0,3);
\draw[line width=1.5pt, color=blue] (0,5) node[anchor=south]{$n$}--(0,-5)  node[anchor=north]{$n$};
\node at (0,2) [anchor=west, color=blue]{$q$};
\node at (0,0) [anchor=west, color=blue]{$p$};
\node at (0,-2) [anchor=west, color=blue]{$m$};
\draw[line width=.5pt, color=black] (-4,-5) node[anchor=north]{$c$}--(0,-1);
\draw[line width=.5pt, color=gray] (-2,-5) node[anchor=north]{$d$}--(0,-3);
\draw[color=blue, fill=blue] (0,3) circle (.2) node[anchor=west]{$\;\beta$};
\draw[color=blue, fill=blue] (0,1) circle (.2)node[anchor=west]{$\;\alpha$};
\draw[color=blue, fill=blue] (0,-1) circle (.2)node[anchor=west]{$\;\delta$};
\draw[color=blue, fill=blue] (0,-3) circle (.2)node[anchor=west]{$\;\gamma$};
\end{scope}
\end{tikzpicture}
\end{align}

\item  In particular, for $\mam=\mac=\mad$ identity  \eqref{eq:6jb}  yields  the usual 6j symbol for a spherical fusion category $\mac$ that encodes its associator, but in  
 a  diagrammatic representation different from \eqref{eq:a6j}
  \begin{align}\label{eq:6ja2}
\begin{tikzpicture}[scale=.38, baseline=(current bounding box.center)]
\node at (-11,0) {$a_{c,n,d}=\sum_{\substack{{m,p,q}\\{\alpha,\beta,\gamma,\delta}}}$};
\node at (-6,1) {$\dim(m)$};
\node at (-6,0) {$\dim(p)$};
\node at (-6,-1) {$\dim(q)$};
\draw[line width=.5pt, color=black] (-1,5)--(1,5) node[anchor=west]{$p$};
\draw[line width=.5pt, color=black] (-1,-5)--(1,-5) node[anchor=west]{$p$};
\draw[line width=.5pt, color=black] (0,5)--(0,-5);
\draw[line width=.5pt, color=black] (0,3)  .. controls (-3,2) and (-3,0) .. (0,-1)  node[midway, anchor=south east]{$d$};
\draw[color=white, fill=white] (-2,0) circle (.4);
\draw[line width=.5pt, color=black] (0,1)  .. controls (-3,0) and (-3,-2) .. (0,-3)  node[midway, anchor=south east]{$c$};
\node at (0,2)[anchor=west, color=black] {$q$};
\node at (0,0)[anchor=west, color=black] {$n$};
\node at (0,-2)[anchor=west, color=black] {$m$};
\draw[color=black, fill=black] (0,3) circle (.2) node[anchor=west]{$\;\alpha$};
\draw[color=black, fill=black] (0,1) circle (.2) node[anchor=west]{$\;\beta$};
\draw[color=black, fill=black] (0,-1) circle (.2) node[anchor=west]{$\;\gamma$};
\draw[color=black, fill=black] (0,-3) circle (.2) node[anchor=west]{$\;\delta$};
%%%%%%%
\begin{scope}[shift={(7,0)}]
\draw[line width=.5pt, color=black] (-4,5) node[anchor=south]{$d$}--(0,1);
\draw[line width=.5pt, color=black] (-2,5) node[anchor=south]{$c$}--(0,3);
\draw[line width=.5pt, color=black] (0,5) node[anchor=south]{$n$}--(0,-5)  node[anchor=north]{$n$};
\node at (0,2) [anchor=west, color=black]{$q$};
\node at (0,0) [anchor=west, color=black]{$p$};
\node at (0,-2) [anchor=west, color=black]{$m$};
\draw[line width=.5pt, color=black] (-4,-5) node[anchor=north]{$c$}--(0,-1);
\draw[line width=.5pt, color=black] (-2,-5) node[anchor=north]{$d$}--(0,-3);
\draw[color=black, fill=black] (0,3) circle (.2) node[anchor=west]{$\;\beta$};
\draw[color=black, fill=black] (0,1) circle (.2)node[anchor=west]{$\;\alpha$};
\draw[color=black, fill=black] (0,-1) circle (.2)node[anchor=west]{$\;\delta$};
\draw[color=black, fill=black] (0,-3) circle (.2)node[anchor=west]{$\;\gamma$};
\end{scope}
\end{tikzpicture}
\end{align}

\item If we set $\maq=\map$, $\man=\mam$,  $G=H: \mam\to\map$, $F=\id_\map$, $K=\id_\mam$, $\nu=\id_{G}$ and $c=e$ in Figures \ref{fig:dectet} (b) and \eqref{fig:6jgeneral}, we obtain the 6j symbol for the isomorphism $t^G_{n,d}: G(n)\lhd d\to G(n\lhd d)$ from Definition \ref{def:modulefunc}:
\begin{align}
\begin{tikzpicture}[scale=.38, baseline=(current bounding box.center)]
\node at (-11,0) {$t^G_{n,d}=\sum_{\substack{{m,p,q}\\{\alpha,\beta,\gamma,\delta}}}$};
\node at (-6,1) {$\dim(m)$};
\node at (-6,0) {$\dim(p)$};
\node at (-6,-1) {$\dim(q)$};
\draw[line width=1.5pt, color=violet] (-1,5)--(1,5) node[anchor=west]{$p$};
\draw[line width=1.5pt, color=violet] (-1,-5)--(1,-5) node[anchor=west]{$p$};
\draw[line width=1.5pt, color=violet] (0,5)--(0,1);
\draw[line width=1.5pt, color=violet] (0,-3)--(0,-5);
\draw[line width=1.5pt, color=blue] (0,1)--(0,-3);
\draw[line width=.5pt, color=gray] (0,3)  .. controls (-3,2) and (-3,0) .. (0,-1)  node[midway, anchor=south east]{$d$};
\draw[color=white, fill=white] (-2,0) circle (.4);
\draw[line width=1pt, color=black, dashed] (0,1)  .. controls (-3,0) and (-3,-2) .. (0,-3)  node[midway, anchor=south east]{$G$};
\node at (0,2)[anchor=west, color=violet] {$q$};
\node at (0,0)[anchor=west, color=blue] {$n$};
\node at (0,-2)[anchor=west, color=blue] {$m$};
\draw[color=violet, fill=violet] (0,3) circle (.2) node[anchor=west]{$\;\alpha$};
\draw[color=violet, fill=violet] (0,1) circle (.2) node[anchor=west]{$\;\beta$};
\draw[color=blue, fill=blue] (0,-1) circle (.2) node[anchor=west]{$\;\gamma$};
\draw[color=violet, fill=violet] (0,-3) circle (.2) node[anchor=west]{$\;\delta$};
%%%%%%%%
\begin{scope}[shift={(7,0)}]
\draw[line width=.5pt, color=gray] (-4,5) node[anchor=south]{$d$}--(0,1);
\draw[line width=1pt, color=black, dashed] (-2,5) node[anchor=south]{$G$}--(0,3);
\draw[line width=1.5pt, color=blue] (0,5) node[anchor=south]{$n$}--(0,3);  
\draw[line width=1.5pt, color=blue] (0,-1)-- (0,-5)  node[anchor=north]{$n$};
\draw[line width=1.5pt, color=violet] (0,-1)-- (0,3) ;
\node at (0,2) [anchor=west, color=violet]{$q$};
\node at (0,0) [anchor=west, color=violet]{$p$};
\node at (0,-2) [anchor=west, color=blue]{$m$};
\draw[line width=1pt, color=black, dashed] (-4,-5) node[anchor=north]{$G$}--(0,-1);
\draw[line width=.5pt, color=gray] (-2,-5) node[anchor=north]{$d$}--(0,-3);
\draw[color=violet, fill=violet] (0,3) circle (.2) node[anchor=west]{$\;\beta$};
\draw[color=violet, fill=violet] (0,1) circle (.2)node[anchor=west]{$\;\alpha$};
\draw[color=violet, fill=violet] (0,-1) circle (.2)node[anchor=west]{$\;\delta$};
\draw[color=blue, fill=blue] (0,-3) circle (.2)node[anchor=west]{$\;\gamma$};
\end{scope}
\end{tikzpicture}
\end{align}

\item Setting instead $\map=\mam$, $\maq=\man$, $F=K: \man\to\mam$, $G=\id_\man$, $H=\id_\mam$, $\nu=\id_{F}$ and $d=e$  in Figures \ref{fig:dectet} (b) and  \eqref{fig:6jgeneral}  yields the 6j symbol for the isomorphism $s^F_{c,n}: F(c\rhd n)\to c\rhd F(n)$ from Definition \ref{def:modulefunc}: 
\begin{align}
\begin{tikzpicture}[scale=.38, baseline=(current bounding box.center)]
\node at (-11,0) {$s^F_{c,n}=\sum_{\substack{{m,p,q}\\{\alpha,\beta,\gamma,\delta}}}$};
\node at (-6,1) {$\dim(m)$};
\node at (-6,0) {$\dim(p)$};
\node at (-6,-1) {$\dim(q)$};
\draw[line width=1.5pt, color=blue] (-1,5)--(1,5) node[anchor=west]{$p$};
\draw[line width=1.5pt, color=blue] (-1,-5)--(1,-5) node[anchor=west]{$p$};
\draw[line width=1.5pt, color=blue] (0,5)--(0,3);
\draw[line width=1.5pt, color=blue] (0,-1)--(0,-5);
\draw[line width=1.5pt, color=red] (0,-1)--(0,3);
\draw[line width=1pt, dashed] (0,3)  .. controls (-3,2) and (-3,0) .. (0,-1)  node[midway, anchor=south east]{$F$};
\draw[color=white, fill=white] (-2,0) circle (.4);
\draw[line width=.5pt, color=black] (0,1)  .. controls (-3,0) and (-3,-2) .. (0,-3)  node[midway, anchor=south east]{$c$};
\node at (0,2)[anchor=west, color=red] {$q$};
\node at (0,0)[anchor=west, color=red] {$n$};
\node at (0,-2)[anchor=west, color=blue] {$m$};
\draw[color=blue, fill=blue] (0,3) circle (.2) node[anchor=west]{$\;\alpha$};
\draw[color=red, fill=red] (0,1) circle (.2) node[anchor=west]{$\;\beta$};
\draw[color=blue, fill=blue] (0,-1) circle (.2) node[anchor=west]{$\;\gamma$};
\draw[color=blue, fill=blue] (0,-3) circle (.2) node[anchor=west]{$\;\delta$};
%%%%%%%%%%%%%
\begin{scope}[shift={(7,0)}]
\draw[line width=1pt, dashed] (-4,5) node[anchor=south]{$F$}--(0,1);
\draw[line width=.5pt, color=black] (-2,5) node[anchor=south]{$c$}--(0,3);
\draw[line width=1.5pt, color=red] (0,5) node[anchor=south]{$n$}--(0,1);
\draw[line width=1.5pt, color=red] (0,-3)--(0,-5)  node[anchor=north]{$n$};
\draw[line width=1.5pt, color=blue] (0,-3)--(0,1);
\node at (0,2) [anchor=west, color=red]{$q$};
\node at (0,0) [anchor=west, color=blue]{$p$};
\node at (0,-2) [anchor=west, color=blue]{$m$};
\draw[line width=.5pt, color=black] (-4,-5) node[anchor=north]{$c$}--(0,-1);
\draw[line width=1pt, dashed] (-2,-5) node[anchor=north]{$F$}--(0,-3);
\draw[color=red, fill=red] (0,3) circle (.2) node[anchor=west]{$\;\beta$};
\draw[color=blue, fill=blue] (0,1) circle (.2)node[anchor=west]{$\;\alpha$};
\draw[color=blue, fill=blue] (0,-1) circle (.2)node[anchor=west]{$\;\delta$};
\draw[color=blue, fill=blue] (0,-3) circle (.2)node[anchor=west]{$\;\gamma$};
\end{scope}
\end{tikzpicture}
\end{align}
\end{itemize}
In all cases the diagrammatic expressions for the inverses of the coherence isomorphisms 
are obtained  by reflecting  the diagrams on a horizontal axis. This corresponds to evaluating the associated polygon diagrams counterclockwise and to an orientation reversal of the associated defect surface. This is possible, because all of the bimodule natural transformations in these tetrahedra are bimodule natural isomorphisms and hence their directions can be reversed, see the discussion before Definition \ref{def:oppoly}.

We now investigate how the 6j symbols depend on the orientations of  edges  labeled by objects  in spherical fusion categories. 
Recall that reversing the orientation of such an edge does not affect  the orientations of the adjacent  triangles and tetrahedra. 
Although it induces an odd permutation of the vertex numbers, which may be viewed as an orientation reversal for the associated \emph{combinatorial} tetrahedron, the orientations in our formalism are fixed as part of the data. 

Reversing the orientation of such an edge  and replacing its label by its dual induces natural isomorphisms between the morphism spaces for the adjacent triangles. 
 
 For a $\mac$-module functor  $F:\mam\to\man$ and the first  morphism space from \eqref{pic:homspace}, we obtain
 the following isomorphism
that  is natural in $c\in \Ob\mac$, $m\in\Ob \mam$ and $n\in \Ob\man$
\begin{align}\label{eq:reverse edge}
&\begin{tikzpicture}[scale=.45, baseline=(current bounding box.center)]
\draw[line width=1.5pt, color=red] (0,3)--(3,0) node[midway, anchor=south west]{$n$} node[sloped, pos=0.5, allow upside down]{\arrowIn};
\draw[line width=1.5pt, color=blue] (0,3)--(-3,0) node[midway, anchor=south east]{$m$}node[sloped, pos=0.5, allow upside down]{\arrowIn};
\draw[line width=.5pt] (-3,0)--(3,0) node[midway, anchor=north]{$c$} node[sloped, pos=0.5, allow upside down]{\arrowOut};
\draw[line width=1pt, stealth-] (0,.8)--(0,1.4) node[midway,anchor=west]{$F$};
\node at (0,-2) {$\Hom_\man(F(c\rhd m), n)$};
\begin{scope}[shift={(7,0)}]
\draw[line width=1.5pt, color=blue] (0,2) node[anchor=south]{$m$}--(0,0);
\draw[line width=1.5pt, color=red] (0,-2) node[anchor=north]{$n$}--(0,0);
\draw[line width=.5pt, color=black] (-1,2) node[anchor=south]{$c$}--(0,0);
\draw[line width=1pt, color=black, dashed] (-2,2) node[anchor=south]{$F$}--(0,0);
\draw[fill=red, color=red] (0,0) circle (.2) node[anchor=west]{$\alpha$};
 \end{scope}
\end{tikzpicture}
&
&\longrightarrow
&
&\begin{tikzpicture}[scale=.45, baseline=(current bounding box.center)]
\draw[line width=1.5pt, color=red] (0,3)--(3,0) node[midway, anchor=south west]{$n$} node[sloped, pos=0.5, allow upside down]{\arrowIn};
\draw[line width=1.5pt, color=blue] (0,3)--(-3,0) node[midway, anchor=south east]{$m$}node[sloped, pos=0.5, allow upside down]{\arrowIn};
\draw[line width=.5pt] (3,0)--(-3,0) node[midway, anchor=north]{$c^*$} node[sloped, pos=0.5, allow upside down]{\arrowOut};
\draw[line width=1pt, stealth-] (0,.8)--(0,1.4) node[midway,anchor=west]{$F$};
\node at (0,-2) {$\Hom_\man(F(m), c^*\rhd n)$};
\begin{scope}[shift={(8,0)}]
\draw[line width=1.5pt, color=blue] (0,2) node[anchor=south]{$m$}--(0,0);
\draw[line width=1.5pt, color=red] (0,-2) node[anchor=north]{$n$}--(0,0);
\draw[line width=1pt, color=black, dashed] (-2,2) node[anchor=south]{$F$}--(0,0);
\draw[color=white, fill=white]  (-1.4,1.4) circle (.4);
\draw[line width=.5pt, color=black] (0,0).. controls (-.5,2) and (-2,2).. (-2,0);
\draw[line width=.5pt, color=black] (-2,-2) node[anchor=north]{$c$} -- (-2,0) node[sloped, pos=0.5, allow upside down]{\arrowOut} ;
\draw[fill=red, color=red] (0,0) circle (.2) node[anchor=west]{$\alpha$};
 \end{scope}
\end{tikzpicture}\\
&\begin{tikzpicture}[scale=.45, baseline=(current bounding box.center)]
\draw[line width=1.5pt, color=red] (0,3)--(3,0) node[midway, anchor=south west]{$n$} node[sloped, pos=0.5, allow upside down]{\arrowIn};
\draw[line width=1.5pt, color=blue] (0,3)--(-3,0) node[midway, anchor=south east]{$m$}node[sloped, pos=0.5, allow upside down]{\arrowIn};
\draw[line width=.5pt] (-3,0)--(3,0) node[midway, anchor=north]{$c$} node[sloped, pos=0.5, allow upside down]{\arrowOut};
\draw[line width=1pt, stealth-] (0,.8)--(0,1.4) node[midway,anchor=west]{$F$};
\node at (0,-2) {$\Hom_\man(F(c\rhd m), n)$};
\begin{scope}[shift={(7,0)}]
\draw[line width=1.5pt, color=blue] (0,2) node[anchor=south]{$m$}--(0,0);
\draw[line width=1.5pt, color=red] (0,-2) node[anchor=north]{$n$}--(0,0);
\draw[line width=1pt, color=black, dashed] (-2,2) node[anchor=south]{$F$}--(0,0);
\draw[color=white, fill=white] (-1,1) circle (.3);
\draw[line width=.5pt, color=black] (-1,2) node[anchor=south]{$c$}--(-1,0);
\draw[line width=.5pt, color=black] (-1,0) .. controls  (-1,-1) and (-.25, -1) .. (0,0) node[sloped, pos=0.8, allow upside down]{\arrowOut};
\draw[fill=red, color=red] (0,0) circle (.2) node[anchor=west]{$\beta$};
 \end{scope}
\end{tikzpicture}
&
&\longleftarrow
&
&\begin{tikzpicture}[scale=.45, baseline=(current bounding box.center)]
\draw[line width=1.5pt, color=red] (0,3)--(3,0) node[midway, anchor=south west]{$n$} node[sloped, pos=0.5, allow upside down]{\arrowIn};
\draw[line width=1.5pt, color=blue] (0,3)--(-3,0) node[midway, anchor=south east]{$m$}node[sloped, pos=0.5, allow upside down]{\arrowIn};
\draw[line width=.5pt] (3,0)--(-3,0) node[midway, anchor=north]{$c^*$} node[sloped, pos=0.5, allow upside down]{\arrowOut};
\draw[line width=1pt, stealth-] (0,.8)--(0,1.4) node[midway,anchor=west]{$F$};
\node at (0,-2) {$\Hom_\man(F(m), c^*\rhd n)$};
\begin{scope}[shift={(8,0)}]
\draw[line width=1.5pt, color=blue] (0,2) node[anchor=south]{$m$}--(0,0);
\draw[line width=1.5pt, color=red] (0,-2) node[anchor=north]{$n$}--(0,0);
\draw[line width=1pt, color=black, dashed] (-2,2) node[anchor=south]{$F$}--(0,0);
\draw[line width=.5pt, color=black] (-2,-2) node[anchor=north]{$c$} -- (0,0) node[sloped, pos=0.5, allow upside down]{\arrowOut} ;
\draw[fill=red, color=red] (0,0) circle (.2) node[anchor=west]{$\beta$};
 \end{scope}
\end{tikzpicture}\nonumber
\end{align}
For a $\mac$-module functor  $G:\man\to\mam$ and the third  morphism space from  \eqref{pic:homspace} we have the isomorphism 
\begin{align}\label{eq:reverse edge2}
&\begin{tikzpicture}[scale=.45, baseline=(current bounding box.center)]
\draw[line width=1.5pt, color=red] (0,3)--(3,0) node[midway, anchor=south west]{$n$} node[sloped, pos=0.5, allow upside down]{\arrowIn};
\draw[line width=1.5pt, color=blue] (0,3)--(-3,0) node[midway, anchor=south east]{$m$}node[sloped, pos=0.5, allow upside down]{\arrowIn};
\draw[line width=.5pt] (-3,0)--(3,0) node[midway, anchor=north]{$c$} node[sloped, pos=0.5, allow upside down]{\arrowOut};
\draw[line width=1pt, -stealth] (0,.8)--(0,1.4) node[midway,anchor=west]{$G$};
\node at (0,-2) {$\Hom_\man(c\rhd m, G(n))$};
\begin{scope}[shift={(7,0)}]
\draw[line width=1.5pt, color=blue] (0,2) node[anchor=south]{$m$}--(0,0);
\draw[line width=1.5pt, color=red] (0,-2) node[anchor=north]{$n$}--(0,0);
\draw[line width=.5pt, color=black] (-2,2) node[anchor=south]{$c$}--(0,0);
\draw[line width=1pt, color=black, dashed] (-2,-2) node[anchor=north]{$G$}--(0,0);
\draw[fill=red, color=red] (0,0) circle (.2) node[anchor=west]{$\alpha$};
 \end{scope}
\end{tikzpicture}
&
&\longrightarrow
&
&\begin{tikzpicture}[scale=.45, baseline=(current bounding box.center)]
\draw[line width=1.5pt, color=red] (0,3)--(3,0) node[midway, anchor=south west]{$n$} node[sloped, pos=0.5, allow upside down]{\arrowIn};
\draw[line width=1.5pt, color=blue] (0,3)--(-3,0) node[midway, anchor=south east]{$m$}node[sloped, pos=0.5, allow upside down]{\arrowIn};
\draw[line width=.5pt] (3,0)--(-3,0) node[midway, anchor=north]{$c^*$} node[sloped, pos=0.5, allow upside down]{\arrowOut};
\draw[line width=1pt, -stealth] (0,.8)--(0,1.4) node[midway,anchor=west]{$G$};
\node at (0,-2) {$\Hom_\man(m, G(c^*\rhd n))$};
\begin{scope}[shift={(8,0)}]
\draw[line width=1.5pt, color=blue] (0,2) node[anchor=south]{$m$}--(0,0);
\draw[line width=1.5pt, color=red] (0,-2) node[anchor=north]{$n$}--(0,0);
\draw[line width=1pt, color=black, dashed] (-2,-2) node[anchor=north]{$G$}--(0,0);
\draw[color=white, fill=white]  (-1,-1) circle (.4);
\draw[line width=.5pt, color=black] (-1,0).. controls (-1,1) and (-.25,1).. (0,0);
\draw[line width=.5pt, color=black] (-1,-2) node[anchor=north]{$c$} -- (-1,0) node[sloped, pos=0.7, allow upside down]{\arrowOut} ;
\draw[fill=red, color=red] (0,0) circle (.2) node[anchor=west]{$\alpha$};
 \end{scope}
\end{tikzpicture}\\
&\begin{tikzpicture}[scale=.45, baseline=(current bounding box.center)]
\draw[line width=1.5pt, color=red] (0,3)--(3,0) node[midway, anchor=south west]{$n$} node[sloped, pos=0.5, allow upside down]{\arrowIn};
\draw[line width=1.5pt, color=blue] (0,3)--(-3,0) node[midway, anchor=south east]{$m$}node[sloped, pos=0.5, allow upside down]{\arrowIn};
\draw[line width=.5pt] (-3,0)--(3,0) node[midway, anchor=north]{$c$} node[sloped, pos=0.5, allow upside down]{\arrowOut};
\draw[line width=1pt, -stealth] (0,.8)--(0,1.4) node[midway,anchor=west]{$G$};
\node at (0,-2) {$\Hom_\man(c\rhd m, G( n))$};
\begin{scope}[shift={(7,0)}]
\draw[line width=1.5pt, color=blue] (0,2) node[anchor=south]{$m$}--(0,0);
\draw[line width=1.5pt, color=red] (0,-2) node[anchor=north]{$n$}--(0,0);
\draw[line width=1pt, color=black, dashed] (-2,-2) node[anchor=north]{$G$}--(0,0);
\draw[color=white, fill=white] (-1.5,-1.5) circle (.3);
\draw[line width=.5pt, color=black] (-2,2) node[anchor=south]{$c$}--(-2,0);% ;
\draw[line width=.5pt, color=black] (-2,0) .. controls  (-2,-2) and (-.5, -2) .. (0,0) node[sloped, pos=0.8, allow upside down]{\arrowOut};
\draw[fill=red, color=red] (0,0) circle (.2) node[anchor=west]{$\beta$};
 \end{scope}
\end{tikzpicture}
&
&\longleftarrow
&
&\begin{tikzpicture}[scale=.45, baseline=(current bounding box.center)]
\draw[line width=1.5pt, color=red] (0,3)--(3,0) node[midway, anchor=south west]{$n$} node[sloped, pos=0.5, allow upside down]{\arrowIn};
\draw[line width=1.5pt, color=blue] (0,3)--(-3,0) node[midway, anchor=south east]{$m$}node[sloped, pos=0.5, allow upside down]{\arrowIn};
\draw[line width=.5pt] (3,0)--(-3,0) node[midway, anchor=north]{$c^*$} node[sloped, pos=0.5, allow upside down]{\arrowOut};
\draw[line width=1pt, -stealth] (0,.8)--(0,1.4) node[midway,anchor=west]{$G$};
\node at (0,-2) {$\Hom_\man(m, G(c^*\rhd n))$};
\begin{scope}[shift={(8,0)}]
\draw[line width=1.5pt, color=blue] (0,2) node[anchor=south]{$m$}--(0,0);
\draw[line width=1.5pt, color=red] (0,-2) node[anchor=north]{$n$}--(0,0);
\draw[line width=1pt, color=black, dashed] (-2,-2) node[anchor=north]{$G$}--(0,0);
\draw[line width=.5pt, color=black] (-1,-2) node[anchor=north]{$c$} -- (0,0) node[sloped, pos=0.5, allow upside down]{\arrowOut} ;
\draw[fill=red, color=red] (0,0) circle (.2) node[anchor=west]{$\beta$};
 \end{scope}
\end{tikzpicture}\nonumber.
\end{align}
The isomorphisms for the other  morphism spaces  from \eqref{pic:homspace}
are analogous, just that the black line labeled $c$ may be replaced by a grey line labeled $d$ that crosses under  the dashed lines for the functors. 

With these transformations of the morphism spaces one finds that reversing the orientation of an edge labeled by a spherical fusion category does not affect the generalised 6j symbols of defect tetrahedra.

\begin{lemma}\label{lem:sphlemma}  The 6j symbol of  a labeled generic transversal tetrahedron is invariant under reversing an edge labeled with an object $c\in I_\mac$ or $d\in I_\mad$, replacing its label with the dual and transforming the morphisms at its triangles according to \eqref{eq:reverse edge} and \eqref{eq:reverse edge2}.
\end{lemma}

\begin{proof} This follows by a direct diagrammatic computation from  \eqref{eq:reverse edge} and \eqref{eq:reverse edge2},  together with identity \eqref{pic:snake}  and the second  identity in \eqref{eq:functortrpivot} for the trace, applied to the black and grey lines of the diagram. It also uses identities \eqref{pic:rm2} to \eqref{pic:bimodulenat}  that encode the  Reidemeister moves and  the naturality conditions for the over- and undercrossings. 
For the first diagram in \eqref{fig:6jgeneral} this yields 
\begin{align*}
\begin{tikzpicture}[scale=.5, baseline=(current bounding box.center)]
\draw[line width=1.5pt, color=blue] (-1,5)--(1,5) node[anchor=west]{$m$};
\draw[line width=1.5pt, color=blue] (-1,-5)--(1,-5) node[anchor=west]{$m$};
\draw[line width=1.5pt, color=blue] (0,5)--(0,4);
\draw[line width=1.5pt, color=violet] (0,4)--(0,-1)node[midway, anchor=west]{$p$};
\draw[line width=1.5pt, color=red] (0,-1)--(0,-4) node[midway, anchor=west]{$n$};
\draw[line width=1.5pt, color=blue] (0,-4)--(0,-5);
\draw[line width=1pt, color=black, style=dashed] (0,4)--(-5,1)  node[midway, anchor=south east]{$F$};
\draw[line width=.5pt, color=black] (0,4)--(-2,1)  node[midway, anchor=west]{$i$};
\draw[line width=1pt, color=black, style=dashed] (-5,1) --(0,-1)  node[midway, anchor=north]{$H$};
\draw[line width=.5pt, color=black] (-2,1) --(0,-1)  node[midway, anchor=south]{$k$};
\draw[color=white, fill=white] (-1.4,-.5) circle (.4);
\draw[line width=1pt, color=black, style=dashed] (-5,1) --(0,-4)  node[midway, anchor=north east]{$G$};
\draw[line width=.5pt, color=black] (-2,1) --(0,-4)  node[midway, anchor=east]{$j$};
\draw[color=blue, fill=blue] (0,4) circle (.15) node[anchor=west]{$\;\beta$};
\draw[color=red, fill=red] (0,-1) circle (.15) node[anchor=west]{$\;\gamma$};
\draw[color=blue, fill=blue] (0,-4) circle (.15) node[anchor=west]{$\;\delta$};
\draw[ fill=black] (-2,1) circle (.15) node[anchor=east]{$\alpha\;$};
\draw[ fill=black] (-5,1) circle (.15) node[anchor=east]{$\nu\;$};
\end{tikzpicture}
\stackrel{\eqref{pic:snake}}=
%%%%%%%%%%%%%%%%%
\begin{tikzpicture}[scale=.5, baseline=(current bounding box.center)]
\draw[line width=1.5pt, color=blue] (-1,5)--(1,5) node[anchor=west]{$m$};
\draw[line width=1.5pt, color=blue] (-1,-5)--(1,-5) node[anchor=west]{$m$};
\draw[line width=1.5pt, color=blue] (0,5)--(0,4);
\draw[line width=1.5pt, color=violet] (0,4)--(0,-1)node[midway, anchor=west]{$p$};
\draw[line width=1.5pt, color=red] (0,-1)--(0,-4) node[midway, anchor=west]{$n$};
\draw[line width=1.5pt, color=blue] (0,-4)--(0,-5);
\draw[line width=1pt, color=black, style=dashed] (0,4)--(-5,1)  node[midway, anchor=south east]{$F$};
\draw[line width=.5pt, color=black] (0,4)--(-2,1)  node[midway, anchor=west]{$i$};
\draw[line width=1pt, color=black, style=dashed] (-5,1) --(0,-1)  node[midway, anchor=north]{$H$};
\draw[color=white, fill=white] (-1.4,-.5) circle (.4);
\draw[color=white, fill=white] (-.8,-.8) circle (.4);
\draw[line width=1pt, color=black, style=dashed] (-5,1) --(0,-4)  node[midway, anchor=north east]{$G$};
\draw[line width=.5pt, color=black] (-2,1) --(0,-4)  node[midway, anchor=east]{$j$};
\draw[line width=.5pt, color=black] (0,-1) .. controls (-1,-3) and (-.7,3)..(-2,1) node[sloped, pos=0.5, allow upside down]{\arrowOut} node[midway, anchor=west]{$k$};
\draw[color=blue, fill=blue] (0,4) circle (.15) node[anchor=west]{$\;\beta$};
\draw[color=red, fill=red] (0,-1) circle (.15) node[anchor=west]{$\;\gamma$};
\draw[color=blue, fill=blue] (0,-4) circle (.15) node[anchor=west]{$\;\delta$};
\draw[ fill=black] (-2,1) circle (.15) node[anchor=east]{$\alpha\;$};
\draw[ fill=black] (-5,1) circle (.15) node[anchor=east]{$\nu\;$};
\end{tikzpicture}
\stackrel{\eqref{pic:snake},\eqref{eq:functortrpivot}}=
\begin{tikzpicture}[scale=.5, baseline=(current bounding box.center)]
\draw[line width=1.5pt, color=blue] (-3.5,5)--(1,5) node[anchor=west]{$m$};
\draw[line width=1.5pt, color=blue] (-3.5,-5)--(1,-5) node[anchor=west]{$m$};
\draw[line width=1.5pt, color=blue] (0,5)--(0,4);
\draw[line width=1.5pt, color=violet] (0,4)--(0,-1)node[midway, anchor=west]{$p$};
\draw[line width=1.5pt, color=red] (0,-1)--(0,-4) node[midway, anchor=west]{$n$};
\draw[line width=1.5pt, color=blue] (0,-4)--(0,-5);
\draw[line width=1pt, color=black, style=dashed] (0,4)--(-5,1)  node[pos=.2, anchor=south]{$F$};
\draw[line width=.5pt, color=black] (0,4)--(-2,1)  node[midway, anchor=west]{$i$};
\draw[line width=1pt, color=black, style=dashed] (-5,1) --(0,-1)  node[pos=.7, anchor=north]{$H$};
\draw[line width=.5pt, color=black] (-2,1) --(0,-1)  node[pos=.7, anchor=south]{$k$};
\draw[line width=1pt, color=black, style=dashed] (-5,1) --(0,-4)  node[pos=.7, anchor=north east]{$G$};
\draw[color=white, fill=white] (-3,2) circle (.4);
\draw[color=white, fill=white] (-2.2,-2) circle (.4);
\draw[line width=.5pt, color=black] (-2,1) .. controls (-2,0) and (-3,0).. (-3,1);
\draw[line width=.5pt, color=black] (-3,1) --(-3,5) node[sloped, pos=0.8, allow upside down]{\arrowOut}  node[pos=.8, anchor=east]{$j$};
\draw[line width=.5pt, color=black] (-3,-5) --(-3,-4) node[sloped, pos=0.5, allow upside down]{\arrowOut}  node[midway, anchor=east]{$j$};
\draw[line width=.5pt, color=black] (-3,-4) .. controls (-3,-1) and (-.5,-1).. (0,-4);
\draw[color=blue, fill=blue] (0,4) circle (.15) node[anchor=west]{$\;\beta$};
\draw[color=red, fill=red] (0,-1) circle (.15) node[anchor=west]{$\;\gamma$};
\draw[color=blue, fill=blue] (0,-4) circle (.15) node[anchor=west]{$\;\delta$};
\draw[ fill=black] (-2,1) circle (.15) node[anchor=west]{$\;\alpha$};
\draw[ fill=black] (-5,1) circle (.15) node[anchor=east]{$\nu\;$};
\end{tikzpicture}
\end{align*}
\begin{align*}
\stackrel{\eqref{pic:snake}}=
\begin{tikzpicture}[scale=.5, baseline=(current bounding box.center)]
\draw[line width=1.5pt, color=blue] (-3.5,5)--(1,5) node[anchor=west]{$m$};
\draw[line width=1.5pt, color=blue] (-3.5,-5)--(1,-5) node[anchor=west]{$m$};
\draw[line width=1.5pt, color=blue] (0,5)--(0,4);
\draw[line width=1.5pt, color=violet] (0,4)--(0,-1)node[midway, anchor=west]{$p$};
\draw[line width=1.5pt, color=red] (0,-1)--(0,-4) node[midway, anchor=west]{$n$};
\draw[line width=1.5pt, color=blue] (0,-4)--(0,-5);
\draw[line width=1pt, color=black, style=dashed] (0,4)--(-5,1)  node[pos=.35, anchor=south]{$F$};
\draw[line width=1pt, color=black, style=dashed] (-5,1) --(0,-1)  node[pos=.7, anchor=north]{$H$};
\draw[line width=.5pt, color=black] (-2,1) --(0,-1)  node[pos=.7, anchor=south]{$k$};
\draw[line width=1pt, color=black, style=dashed] (-5,1) --(0,-4)  node[pos=.7, anchor=north east]{$G$};
\draw[color=white, fill=white] (-3,2) circle (.4);
\draw[color=white, fill=white] (-2.2,-2) circle (.4);
\draw[color=white, fill=white] (-1,3.2) circle (.4);
\draw[color=white, fill=white] (-.6,3.6) circle (.4);
\draw[line width=.5pt, color=black] (-2,1) .. controls (-2,0) and (-3,0).. (-3,1);
\draw[line width=.5pt, color=black] (-3,1) --(-3,5) node[sloped, pos=0.8, allow upside down]{\arrowOut}  node[pos=.8, anchor=east]{$j$};
\draw[line width=.5pt, color=black] (-3,-5) --(-3,-4) node[sloped, pos=0.5, allow upside down]{\arrowOut}  node[midway, anchor=east]{$j$};
\draw[line width=.5pt, color=black] (-3,-4) .. controls (-3,-1) and (-.5,-1).. (0,-4);
\draw[line width=.5pt, color=black] (0,4) .. controls (-.4,2) and (-.75,6).. (-1,3)  node[sloped, pos=0.5, allow upside down]{\arrowOut};
\draw[line width=.5pt, color=black] (-1,3)--(-1,1)   node[midway, anchor=west]{$i$};
\draw[line width=.5pt, color=black] (-1,1) .. controls (-1.25,0) and (-1.75,3).. (-2,1)  node[sloped, pos=0.5, allow upside down]{\arrowOut};
\draw[color=blue, fill=blue] (0,4) circle (.15) node[anchor=west]{$\;\beta$};
\draw[color=red, fill=red] (0,-1) circle (.15) node[anchor=west]{$\;\gamma$};
\draw[color=blue, fill=blue] (0,-4) circle (.15) node[anchor=west]{$\;\delta$};
\draw[ fill=black] (-2,1) circle (.15) node[anchor=east]{$\alpha\,$};
\draw[ fill=black] (-5,1) circle (.15) node[anchor=east]{$\nu\;$};
\end{tikzpicture}
\stackrel{\eqref{pic:snake},\eqref{eq:functortrpivot}}=
\begin{tikzpicture}[scale=.5, baseline=(current bounding box.center)]
\draw[line width=1.5pt, color=blue] (-4,5)--(1,5) node[anchor=west]{$m$};
\draw[line width=1.5pt, color=blue] (-4,-5)--(1,-5) node[anchor=west]{$m$};
\draw[line width=1.5pt, color=blue] (0,5)--(0,4);
\draw[line width=1.5pt, color=violet] (0,4)--(0,-1)node[midway, anchor=west]{$p$};
\draw[line width=1.5pt, color=red] (0,-1)--(0,-4) node[midway, anchor=west]{$n$};
\draw[line width=1.5pt, color=blue] (0,-4)--(0,-5);
\draw[line width=1pt, color=black, style=dashed] (0,4)--(-5,1)  node[pos=.35, anchor=south]{$F$};
\draw[line width=1pt, color=black, style=dashed] (-5,1) --(0,-1)  node[pos=.7, anchor=north]{$H$};
\draw[line width=1pt, color=black, style=dashed] (-5,1) --(0,-4)  node[pos=.7, anchor=north east]{$G$};
\draw[color=white, fill=white] (-3,2) circle (.4);
\draw[color=white, fill=white] (-3.5,1.6) circle (.4);
\draw[color=white, fill=white] (-2.2,-2) circle (.4);
\draw[color=white, fill=white] (-1,3.2) circle (.4);
\draw[color=white, fill=white] (-.6,3.6) circle (.4);
\draw[color=white, fill=white] (-3.2,-1.2) circle (.4);
\draw[color=white, fill=white] (-2.5,0) circle (.4);
\draw[line width=.5pt, color=black] (-2,1) .. controls (-2,.5) and (-3,.5).. (-3,1);
\draw[line width=.5pt, color=black] (-3,1) --(-3,5) node[sloped, pos=0.8, allow upside down]{\arrowOut}  node[pos=.8, anchor=west]{$j$};
\draw[line width=.5pt, color=black] (-3,-5) --(-3,-4) node[sloped, pos=0.5, allow upside down]{\arrowOut}  node[midway, anchor=west]{$j$};
\draw[line width=.5pt, color=black] (-3,-4) .. controls (-3,-1) and (-.5,-1).. (0,-4);
\draw[line width=.5pt, color=black] (-2,1) .. controls (-2,0) and (-3.5,0).. (-3.5,1.5);
\draw[line width=.5pt, color=black] (-3.5,1.5) --(-3.5,5) node[sloped, pos=0.8, allow upside down]{\arrowOut}  node[pos=.8, anchor=east]{$k$};
\draw[line width=.5pt, color=black] (-3.5,-5) --(-3.5,-3) node[sloped, pos=0.5, allow upside down]{\arrowOut}  node[midway, anchor=east]{$k$};
\draw[line width=.5pt, color=black] (0,-1) .. controls (-1,1) and (-3.5,1).. (-3.5,-3);
\draw[line width=.5pt, color=black] (0,4) .. controls (-.4,2) and (-.75,6).. (-1,3)  node[sloped, pos=0.5, allow upside down]{\arrowOut};
\draw[line width=.5pt, color=black] (-1,3)--(-1,1)   node[midway, anchor=west]{$i$};
\draw[line width=.5pt, color=black] (-1,1) .. controls (-1.25,0) and (-1.75,3).. (-2,1)  node[sloped, pos=0.5, allow upside down]{\arrowOut};
\draw[color=blue, fill=blue] (0,4) circle (.15) node[anchor=west]{$\;\beta$};
\draw[color=red, fill=red] (0,-1) circle (.15) node[anchor=west]{$\;\gamma$};
\draw[color=blue, fill=blue] (0,-4) circle (.15) node[anchor=west]{$\;\delta$};
\draw[ fill=black] (-2,1) circle (.15) node[anchor=east]{$\alpha\,$};
\draw[ fill=black] (-5,1) circle (.15) node[anchor=east]{$\nu\;$};
\end{tikzpicture}
\stackrel{\eqref{pic:snake}}=
\begin{tikzpicture}[scale=.5, baseline=(current bounding box.center)]
\draw[line width=1.5pt, color=blue] (-4,5)--(1,5) node[anchor=west]{$m$};
\draw[line width=1.5pt, color=blue] (-4,-5)--(1,-5) node[anchor=west]{$m$};
\draw[line width=1.5pt, color=blue] (0,5)--(0,4);
\draw[line width=1.5pt, color=violet] (0,4)--(0,-1)node[midway, anchor=west]{$p$};
\draw[line width=1.5pt, color=red] (0,-1)--(0,-4) node[midway, anchor=west]{$n$};
\draw[line width=1.5pt, color=blue] (0,-4)--(0,-5);
\draw[line width=1pt, color=black, style=dashed] (0,4)--(-5,1)  node[pos=.35, anchor=south]{$F$};
\draw[line width=1pt, color=black, style=dashed] (-5,1) --(0,-1)  node[pos=.7, anchor=north]{$H$};
\draw[line width=1pt, color=black, style=dashed] (-5,1) --(0,-4)  node[pos=.7, anchor=north east]{$G$};
\draw[color=white, fill=white] (-3,2) circle (.4);
\draw[color=white, fill=white] (-3.5,1.6) circle (.4);
\draw[color=white, fill=white] (-2.2,-2) circle (.4);
\draw[color=white, fill=white] (-3.2,-1.2) circle (.4);
\draw[color=white, fill=white] (-2.5,0) circle (.4);
\draw[line width=.5pt, color=black] (-2,1) .. controls (-2,.5) and (-3,.5).. (-3,1);
\draw[line width=.5pt, color=black] (-3,1) --(-3,5) node[sloped, pos=0.8, allow upside down]{\arrowOut}  node[pos=.8, anchor=west]{$j$};
\draw[line width=.5pt, color=black] (-3,-5) --(-3,-4) node[sloped, pos=0.5, allow upside down]{\arrowOut}  node[midway, anchor=west]{$j$};
\draw[line width=.5pt, color=black] (-3,-4) .. controls (-3,-1) and (-.5,-1).. (0,-4);
\draw[line width=.5pt, color=black] (-2,1) .. controls (-2,0) and (-3.5,0).. (-3.5,1.5);
\draw[line width=.5pt, color=black] (-3.5,1.5) --(-3.5,5) node[sloped, pos=0.8, allow upside down]{\arrowOut}  node[pos=.8, anchor=east]{$k$};
\draw[line width=.5pt, color=black] (-3.5,-5) --(-3.5,-3) node[sloped, pos=0.5, allow upside down]{\arrowOut}  node[midway, anchor=east]{$k$};
\draw[line width=.5pt, color=black] (0,-1) .. controls (-1,1) and (-3.5,1).. (-3.5,-3);
\draw[line width=.5pt, color=black] (0,4)--(-2,1)   node[midway, anchor=west]{$i$};
\draw[color=blue, fill=blue] (0,4) circle (.15) node[anchor=west]{$\;\beta$};
\draw[color=red, fill=red] (0,-1) circle (.15) node[anchor=west]{$\;\gamma$};
\draw[color=blue, fill=blue] (0,-4) circle (.15) node[anchor=west]{$\;\delta$};
\draw[ fill=black] (-2,1) circle (.15) node[anchor=east]{$\alpha\,$};
\draw[ fill=black] (-5,1) circle (.15) node[anchor=east]{$\nu\;$};
\end{tikzpicture}
\end{align*}
and for the second diagram in \eqref{fig:6jgeneral}
\begin{align*}
\begin{tikzpicture}[scale=.5, baseline=(current bounding box.center)]
\draw[line width=1.5pt, color=violet] (-1,5)--(1,5) node[anchor=west]{$p$};
\draw[line width=1.5pt, color=violet] (-1,-5)--(1,-5) node[anchor=west]{$p$};
\draw[line width=1.5pt, color=cyan] (0,4)--(0,1) node[midway, anchor=west]{$q$};
\draw[line width=1.5pt, color=red] (0,1)--(0,-1) node[midway, anchor=west]{$n$};
\draw[line width=1.5pt, color=blue] (0,-4)--(0,-1)node[midway, anchor=west]{$m$};
\draw[line width=1.5pt, color=violet] (0,5)--(0,4);
\draw[line width=1.5pt, color=violet] (0,-4)--(0,-5);
\draw[line width=.5pt, color=gray] (0,4) .. controls (-2,2.5)  and (-2,.5).. (0,-1) node[midway, anchor=west]{$d$};
\draw[color=white, fill=white] (-1.1,.83) circle (.4);
\draw[color=white, fill=white] (-.8,0) circle (.4);
\draw[line width=1pt, color=black, style=dashed] (0,4) --(-5,0)  node[midway, anchor=south east]{$F$};
\draw[line width=1pt, color=black, style=dashed] (0,1) --(-5,0) node[midway, anchor=south]{$\;\;G$};
\draw[line width=1pt, color=black, style=dashed] (-5,0) --(0,-1)  node[midway, anchor=north]{$\;\;K$};
\draw[line width=1pt, color=black, style=dashed] (-5,0) --(0,-4) node[midway, anchor=north east]{$H$};
\draw[color=white, fill=white] (-1.2,-.83) circle (.4);
\draw[line width=.5pt, color=black] (0,-4) .. controls (-2,-2.5)  and (-2,-.5).. (0,1) node[midway, anchor=west]{$c$};
\draw[color=violet, fill=violet] (0,4) circle (.15) node[anchor=west]{$\;\alpha$};
\draw[color=cyan, fill=cyan] (0,1) circle (.15) node[anchor=west]{$\;\beta$};
\draw[color=blue, fill=blue] (0,-1) circle (.15) node[anchor=west]{$\;\gamma$};
\draw[color=violet, fill=violet] (0,-4) circle (.15) node[anchor=west]{$\;\delta$};
\draw[ fill=black] (-5,0) circle (.15) node[anchor=east]{$\nu\;$};
\end{tikzpicture}
\stackrel{\eqref{pic:snake},\eqref{eq:functortrpivot}}=
\begin{tikzpicture}[scale=.5, baseline=(current bounding box.center)]
\draw[line width=1.5pt, color=violet] (-2,5)--(1,5) node[anchor=west]{$p$};
\draw[line width=1.5pt, color=violet] (-2,-5)--(1,-5) node[anchor=west]{$p$};
\draw[line width=1.5pt, color=cyan] (0,4)--(0,1) node[midway, anchor=west]{$q$};
\draw[line width=1.5pt, color=red] (0,1)--(0,-1) node[midway, anchor=west]{$n$};
\draw[line width=1.5pt, color=blue] (0,-4)--(0,-1)node[midway, anchor=west]{$m$};
\draw[line width=1.5pt, color=violet] (0,5)--(0,4);
\draw[line width=1.5pt, color=violet] (0,-4)--(0,-5);
\draw[line width=.5pt, color=gray] (0,4) .. controls (-.5,2)  and (-1.5,2).. (-1.5,4);
\draw[line width=.5pt, color=gray] (-1.5,4)--(-1.5,5) node[sloped, pos=0.2, allow upside down]{\arrowOut} node[pos=.2, anchor=east]{$d$};
\draw[line width=.5pt, color=gray] (0,-1) .. controls (-.5,1)  and (-1.5,1).. (-1.5,-1);
\draw[line width=.5pt, color=gray] (-1.5,-5)--(-1.5,-1) node[sloped, pos=0.2, allow upside down]{\arrowOut} node[pos=.2, anchor=east]{$d$};
\draw[color=white, fill=white] (-1.4, 3) circle (.4);
\draw[color=white, fill=white] (-1.5,-2.8) circle (.4);
\draw[color=white, fill=white] (-.5,.2) circle (.4);
\draw[line width=1pt, color=black, style=dashed] (0,4) --(-5,0)  node[midway, anchor=south east]{$F$};
\draw[line width=1pt, color=black, style=dashed] (0,1) --(-5,0) node[midway, anchor=south]{$\;\;G$};
\draw[line width=1pt, color=black, style=dashed] (-5,0) --(0,-1)  node[midway, anchor=north]{$\;\;K$};
\draw[line width=1pt, color=black, style=dashed] (-5,0) --(0,-4) node[midway, anchor=north east]{$H$};
\draw[color=white, fill=white] (-1.2,-.83) circle (.4);
\draw[line width=.5pt, color=black] (0,-4) .. controls (-1.5,-2.5)  and (-1.5,-.5).. (0,1) node[midway, anchor=west]{$c$};
\draw[color=violet, fill=violet] (0,4) circle (.15) node[anchor=west]{$\;\alpha$};
\draw[color=cyan, fill=cyan] (0,1) circle (.15) node[anchor=west]{$\;\beta$};
\draw[color=blue, fill=blue] (0,-1) circle (.15) node[anchor=west]{$\;\gamma$};
\draw[color=violet, fill=violet] (0,-4) circle (.15) node[anchor=west]{$\;\delta$};
\draw[ fill=black] (-5,0) circle (.15) node[anchor=east]{$\nu\;$};
\end{tikzpicture}
\end{align*}
\begin{align*}
\stackrel{\eqref{pic:snake},\eqref{eq:functortrpivot}}=
\begin{tikzpicture}[scale=.5, baseline=(current bounding box.center)]
\draw[line width=1.5pt, color=violet] (-3,5)--(1,5) node[anchor=west]{$p$};
\draw[line width=1.5pt, color=violet] (-3,-5)--(1,-5) node[anchor=west]{$p$};
\draw[line width=1.5pt, color=cyan] (0,4)--(0,1) node[midway, anchor=west]{$q$};
\draw[line width=1.5pt, color=red] (0,1)--(0,-1) node[midway, anchor=west]{$n$};
\draw[line width=1.5pt, color=blue] (0,-4)--(0,-1)node[midway, anchor=west]{$m$};
\draw[line width=1.5pt, color=violet] (0,5)--(0,4);
\draw[line width=1.5pt, color=violet] (0,-4)--(0,-5);
\draw[line width=.5pt, color=gray] (0,4) .. controls (-.5,2)  and (-1.5,2).. (-1.5,4);
\draw[line width=.5pt, color=gray] (-1.5,4)--(-1.5,5) node[sloped, pos=0.2, allow upside down]{\arrowOut} node[pos=.2, anchor=west]{$d$};
\draw[line width=.5pt, color=gray] (0,-1) .. controls (-.5,.2)  and (-1.5,.2).. (-1.5,-1);
\draw[line width=.5pt, color=gray] (-1.5,-5)--(-1.5,-1) node[sloped, pos=0.2, allow upside down]{\arrowOut} node[pos=.2, anchor=west]{$d$};
\draw[color=white, fill=white] (-1.4, 3) circle (.4);
\draw[color=white, fill=white] (-1.5,-2.8) circle (.4);
\draw[color=white, fill=white] (-1.5,-.8) circle (.4);
\draw[color=white, fill=white] (-1.5,-1.8) circle (.2);
\draw[line width=1pt, color=black, style=dashed] (0,4) --(-5,0)  node[pos=.6, anchor=south east]{$F$};
\draw[line width=1pt, color=black, style=dashed] (0,1) --(-5,0) node[pos=.3, anchor=south]{$\;\;G$};
\draw[line width=1pt, color=black, style=dashed] (-5,0) --(0,-1)  node[midway, anchor=north]{$\;\;K$};
\draw[line width=1pt, color=black, style=dashed] (-5,0) --(0,-4) node[midway, anchor=north east]{$H$};
\draw[color=white, fill=white] (-2.2,-2.5) circle (.4);
\draw[color=white, fill=white] (-2,.8) circle (.4);
\draw[color=white, fill=white] (-2.5,2) circle (.4);
\draw[line width=.5pt, color=black] (0,-4) .. controls (-.5,-1)  and (-2.5,-1).. (-2.5,-4);
\draw[line width=.5pt, color=black] (0,1) .. controls (-.5,0)  and (-2.5,0).. (-2.5,1.5);
\draw[line width=.5pt, color=black] (-2.5,-5)--(-2.5,-4) node[sloped, pos=0.8, allow upside down]{\arrowOut} node[pos=.8, anchor=east]{$c$};
\draw[line width=.5pt, color=black] (-2.5,1.5)--(-2.5,5) node[sloped, pos=0.8, allow upside down]{\arrowOut} node[pos=.8, anchor=east]{$c$};
\draw[color=violet, fill=violet] (0,4) circle (.15) node[anchor=west]{$\;\alpha$};
\draw[color=cyan, fill=cyan] (0,1) circle (.15) node[anchor=west]{$\;\beta$};
\draw[color=blue, fill=blue] (0,-1) circle (.15) node[anchor=west]{$\;\gamma$};
\draw[color=violet, fill=violet] (0,-4) circle (.15) node[anchor=west]{$\;\delta$};
\draw[ fill=black] (-5,0) circle (.15) node[anchor=east]{$\nu\;$};
\end{tikzpicture}
\stackrel{\eqref{pic:snake},\eqref{eq:functortrpivot}}=
\begin{tikzpicture}[scale=.5, baseline=(current bounding box.center)]
\draw[line width=1.5pt, color=violet] (-1.5,5)--(1,5) node[anchor=west]{$p$};
\draw[line width=1.5pt, color=violet] (-1.5,-5)--(1,-5) node[anchor=west]{$p$};
\draw[line width=1.5pt, color=cyan] (0,4)--(0,1) node[midway, anchor=west]{$q$};
\draw[line width=1.5pt, color=red] (0,1)--(0,-1) node[midway, anchor=west]{$n$};
\draw[line width=1.5pt, color=blue] (0,-4)--(0,-1)node[midway, anchor=west]{$m$};
\draw[line width=1.5pt, color=violet] (0,5)--(0,4);
\draw[line width=1.5pt, color=violet] (0,-4)--(0,-5);
\draw[line width=.5pt, color=gray] (0,4) .. controls (-2,2.5)  and (-2,.5).. (0,-1) node[pos=.4, anchor=east]{$d$};
\draw[color=white, fill=white] (-1.1,.83) circle (.4);
\draw[line width=1pt, color=black, style=dashed] (0,4) --(-5,0)  node[midway, anchor=south east]{$F$};
\draw[line width=1pt, color=black, style=dashed] (0,1) --(-5,0) node[midway, anchor=south]{$\;\;G$};
\draw[line width=1pt, color=black, style=dashed] (-5,0) --(0,-1)  node[midway, anchor=north]{$\;\;K$};
\draw[line width=1pt, color=black, style=dashed] (-5,0) --(0,-4) node[midway, anchor=north east]{$H$};
\draw[color=white, fill=white] (-1.2,-.83) circle (.4);
\draw[color=white, fill=white] (-.9,-3.2) circle (.4);
\draw[color=white, fill=white] (-1,1) circle (.4);
\draw[color=white, fill=white] (-1,3) circle (.4);
\draw[line width=.5pt, color=black] (0,-4) .. controls (-.25,-2)  and (-1,-2).. (-1,-4);
\draw[line width=.5pt, color=black] (0,1) .. controls (-.25,0)  and (-1,0).. (-1,1);
\draw[line width=.5pt, color=black] (-1,-5)--(-1,-4) node[sloped, pos=0.8, allow upside down]{\arrowOut} node[pos=.8, anchor=east]{$c$};
\draw[line width=.5pt, color=black] (-1,1)--(-1,5) node[sloped, pos=0.8, allow upside down]{\arrowOut} node[pos=.8, anchor=east]{$c$};
\draw[color=violet, fill=violet] (0,4) circle (.15) node[anchor=west]{$\;\alpha$};
\draw[color=cyan, fill=cyan] (0,1) circle (.15) node[anchor=west]{$\;\beta$};
\draw[color=blue, fill=blue] (0,-1) circle (.15) node[anchor=west]{$\;\gamma$};
\draw[color=violet, fill=violet] (0,-4) circle (.15) node[anchor=west]{$\;\delta$};
\draw[ fill=black] (-5,0) circle (.15) node[anchor=east]{$\nu\;$};
\end{tikzpicture}
\end{align*}
\end{proof}

The generalised 6j symbols are also invariant under a simultaneous orientation reversal  of the defect surface and all  defect lines on the surface, if the defect data is transformed accordingly.

\begin{lemma}\label{lem:mirror}
The 6j symbol of  a labeled generic transversal tetrahedron is invariant under 
\begin{compactitem} 
\item reversing the orientation of all defect areas and all defect lines, 
\item replacing each $(\mac,\mad)$-bimodule category $\mam$ by the $(\mad,\mac)$-bimodule category $\mam^\#$, 
\item replacing each bimodule functor $F:\mam\to\man$ by $F^\#:\mam^\#\to \man^\#$,
\item replacing each bimodule natural transformation $\nu: F\Rightarrow G$ by $\nu^\#: G^\#\Rightarrow F^\#$.
\end{compactitem}
\end{lemma}

\begin{proof} 
By Lemma \ref{lem:sphlemma}, one may additionally reverse the orientation of all edges labeled with objects in spherical fusion categories and replace those objects by their duals without changing the associated 6j symbol. 

Let $t$ be a generic transversal tetrahedron and $t'$ the tetrahedron obtained  from $t$ by reversing the orientations of all defect areas,  defect lines and edges of $t$ 
 and transforming their labels accordingly.  For each triangle of $t$ the corresponding triangle of $t'$ has  the same morphism space assigned to it  by \eqref{pic:homspace}, but interpreted as a morphism space in the opposite category. For instance, $\Hom_\mam(F(c\rhd m), n)$ is viewed as $\Hom_{\mam^\#}(n, F(m\lhd c^*))$  and $\Hom_{\mam}(m\lhd d, G(n))$ is viewed as $\Hom_{\mam^\#}(G(n), d^*\rhd m)$. 

Projecting  the data from the spherical fusion categories on the defect plane of $t'$ then  yields a mirror image of the polygon diagram for $t$, in which the orientation of all interior lines is reversed, each bimodule category $\mam$ is replaced by its opposite $\mam^\#$, each bimodule functor $F$ by $F^\#$, each bimodule natural transformation $\nu$ by $\nu^\#$
and each object in a spherical fusion category by its dual. As the left action of $c\in \Ob\mac$ or the right action of $d\in \Ob\mad$ on $\mam$ corresponds to the right action of $c^*$ and the left action of $d^*$ on $\mam^\#$, the evaluations of the polygon diagrams for $t'$ and $t$ are equal by Corollary \ref{rem:polyback}.\end{proof}

\subsection{State sum model with defects}
In this section we show how the generalised 6j symbols for generic transversal defect tetrahedra  combine to state sums for a generic transversal triangulated 3-manifolds with defect data.

Let $M$ be a generic transversal triangulated 3-manifold with defect data, $R$ its set of regions and $A$ its set of oriented defect areas.  We denote by $\mac_r$ the spherical fusion category assigned to the region $r$ and  by $\mam_a$ the bimodule category assigned to the defect area $a$. A vertex, edge or triangle of the triangulation is called {\bf boundary} if its is contained
 in $\partial M$ and  {\bf internal} otherwise. 
We use the following notation:
\begin{compactitem}
\item $\mathrm{Tet}$ for the set of tetrahedra of the triangulation, 
\item $\Delta$ for its set of triangles, $\Delta_{\partial M}$ for its set of boundary triangles and $\Delta_{i}$ for its set of internal triangles,
\item $E$ for its set of edges, $E_{\partial M}$ for its set of boundary edges and $E_{i}$ for its set of internal edges, $E_r$ for  its set of edges  in the region $r$ and $E_a$ for its set of edges that intersect a defect area $a$,
\item $V$ for its set of vertices, $V_{\partial M}$ for its set of boundary vertices,  $V_{i}$ for its set of internal vertices and $V_r$ for its set of vertices  in the region $r$.  
\end{compactitem}

We fix for each spherical fusion category $\mac$ and each bimodule category $\mam$ sets $I_{\mac}$ and $I_{\mam}$ of representatives of the isomorphism classes of simple objects and
associated projection and inclusion morphisms  as  in Section \ref{sec:projincl}.  
For each labeling $l$ of $M$  as in Definition \ref{def:label} 
 we denote by $l(e)\in I_{\mac_r}$ or $l(e')\in I_{\mam_a}$ the simple object assigned to  $e\in E_r$ or $e'\in E_a$. For an oriented triangle $f$ labeled with simple objects in the sets of representatives and associated projection or inclusion morphisms, we write $\Hom(f)$ for the associated morphism space from \eqref{pic:homspace}.

Each tetrahedron $t\in \mathrm{Tet}$ is equipped with the orientation induced by $M$. For the tetrahedra that intersect defect surfaces, this is already encoded in their labeling with defect data. For  tetrahedra that do not intersect defect surfaces, we always assume this orientation choice without emphasising it notationally.

\begin{definition} \label{def:statesum} Let $M$ be a generic transversal triangulated  3-manifold with defect data,  $l_{\partial M}$ a labeling of its boundary edges with simple objects in the chosen sets of representatives and $b_{\partial M}$ an assignment of  morphisms from the morphism spaces \eqref{pic:homspace} to its boundary triangles. The state sum of $(M,l_{\partial M}, b_{\partial M})$ is
\begin{align}\label{eq:state sum}
\mathcal Z(M,l_{\partial M},b_{\partial M})=
\prod_{r\in R} \dim(\mathcal C_r)^{-\left(\Sigma_{v\in V_r} \epsilon(v)\right)}
\sum_{l}\sum_{b_i}   \prod_{t\in \mathrm{Tet}} \mathrm{6j}(t, l,b)\prod_{e\in E} \dim l(e)^{\epsilon(e)},
\end{align} 
where
\begin{compactitem}
\item  $\epsilon(v)=1$ for $v\in V_{i}$, $\epsilon(e)=1$ for $e\in E_i$ and $\epsilon(v)=1/2$ for $v\in V_{\partial M}$,  $\epsilon(e)=1/2$ for $e\in E_{\partial M}$,
\item $\Sigma_l$ is the sum over all labelings $l$ of  edges of $M$  with simple objects in  $I_{\mac_r}$ and $I_{\mam_a}$ that coincide with  $l_{\partial M}$ on the boundary,
\item $\Sigma_{b_i}$ is the sum over all labelings $b_i$ of internal triangles of $M$ with the projection and inclusion morphisms in the morphism spaces \eqref{pic:homspace},
\item $\Pi_{t\in \mathrm{Tet}}\, \mathrm{6j}(t, l, b)$ stands for the  product of the generalised 6j symbols $\mathrm{6j}(t, l, b)\in\C$ of all tetrahedra given by the labeling $l$ and the assignment $b=(b_i, b_{\partial M})$ of morphisms to the triangles. 
\end{compactitem}
We denote by $\mathcal Z(M, l_{\partial M})$ the associated  linear map 
\begin{align}\label{eq:statesummap}
\mathcal Z(M, l_{\partial M}): \bigotimes_{f\in \Delta_{\partial M}} \Hom(f)\to\C, \quad b_{\partial M}\mapsto \mathcal Z(M,l_{\partial M},b_{\partial M})
\end{align}
and by $\mathcal Z'(M,l_{\partial M},b_{\partial M})$ the rescaled state sum without the dimension factors of boundary edges and vertices
\begin{align}\label{eq:rescstate}
\mathcal Z(M, l_{\partial M}, b_{\partial M})=\left(\frac {\prod_{e\in E_{\partial M}} \dim l(e)} {\prod_{r\in R}\prod_{v\in V_r\cap V_{\partial M}} \dim \mac_r}  \right)^{1/2} \mathcal Z'(M, l_{\partial M}, b_{\partial M}).
\end{align}
\end{definition}

The product $\Pi_{t\in \mathrm{Tet}} \mathrm{6j}(t,l,b)$  in formula \eqref{eq:state sum} can also be viewed as the evaluation of the (unordered) tensor product $\otimes_{t\in \mathrm{Tet}} \mathrm{6j}(t)$ of the linear forms $\mathrm{6j}(t)$ in Definition \ref{def:gen6j} on a vector in the  tensor product of the morphism spaces for  the triangles. This vector is  defined by the labelings $l,b_i, b_{\partial M}$, and the sums over $l$ and $b_i$ eliminate the dependence on the labeling in the interior of $M$.

It is clear  that the linear map $\mathcal Z(M, l_{\partial M})$ does not depend on the choice of projection and inclusion morphisms for the interior triangles. If the simple objects assigned to the edges of the triangulation are fixed, then the projection and inclusion morphisms assigned to the interior triangles are fixed by \eqref{eq:projincl} up to simultaneous rescaling $p^\alpha_{xm}\to \lambda p^\alpha_{xm}$ and $j^\alpha_{xm}\to \lambda^\inv j^\alpha_{xm}$ for $\lambda=\lambda(x,m,\alpha)\in \C^\times$. As each interior triangle arises in exactly two tetrahedra with opposite orientations, the state sum is invariant under such rescalings.

To determine the dependence of the state sum   on the choice of representatives of simple objects, let $M$ be a  triangulated 3-manifold with defect data. Fix two sets of representatives $I,I'$ for each category labeling a region or  defect area and an isomorphism $\phi_{xx'}: x\to x'$ for each pair of isomorphic objects $x\in I$, $x'\in I'$. Note that this isomorphism is unique up to rescaling with $\C^\times$ and that it induces an isomorphism $\phi_f: \Hom(f)\to \Hom(f)'$ between the morphism spaces for a triangle $f$ labeled by isomorphic objects in $I$ and $I'$. 
More explicitly, these isomorphisms are given by 
\begin{align*}
&\phi_f: \Hom(m, F(j\rhd n))\to \Hom(m', F(j'\rhd n')), \quad \alpha\mapsto F(\phi_{jj'}\rhd \phi_{nn'})\circ \alpha\circ \phi^\inv_{mm'}\\
&\phi_f: \Hom(F(j\rhd n),n)\to \Hom(F(j'\rhd n'),m'), \quad \beta \mapsto \phi_{mm'} \circ \beta\circ F(\phi_{jj'}^\inv\rhd \phi_{nn'}^\inv).
\end{align*}
Definition \ref{def:gen6j} and the cyclic invariance of the trace imply that the 6j symbols are invariant under these isomorphisms, if they are applied to all morphism spaces simultaneously.  This yields

\begin{corollary} Let $M$ be a generic transversal  triangulated 3-manifold with defect data. Let $l_{\partial M}$ and $l_{\partial M}'$ be labelings of the boundary triangles with simple objects in two sets of representatives $I,I'$. Then 
$$\mathcal Z(M,l_{\partial_M})=\mathcal Z(M,l'_{\partial_M})\circ (\otimes_{f\in \Delta_{\partial M}}\phi_f):\bigotimes_{f\in \Delta_{\partial M}}\Hom(f)\to \C.$$
\end{corollary}

In particular, it follows that the state sum for a \emph{closed} generic transversal  triangulated 3-manifold with defect data does not depend on the choice of representatives of simple objects. 

It is also apparent from  Definition \ref{def:statesum} that the state sum \eqref{eq:state sum} is a direct generalisation of the usual Turaev-Viro-Barrett-Westbury state sum from \cite{BW,TV} and coincides with it if all defect areas, defect lines and defect vertices are trivial. This amounts to replacing all bimodule categories in \eqref{eq:state sum} by a single spherical fusion category $\mac$, viewed as a $(\mac,\mac)$-bimodule category, and all bimodule functors and natural transformations by identity functors and natural transformations. 

\begin{corollary}\label{cor:reducedstatesum}
For any triangulated 3-manifold without defects and any generic transversal triangulated 3-manifold with trivial defect data, labeled by a spherical fusion category $\mac$, the state sum from Definition \ref{def:statesum} reduces to the Turaev-Viro-Barrett-Westbury state sum
\begin{align}\label{eq:statesumsph}
\mathcal Z(M, l_{\partial M}, b_{\partial M})=\frac 1 {\dim\mac^{|V_{i}|+|V_{\partial M}|/2}}  \sum_{l}\sum_{b_i}   \prod_{t\in \mathrm{Tet}} \mathrm{6j}(t,l,b)\prod_{e\in E} \dim l(e)^{\epsilon(e)}.
\end{align} 
\end{corollary}

Another direct consequence of  Definition \ref{def:statesum} is that the state sum can be computed by gluing  triangulated 3-manifolds  along  boundary components. As in the case  without defects, this allows one to compute state sums by cutting  3-manifolds into simpler pieces. The only additional restriction  is that the defect surfaces must meet the cut transversally and all pieces must be generic and transversal.

\begin{corollary}\label{cor:gluing} Let $M_1$ and $M_2$ be generic  transversal  triangulated 3-manifolds with defect data and $M$ obtained by gluing $M_1$ and $M_2$ along a boundary component $\Sigma\subset \partial M_i$. Then
\begin{align}
\mathcal Z(M,l_{\partial M}, b_{\partial M})=\sum_{l_\Sigma} \mathcal Z(M_1,l_{\partial M_1}, b_{\partial M_1})\cdot \mathcal Z(M_2,l_{\partial M_2}, b_{\partial M_2}),
\end{align}
where the labelings $l_{\partial M}$ and $b_{\partial M}$ of $\partial M$ are induced by the labelings $l_{\partial M_i}$ and $b_{\partial M_i}$ of $\partial M_i$ and the
 sum is over the labelings of  edges in $\Sigma$ with simple objects and over dual bases of the  morphism spaces for  triangles in $\Sigma$.
\end{corollary}

\section{Topological invariance}
\label{sec:topinvariance}

\subsection{Moves between triangulated PL manifolds}
\label{subsec:PLsubsec}

To establish triangulation independence for  state sum models with defects, we require some background from PL topology. We consider different moves on triangulations that relate PL homeomorphic 3-manifolds with and without boundary. The main references  are the articles \cite{P, Pshell}  by Pachner, see also  the articles \cite{L} by Lickorish and \cite{Cas} by Casali.

\begin{theorem}\cite[Th.~5.5]{P},\cite[Th.~5.9]{L}\\
Two closed triangulated $n$-dimensional PL manifolds are PL homeomorphic iff they are related by  a finite sequence of bistellar  moves.
\end{theorem}

\begin{theorem}\cite{Cas}\label{th:pachbound}\\
Two triangulated $n$-dimensional PL manifolds with boundary, whose boundary triangulations coincide, are PL homeomorphic iff they are related by bistellar moves.
\end{theorem}

The {\bf bistellar moves} for a 3-manifold are the  2-3 move and the 1-4 move depicted in Figure \ref{fig:bistellar}. The 1-4 move  subdivides a tetrahedron into four  by inserting a new vertex in its interior  or reverses this step. The 2-3 move replaces two tetrahedra that share a triangle by three tetrahedra that share an edge  or reverses this step.

\begin{figure}
\begin{center}
\def\svgwidth{.41\columnwidth}
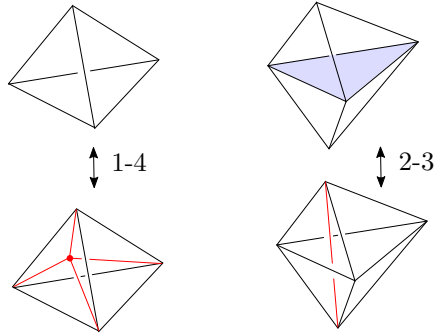 
\end{center}
\caption{Bistellar moves for a tetrahedron.}
\label{fig:bistellar}
\end{figure}

We also consider {\bf stellar subdivisions}, see for instance \cite{RS}. The stellar subdivision for a tetrahedron is  the 1-4 move  from Figure \ref{fig:bistellar}. The stellar subdivisions for a triangle and an edge are shown in Figure \ref{fig:stellar}.
A stellar subdivision of a triangle $\Delta$ inserts a vertex in the interior of $\Delta$ and connects it to the vertices of $\Delta$ and to the  opposite vertices of all tetrahedra containing $\Delta$. 
A stellar subdivision of an edge $e$ inserts a vertex in the interior of $e$ and connects it to the opposite vertices of all triangles containing $e$. 

It is shown in \cite{BW} that stellar subdivisions are composites of 1-4 and 2-3 moves. The stellar subdivision for a triangle $\Delta$ is obtained by to first applying a 1-4 move   to one of the two tetrahedra containing $\Delta$ and then a 2-3 move to the other tetrahedron and the adjacent tetrahedron created by the 1-4 move, see the proof of Theorem 4.6 in \cite{BW}. The stellar subdivision for an edge $e$ shared by $n$ tetrahedra  is obtained by first applying a 1-4  move  to one of the tetrahedra containing $e$.  Then one applies a sequence of $(n-2)$  2-3  moves  clockwise to the remaining tetrahedra containing  $e$. These 2-3 moves  connect the new vertex  to the vertices of these  tetrahedra.  Finally one applies a 2-3  move that removes $e$, see the proof of Theorem 4.6 in  \cite{BW}.

\begin{figure}
\begin{center}
\def\svgwidth{.41\columnwidth}
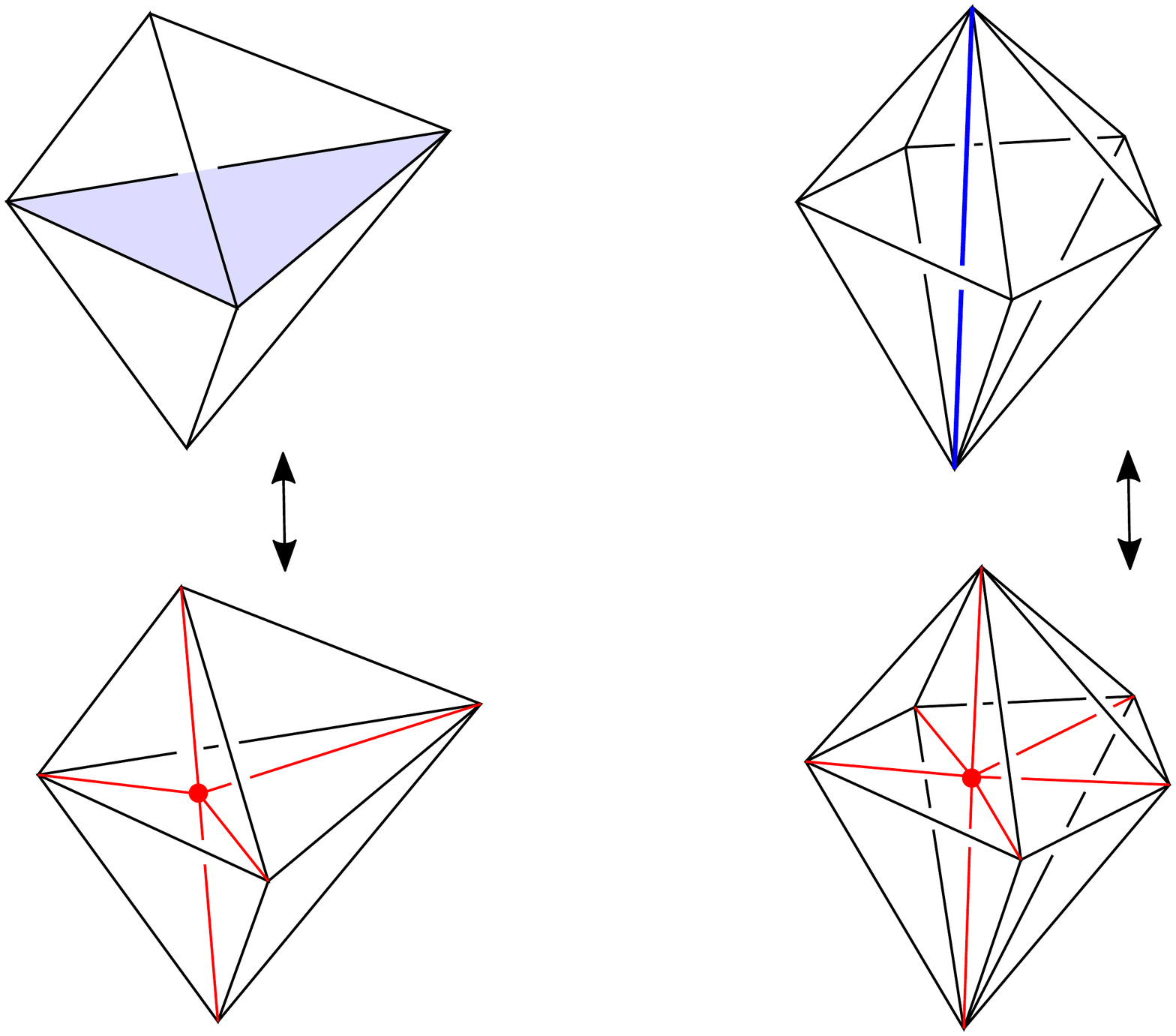 
\end{center}
\caption{Stellar subdivisions for a triangle and an edge of a triangulation.}
\label{fig:stellar}
\end{figure}

For 3-manifolds with boundaries one also needs to consider elementary shellings and their inverses. An {\bf elementary shelling} of a PL 3-manifold $M$ removes a 3-simplex $t$ that is the  join $t=A\star B=\mathrm{conv}(A\cup B)$ of two simplexes $A$ and $B$ with $A\cap M=\partial A$ and $B\star \partial A\subset \partial M$ and replaces $M$ by the closure of $M\setminus A\star B$, see \cite[Def.~5.2]{L}. This amounts to removing the tetrahedron $t$, which has  one, two or three faces in the boundary $\partial M$, such that its remaining faces become boundary faces, as shown in Figure \ref{fig:shellingdef}. Correspondingly, an {\bf inverse shelling} amounts to adding a tetrahedron to $M$ that shares  one, two  or three triangles with $\partial M$. 
 
 \begin{figure}
\begin{center}
\def\svgwidth{.37\columnwidth}
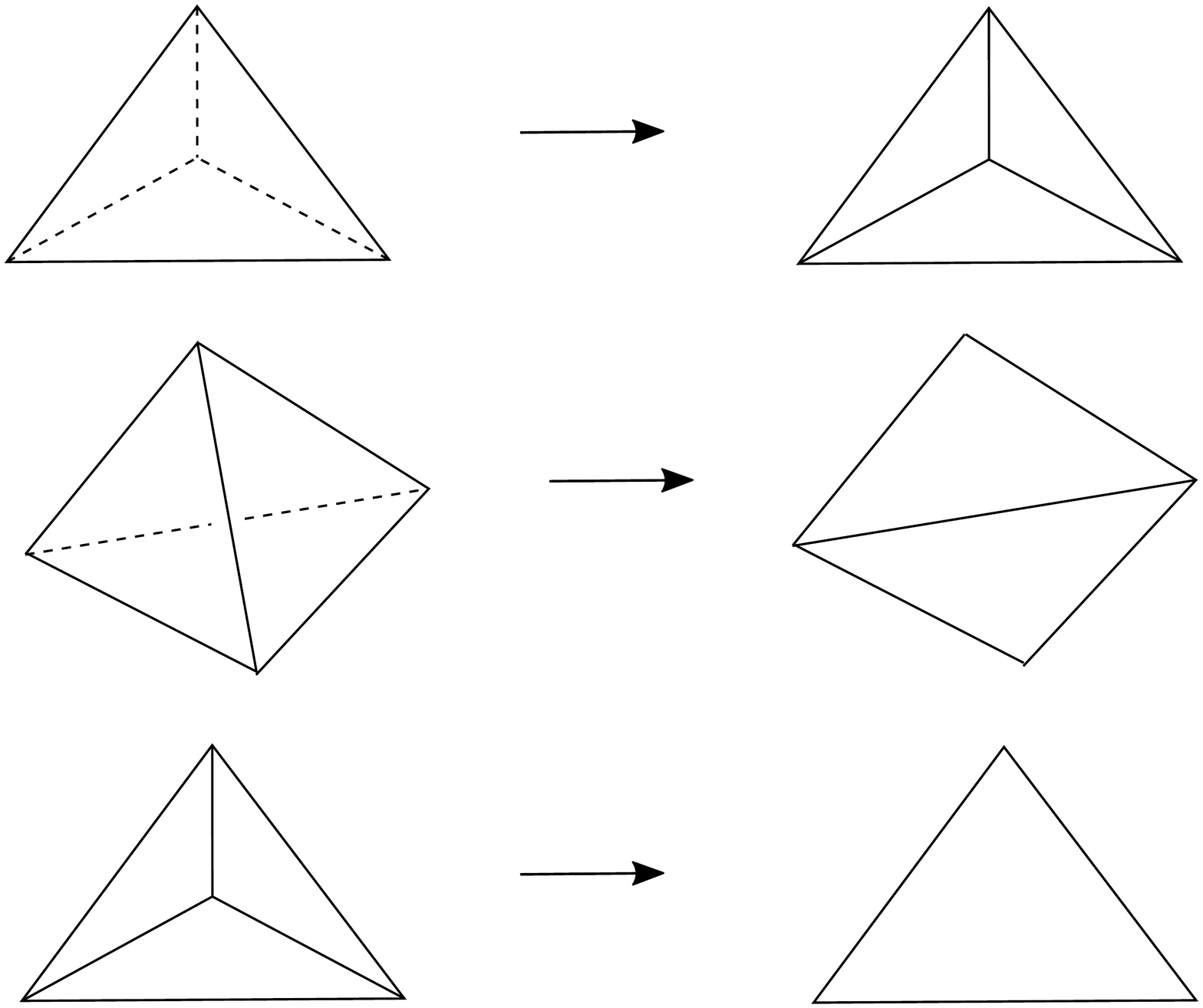 
\end{center}
\caption{Elementary shellings: removing a tetrahedron with one, two or three faces in the boundary.}
\label{fig:shellingdef}
\end{figure}

Elementary shellings and their inverses on a 3-manifold $M$ induce bistellar moves on the boundary $\partial M$. These are the 1-3 move, which subdivides a boundary triangle into three  by adding a vertex in its interior or reverses this step, and the 2-2 move, which replaces a common edge  shared by two boundary triangles with the edge connecting their remaining two vertices, as  in Figure \ref{fig:shellingdef}. It is shown  in \cite{Pshell, P} that  triangulations of a PL manifold with boundary are related by bistellar moves, elementary shellings and their inverses. 

\begin{theorem}\cite[Th.~4.14]{P}, \cite[Th.~2]{Pshell} \label{the:pachschellstell}\\
Two triangulated $n$-dimensional PL manifolds $M,M'$ with boundary are PL homeomorphic iff $M$ can be transformed into $M'$ by a finite number of elementary shellings and bistellar moves.
\end{theorem}

\begin{theorem}\cite[Th.~6.3]{P}, \cite[Th.~5.10]{L}\\  \label{th:pachshell}Two connected triangulated $n$-dimensional PL manifolds with non-empty boundary are PL homeomorphic iff they are related by a finite sequence of elementary shellings and  inverse shellings.
\end{theorem}

If a triangulated $n$-manifold $M$ is transformed  into  $M'$ by a finite number of elementary shellings, one says that $M$ {\bf shells} to $M'$. A triangulated $n$-manifold $M$ is called {\bf shellable} if it shells to an $n$-simplex. It is well-known that every triangulation and, more generally, CW-decomposition of a $2$-disc is shellable. 
This does not hold for  triangulations or CW-decompositions of the 3-ball, see Ziegler \cite{Zi} and the references therein.

The triangulation independence of the  Turaev-Viro-Barrett-Westbury  state sum  from \cite{TV,BW} follows from its invariance under bistellar moves.
It is shown in \cite{BW} and \cite{T} that the invariance of the state sum
 under the 2-3 move in Figure \ref{fig:bistellar} encodes the pentagon relation for the associator of a spherical fusion category. Its invariance under the 1-4 move in Figure \ref{fig:bistellar}  corresponds to an equivalent relation  obtained by inverting one of the associators in the pentagon relation.  It follows from the invariance under the 2-3 move and  under a \emph{bubble move}. The latter corresponds to the invertibility of the associator and  states that gluing two tetrahedra with  opposite orientations along  four faces yields the trivial state sum.  
 
 The pentagon relation for the associator is encoded in the Biedenharn-Elliott relations for the 6j symbols and the invertibility of the associator in their orthogonality relations
\cite[Prop.~5.2]{BW}.  They  are obtained by expressing the associator on simple objects of $\mac$ as in \eqref{eq:a6j} and using its naturality.
In our notation they read
\begin{align}
\label{eq:biedenharneliotsph}
&\begin{tikzpicture}[scale=.3]
\node at (-5,0) [anchor=east]{$\sum_{\beta}$};
\draw[line width=.5pt, color=black] (0,5)--(0,-5);
\draw[line width=.5pt, color=black] (-.5,5)--(.5,5) node[anchor=west] {$m$};
\draw[line width=.5pt, color=black] (-.5,-5)--(.5,-5) node[anchor=west] {$m$};
\draw[line width=.5pt, color=black] (0,3)--(-3,1) node[midway, anchor=south east]{$k$};
\draw[line width=.5pt, color=black] (-3,1)--(0,0) node[pos=.7, anchor=south]{$\;j$};
\draw[line width=.5pt, color=black] (-3,1)--(0,-3) node[midway, anchor=north east]{$i$};
\node at (0,1.5) [anchor=west, color=black]{$n$};
\node at (0,-1.5) [anchor=west, color=black]{$p$};
\draw[color=black, fill=black] (-3,1) circle (.2) node[anchor=east]{$\alpha\;$};
\draw[color=black, fill=black] (0,3) circle (.2) node[anchor=west]{$\beta$};
\draw[color=black, fill=black] (0,0) circle (.2) node[anchor=west]{$\gamma$};
\draw[color=black, fill=black] (0,-3) circle (.2) node[anchor=west]{$\delta$};
\begin{scope}[shift={(7,0)}]
\draw[line width=.5pt, color=black] (0,5)--(0,-5);
\draw[line width=.5pt, color=black] (-.5,5)--(.5,5) node[anchor=west] {$m$};
\draw[line width=.5pt, color=black] (-.5,-5)--(.5,-5) node[anchor=west] {$m$};
\draw[line width=.5pt, color=black] (0,3)--(-3,1) node[midway, anchor=south east]{$a$};
\draw[line width=.5pt, color=black] (-3,1)--(0,0) node[pos=.7, anchor=south]{$\;b$};
\draw[line width=.5pt, color=black] (-3,1)--(0,-3) node[midway, anchor=north east]{$k$};
\node at (0,1.5) [anchor=west, color=black]{$q$};
\node at (0,-1.5) [anchor=west, color=black]{$n$};
\draw[color=black, fill=black] (-3,1) circle (.2) node[anchor=east]{$\phi\;$};
\draw[color=black, fill=black] (0,3) circle (.2) node[anchor=west]{$\rho$};
\draw[color=black, fill=black] (0,0) circle (.2) node[anchor=west]{$\sigma$};
\draw[color=black, fill=black] (0,-3) circle (.2) node[anchor=west]{$\beta$};
\end{scope}
\node at (9,0)[anchor=west]{$=\sum_{c,\lambda,\mu,\nu}\dim(c)$};
\begin{scope}[shift={(24,0)}]
\draw[line width=.5pt, color=black] (0,5)--(0,-5);
\draw[line width=.5pt, color=black] (-.5,5)--(.5,5) node[anchor=west] {$m$};
\draw[line width=.5pt, color=black] (-.5,-5)--(.5,-5) node[anchor=west] {$m$};
\draw[line width=.5pt, color=black] (0,3)--(-3,1) node[midway, anchor=south east]{$a$};
\draw[line width=.5pt, color=black] (-3,1)--(0,0) node[pos=.7, anchor=south]{$\;c$};
\draw[line width=.5pt, color=black] (-3,1)--(0,-3) node[midway, anchor=north east]{$i$};
\node at (0,1.5) [anchor=west, color=black]{$q$};
\node at (0,-1.5) [anchor=west, color=black]{$n$};
\draw[color=black, fill=black] (-3,1) circle (.2) node[anchor=east]{$\lambda\;$};
\draw[color=black, fill=black] (0,3) circle (.2) node[anchor=west]{$\rho$};
\draw[color=black, fill=black] (0,0) circle (.2) node[anchor=west]{$\mu$};
\draw[color=black, fill=black] (0,-3) circle (.2) node[anchor=west]{$\delta$};
\end{scope}
\begin{scope}[shift={(31,0)}]
\draw[line width=.5pt, color=black] (0,5)--(0,-5);
\draw[line width=.5pt, color=black] (-.5,5)--(.5,5) node[anchor=west] {$a$};
\draw[line width=.5pt, color=black] (-.5,-5)--(.5,-5) node[anchor=west] {$a$};
\draw[line width=.5pt, color=black] (0,3)--(-3,1) node[midway, anchor=south east]{$k$};
\draw[line width=.5pt, color=black] (-3,1)--(0,0) node[pos=.7, anchor=south]{$\;j$};
\draw[line width=.5pt, color=black] (-3,1)--(0,-3) node[midway, anchor=north east]{$i$};
\node at (0,1.5) [anchor=west, color=black]{$b$};
\node at (0,-1.5) [anchor=west, color=black]{$c$};
\draw[color=black, fill=black] (-3,1) circle (.2) node[anchor=east]{$\alpha\;$};
\draw[color=black, fill=black] (0,3) circle (.2) node[anchor=west]{$\phi$};
\draw[color=black, fill=black] (0,0) circle (.2) node[anchor=west]{$\nu$};
\draw[color=black, fill=black] (0,-3) circle (.2) node[anchor=west]{$\lambda$};
\end{scope}
\begin{scope}[shift={(38,0)}]
\draw[line width=.5pt, color=black] (0,5)--(0,-5);
\draw[line width=.5pt, color=black] (-.5,5)--(.5,5) node[anchor=west] {$p$};
\draw[line width=.5pt, color=black] (-.5,-5)--(.5,-5) node[anchor=west] {$p$};
\draw[line width=.5pt, color=black] (0,3)--(-3,1) node[midway, anchor=south east]{$c$};
\draw[line width=.5pt, color=black] (-3,1)--(0,0) node[pos=.7, anchor=south]{$\;b$};
\draw[line width=.5pt, color=black] (-3,1)--(0,-3) node[midway, anchor=north east]{$j$};
\node at (0,1.5) [anchor=west, color=black]{$q$};
\node at (0,-1.5) [anchor=west, color=black]{$n$};
\draw[color=black, fill=black] (-3,1) circle (.2) node[anchor=east]{$\nu\;$};
\draw[color=black, fill=black] (0,3) circle (.2) node[anchor=west]{$\mu$};
\draw[color=black, fill=black] (0,0) circle (.2) node[anchor=west]{$\sigma$};
\draw[color=black, fill=black] (0,-3) circle (.2) node[anchor=west]{$\gamma$};
\end{scope}
\end{tikzpicture}
\end{align}
\begin{align}
\label{eq:orthogonalitysph}
&\begin{tikzpicture}[scale=.3]
\node at (-5,0) [anchor=east]{$\sum_{p,\gamma,\delta} \dim(p)\dim(k)$};
\node at (9,0)[anchor=west]{$=\delta_{\alpha\sigma}\delta_{\beta\rho}\delta_{kl}$};
\draw[line width=.5pt, color=black] (0,5)--(0,-5);
\draw[line width=.5pt, color=black] (-.5,5)--(.5,5) node[anchor=west] {$m$};
\draw[line width=.5pt, color=black] (-.5,-5)--(.5,-5) node[anchor=west] {$m$};
\draw[line width=.5pt, color=black] (0,3)--(-3,1) node[midway, anchor=south east]{$k$};
\draw[line width=.5pt, color=black] (-3,1)--(0,0) node[pos=.7, anchor=south]{$\;j$};
\draw[line width=.5pt, color=black] (-3,1)--(0,-3) node[midway, anchor=north east]{$i$};
\node at (0,1.5) [anchor=west, color=black]{$n$};
\node at (0,-1.5) [anchor=west, color=black]{$p$};
\draw[color=black, fill=black] (-3,1) circle (.2) node[anchor=east]{$\alpha\;$};
\draw[color=black, fill=black] (0,3) circle (.2) node[anchor=west]{$\beta$};
\draw[color=black, fill=black] (0,0) circle (.2) node[anchor=west]{$\gamma$};
\draw[color=black, fill=black] (0,-3) circle (.2) node[anchor=west]{$\delta$};
\begin{scope}[shift={(7,0)}]
\draw[line width=.5pt, color=black] (0,5)--(0,-5);
\draw[line width=.5pt, color=black] (-.5,5)--(.5,5) node[anchor=west] {$m$};
\draw[line width=.5pt, color=black] (-.5,-5)--(.5,-5) node[anchor=west] {$m$};
\draw[line width=.5pt, color=black] (0,3)--(-3,-1) node[midway, anchor=south east]{$i$};
\draw[line width=.5pt, color=black] (-3,-1)--(0,0) node[pos=.7, anchor=south]{$\;j$};
\draw[line width=.5pt, color=black] (-3,-1)--(0,-3) node[midway, anchor=north east]{$l$};
\node at (0,1.5) [anchor=west, color=black]{$p$};
\node at (0,-1.5) [anchor=west, color=black]{$n$};
\draw[color=black, fill=black] (-3,-1) circle (.2) node[anchor=east]{$\sigma\;$};
\draw[color=black, fill=black] (0,3) circle (.2) node[anchor=west]{$\delta$};
\draw[color=black, fill=black] (0,0) circle (.2) node[anchor=west]{$\gamma$};
\draw[color=black, fill=black] (0,-3) circle (.2) node[anchor=west]{$\rho$};
\end{scope}
\end{tikzpicture}
\end{align}

These relations yield a direct proof of the triangulation independence of the state sum for closed 3-manifolds \cite[Th.~5.1]{BW}.  For 3-manifolds  with boundary one has the following result.

\begin{theorem} \label{th:pachinv}\cite[Th.~1.4]{TV}, \cite[Th.~1.7]{T}\\
Let $M, M'$ be compact triangulated 3-manifolds with boundary whose boundary triangulations and boundary labelings $l_{\partial M}$ and $b_{\partial M}$ coincide. Then their  state sums \eqref{eq:statesumsph} are equal.
\end{theorem}

\subsection{Neighbourhoods of defect surfaces}
\label{subsec:gentrans}

 In the presence of defects,  applying a bistellar move to a generic transversal triangulation may yield a triangulation that is no longer transversal or generic. 
  For this reason, it is not straightforward to establish invariance of the state sums via bistellar moves.  
  Instead, we refine the triangulations and decompose the underlying  defect 3-manifold   into certain neighbourhoods of defect surfaces and components contained in the regions between them. Their state sums are  computed separately and then  glued with Corollary \ref{cor:gluing}. This also gives a geometrical interpretation to the state sums and facilitates the computation of  examples.

The neighbourhoods of the defect surfaces we consider are similar to regular neighbourhoods, see \cite[Ch.~3]{RS}, but our defect surfaces are not realised as subcomplexes of the triangulation. 

\begin{definition} A 
 {\bf fine neighbourhood} of a defect surface $D$ is a connected
generic transversal  triangulated  3-manifold $M$ with defect surface $D\subset M$ such that
 $D\cap \partial M=\partial D$,  every vertex of $M$ is contained in $\partial M$  and every triangle or edge  in $M$ that does not intersect $D$ is contained in $\partial M$.
\end{definition}

By applying a finite number of  bistellar moves between generic transversal triangulations, any 
 generic transversal triangulation of a defect 3-manifold $M$ can be subdivided in such a way that the tetrahedra intersecting a given defect surface form a fine neighbourhood of the defect surface.  
 A bistellar move between two generic transversal triangulations is called {\bf generic transversal bistellar move} in the following.

\begin{lemma} \label{lem:finetriang}Let $(M, D)$ be a defect 3-manifold  and $T$ a generic transversal  triangulation of $M$. 
Then there is a finite sequence of  generic transversal bistellar  moves  that transforms  $T$ into a generic transversal  triangulation $S$ such that:
\begin{compactenum}[(i)]
\item $S$ coincides with $T$ on $\partial M$,
\item the tetrahedra in $S$ contained in any region $r\subset M$ form a  3-manifold $R$ homeomorphic to $\bar r$,
\item the tetrahedra in $S$ that intersect any defect surface $\Sigma\subset D$ form a fine neighbourhood of $\Sigma$.
 \end{compactenum}
\end{lemma}

\begin{proof}  Let $D=\amalg_{j=1}^n \Sigma_j$ be the defect manifold, whose connected components are the defect surfaces $\Sigma_j$, and $M\setminus D=\amalg_{k=1}^m r_k$ with the regions $r_k$ as connected components.  

For each oriented edge in $T$ that intersects $D$ and hence is  oriented by $D$,  denote by $r_f$ the region that contains its target vertex  and by  $r_i$ the region that contains its starting vertex. Analogously, for each triangle and tetrahedron  in $T$ that intersects $D$,  denote by $r_f$ and $r_i$ the regions that contain the target vertices and starting vertices of its edges that intersect $D$.

The triangulation $S$ is constructed by a finite sequence of stellar subdivisions from  Section \ref{subsec:PLsubsec}, which are finite sequences of bistellar moves. 
We show  that this can be achieved with generic transversal bistellar moves.

1.~In the first step, 
for each tetrahedron $t$ in $T$ that intersects $D$, choose a point $t_f\in r_f\cap \mathring{t}$  and  perform a 1-4  move  with $t_f$ as the new vertex. 
This  subdivides $t$ into four transversal tetrahedra. By adjusting the position of $t_f$ one can achieve that none of the new triangles contains a defect vertex and none of the new edges intersect defect lines. Hence,  all four tetrahedra are generic. This yields a generic transversal   triangulation $T'$ related to $T$ by generic transversal bistellar moves.

2.~In the second step, for each triangle $\Delta$ in $T'$ that intersects $D$ and is not contained in $\partial M$, choose a point $\Delta_f\in r_f\cap\mathring{\Delta}$ and perform the stellar subdivision on $\Delta$ with $\Delta_f$ as the new vertex. The  two bistellar moves that combine to this stellar subdivision involve only transversal tetrahedra. By adjusting the position of $\Delta_f$ one can achieve that  none of the  triangles arising in these moves  contains a defect vertex and none of the edges  intersect a defect line. With this, all of the tetrahedra become generic and we obtain a generic transversal  triangulation $T''$ related to $T'$ by generic transversal bistellar moves. 

3.~In the third step, for each edge $e$ in $T''$ that intersects $D$ and is not contained in  $\partial M$, choose a point $e_f\in r_f\cap\mathring{e}$  and perform a stellar subdivision on $e$ with $e_f$ as the new vertex. All tetrahedra arising in the bistellar moves that combine to this stellar subdivision
 are transversal. By adjusting the position of $e_f$, one can achieve that none of their triangles  contains a defect vertex and none of their edges intersects a defect line. Hence,  all tetrahedra in this sequence of bistellar moves are generic. 

This yields a generic  transversal  triangulation $T'''$ related to $T$ by generic transversal bistellar moves. The associated subdivisions of defect tetrahedra are shown in Figure \ref{fig:barycentric}.  We equip each new edge in $T'''$ that intersects $D$ with the induced orientation and choose an arbitrary orientation for the other new edges.  

4.~We then apply steps 1 to 3  to the triangulation $T'''$, but now with  additional vertices $t_i\in r_i\cap\mathring{t}$, $\Delta_i\in r_i\cap \mathring{\Delta}$, $e_i\in r_i\cap\mathring{e}$ in the region $r_i$ instead of $r_f$. The positions of these new vertices can again be chosen in such a way that this defines a finite sequence of  bistellar moves between generic transversal tetrahedra. 

The result is a generic transversal triangulation $S$  that coincides with $T$ on $\partial M$ by construction.  
We denote by $S_j$ the triangulated 3-manifold formed by all tetrahedra in $S$ that intersect a defect surface $\Sigma_j$. 
By construction, every vertex of $S_j$ and every edge or triangle of $S_j$ that does not intersect $\Sigma_j$ is contained in $\partial S_j$, and hence $S_j$ is a fine neighbourhood of $\Sigma_j$.  For each tetrahedron $t$ in the  triangulation $T$ that intersects  $\Sigma_j$,   the union  of the tetrahedra in its subdivision that are contained in a region $r_k$  is PL homeomorphic to  the closure of $(t\setminus D)\cap r_k$. These PL homeomorphisms and the identity maps for tetrahedra in $T$ that do not intersect $D$  glue to a PL homeomorphism from  the union of all tetrahedra in $S$  in $r_k$ to $\bar r_k$.  
\end{proof}

\begin{figure}
\begin{center}
\def\svgwidth{.99\columnwidth}
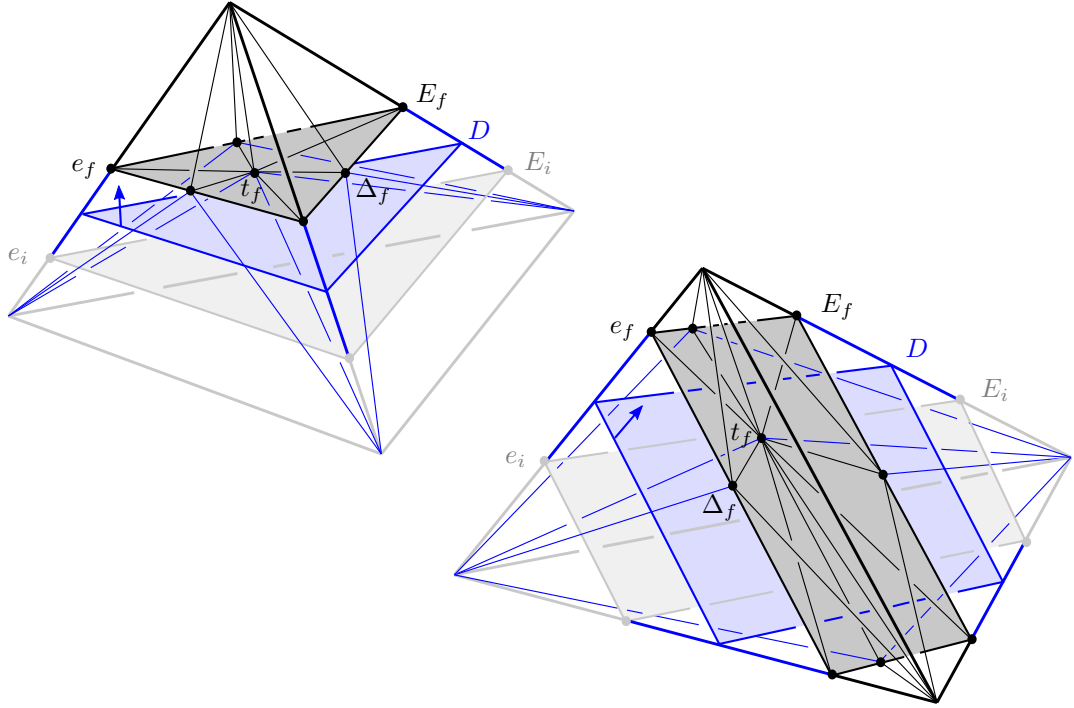 
\end{center}
$\quad$

\vspace{-2cm}
\caption{Subdivisions of tetrahedra intersecting the defect manifold $D$ with new vertices $t_f$, $\Delta_f$ and $e_f$ that span the surface $E_f$.  Edges in the  subdivision that intersect $D$ are drawn in blue, edges  that lie in the region  $r_f$ in black and edges in the region $r_i$ in grey.}
\label{fig:barycentric}
\end{figure}

 A  fine neighbourhood of a defect surface $D$ is shown in Example \ref{ex:surface} and Figure \ref{eq:prism}. Generic transversal triangulations that do not form fine neighbourhoods  of their defect surfaces are given in  Examples \ref{ex:tetrahedronsphere}, \ref{ex:genusg} and \ref{ex:knotcomplement}, see also Figures \ref{fig:handle} and \ref{fig:knotcomplement}. In particular, any fine neighbourhood of  $D$ must have topology $[0,1]\times D$. 

\begin{lemma}\label{lem:finenbtop}
Let $M$ be a triangulated 3-manifold that shells to  a fine neighbourhood of a defect surface $D$  by elementary shellings with tetrahedra that do not intersect $D$. Then $M$  is PL homeomorphic to $[0,1]\times D$.
\end{lemma}

\begin{proof} If $M$ is a fine neighbourhood of $D$, this follows by applying  the same subdivision procedure as in the proof of Lemma \ref{lem:finetriang}, but to every edge and triangle that intersects the defect surface $D$, including the ones in $\partial M$.  Denote by $E_f$ the 2d PL surface spanned by the vertices $t_f$, $\Delta_f$ and $e_f$,  by $E_i$ the 2d PL surface spanned by the new vertices $t_i$, $\Delta_i$ and $e_i$, as shown in Figure \ref{fig:barycentric}. For each tetrahedron $t$ in the original triangulation  with $t\cap D\neq \emptyset$  the tetrahedra in its subdivision that intersect $D$ form  the region between the surfaces $E_t\cap t$ and $E_i\cap t$ in Figure \ref{fig:barycentric} This region  is PL homeomorphic to $[0,1]\times (t\cap D)$.  
The PL homeomorphisms for the tetrahedra $t$ with $t\cap D\neq \emptyset$ glue to a PL homeomorphism that sends the union  of all tetrahedra in the subdivision that intersect $D$  to $[0,1]\times D$.  As $M$ is related to it by elementary shellings, it is PL homeomorphic to $[0,1]\times D$ by Theorem \ref{th:pachshell}.
For a 3-manifold $M$  that shells to a fine neighbourhood of $D$ with elementary shellings by tetrahedra that do not intersect $D$ the claim follows  with Theorem \ref{th:pachshell}.
\end{proof}

In the following, we will often project graphs on the boundary of a manifold $M$ with a defect surface $D\subset M$ to the defect surface $D$. This is achieved via  orientation preserving  PL homeomorphisms $\phi: M\to [0,1]\times D$. We denote by $p_D: [0,1]\times D\to  D$, $(t,d)\mapsto d$ and $p_t:[0,1]\times D\to [0,1]$, $(t,d)\mapsto t$ the projection maps for the product, which are PL maps by construction.

\begin{definition} \label{def:projection}Let $M$ be a triangulated 3-manifold with a defect surface $D\subset M$  and 
$\phi: M\to [0,1]\times D$ an orientation preserving PL homeomorphism with $p_D\circ \phi\vert_D=\mathrm{id}_D$. 
The {\bf  projections} induced by $\phi$ are the PL homeomorphisms
$P_i(\phi): \phi^\inv(\{i\}\times D)\to D$, $m\mapsto p_D\circ \phi(m)$ for $ i=0,1$.
\end{definition}

The diagrams obtained by projecting graphs on $\partial M$ to a defect surface $D\subset M$  depend on the choice of the underlying  PL homeomorphism. That their \emph{evaluations} are independent of this choice follows from the fact that different choices of PL homeomorphisms yield PL isotopic projection maps.

\begin{lemma} \label{lem:projiso}Let $\phi_1,\phi_2: M\to [0,1]\times D$ be orientation preserving PL homeomorphisms with $p_D\circ \phi_1\vert_D=p_D\circ \phi_2\vert_D=\mathrm{id}_D$. Then the projection maps  $P_i(\phi_1)$ and $P_i(\phi_2)$ are PL isotopic for $i=0,1$. 
\end{lemma}

\begin{proof} By a well-known result of Epstein \cite[Th.~6.4]{E}, an orientation preserving PL homeomorphism from a 2-manifold with compact boundary components to itself is  
PL isotopic to the identity iff it is  homotopic to the identity. It is therefore sufficient to show that $P_i(\phi_2)$ and $P_i(\phi_1)$ are homotopic.  For this, we construct homeomorphisms $\phi'_i: M\to [0,1]\times D$ that coincide with $\phi_i$ on $\partial M_i:=\phi_1^\inv(\{i\}\times D)=\phi_2^\inv(\{i\}\times D)$ and  homotopies from $\phi'_2\vert_{\partial M_i}$ to $\phi'_1\vert_{\partial M_i}$.  For $i=0$ and $j=1,2$, we consider the continuous maps
\begin{align*}
F_j:[0,1]\times D\to [0,1], \;(t,d)\mapsto \begin{cases}\frac t {2 p_t(\phi_j(d))} & t\in[0, p_t(\phi_j(d))]\\
\frac {1-t}{2(1-p_t(\phi_j(d)))}+\frac {t-p_t(\phi_j(d))}{1-p_t(\phi_j(d))} & t\in[p_t(\phi_j(d)), 1]
\end{cases}
\end{align*}
and the homeomorphisms
$\phi_j': M\to [0,1]\times D$, $m\mapsto (F_j(\phi_j(m)), p_D(\phi_j(m)))$
that satisfy $\phi'_j(m)=\phi_j(m)$ for all $m\in \partial M_0\cup \partial M_1$, $p_D\circ \phi'(d)=d$ and $p_t\circ \phi'_j(d)=\tfrac 1 2$ for all $d\in D$. Then the continuous map
\begin{align*}
h:[0,1]\times \partial M_0\to D,\; (t,m)\mapsto p_D\circ \phi_2'\circ \phi'^\inv_1(\tfrac t 2 , p_D(\phi'_1(m)))
\end{align*}
satisfies $h(0,m)=P_0(\phi_2)(m)$ and $h(1,m)=P_0(\phi_1)(m)$ for all $m\in \partial M_0$ and is a homotopy from $P_0(\phi_1)$ to $P_0(\phi_2)$.  
The proof that $P_1(\phi_1)$ and $P_1(\phi_2)$ are homotopic is analogous.
\end{proof}

\subsection{State sums with defect discs}

To prove the triangulation independence of the state sum, we first show that the state sums of 3-balls $M$ that shell to fine neighbourhoods of defect 
discs $D$ are given as cyclic evaluations of certain polygon diagrams. 

These polygon diagrams are obtained by projecting the dual graphs of the boundary triangulations on $\partial M$ to the defect disc. As  $M$ is PL homeomorphic to $[0,1]\times D$ via an orientation preserving PL homeomorphism  $\phi: M\to [0,1]\times D$ by Lemma \ref{lem:finenbtop}, this projection is given by Definition \ref{def:projection}. Lemma \ref{lem:projiso} ensures that different choices of $\phi$ yield graphs on $D$ that are related by isotopies.

\begin{proposition}\label{th:topological invariance}
Let $M$ be a 3-manifold with defect data that shells to a fine neighbourhood 
 of  a defect disc $D$ by  shellings with tetrahedra that do not intersect $D$. 
Then for all labelings  $l_{\partial M}$, $b_{\partial M}$  as in Definition \ref{def:statesum}  
$$
\mathcal Z'(M, l_{\partial M}, b_{\partial M})=\ev(P),
$$
where 
$\ev(P)$ is
the cyclic evaluation of the polygon diagram $P$ obtained by
\begin{compactitem}
\item projecting the dual graph of the boundary triangulation on $\partial M$ to $D$,
\item labeling the edges and vertices of the resulting diagram with the data  assigned by $l_{\partial M}$ and $b_{\partial M}$.
\end{compactitem}
\end{proposition}

\begin{proof}  
1.~We first suppose that $M$ is a fine neighbourhood of the defect disc $D$ and prove the claim by induction over the number of tetrahedra in the triangulation. For a single tetrahedron, this holds by definition of the generalised 6j symbol, see Figure \ref{fig:projecting}, and of the rescaled state sum. 

Any fine neighbourhood $M$ of $D$ with $n+1$ tetrahedra is obtained by gluing a single tetrahedron $t$ to a fine neighbourhood $M'$ of $D$ with $n$ tetrahedra.  
This follows, because  the triangulation of $M$ defines a cell complex structure on $D$ and 
 every cell decomposition of a disc is shellable, see for instance Sanderson \cite[Lemma 1]{S}. Shellings of the defect disc correspond to 3d shellings of its fine neighbourhood $M$.  
 
 The tetrahedron $t$ can be glued to $M'$ either along one or two faces. Gluing along three faces is not possible, as $t$ intersects $D$ either in three edges incident at a common vertex as in Figure \ref{fig:dectet} (a)  or in all except two opposite edges as in Figure \ref{fig:dectet} (b). The first case would yield a contradiction to the assumption that the defect surfaces in $M$ and $M'$ are both discs. In the second case there would be an internal  edge in $M$ common to two glued faces that does not intersect $D$, contradicting that
 $M$ is a fine neighbourhood of $D$.

 By Corollary \ref{cor:gluing}, formula \eqref{eq:rescstate} and the induction hypothesis, the rescaled state sum for $M$ is
\begin{align}\label{eq:indstate}
\mathcal Z'(M, l_{\partial M}, b_{\partial M})
=\sum_{l}\sum_{b_g}
\prod_{e\in E_g} \dim l(e) \ev(P')\, \mathrm{6j}(t,l,b),
\end{align}
where 
\begin{compactitem}
\item $\ev(P')$ denotes the cyclic evaluation of the  corresponding polygon diagram for $M'$,  
\item $E_{g}$ is the set of internal edges in $M$ that correspond to boundary edges of $t$ and $M'$,
\item the sum $\Sigma_{l}$ runs over all labelings of edges in $E_g$ with simple objects in the chosen sets of representatives,
\item  the sum $\Sigma_{b_g}$ runs over bases of the morphism spaces for glued faces of $t$ and $M'$. 
\end{compactitem}
We compute the sum in \eqref{eq:indstate} depending on the number of glued faces of $t$:
 
 (i) If $t$ is glued to $M'$ along a single face $f$, then $E_g=\emptyset$.  
Identity \eqref{pic:gluepoly}  implies that the sum in \eqref{eq:indstate} is given by
$$\mathcal Z'(M, l_{\partial M}, b_{\partial M})=\sum_{b_f} \ev(P')\, \mathrm{6j}(t,l,b)\stackrel{\eqref{pic:gluepoly}}=\ev(P).$$
(ii) If $t$ is glued to $M'$ along two faces $f,f'$ that share an edge $e$, then $E_g=\{e\}$ and $e$ intersects $D$.  Otherwise, $t$ would have to intersect 
$D$ in four edges, as in Figure \ref{fig:dectet} (b). As the intersections of the defect disc in $t$ with $f$ and $f'$ are both glued to the boundary of the defect disc in $M'$, 
the defect surface $D$ could not be a disc.

We apply  identity \eqref{pic:gluepoly}  to 
 glue $t$ on $M'$ along $f$. This yields a polygon diagram $P''$ with the morphisms for $f'$ assigned to adjacent sides of $P''$ that are separated by a vertex labeled $l(e)$. 
Identity \eqref{pic:closepoly} implies that the state sum 
 in \eqref{eq:indstate} is given by
$$
\mathcal Z'(M, l_{\partial M}, b_{\partial M})=\sum_{l(e)} \sum_{b_f, b_{f'}} \dim l(e)\, \ev(P')\, \mathrm{6j}(t,l,b)\stackrel{\eqref{pic:gluepoly}}=\sum_{l(e)} \sum_{b_{f'}} \dim l(e) \ev(P'') \stackrel{\eqref{pic:closepoly}}=\ev(P).
$$
This proves the claim for the case where $M$ is a fine neighbourhood of the defect disc.

2.~Suppose now that $M$ is obtained from a fine neighbourhood of the defect disc by inverse shellings with tetrahedra that do not intersect $D$. We prove the claim by induction over the number $n$ of tetrahedra in the triangulation that  do not intersect $D$. If $n=0$ the claim holds by 1.  

If $M$ contains $n+1$ tetrahedra that do not intersect $D$, then $M$ is obtained by gluing a tetrahedron $t$ that does not intersect $D$ to a defect 3-manifold $M'$ with $n$ tetrahedra that do not intersect $D$ and also satisfies the assumptions. If $t$  is contained in a region labeled by a spherical fusion category  $\mac$, then 
by Corollary \ref{cor:gluing} and the induction hypothesis, the rescaled state sum for $M$ is  given by
\begin{align}\label{eq:indstate2}
\mathcal Z'(M, l_{\partial M}, b_{\partial M})=
\sum_{l}\sum_{b_g}
\frac{ \prod_{e\in E_g} \dim l(e)}{\dim \mac^{v_g}}\, \ev(P')\, \mathrm{6j}(t,l,b),
\end{align}
where the notation is as in \eqref{eq:indstate} and $v_g\in\{0,1\}$ is the number of internal vertices in $M$ that correspond to boundary vertices of $t$ and $M'$.
We  compute the sum in \eqref{eq:indstate2} depending on the number of glued faces of $t$:

(i) If $t$ is glued to $M'$ along a single face $f$, then $E_g=\emptyset$ and $v_g=0$. 
Identity \eqref{pic:insert}  then implies that the state sum in \eqref{eq:indstate2} is given by
$$\mathcal Z'(M, l_{\partial M}, b_{\partial M})=\sum_{b_f} \ev(P')\, \mathrm{6j}(t,l,b)\stackrel{\eqref{pic:insert}}=\ev(P).$$
(ii) If $t$ is glued to $M'$ along two faces $f,f'$ that share an edge $e$, then $E_g=\{e\}$ and $v_g=0$. 
We use identity \eqref{pic:insert} to insert the 6j symbol for $t$ into the diagram $P'$. This yields a polygon diagram $P''$ with the morphisms for $f'$  assigned to vertices in the interior
 that are connected by an edge labeled $l(e)$. 
Applying the second identity in \eqref{pic:semisimpleids} then yields
$$
\mathcal Z'(M, l_{\partial M}, b_{\partial M})=\sum_{l(e)} \sum_{b_f, b_{f'}} \dim l(e)\, \ev(P')\, \mathrm{6j}(t,l,b)\stackrel{\eqref{pic:insert}}=\sum_{l(e)} \sum_{b_{f'}} \dim l(e) \ev(P'') \stackrel{\eqref{pic:semisimpleids}}=\ev(P).
$$
(iii) If $t$ is glued to $M'$ along three faces $f,f',f''$, then they share three edges $e=f\cap f'$, $e'=f\cap f''$ and $e''=f'\cap f''$, and there is exactly one vertex $v$ in $t$ that becomes internal after the gluing. Thus we have  $E_g=\{e,e',e''\}$ and  $v_g=1$.  
We first glue $t$ to $M'$ along $f$  with  identity \eqref{pic:insert}.  This eliminates the summation over $b_f$ in  \eqref{eq:indstate2} 
and  yields a polygon $P''$ with the morphisms for $f'$ assigned  to adjacent vertices in the interior that are connected by an edge labeled $l(e)$.  
Applying the second identity in \eqref{pic:semisimpleids} then eliminates the summation over $l(e)$ and  $b_{f'}$ and  the factor $\dim l(e)$ in \eqref{eq:indstate2}, as in (ii). This yields a polygon diagram $P'''$ that  differs from $P$ only in a disc in its interior, where one has
\begin{align}\label{eq:differenceii}
&\begin{tikzpicture}[scale=.3]
\node at (-3,0){$P'''$};
\draw[line width=.5pt] (0,0)--(0,-2) node[midway, anchor=west]{$i$};
\draw[line width=.5pt] (0,-2).. controls (-2,-4).. (0,-6) node[midway, anchor=east]{$l(e')$};
\draw[line width=.5pt] (0,-2).. controls (2,-4).. (0,-6) node[midway, anchor=west]{$l(e'')$};
\draw[line width=.5pt] (0,-6)--(0,-8) node[midway, anchor=west]{$i$};
\draw[fill=black] (0,-2) circle (.2) node[anchor=east]{$\gamma$};
\draw[fill=black] (0,-6) circle (.2) node[anchor=east]{$\gamma$};
\end{tikzpicture}
\qquad 
&\begin{tikzpicture}[scale=.3]
\node at (3,0){$P$};
\draw[line width=.5pt] (0,0)--(0,-8) node[midway, anchor=west]{$i$};
\end{tikzpicture}
\end{align}
with $\gamma\in \Hom(f'')$  and $i\in I_\mac$  labeling the remaining edge of $f''$.  Multiplying $\ev(P''')$ by $\dim l(e')\dim l(e'') \dim \mac^\inv$ and summing over $l(e'), l(e'')\in I_\mac$ and over  $\gamma\in \Hom(f'')$ then yields 
\begin{align*}
\mathcal Z'(M, l_{\partial M}, b_{\partial M})= \dim \mac^\inv \!\!\!\! \!\! \sum_{\gamma, l(e'), l(e'')} \!\!\!\!\!\!\ \dim l(e')\dim l(e'')\ev(P''')
\stackrel{\eqref{eq:2gonformula}}
=\ev(P).
\end{align*}
\end{proof}
 
Note that the evaluation of the diagram obtained by projecting the dual of the boundary triangulation on $M$ to the defect surface $D$ does not depend on the choice of the underlying PL homeomorphism $\phi: M\to [0,1]\times D$. By Lemma \ref{lem:projiso}  different choices of $\phi$ yield isotopic projection maps. Diagrams for spherical fusion categories that are related by PL isotopies have the same evaluation. This is well-known and  can be viewed as a special case of  \cite[Th.~1.12]{JS}, \cite[Th.~2.12]{BMS} for 2-category diagrams and \cite[Th.~3.9]{BMS}  for planar 2-category diagrams, see also the paragraphs after Example \ref{ex:2catdiagram} and Definition \ref{def:cyclictrans}. Together with identities \eqref{pic:rm2} to \eqref{pic:bimodulenat} this ensures that the evaluation of the resulting diagram is independent of $\phi$. 

As the evaluation of the polygon diagram in Proposition \ref{th:topological invariance} depends only on the  triangulation and labeling of $\partial M$ and on the defect data, we obtain

\begin{corollary} Let $M,M'$ be  triangulated 3-manifolds with defect data that satisfy the assumptions of Proposition \ref{th:topological invariance} and whose boundary triangulations coincide. Then for all labelings $l_{\partial M}$ of the boundary one has $\mathcal Z(M, l_{\partial M})=\mathcal Z(M', l_{\partial M})$.
\end{corollary}

In particular, this implies  that the state sum of any  generic transversal triangulated 3-manifold with defect data is invariant under generic transversal bistellar moves.  This follows with Corollary \ref{cor:gluing}, as the tetrahedra  in such moves define  triangulations  of a 3-ball
that satisfy the assumptions of Proposition \ref{th:topological invariance} and agree on its boundary.

\begin{corollary}\label{cor:pachner} Let $M$ be a generic and transversal  triangulated 3-manifold with defect data and $l_{\partial M}$ a labeling of its boundary. 
Then the state sum $\mathcal Z(M, l_{\partial M})$ from Definition \ref{def:statesum} is invariant under  generic transversal bistellar moves in the interior of $M$.
\end{corollary}

\begin{figure}
\def\svgwidth{.9\columnwidth}
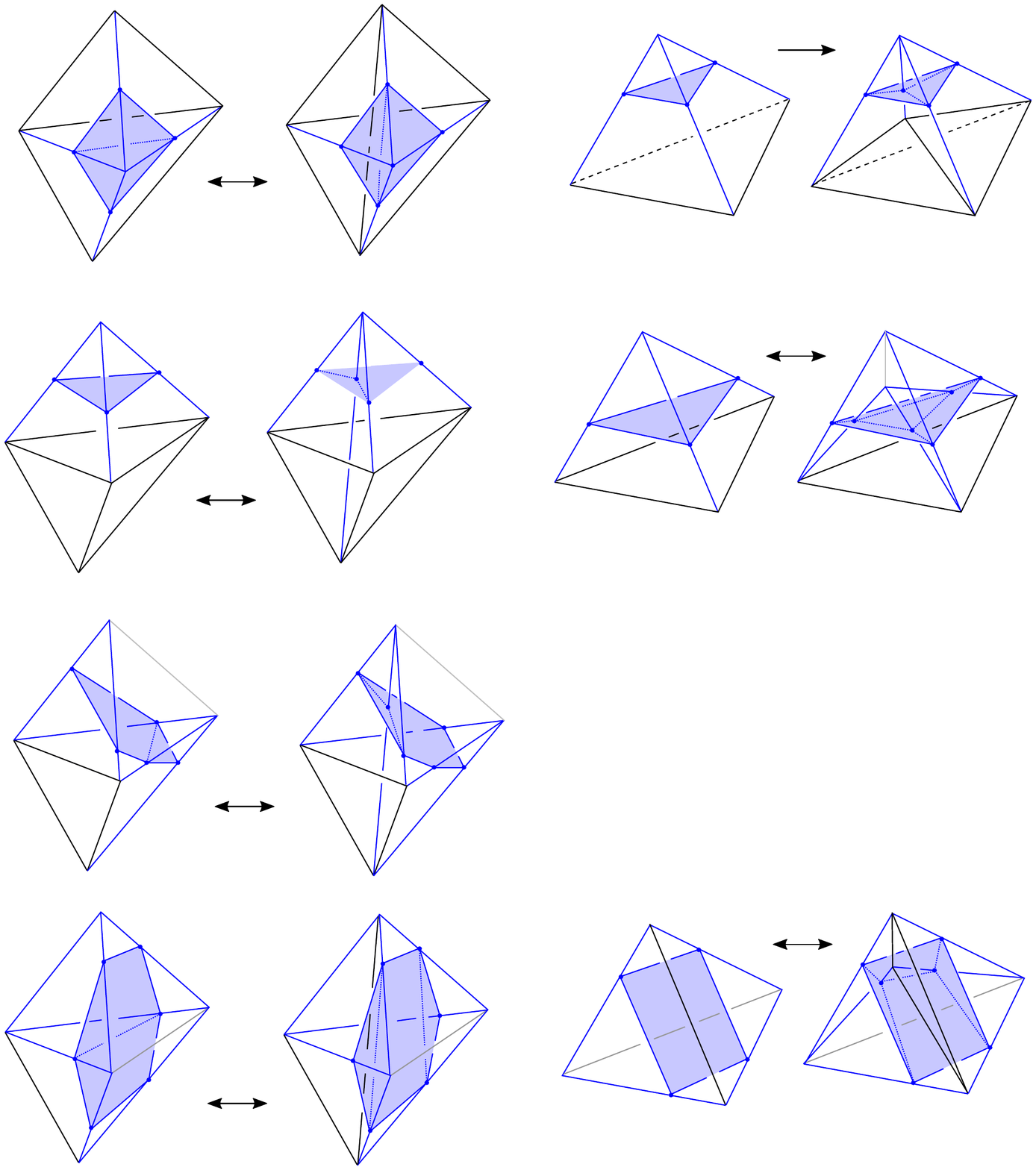
\caption{bistellar moves for tetrahedra without and with defect discs:\protect\newline
(a),(b):  encoding the associator of a spherical fusion category $\mac$,\protect\newline
(c)-(f):   encoding the natural isomorphism $c$ for a $\mac$-module category $\mam$,\protect\newline
(g)-(i):  encoding the natural isomorphism $b$ for a $(\mac,\mad)$-bimodule category $\mam$.
}
\label{fig:pachnerdef}
\end{figure}

All  generic transversal bistellar moves  for tetrahedra with defect surfaces are shown in Figure \ref{fig:pachnerdef}.  The moves for tetrahedra that also involve defect lines and defect vertices are analogous, but with the additional conditions that  defect vertices must not lie on faces  and defect edges must not intersect edges of the tetrahedra.

The bistellar invariance of state sums for 3-manifolds with defect surfaces has a similar interpretation to the one for Turaev-Viro-Barrett-Westbury state sums outlined at the end of Section \ref{subsec:PLsubsec}. The moves  in Figure \ref{fig:pachnerdef} (c) to (f) encode the pentagon relation \eqref{eq:pentagoncdef} for the structure  isomorphisms $c_{i,j,m}: (i\oo j)\rhd m\to i\rhd (j\rhd m)$ for a $\mac$-module category and equivalent relations obtained from it by taking inverses. The  moves in Figure \ref{fig:pachnerdef} (g) to (i) encode the pentagon relation \eqref{eq:pentagonbdef} and the invertibility of the structure isomorphisms $b_{c,m,d}: (c\rhd m)\lhd d\to c\rhd (m\lhd d)$ in a $(\mac,\mad)$-bimodule category.  
The corresponding moves for surfaces with line and point defects combine these pentagon relations with the pentagon relations \eqref{eq:pentagoncfunc}, \eqref{eq:pentagondfunc} for $\mac$-module and $\mad$-right module functors and the hexagon relation \eqref{eq:hexa} for $(\mac,\mad)$-bimodule functors.

As for a spherical fusion category, the bistellar invariance for manifolds with defects can  be formulated as generalised  Biedenharn-Elliott and orthogonality relations for  6j symbols, which generalise the diagrammatic identities  \eqref{eq:biedenharneliotsph} and \eqref{eq:orthogonalitysph}. However, it becomes cumbersome to list all of these relations.

\subsection{Triangulation independence}

In this section we  show  that  any two generic transversal triangulations of a  3-manifold with defect data yield the same state sum if they agree at its boundary. 
Note that this does not follow directly from  Corollary \ref{cor:pachner}, because it is not guaranteed a priori  that such triangulations are related by \emph{generic  transversal} bistellar moves. 
As a first step, we derive a counterpart  of Proposition \ref{th:topological invariance}  for general defect surfaces. 

We again consider 3-manifolds $M$  that shell to fine neighbourhoods of a defect surface $D$ and project the dual of the boundary triangulation on $\partial M$  to $D$  as in Definition \ref{def:projection} with an orientation preserving PL homeomorphism $\phi: M\to [0,1]\times D$. The only difference to Proposition \ref{th:topological invariance}  is that $D$ is no longer a disc and the resulting diagram on $D$ must be cut along the triangulation to obtain a polygon diagram.

\begin{proposition}\label{lem:polystatesum}
Let  $M$ be a triangulated 3-manifold with defect data that shells to a fine neighbourhood  of a defect surface $D$ by shellings of tetrahedra that do not intersect $D$. 
Let $P$  be a polygon diagram obtained by projecting the dual of the  triangulation on $\partial M$ to $D$ and  cutting $D$ along the triangulation  to a disc. 
 Then 
\begin{align*}
\mathcal Z'(M,  l_{\partial M}, b_{\partial M})=
\sum_{l}\sum_{b_{\Delta_i}} \mathrm{ev}(P)\prod_{e\in E_{i}} \dim l(e),\\[-6ex]
\end{align*}
where 
\begin{compactitem}
\item $E_{i}$ and $\Delta_i$ are the sets of internal edges and triangles of $M$ that 
 correspond to boundary vertices and edges in $P$, 
\item the sum is over all assignments of simple objects to  edges in $E_i$ and bases of the morphism spaces of  triangles in  $\Delta_i$. 
\end{compactitem}
\end{proposition}

\begin{proof} 
The claim  follows by induction over the number $n$  of tetrahedra in $M$ that do not intersect $D$. 

1.~For $n=0$,  $M$ is a fine neighbourhood of $D$. Cutting $D$ along the triangulation to a disc corresponds to cutting $M$ to a 3-ball $B$ that is a fine neighbourhood of a defect disc. 
Edges in $E_i$ correspond to  boundary edges of $B$  and elements of $\Delta_i$ to boundary triangles in $B$.   Every labeling of internal edges in $M$  defines a labeling $l$ of  $E_i$ and a labeling $l'$ of the internal edges of $B$. 

The state sum for $M$ is given by formula \eqref{eq:state sum}. 
By Proposition \ref{th:topological invariance}, summing over the labeling $l'$ and the labelings of internal triangles of $B$  yields the state sum for $B$, which is the evaluation of $P$. 
As every edge in $E_i$ is internal in $M$, but  in the boundary of $B$, the state sum for $M$ is obtained by multiplying this expression with $\Pi_{e\in E_{i}} \dim l(e)$
and summing over the labelings of edges in $E_i$ and morphism spaces for triangles in $\Delta_i$.

2.~Suppose that  $M$ contains $n+1$ tetrahedra that do not intersect $D$. 
As $M$ shells to a fine neighbourhood of $D$, it is obtained by gluing a tetrahedron $t$  that does not intersect $D$ to triangulated manifold  $M'$ that also satisfies the assumptions and has $n$ tetrahedra that do not intersect $D$.  
If $t$ is contained in the region labeled by a spherical fusion category $\mac$, then by the induction hypothesis and 
Corollary \ref{cor:gluing} the state sum of $M$ is given by formula  \eqref{eq:indstate2} in the proof of Proposition \ref{th:topological invariance}. 

It remains to show that  \eqref{eq:indstate2} yields  $\ev(P)$. 
If $t$ is glued to $M'$ along a single face, this follows as in 2.(i) in the proof of Proposition \ref{th:topological invariance}. 
 If $t$ is glued to $M'$ along two faces $f,f'$ that share an edge $e=f\cap f''$ or along three faces $f,f', f''$ that share edges $e=f\cap f'$, $e'=f\cap f''$ and $e''=f'\cap f''$, 
  we can assume without loss of generality that the edges in the polygon diagram  $P'$ for $M'$  that correspond to $e$, $e'$, $e''$ do not intersect  $\partial P'$. This is possible, because the assumptions imply that they do not intersect a side of $P'$ that is contained in $\partial D$. If they intersect  sides of $P'$ that are in the interior of  $D$, we 
can apply identity \eqref{pic:gluepoly} to cut   $P'$, glue it back together in such a way that these edges do not intersect $\partial P'$ and reverse this transformation after the gluing of $t$.
The claim then  follows as in 2.(ii) and 2.(iii) in the proof of Proposition \ref{th:topological invariance}. \end{proof}

After summation over the objects and morphisms at its boundary, the evaluation of the polygon diagram in Proposition \ref{lem:polystatesum} is again independent of the choice of the  PL homeomorphism $\phi: M\to [0,1]\times D$ that defines the projection of the boundary graph on $D$. Lemma \ref{lem:projiso} ensures that different choices for $\phi$ yield PL isotopic projection maps and hence diagrams for spherical fusion categories with the same evaluations. Together with identities   \eqref{pic:rm2} to \eqref{pic:bimodulenat} 
and identities \eqref{pic:gluepoly} and \eqref{pic:closepoly} this ensures the polygon evaluations for different choice of $\phi$ are equal. 
As the latter  are precisely the cutting and gluing identities that relate different polygon presentations of surfaces, see for instance \cite[Ch.~6]{Lee} and \cite[Ch.~6.38 and 6.40]{ST}, it also follows that  all ways of cutting the surface in Proposition \ref{lem:polystatesum} 
yield the same value for the state sum. 

Note also that Proposition \ref{lem:polystatesum}  allows one to  drop the requirement that a fine triangulation is generic in the interior of $M$.
For a triangulation that is non-generic in the interior of $M$, a small displacement of the defect surfaces, defect lines and defect vertices yields a generic triangulation. As the cyclic evaluations of the associated polygon diagrams coincide, all such displacements have the same state sum.

It remains to show that the state sum from Proposition \ref{lem:polystatesum}  does not depend on the choice of the triangulation in the interior of $M$.  
This follows again from identities \eqref{pic:gluepoly} and \eqref{pic:closepoly}.

\begin{lemma} \label{lem:subdivide}
Let $M,M'$ be triangulated 3-manifolds with defect data that shell to fine neighbourhoods of a defect surface $D$ with shellings by tetrahedra that do not intersect $D$. Suppose that the triangulations of $M,M'$ coincide on the boundary. Then for all labelings $l_{\partial M}$ and $b_{\partial M}$ 
$$
\mathcal Z(M,l_{\partial M}, b_{\partial M})=\mathcal Z(M',l_{\partial M}, b_{\partial M}).
$$
\end{lemma}

\begin{proof}
By Proposition \ref{lem:polystatesum} the rescaled state sums for $M,M'$ are obtained by  projecting the dual graph of the boundary triangulation on $\partial M$ to $D$,  cutting $D$ along the triangulations to discs and evaluating the resulting  polygon diagrams $P,P'$. With identities \eqref{pic:gluepoly} and \eqref{pic:closepoly}  we can cut the polygon diagrams $P,P'$ into smaller pieces that are the intersections of $P,P'$ with the tetrahedra in the two  triangulations. 

More precisely, one obtains embedded graphs $\Gamma,\Gamma'$  on $D$, whose vertices, edges and faces are the intersections of $D$ with the edges, triangles and tetrahedra in the triangulations.   We denote by the same letters the associated graphs  on $P,P'$ and call  vertices or edges of $\Gamma,\Gamma'$ internal, if they are not contained in $\partial D$ and boundary otherwise.   By adjusting the triangulations, we can assume that $\Gamma,\Gamma'$ are \emph{generic}, meaning that
(i) defect vertices  and vertices labeled by  spherical fusion categories 
 are in the interior of faces of $\Gamma,\Gamma'$,
(ii) defect edges and edges labeled by spherical fusion categories intersect the edges of $\Gamma,\Gamma'$ transversally in the interior. 

With the labelings of the defects and the triangulations every face of $\Gamma$ and $\Gamma'$ becomes a polygon diagram. 
By Proposition \ref{lem:polystatesum},   \eqref{pic:gluepoly} and \eqref{pic:closepoly}, the rescaled state sums for $M,M'$  are
  obtained by multiplying the evaluations of these polygon diagrams for $\Gamma,\Gamma'$, rescaling with the dimensions of the simple objects at internal vertices and summing  over all  labelings of  internal vertices and over bases of the morphism spaces at internal edges.
 
Using again identities \eqref{pic:gluepoly} and \eqref{pic:closepoly}, we cut these polygon diagrams into smaller pieces that correspond to a common subdivision of $\Gamma$ and $\Gamma'$.  By slightly displacing the vertices of $M,M'$  one can achieve that the graphs $\Gamma,\Gamma'$ remain generic  and that  every  point $v''\in (\Gamma\cap \Gamma')\setminus \partial D$ is a transversal intersection point of an edge $e$ of $\Gamma$ and an edge $e'$  of $\Gamma'$, while all boundary edges and vertices of $\Gamma,\Gamma'$ remain fixed.  If the displacements of vertices  are sufficiently small, this does not change the evaluations of the polygon diagrams for $\Gamma,\Gamma'$ and hence preserves the state sums.

Let $\Gamma''$ be the embedded graph  on $D$
obtained by superimposing $\Gamma$ and $\Gamma'$ and inserting a new four-valent vertex at each intersection point $v''=e\cap e'$.  Then every edge of $\Gamma''$ is either on an edge of $\Gamma$ or of $\Gamma'$. Every vertex of $\Gamma''$ is either a vertex of $\Gamma$, of $\Gamma'$ or an intersection point of an edge of $\Gamma$ and an edge of $\Gamma'$.
  Summing the  product of polygon  evaluations for  $\Gamma''$ over the labelings at those  internal  edges and vertices  that lie on $\Gamma'$ 
  yields the corresponding  product of  polygon evaluations for $\Gamma$. Summing it instead over  the labelings at the internal  edges and vertices that lie on $\Gamma$ gives the corresponding  product of  polygon evaluations for $\Gamma'$. By combining these two summations one obtains that the state sums are equal. 
\end{proof}

By combining Lemma \ref{lem:subdivide} with the triangulation independence of Turaev-Viro-Barrett-Westbury state sums for spherical fusion categories, we can now show that the state sums for two generic transversal triangulations of a 3-manifold $M$ with defect data coincide, whenever their triangulations coincide at the boundary.

\begin{theorem}\label{th:fulltopinv} Let $T,T'$ be generic transversal triangulations of a 3-manifold $M$ with defect data that coincide on $\partial M$.  Then for all boundary labelings $l_{\partial M}$  and $b_{\partial M}$   the state sums for  $T$ and $T'$ are equal:
$$
\mathcal Z(M,l_{\partial M}, b_{\partial M},T)=\mathcal Z(M, l_{\partial M}, b_{\partial M}, T').
$$
\end{theorem}

\begin{proof} By Lemma \ref{lem:finetriang} there are  finite sequences of  bistellar moves  between generic  transversal  tetrahedra that transform the triangulations $T,T'$ into triangulations $S,S'$ that satisfy (i)-(iii) in Lemma \ref{lem:finetriang}. By Corollary \ref{cor:pachner} one has $\mathcal Z(M,l_{\partial M}, b_{\partial M},T)=\mathcal Z(M,l_{\partial M}, b_{\partial M},S)$ and $\mathcal Z(M,l_{\partial M}, b_{\partial M},T')=\mathcal Z(M,l_{\partial M}, b_{\partial M},S')$. 
To show that $\mathcal Z(M,l_{\partial M}, b_{\partial M},S)=\mathcal Z(M,l_{\partial M}, b_{\partial M},S')$, we
denote by $\Sigma_i$ the connected components of the defect manifold $D=\amalg_{i=1}^n \Sigma_i$ and by $r_k$ the regions of $M$, which are the connected components of $M\setminus D=\amalg_{k=1}^m r_k$.

 By Lemma \ref{lem:finetriang} (ii)  gluing all tetrahedra in $S, S'$ that are contained in a region $r_k$ yields  triangulated  3-manifolds $R_k, R'_k$ that are PL homeomorphic to $\bar r_k$.  
 By Lemma \ref{lem:finetriang} (ii) gluing all tetrahedra in $S, S'$ that intersect $\Sigma_i$ yields fine neighbourhoods $D_i, D'_i$ of $\Sigma_i$.
 This implies that all boundary triangles in $D_i, D'_i$ that are internal triangles of $M$ are glued to boundary triangles in one of the 3-manifolds $R_k, R'_k$.
 More specifically, each connected component of $\partial D_i\setminus \partial M$ coincides with a unique boundary component of  some $R_k$,  up to the fact that they have  opposite orientations. 
 As $\partial D\subset \partial M$, all boundary triangles of $D_i$ and $D'_i$ that intersect $\Sigma_i$ are contained in $\partial M$ and hence coincide. 
 
As  $D_i$ and $D'_i$ are both PL homeomorphic  to  $[0,1]\times \Sigma_i$, their boundaries are PL homeomorphic. 
As their boundary triangulations agree on $\partial D\subset \partial M$, this allows one to transform each connected component of $\partial D'_i\setminus \partial M$ into the corresponding  connected component of  $\partial D_i\setminus \partial M$ by a finite number  shellings  or inverse shellings with tetrahedra that do not intersect any defect surface. Performing the same  shellings on the associated boundary component of $R'_k$   transforms it into the corresponding boundary component of  $R_k$.

For each labeling of the edges and triangles, performing  such a simultaneous  shelling or inverse shelling at $D'_i$ and $R'_k$ amounts to multiplying or dividing the state sum of $M'$ by the 6j symbol of a  tetrahedron $t$  labeled by a spherical fusion category, with the 6j symbol of the tetrahedron $\bar t$ with the same labeling, but the opposite orientation, and with  dimension factors associated to their edges.   The orthogonality relation \eqref{eq:orthogonalitysph} and identity \eqref{eq:dimformula} ensure that  this does not change the state sum. 

Hence, we can assume without loss of generality that the triangulations of all 3-manifolds $D'_i$ and $D_i$ coincide on the boundary and that the same holds for all  triangulated 3-manifolds $R_k$ and $R'_k$.  
  Under this assumption, the state sums of $D_i$ and $D'_i$  agree by Lemma \ref{lem:subdivide} for
 all labelings. 
 The state sums of $R_k$ and $R'_k$ are the usual Turaev-Viro-Barrett-Westbury state sums for a  spherical fusion category and are equal by Theorem \ref{th:pachinv}. 
 With Corollary \ref{cor:gluing} this yields 
 \begin{align*}
\mathcal Z(M, l_{\partial M}, b_{\partial M}, S)&=\sum_{l_{\partial}}\sum_{b_\partial}  \prod_{i=1}^n\mathcal Z(D_i, l_{\partial D_i}, b_{\partial D_i})\prod_{k=1}^m  \mathcal Z(R_k, l_{\partial R_k}, b_{\partial R_k})\\
&=\sum_{l_{\partial}}\sum_{b_\partial}  \prod_{i=1}^n\mathcal Z(D'_i, l_{\partial D'_i}, b_{\partial D'_i})\prod_{k=1}^m  \mathcal Z(R'_k, l_{\partial R'_k}, b_{\partial R'_k})=\mathcal Z(M, l_{\partial M}, b_{\partial M}, S'),
 \end{align*}
 where the sums are over all labelings $l_\partial$ of internal edges in $M$ that are boundary edges of  $D_i$, $D'_i$ and over dual bases of the morphism spaces for internal triangles  in $M$ that are boundary triangles of  $D_i$, $D'_i$. 
\end{proof}

Theorem \ref{th:fulltopinv}  extends the definition of the state sum for a defect manifold to triangulations that are non-generic or non-transversal in the interior of $M$.  
Together with Proposition \ref{lem:polystatesum} it shows that slightly perturbing the defect surface, defect lines or defect vertices in the interior of $M$ does not change the value of the state sum. This allows one to define the state sum of a non-generic triangulation as the state sum of a small generic deformation.
Similarly, any triangulation of $M$ with non-transversal tetrahedra in the interior can be refined 
 by  stellar subdivisions to a triangulation that is transversal. 
 One can then perturb the defect surfaces, lines and vertices  to make it generic. 
  By Theorem \ref{th:fulltopinv}  all ways of doing so yield the same  state sum.

\begin{definition} Let $M$ be a 3-manifold with defect data and $T$ a triangulation of $M$ whose boundary triangles are transversal and generic. Then for all boundary labelings $l_{\partial M}$,  $b$   the state sum of $M$ is  defined as
\begin{align}
\mathcal Z(M,l_{\partial M},b_{\partial M}, T):=\mathcal Z(M,l_{\partial M}, b_{\partial M}, T'),
\end{align}
for any generic transversal triangulation $T'$ that coincides with $T$ on $\partial M$.
\end{definition}

\begin{corollary}\label{cor:generalinvariant} Let $M$ be a 3-manifold with defect data and $T,T'$ triangulations of $M$ that coincide on $\partial M$ and whose boundary triangles are transversal and generic.
Then their state sums coincide for all boundary  labelings $l_{\partial M}$,  $b_{\partial M}$. 
In particular, if $M$ is closed, its state sum does not depend on the choice of the triangulation.
\end{corollary}

In principle the dependence  of the state sum on boundary triangulations  and their labeling could be removed in the same way as  for  Turaev-Viro-Barrett-Westbury state sums without defects.  To obtain a topological quantum field theory from the latter one considers 
the vector spaces $V_{\partial M}(t,l)=\otimes_{f\in \Delta_{t}} \Hom(f,l)$ for a given triangulation $t$ of $\partial M$ and a labeling $l$ with simple objects. The state sum of a triangulated cyclinder $\partial M\times[0,1]$  then defines a linear map $\mathcal Z(\partial M\times[0,1], l): V_{\partial M}(t_0,l_0)\to V_{\partial M}(t_1,l_1)$, where $t_i$ is the boundary triangulation of $\partial M\times \{i\}$ and $l_i$ its labeling. By taking the projective limit of the system of all such state sums, one obtains the vector space that the associated TQFT assigns to $\partial M$, see  \cite[Sec.~13.3.2]{TVbook}. 

It is clear that an analogous procedure could be applied to state sum models with defects and seems plausible that this would result in a defect TQFT in the sense of \cite{CMS,CRS} by Carqueville et al.  For the spherical fusion categories $\mathrm{Vec}_G$ from Example \ref{ex:vectgomega} and oriented surfaces of genus $g\geq 1$  cylinders with  defects are considered in Example \ref{ex:surface}.

\section{Examples of state sums}
\label{sec:examples}

In this section, we  compute the state sums for some examples of defect 3-manifolds. We first consider state sums with defect surfaces, but  without defect lines or defect vertices. We  then show how to recover ribbon invariants from defect lines and vertices on trivial defect surfaces.

\subsection{State sums with defect surfaces}

In the following, we assume that all triangulations are generic and transversal at the boundary. 
Our first example is a defect sphere without defect lines or vertices embedded inside a 3-ball.  We suppose that the embedding is \emph{unknotted} in the sense of Zeeman \cite{Z}. This means that there is a PL homeomorphism from the 3-ball to itself that  sends the defect sphere to the boundary of a tetrahedron, see also Footnote 1  \cite{Z}.

\begin{example} \label{ex:tetrahedronsphere} Let $M$ be a 3-ball with an unknotted defect sphere  in its interior. Suppose the boundary of $M$ is labeled by a spherical fusion category $\mac$ and the defect sphere by 
 a $(\mac, \mad)$-bimodule category  $\mam$ with a bimodule trace.  Then for all boundary labelings $l_{\partial M}$, $b_{\partial M}$ the state sum of $M$ is
 \begin{align}\label{eq:bublle}
\mathcal Z(M, l_{\partial M},b_{\partial M})=\frac{\dim\mam}{\dim \mad}\cdot \mathcal Z(M', l_{\partial M}, b_{\partial M}), 
\end{align}
where $M'$ is a triangulated 3-ball labeled by $\mac$ with the same  boundary triangulation, but without defects. 
\end{example}

\begin{proof} As the defect sphere is unknotted, there is a PL homeomorphism from $M$ to itself that sends $D$ to the boundary of a tetrahedron.
By Theorem \ref{th:fulltopinv}, Corollary \ref{cor:generalinvariant} and because 
 different boundary triangulations  are related by elementary shellings and inverse shellings, it is sufficient to prove the claim for the triangulation
 \begin{center}
\def\svgwidth{.52\columnwidth}
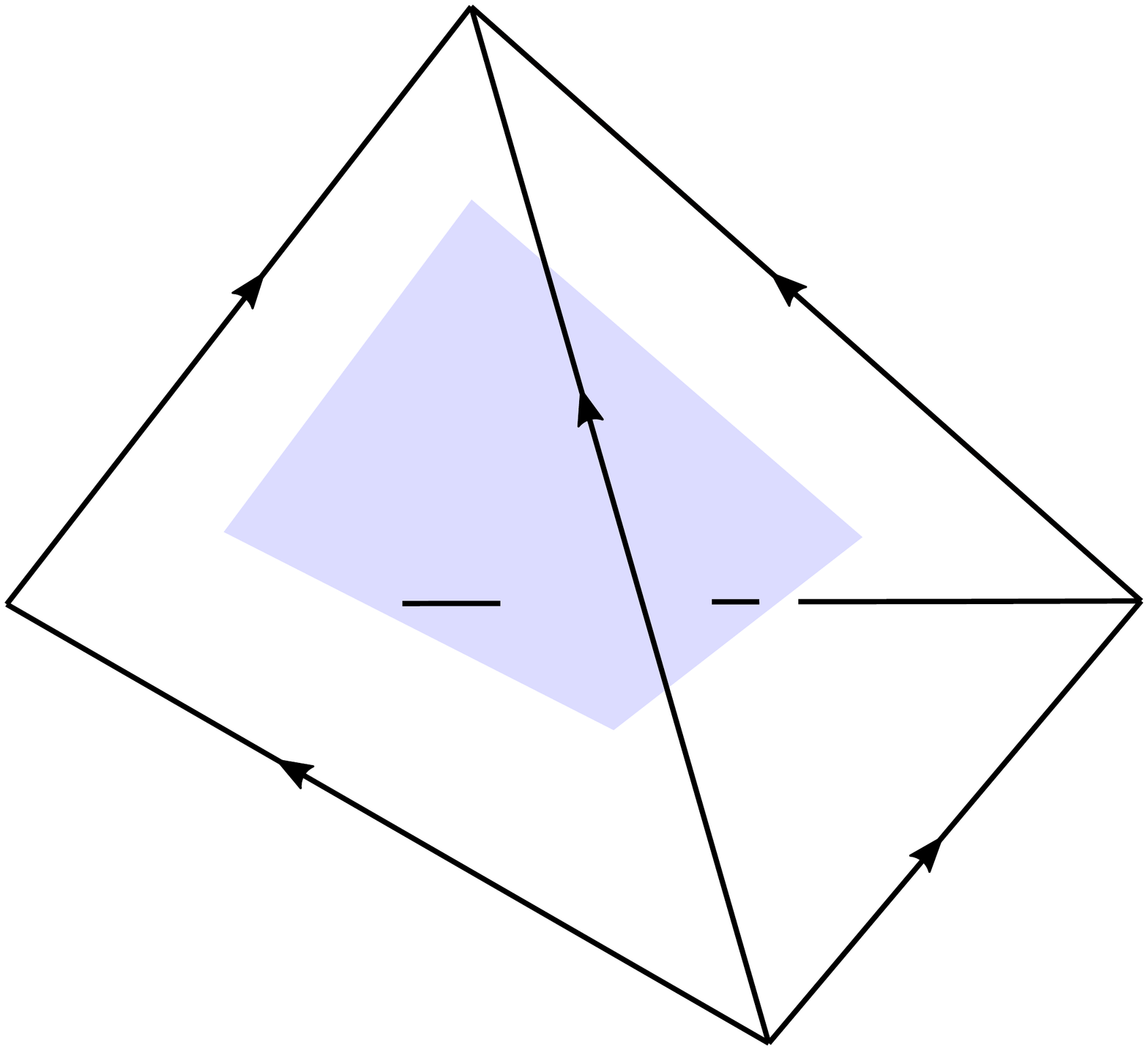 
\end{center}

We evaluate the state sum by gluing the polygon diagrams for the four tetrahedra in the triangulation. Gluing the bottom tetrahedron to the three tetrahedra above   with identity \eqref{pic:gluepoly} yields for the rescaled state sum 
\begin{align*}
&\begin{tikzpicture}[scale=.24]
\node at (-28,8) {$\mathcal Z'(M, l_{\partial M},b_{\partial M})=\frac 1 {\dim \mad}
\sum_{\substack{ {m,n,p,q}\\{\alpha,\beta,\gamma}}} \dim m\dim n\dim p\dim q$};
\draw[line width=1.5pt, color=blue, dotted] (-6,8) node[anchor=east]{$p$}--(6,8) node[anchor=west]{$n$};
\draw[line width=1.5pt, color=blue, dotted] (-6,8)--(0,0);
\draw[line width=1.5pt, color=blue, dotted] (6,8)--(0,0);
\draw[line width=1.5pt, color=blue] (-6,8)--(0,16) node[anchor=south]{$m$};
\draw[line width=1.5pt, color=blue] (6,8)--(0,16);
\draw[line width=1.5pt, color=blue] (-12,0)node[anchor=east]{$m$}--(-6,8);
\draw[line width=1.5pt, color=blue] (-12,0)--(0,0) node[anchor=north]{$q$};
\draw[line width=1.5pt, color=blue] (12,0) node[anchor=west]{$m$}--(6,8);
\draw[line width=1.5pt, color=blue] (12,0)--(0,0);
\draw[line width=.5pt] (0,5)--(0,10) node[pos=.3, anchor=east]{$g$} node[sloped, pos=.3, allow upside down]{\arrowOut};
\draw[line width=.5pt] (-3,12)--(0,10)node[pos=.5, anchor=north]{$i$} node[sloped, pos=.5, allow upside down]{\arrowOut};
\draw[line width=.5pt] (0,10)--(3,12) node[pos=.5, anchor=north]{$k$} node[sloped, pos=.5, allow upside down]{\arrowOut};
\draw[line width=.5pt] (-6,3)--(0,5) node[pos=.7, anchor=south]{$j$} node[sloped, pos=.7, allow upside down]{\arrowOut};
\draw[line width=.5pt] (-6,3)--(-9,4) node[pos=.5, anchor=south]{$i$} node[sloped, pos=.5, allow upside down]{\arrowOut};
\draw[line width=.5pt] (-6,0)--(-6,3) node[pos=.5, anchor=east]{$h$} node[sloped, pos=.5, allow upside down]{\arrowOut};
\draw[line width=.5pt] (0,5)--(6,3) node[pos=.3, anchor=south]{$l$} node[sloped, pos=.3, allow upside down]{\arrowOut};
\draw[line width=.5pt] (9,4)--(6,3) node[pos=.5, anchor=south]{$k$} node[sloped, pos=.5, allow upside down]{\arrowOut};
\draw[line width=.5pt] (6,3)--(6,0) node[pos=.5, anchor=east]{$h$} node[sloped, pos=.5, allow upside down]{\arrowOut};
\draw[fill=black] (0,10) circle (.2) node[anchor=south]{$\lambda$};
\draw[fill=black] (0,5) circle (.2) node[anchor=north]{$\mu$};
\draw[fill=black] (-6,3) circle (.2) node[anchor=south]{$\nu$};
\draw[fill=black] (6,3) circle (.2) node[anchor=south]{$\rho$};
\draw[fill=blue, color=blue] (-3,12) circle (.2) node[anchor=east]{$\alpha\;$};
\draw[fill=blue, color=blue] (3,12) circle (.2) node[anchor=west]{$\;\beta$};
\draw[fill=blue, color=blue] (-9,4) circle (.2) node[anchor=east]{$\alpha\;$};
\draw[fill=blue, color=blue] (9,4) circle (.2) node[anchor=west]{$\;\beta$};
\draw[fill=blue, color=blue] (-6,0) circle (.2) node[anchor=north]{$\;\gamma$};
\draw[fill=blue, color=blue] (6,0) circle (.2) node[anchor=north]{$\;\gamma$};
\end{tikzpicture}
\end{align*}
Applying identity \eqref{pic:closepoly} to eliminate the  summation over $\alpha$, $p$ yields
\begin{align*}
&\begin{tikzpicture}[scale=.24]
\node at (-30,0) {$\mathcal Z'(M, l_{\partial M},b)=\frac 1 {\dim \mad}\sum_{\substack{ {m,n,q}\\{\beta,\gamma}}} \dim m\dim n\dim q$};
\draw[line width=1.5pt, color=blue] (-10,0) node[anchor=east]{$m$}--(0,10) node[anchor=south]{$n$};
\draw[line width=1.5pt, color=blue] (10,0) node[anchor=west]{$m$}--(0,10);
\draw[line width=1.5pt, color=blue] (-10,0)--(0,-10) node[anchor=north]{$q$};
\draw[line width=1.5pt, color=blue] (10,0)--(0,-10);
\draw[line width=.5pt](5,5)--(5,0) node[pos=.5, anchor=west]{$k$} node[sloped, pos=.5, allow upside down]{\arrowOut};
\draw[line width=.5pt](5,0)--(5,-5) node[pos=.5, anchor=west]{$h$} node[sloped, pos=.5, allow upside down]{\arrowOut};
\draw[line width=.5pt](0,0)--(5,0)node[pos=.5, anchor=south]{$l$} node[sloped, pos=.5, allow upside down]{\arrowOut};
\draw[line width=.5pt](-5,-2)--(0,0) node[pos=.5, anchor=north]{$j$} node[sloped, pos=.5, allow upside down]{\arrowOut};
\draw[line width=.5pt](0,0)--(-5,2) node[pos=.5, anchor=south]{$g$} node[sloped, pos=.5, allow upside down]{\arrowOut};
\draw[line width=.5pt](-5,-2)--(-5,2) node[pos=.7, anchor=west]{$i$} node[sloped, pos=.7, allow upside down]{\arrowOut};
\draw[line width=.5pt](-5,-5)--(-5,-2) node[pos=.5, anchor=west]{$h$} node[sloped, pos=.5, allow upside down]{\arrowOut};
\draw[line width=.5pt](-5,2)--(-5,5) node[pos=.5, anchor=west]{$k$} node[sloped, pos=.5, allow upside down]{\arrowOut};
\draw[fill=blue, color=blue] (-5,5) circle (.2) node[anchor=south east]{$\;\beta$};
\draw[fill=blue, color=blue] (5,5) circle (.2) node[anchor=south west]{$\;\beta$};
\draw[fill=blue, color=blue] (-5,-5) circle (.2) node[anchor=north east]{$\;\gamma$};
\draw[fill=blue, color=blue] (5,-5) circle (.2) node[anchor=north west]{$\;\gamma$};
\draw[fill=black] (-5,2) circle (.2) node[anchor=east]{$\;\lambda$};
\draw[fill=black] (-5,-2) circle (.2) node[anchor=south east]{$\;\nu$};
\draw[fill=black] (0,0) circle (.2) node[anchor=south]{$\;\mu$};
\draw[fill=black] (5,0) circle (.2) node[anchor=west]{$\;\rho$};
\end{tikzpicture}
\end{align*}
and applying identity \eqref{pic:closepoly} to eliminate the summation over $\beta,n$ gives 
\begin{align*}
&\begin{tikzpicture}[scale=.24]
\node at (-28,0) {$\mathcal Z'(M, l_{\partial M},b_{\partial M})=\frac 1 {\dim \mad}\sum_{m,q,\gamma} \dim m\dim q$};
\draw[line width=1.5pt, color=blue] (10,0) node[anchor=west]{$m$}.. controls (2,-10) and (-2,-10)..(-10,0) node[anchor=east]{$q$};
\draw[line width=1.5pt, color=blue] (10,0) .. controls (2,10) and (-2,10)..(-10,0) ;
\draw[line width=.5pt] (0,7.5)--(0,4) node[pos=.5, anchor=west]{$h$} node[sloped, pos=.5, allow upside down]{\arrowOut};
\draw[line width=.5pt] (0,-4)--(0,-7.5) node[pos=.5, anchor=west]{$h$} node[sloped, pos=.5, allow upside down]{\arrowOut};
\draw[line width=.5pt] (0,4)--(3,-1) node[pos=.5, anchor=south west]{$i$} node[sloped, pos=.5, allow upside down]{\arrowOut};
\draw[line width=.5pt] (0,4)--(-3,1) node[pos=.5, anchor=south east]{$j$} node[sloped, pos=.5, allow upside down]{\arrowOut};
\draw[line width=.5pt] (-3,1)--(3,-1) node[pos=.5, anchor=south]{$g$} node[sloped, pos=.5, allow upside down]{\arrowOut};
\draw[line width=.5pt] (3,-1)--(0,-4) node[pos=.5, anchor=north west]{$k$} node[sloped, pos=.5, allow upside down]{\arrowOut};
\draw[line width=.5pt] (-3,1)--(0,-4) node[pos=.5, anchor=north east]{$l$} node[sloped, pos=.5, allow upside down]{\arrowOut};
\draw[fill=blue, color=blue] (0,7.5) circle (.2) node[anchor=south]{$\gamma$};
\draw[fill=blue, color=blue] (0,-7.5) circle (.2) node[anchor=north]{$\gamma$};
\draw[fill=black] (0,4) circle (.2) node[anchor=west]{$\;\nu$};
\draw[fill=black] (0,-4) circle (.2) node[anchor=west]{$\;\rho$};
\draw[fill=black] (-3,1) circle (.2) node[anchor=east]{$\mu\;$};
\draw[fill=black] (3,-1) circle (.2) node[anchor=west]{$\;\lambda$};
\end{tikzpicture}
\end{align*}
Applying identity \eqref{pic:2gon} to this diagram yields the result
\begin{align*}
\quad\begin{tikzpicture}[scale=.23]
\node at (-15,0) {$\mathcal Z'(M, l_{\partial M},b_{\partial M})=\frac{\dim \mam}{\dim\mad}\quad$};
\draw[line width=.5pt] (0,7.5)--(0,4) node[pos=.5, anchor=west]{$h$} node[sloped, pos=.5, allow upside down]{\arrowOut};
\draw[line width=.5pt] (0,-4)--(0,-7.5) node[pos=.5, anchor=west]{$h$} node[sloped, pos=.5, allow upside down]{\arrowOut};
\draw[line width=.5pt] (0,4)--(3,-1) node[pos=.5, anchor=south west]{$i$} node[sloped, pos=.5, allow upside down]{\arrowOut};
\draw[line width=.5pt] (0,4)--(-3,1) node[pos=.5, anchor=south east]{$j$} node[sloped, pos=.5, allow upside down]{\arrowOut};
\draw[line width=.5pt] (-3,1)--(3,-1) node[pos=.5, anchor=south]{$g$} node[sloped, pos=.5, allow upside down]{\arrowOut};
\draw[line width=.5pt] (3,-1)--(0,-4) node[pos=.5, anchor=north west]{$k$} node[sloped, pos=.5, allow upside down]{\arrowOut};
\draw[line width=.5pt] (-3,1)--(0,-4) node[pos=.5, anchor=north east]{$l$} node[sloped, pos=.5, allow upside down]{\arrowOut};
\draw[line width=.5pt] (0,7.5).. controls (0,10.5) and (-6,10.5)..(-6,7.5);
\draw[line width=.5pt] (-6,-7.5)--(-6,7.5) node[sloped, pos=.5, allow upside down]{\arrowOut};
\draw[line width=.5pt] (0,-7.5).. controls (0,-10.5) and (-6,-10.5)..(-6,-7.5);
\draw[fill=black] (0,4) circle (.2) node[anchor=west]{$\;\nu$};
\draw[fill=black] (0,-4) circle (.2) node[anchor=west]{$\;\rho$};
\draw[fill=black] (-3,1) circle (.2) node[anchor=east]{$\mu\;$};
\draw[fill=black] (3,-1) circle (.2) node[anchor=west]{$\;\lambda$};
\node at (14,0) {$=\frac{\dim \mam}{\dim \mad}\cdot \mathcal Z'(M', l_{\partial M},b_{\partial M}).$};
\end{tikzpicture}\\[-8ex]
\end{align*}
\end{proof}

Note that  Example \ref{ex:tetrahedronsphere}  yields the  usual Turaev-Viro-Barrett-Westbury invariant for a 3-ball without defects,
if one chooses for $\mam$ a spherical fusion category $\mac=\mad$ as a bimodule category over itself. 

State sums with defect surfaces of genus $g\geq 1$ are more difficult to compute.  We therefore restrict attention to the spherical fusion category $\mac=\mathrm{Vec}_G$ for a finite group $G$ from Example \ref{ex:vectgomega},  equipped  with  a trivial 3-cocycle and the standard spherical structure. 
The associated Turaev-Viro-Barrett-Westbury state sums are a special case of  the Dijkgraaf-Witten models from \cite{DW}, see also  Altschuler and Coste \cite{AC}.  

By Example \ref{ex:vectgomega} simple objects in $\mathrm{Vec}_G$ correspond to group elements $g\in G$. For simplicity, we denote them by $g$ instead of $\delta^g$ in the following. One has $\dim(g)=1$ for all $g\in G$ and $\dim(\mac)=|G|$ as well as
\begin{align}\label{homvecg0}
 \Hom_\mac(k, i\oo j)\cong \Hom_\mac(i\oo j, k)\cong \delta_k(ij)\C\qquad\qquad\forall i,j,k\in G.
 \end{align} 
By Examples \ref{ex:modvectgomega} and \ref{ex:tracevect}, every finite transitive $G\times G'^{op}$-set $X$ defines an  indecomposable  semisimple $(\mathrm{Vec}_G, \mathrm{Vec}_{G'})$-bimodule category $\mam$ with a bimodule trace, whose simple objects are in bijection  with elements of $X$. In the following, we only consider the case, where the cocycle defining $\psi \in H^2(L,\C^{\times})$ from Example  \ref{ex:modvectgomega} is trivial.  One  has $\dim(x)=1$ for all $x\in X$ and $\dim(\mam)=|X|$ as well as  
\begin{align}\label{eq:homvecg}
&\Hom_\mam(x, g\rhd y)\cong \Hom_\mam (g\rhd y, x) \cong \delta_x(g\rhd y) \C\\  
&\Hom_\mam(x,  y\lhd g')\cong \Hom_\mam (y\lhd g', x) \cong \delta_x( y\lhd g') \C \qquad \forall x,y\in X,\, g\in G, \, g'\in G'.\nonumber
\end{align} 
Hence, any 6j symbol vanishes, unless for each triangle with edges labeled by  $x, y\in X$ and by $g\in G$ or $g'\in G'$, the objects $x$ and $y$ are related by the action of $g$ or $g'$. Similarly,  for any triangle with labeled  elements of $G$ or $G'$, the oriented product of the group elements must be trivial by \eqref{homvecg0}. 
Morphism labels for triangles are just numbers in $\C$, and all sums over bases of morphism spaces become trivial.

\begin{example}  \label{ex:surface}Let  $\mac=\mathrm{Vec}_G$ and $\mad=\mathrm{Vec}_{G'}$  for finite groups $G,G'$ and $\mam$ the indecomposable $(\mac,\mad)$-bimodule category with bimodule trace defined by a finite transitive  $G\times G'^{op}$-set $X$. 

Let $\Sigma$ be a closed oriented triangulated surface of genus $g\geq 1$ and $M=[0,1]\times\Sigma$ the fine neighbourhood of $\Sigma$
obtained by replacing each triangle in $\Sigma$ by a prism as in Figure \ref{eq:prism}, with a defect surface $\{\tfrac 1 2\}\times\Sigma$ labeled by  $\mam$ and boundary components $\{1\}\times\Sigma$ and $\{0\}\times\Sigma$ labeled by $\mac$ and $\mad$.  Suppose that all boundary triangles with non-trivial morphism spaces are labeled with identity morphisms. 

Then the rescaled state sum of $M$ is 
\begin{align*}
 &\mathcal Z'(M, l_{\partial M}, b_{\partial M})=\begin{cases}
|X^{\rho(\pi_1(\Sigma))}|   & \text{if the oriented product of the group elements in each triangle}\\[-.5ex]
&  \text{on $\{0\}\times\Sigma$ and $\{1\}\times\Sigma$ is trivial,}\\[+1ex]
0 & \text{else},
\end{cases}
\end{align*}
where  $\rho: \pi_1(\Sigma)\to G\times G'$ is the group homomorphism defined by the labels on $\{0\}\times\Sigma$ and $\{1\}\times\Sigma$ and
$X^{\rho(\pi_1(\Sigma))}$ the fixed point set for the induced action of $\pi_1(\Sigma)$. 

In particular:
\begin{itemize}
\item {\bf trivial defect surface:} If $G=G'=X$ with the standard $G\times G^{op}$-set structure, then 
$\mathcal Z'(M, l_{\partial M},b_{\partial M})=0$ unless the oriented product of the group elements in each triangle on $\{0\}\times\Sigma$ and $\{1\}\times\Sigma$ is trivial and the associated group homomorphisms $\rho_0,\rho_1: \pi_1(\Sigma)\to G$ are conjugated. Then
$\mathcal Z'(M, l_{\partial M},b_{\partial M})=|\mathrm{Stab}(\rho_0)|$, where $\mathrm{Stab}(\rho_0)$ is the stabiliser of $\rho_0$ for the conjugation action of $G$ on $\Hom_{\mathrm{Grp}}(\pi_1(\Sigma), G)$.

\item {\bf trivial $G\times G'^{op}$-set:} For a trivial $G\times G'$-set $X=\{\bullet\}$, one has  
$\mathcal Z'(M, l_{\partial M}, b_{\partial M})=1$, whenever
the oriented product of the group elements in each triangle on $\{0\}\times\Sigma$ and $\{1\}\times\Sigma$ is trivial.

\item {\bf normal subgroup:} If $X=G\times G'^{op}/N$ for a normal subgroup $N\subset G\times G'^{op}$ one has $\mathrm{Stab}(m)=\mathrm{Stab}(m')=N$ for all $m,m'\in X$. This implies  
$\mathcal Z'(M, l_{\partial M},b_{\partial M})=|X|$, whenever 
the oriented product of the group elements in each triangle on $\{0\}\times\Sigma$ and $\{1\}\times\Sigma$ is trivial and the labelings of $\{1\}\times\Sigma$  and $\{0\}\times\Sigma$  define a group homomorphism $\rho: \pi_1(\Sigma)\to N$. 
\end{itemize}
\end{example}

\begin{proof}
  Orient the top and bottom triangle of each prism as the associated triangle of $\Sigma$,  the vertical edges of each prism according to the orientation of $\Sigma$. Subdivide each prism into three tetrahedra as  in Figure \ref{eq:prism}. Label its top and bottom edges  with 
 elements of $G$ and $G'$   and the vertical  edges with elements of $X$.
  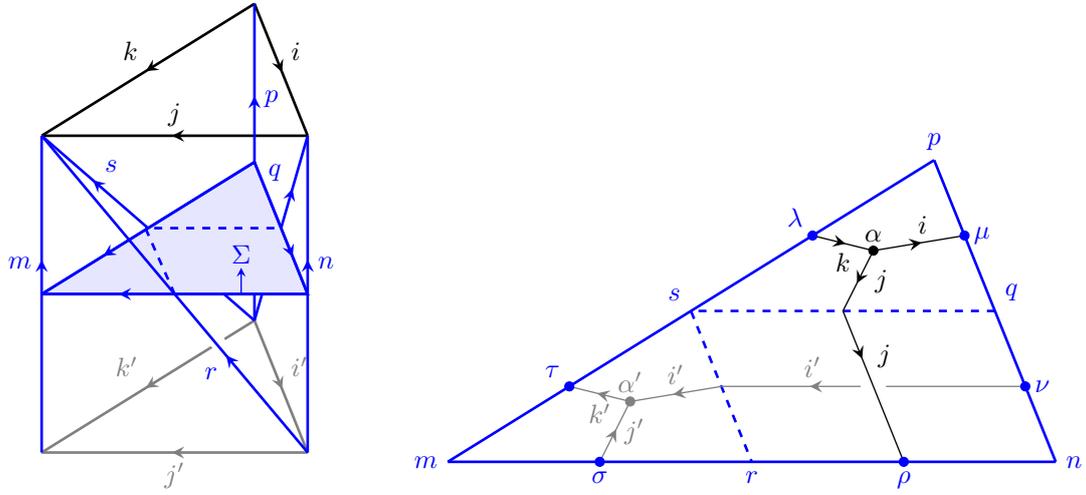
\begin{figure}
\begin{center}
\begin{tikzpicture}[scale=.35]
\draw[color=gray, line width=1pt] (5,0)--(-5,0) node[sloped, pos=0.5, allow upside down]{\arrowIn} node[pos=.5, anchor=north]{$j'$};
\draw[color=gray, line width=1pt] (3,5)--(5,0) node[sloped, pos=0.5, allow upside down]{\arrowIn} node[pos=.5, anchor=south west]{$i'$};
\draw[color=gray, line width=1pt] (3,5)--(-5,0) node[sloped, pos=0.5, allow upside down]{\arrowIn} node[pos=.5, anchor=south east]{$k'$};
\draw[color=black, line width=1pt] (5,12)--(-5,12) node[sloped, pos=0.5, allow upside down]{\arrowIn} node[pos=.5, anchor=south]{$j$};
\draw[color=black, line width=1pt] (3,17)--(5,12) node[sloped, pos=0.5, allow upside down]{\arrowIn} node[pos=.5, anchor=south west]{$i$};
\draw[color=black, line width=1pt] (3,17)--(-5,12) node[sloped, pos=0.5, allow upside down]{\arrowIn} node[pos=.5, anchor=south east]{$k$};
\draw[color=blue, line width=1pt] (-5,0)--(-5,12) node[sloped, pos=0.6, allow upside down]{\arrowIn} node[pos=.6, anchor=east]{$m$};
\draw[color=blue, line width=1pt] (5,0)--(5,12) node[sloped, pos=0.6, allow upside down]{\arrowIn} node[pos=.6, anchor=west]{$n$};
\draw[color=blue, line width=1pt] (3,5)--(3,17) node[sloped, pos=0.7, allow upside down]{\arrowIn} node[pos=.7, anchor=west]{$p$};
\draw[line width=1pt, color=blue] (3,5)--(-5,12) node[sloped, pos=0.75, allow upside down]{\arrowIn} node[pos=.75, anchor=south west]{$s$};
\draw[line width=1pt, color=blue] (3,5)--(5,12) node[sloped, pos=0.7, allow upside down]{\arrowIn} node[pos=.7, anchor=south east]{$q$};
\draw[draw=none, fill=white] (1.6,4.2) circle (.3);
\draw[color=blue, line width=1pt, fill=white] (-5,6)--(5,6)--(3,11)--(-5,6);
\draw[draw=none, fill=blue, fill opacity=.1] (-5,6)--(5,6)--(3,11)--(-5,6);
\node at (2.5,7.5) [color=blue]{$\Sigma$};
\draw[line width=.5pt, color=blue, -stealth] (2.5,6)--(2.5,7);
\draw[line width=1pt, color=blue] (5,6)--(-5,6) node[sloped, pos=0.7, allow upside down]{\arrowIn};
\draw[line width=1pt, color=blue] (3,11)--(-5,6) node[sloped, pos=0.7, allow upside down]{\arrowIn};
\draw[line width=1pt, color=blue] (3,11)--(5,6) node[sloped, pos=0.7, allow upside down]{\arrowIn};
\draw[line width=1pt, color=blue] (5,0)--(-5,12) node[sloped, pos=0.3, allow upside down]{\arrowIn} node[pos=.3, anchor=north east]{$r$};
\draw[line width=1pt, dashed, color=blue] (0,6)--(-1.1,8.5);
\draw[line width=1pt, dashed, color=blue] (3.75,8.5)--(-1.1,8.5);
\end{tikzpicture}
\qquad
\begin{tikzpicture}[scale=.4]
\draw[color=blue, line width=1pt] (10,0) node[anchor=west]{$n$}--(-10,0) node[anchor=east]{$m$}; 
\draw[color=blue, line width=1pt] (6,10) node[anchor=south]{$p$}--(10,0); 
\draw[color=blue, line width=1pt] (6,10)--(-10,0); 
\draw[color=blue, line width=1pt, dashed] (8,5) node[anchor=south west]{$q$}--(-2,5) node[anchor=south east] {$s$} --(0,0) node[anchor=north]{$r$};
\draw[line width=.5pt, color=gray] (9,2.5)--(-1,2.5) node[sloped, pos=0.7, allow upside down]{\arrowIn} node[pos=.7, anchor=south]{$i'$};
\draw[color=white, fill=white] (4,2.5) circle(.4);
\draw[line width=.5pt, color=black] (3,5)--(5,0) node[sloped, pos=0.3, allow upside down]{\arrowIn} node[pos=.3, anchor=west]{$\;j$};
\draw[line width=.5pt,color=black] (4,7)--(3,5) node[sloped, pos=0.5, allow upside down]{\arrowIn} node[pos=.5, anchor=west]{$\;j$};
\draw[line width=.5pt, color=black] (4,7)--(7,7.5) node[sloped, pos=0.5, allow upside down]{\arrowIn} node[pos=.5, anchor=south]{$\;i$};
\draw[line width=.5pt, color=black] (2,7.5)--(4,7) node[sloped, pos=0.5, allow upside down]{\arrowIn} node[pos=.5, anchor=north]{$k$};
\draw[color=black, fill=black] (4,7) circle (.15) node[anchor=south]{$\alpha$};
\draw[line width=.5pt, color=gray] (-1,2.5)--(-4,2)  node[sloped, pos=0.5, allow upside down]{\arrowIn} node[pos=.5, anchor=south]{$i'$};
\draw[line width=.5pt, color=gray] (-5,0)--(-4,2)  node[sloped, pos=0.5, allow upside down]{\arrowIn} node[pos=.5, anchor=west]{$j'$} ;
\draw[line width=.5pt, color=gray] (-4,2)--(-6,2.5)  node[sloped, pos=0.5, allow upside down]{\arrowIn} node[pos=.5, anchor=north]{$k'$};
\draw[color=gray, fill=gray] (-4,2) circle (.15) node[anchor=south]{$\alpha'$};
\draw[color=blue, fill=blue] (2,7.5) circle (.15) node[anchor=south east] {$\lambda$};
\draw[color=blue, fill=blue] (7,7.5) circle (.15) node[anchor=west] {$\mu$};
\draw[color=blue, fill=blue] (9,2.5) circle (.15) node[anchor=west] {$\nu$};
\draw[color=blue, fill=blue] (5,0) circle (.15) node[anchor=north] {$\rho$};
\draw[color=blue, fill=blue] (-5,0) circle (.15) node[anchor=north] {$\sigma$};
\draw[color=blue, fill=blue] (-6,2.5) circle (.15) node[anchor=south east] {$\tau$};
\end{tikzpicture}
\end{center}
\caption{A labeled triangular prism and the  associated polygon diagram.}
\label{eq:prism}
\end{figure}

\begin{figure}
\begin{center}
\def\svgwidth{.75\columnwidth}
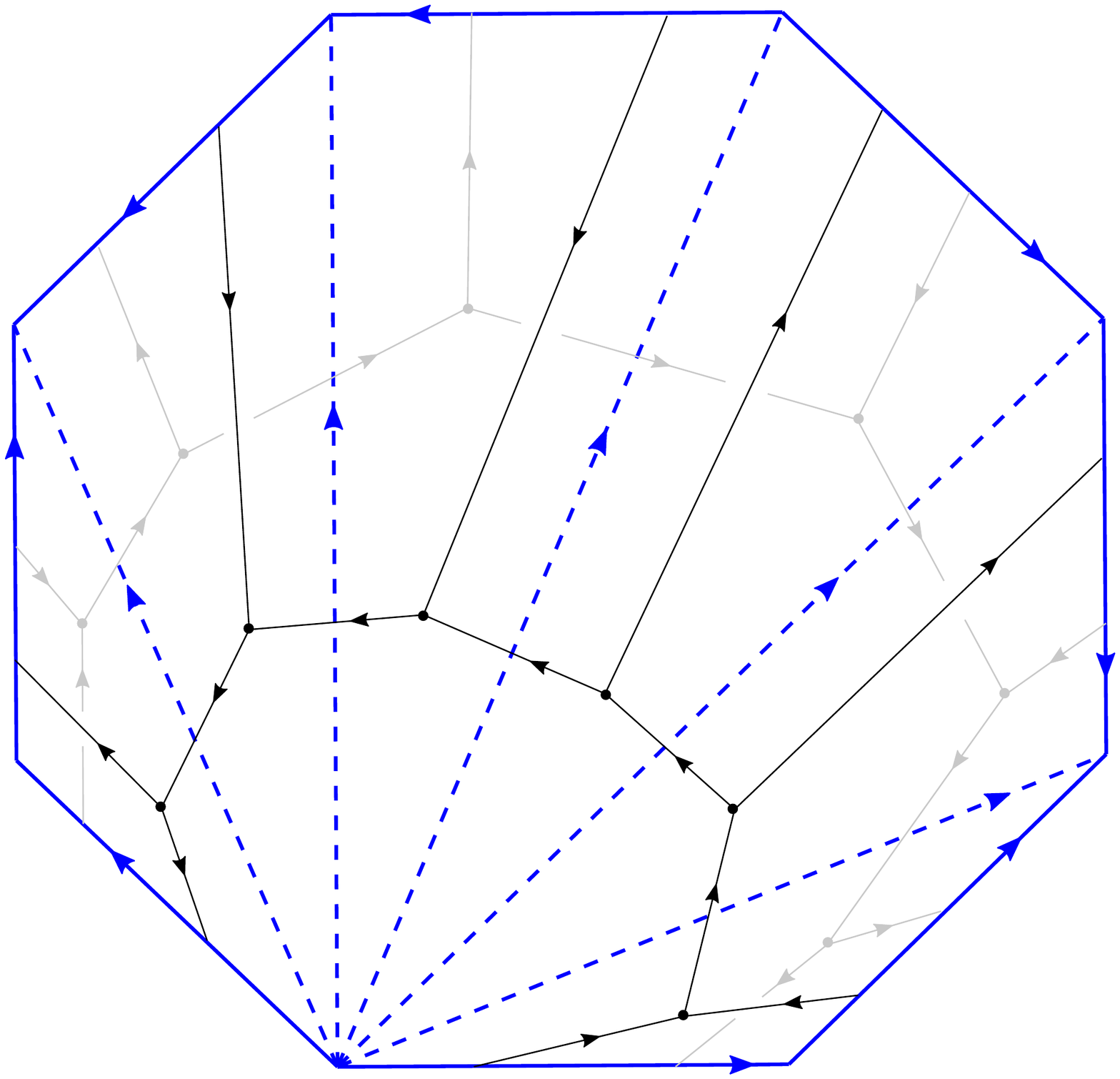 
\end{center}
\caption{Polygon diagram $P$ for $M=[0,1]\times\Sigma$ for triangulated defect surface $\Sigma$ of genus 2. \newline
$\bullet$ The orientation of the blue lines indicates the orientation of the triangulation on $\Sigma$. \newline
$\bullet$ The group elements on the unlabeled edges are determined by the labels.\newline 
$\bullet$ $\mathrm{ev}(P)=0$ unless $[a_1, b_1]\cdots [a_g, b_g]=1$ and $[a'_1,b'_1]\cdots [a'_g,b'_g]=1$ and $(a_i,a'_i), (b_i, b'_i)\in \mathrm{Stab}(m)$ for $i=1,2$.}
\label{fig:polygonsurface}
\end{figure}

By Proposition \ref{lem:polystatesum} the state sum of each labeled  prism is given by the evaluation of  the  polygon diagram in Figure \ref{eq:prism}. Also by Proposition \ref{lem:polystatesum}, the state sum for $M$ is given by the evaluation of any polygon diagram $P$ obtained by  gluing the polygons of the prisms to a disc and summing over the objects at its boundary. We can assume without loss of generality that $P$  is a $4g$-gon with every vertex labeled by  $m\in X$, as in Figure \ref{fig:polygonsurface}. This yields
$$\mathcal Z'(M, l_{\partial M},b_{\partial M})=\sum_{m\in X} \mathrm{ev}(P).$$
By identity \eqref{homvecg0},  the evaluation of $P$ vanishes, unless at each  vertex of $P$ that involves only edges labeled by $G$ or $G'$,  the oriented product of these group elements  is trivial.  By choosing maximal trees in the graphs labeled by $G$ and $G'$, one can express the group elements on the trees in terms of the  $2g$ group elements on edges that intersect $\partial P$, where $g$ is the genus of $\Sigma$. The labels on these $2g$ edges  then define
 a group homomorphism $\rho: \pi_1(\Sigma)\to G\times G'$, as shown in Figure \ref{fig:polygonsurface}.

By identity \eqref{eq:homvecg} the evaluation  $\mathrm{ev}(P)$ vanishes unless the group elements $(a_i,a'_i), (b_i, b'_i)\in G\times G'^{op}$ on each side of $P$ are contained in the stabiliser  of $m\in X$ or, equivalently, 
the group homomorphism $\rho$ takes values in $\mathrm{Stab}(m)$.
In that case, one has $\mathrm{ev}(P)=1$. 
Summation over all elements $m\in X$ then yields the result. 
\end{proof}

\begin{figure}
\begin{center}
\def\svgwidth{.98\columnwidth}
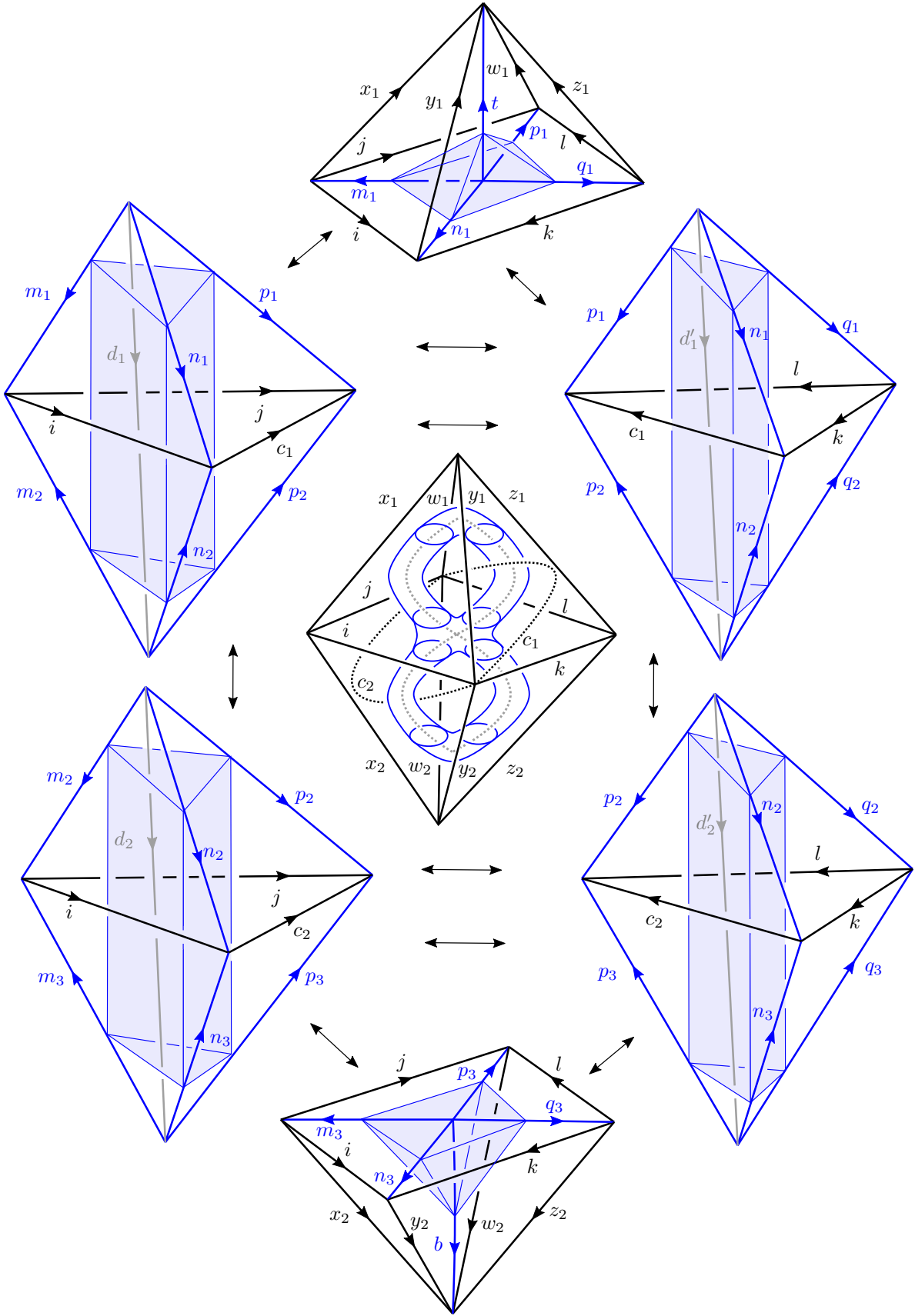 
 \end{center}
\caption{Gluing a surface of genus $g=2$ inside a 3-ball.}
 \label{fig:handle}
\end{figure}

Example \ref{ex:surface} gives intuition how the data for the spherical fusion categories interacts with the defect data in state sums of defect 3-manifolds.  The next example shows that  state sums for defect surfaces in the interior of a 3-ball  detect the genus of the defect surface.

\begin{example} \label{ex:genusg}Let  $\mac=\mathrm{Vec}_G$ and $\mad=\mathrm{Vec}_{G'}$ for finite groups $G,G'$ and
 $\mam$ the $(\mac,\mad)$-bimodule category with bimodule trace  defined by a finite transitive  $G\times G'^{op}$-set $X$. 
 Let $M$ be the triangulated 3-ball from Figure \ref{fig:handle} with an embedded defect surface of genus $g\geq 1$ in its interior. Then 
\begin{align}\label{eq:genussurf}
\mathcal Z(M, l_{\partial M},b_{\partial M})= \frac {|X|\cdot  |\mathrm{Stab}_G|^g\cdot  |\mathrm{Stab}_{G'}|^g}  {|G'|}  \cdot \mathcal Z(M', l_{\partial M}, b_{\partial M}),
\end{align}
where $M'$ is the 3-ball  with the same boundary triangulation and labeling, but without defects and $|\mathrm{Stab}_G|$, $|\mathrm{Stab}_{G'}|$ are the cardinalities of the stabilisers for the $G$- and $G'$-actions on $X$.
\end{example}

\begin{proof}
The triangulated defect 3-manifold in Figure \ref{fig:handle} is constructed as follows:
\begin{compactitem}
\item take $2g$ triangular double-pyramids formed by three tetrahedra each,  with a cylinder labeled by $\mam$ around the interior edge labeled by $\mad$,
\item glue them pairwise along  adjacent top and bottom faces to form $g$ rectangular double-pyramids,
\item glue the four bottom faces of each rectangular double-pyramid to the four top faces of the next one, 
\item glue a rectangular pyramid subdivided into four tetrahedra with defect triangles to the top and the bottom of the resulting polyhedron. 
\end{compactitem}
If the edges are labeled by group elements as in Figure \ref{fig:handle}, then
the edges labeled $ i,j,k,l, w_s,x_s,y_s, z_s\in G$  for $s=1,2$ are boundary edges.  The edges labeled by  $d_r,d'_r\in G'$, by $c_r\in G$ for $r=1,...,g$ and by $m_s,n_s,p_s,q_s, t,b\in X$  for $s=1,...,g+1$ are internal, and their labels are summed over.  Conditions \eqref{homvecg0} and \eqref{eq:homvecg} on the  morphism spaces imply that the state sum  is zero unless the following conditions are met:
\begin{itemize}
\item boundary edges of top and bottom pyramid:
\begin{align}
\label{eq:topcond}
&x_s=y_s\cdot i=w_s\cdot j,\quad z_s= y_s\cdot k=w_s\cdot l & &s\in\{1,2\},
\end{align}
\item defect edges of top and bottom pyramid:
\begin{align}\label{eq:topint}
&m_1=x_1^\inv\rhd t, & &n_1=y_1^\inv \rhd t, & &p_1=w_1^\inv\rhd t, & &q_1=z_1^\inv\rhd t,\\
&m_{g+1}=x_2^\inv\rhd b, & &n_{g+1}=y_2^\inv \rhd b, & &p_{g+1}=w_2^\inv\rhd b, & &q_{g+1}=z_2^\inv\rhd b,\nonumber
\end{align}
\item  double-pyramids: for $r\in \{1,...,g\}$
\begin{align}\label{eq:doublep}
&n_r=i\rhd m_r=k\rhd q_r, \quad p_r=c_r\rhd n_r=j\rhd m_r=l\rhd q_r,\\
&m_{r+1}=m_r\lhd d_r, \quad n_{r+1}=n_r\lhd d_r=n_r\lhd d'_r, \quad p_{r+1}=p_r\lhd d_r=p_r\lhd d'_r, \quad q_{r+1}=q_r\lhd d_r=q_r\lhd d'_r.\nonumber
\end{align}
\end{itemize}
Condition \eqref{eq:topcond} just states that the state sum is zero, unless the oriented product over the group elements of each boundary triangle is trivial. The first condition in \eqref{eq:topint} and the second condition in \eqref{eq:doublep} imply 
\begin{align} \label{eq:iter}
&m_{r+1}=x_1^\inv\rhd t\lhd d_1\cdots d_r & &q_{r+1}=w_1^\inv\rhd t\lhd d'_1\cdots d'_r\\
&n_{r+1}=y_1^\inv\rhd t\lhd d_1\cdots d_r=y_1^\inv\rhd t\lhd d'_1\cdots d'_r & 
&p_{r+1}=z_1^\inv\rhd t\lhd d_1\cdots d_r=z_1^\inv\rhd t\lhd d'_1\cdots d'_r\nonumber
\end{align}
for all $r\in\{1,...,g\}$.  This expresses  $m_s, n_s, p_s, q_s\in X$ as functions of $t\in X$, $d_1,...,d_g, d'_1,....,d'_n\in G'$ and of the group elements on the boundary edges and eliminates the summation over $m_s, n_s, p_s, q_s$. The equations in the second  line of \eqref{eq:iter} are satisfied simultaneously iff for all $r\in \{1,...,.g\}$
\begin{align}\label{eq:dcond}
d_1'\cdots d'_rd_r^\inv\cdots d_1^\inv\in \mathrm{Stab}_{G'}(t)=\{g\in G'\mid t\lhd g'=t\}.
\end{align}
Combining \eqref{eq:iter} with the second line in \eqref{eq:topint} yields
\begin{align}\label{eq:bexp}
b=w_2w_1^\inv\rhd t\lhd d_1\cdots d_r=x_2x_1^\inv\rhd t\lhd d_1\cdots d_r=y_2y_1^\inv\rhd t\lhd d_1\cdots d_r=z_2z_1^\inv\rhd t\lhd d_1\cdots d_r.
\end{align}
As \eqref{eq:topcond} implies $x_2x_1^\inv=y_2y_1^\inv=w_2w_1^\inv=z_2z_1^\inv$, equation \eqref{eq:bexp} just expresses $b$ in terms of $t$ and $d_1,...,d_r$ and thus eliminates the summation over $b$ from the state sum.  The first line in \eqref{eq:doublep} is then equivalent to \eqref{eq:iter}, \eqref{eq:topcond} and the additional condition
\begin{align}\label{eq:crcond}
w_1c_r y_1^\inv\in \mathrm{Stab}_G(t)=\{g\in G\mid g\rhd t=t\}\quad\forall r\in\{1,...,g\}.
\end{align}
 For fixed $t$ and $d_1,...,d_g$, the summations over $d'_1,...,d'_g$ and  $c_1,...,c_g$ 
contribute  factors $|\mathrm{Stab}_{G'}(t)|^g$ and $|\mathrm{Stab}_G(t)|^g$, respectively, to the state sum. As $X$ is a transitive $G\times G'^{op}$-set, all stabilisers have the same cardinality, and the summation over $t$ contributes a factor $|X|$.  
As each of the $g+1$ internal  vertices with incident edges from $\mad$  contributes a factor $\dim(\mad)^\inv=|G'|^\inv$ and the summation over  $d_1,...,d_g$ yields a factor $\dim(\mad)^g=|G'|^g$ we obtain
\begin{align*}
\mathcal Z'(M, l_{\partial M}, b_{\partial M})
&=
 \frac{|X| |\mathrm{Stab}_G|^g |\mathrm{Stab}_{G'}|^g}{|G'| } C,
\end{align*}
where $C$ is a factor that depends only on the labeling of the boundary triangles.
For $\mam=\mac=\mad=\mathrm{Vec}_G$ as a bimodule category over itself,  one has $|X|=|G'|=|G|$, $|\mathrm{Stab}_G|=|\mathrm{Stab}_{G'}|=1$. 
 \end{proof}
 
It is instructive to compare  Examples  \ref{ex:tetrahedronsphere} and  \ref{ex:genusg}.  If one  restricts the data in Example \ref{ex:tetrahedronsphere} to the bimodule category given by a finite $G\times G'^{op}$-set $X$ from Example \ref{ex:genusg}, then formula \eqref{eq:bublle} for the state sum of a 3-ball with a defect sphere becomes
$
\mathcal Z(M, l_{\partial M},b_{\partial M})= |X|/ |G'|  \cdot \mathcal Z(M', l_{\partial M}, b_{\partial M})$,
which is formula \eqref{eq:genussurf} for $g=0$.

 Examples \ref{ex:tetrahedronsphere} to \ref{ex:genusg} show that the state sums detect properties of the bimodule categories labeling a defect surface and properties of the surface. The next example shows that the state sum models with defects  depend on the \emph{embeddings} of defect surfaces  into the  3-manifold, and not just on their topology.  

This is manifest even for very simple non-trivial defect data. 
 By taking a tubular neighbourhood of a knot $K$ embedded into $S^3$ and labeling it with the $\mathrm{Vec}_G$-module category defined by a trivial $G$-set, one obtains a triangulation of the knot complement with trivial data at the boundary. The associated state sum gives the number of conjugacy classes of group homomorphisms $\rho:\pi_1(S^3\setminus K)\to G$ and hence is  sensitive to the embedding of the toroidal defect surface.

 \begin{example}\label{ex:knotcomplement} Let $K\subset S^3$ be a knot, realised as a subgraph of the dual cell complex of 
 a triangulation $T$ of 
  $S^3$. Let $M$ be the triangulated 3-manifold with defect data constructed as follows: 
  \begin{enumerate}
\item   Refine $T$ to a triangulation $T'$ by 
applying a 1-4  move to each tetrahedron $t$  in $T$ with $t\cap K\neq \emptyset$ and a stellar move to each triangle $\Delta$ in $T$ with $\Delta \cap K\neq \emptyset$, such that $K$ connects the midpoints of two faces of $t$ with the vertex in its interior, as shown in Figure \ref{fig:knotcomplement}.  

\item Form a tubular neighbourhood of $K$ by  inserting a defect plane as in Figure \ref{fig:dectet} (b) into each tetrahedron of $T'$ that contains an  edge in $K$ and a defect plane as in  Figure \ref{fig:dectet} (a) into each tetrahedron of $T'$ that contains a vertex, but not an edge of $K$, as shown in Figure \ref{fig:knotcomplement}. 
 
 \item Label the region of $T'$ that does not contain $K$ with  $\mac=\mathrm{Vec}_G$ for a finite group $G$ and the region containing  $K$ with the spherical fusion category $\mad=\mathrm{Vec}_{\{1\}}=\mathrm{Vect}_\C$. Label  the defect surface with the  $(\mac,\mad)$-bimodule category $\mam$ defined by the trivial $G$-set $X=\{\bullet\}$. 
 \end{enumerate}
 Then the state sum of  $M$ is the number of  conjugacy classes of group homomorphisms from the fundamental group $\pi_1(S^3\setminus K)$  into $G$ 
 $$
 \mathcal Z(M)=|\Hom_{\mathrm{Grp}}(\pi_1(S^3\setminus K), G)/G|.
 $$
 \end{example}

 \begin{figure}
\begin{center}
\def\svgwidth{.5\columnwidth}
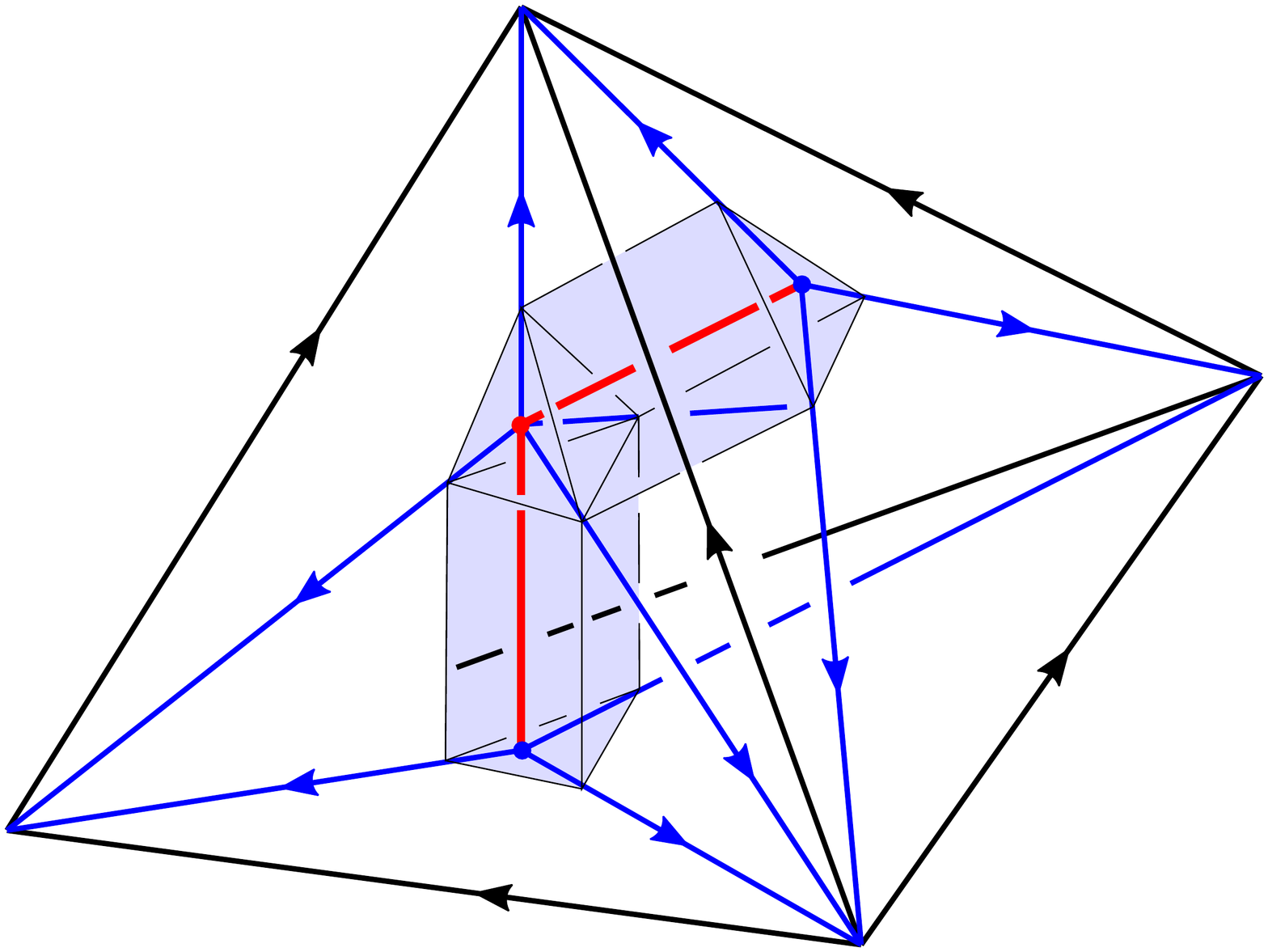 
 \end{center}
\caption{A piece of a knot (red)  in the dual complex of a triangulation and the associated defect surface. The state sum vanishes unless $g_3=g_1g_2$ and $g_5=g_1g_4$.}
 \label{fig:knotcomplement}
\end{figure}

 \begin{proof} Denote by $r_\mac$ the region labeled by $\mac=\mathrm{Vec}_G$.
  By subdividing the triangulation as in Lemma \ref{lem:finetriang}, 
 we can assume that each  vertex, edge or triangle in $T'$ that contained in  $r_\mac$  is part of a tetrahedron contained in $r_\mac$. This does not change the state sum by Theorem \ref{th:fulltopinv}. 
The union of the tetrahedra  in $r_\mac$ is then PL homeomorphic to the knot complement $S^3\setminus K'$, where $K'$ is a solid torus $S^1\times D^2$ embedded as a tubular neighbourhood of $K$. This defines a triangulation $T''$ of $S^3\setminus K'$ labeled by $\mathrm{Vec}_G$.

The remaining tetrahedra are as in Figure \ref{fig:knotcomplement}. By construction, all red edges in Figure \ref{fig:knotcomplement} are labeled with the simple object $\C$ of $\mathrm{Vect}_\C$, all blue edges by the simple object $\bullet$ in $\mam$, while the black edges are labeled with elements of $G$.   Each triangle $\Delta$ with $\Delta\cap t\neq \emptyset$ has two edges labeled by $\bullet$, one edge labeled by a group element $g\in G$ and is assigned the morphism space $\Hom(\bullet, \bullet)\cong \C$ for any $g\in G$. The triangles $\Delta$ in Figure \ref{fig:knotcomplement} with $\Delta\cap D=\emptyset$ are part of tetrahedra contained in $r_\mac$. Their labeling is determined by the labeling of $T''$, and their morphism spaces are $\Hom(\Delta)\cong \C$ if the oriented product of the group elements on their edges
 is trivial and $\Hom(\Delta)\cong \{0\}$ else.  
 
 Let $l: E''\to G$ be a labeling  of the oriented edges of $T''$ with elements of $G$.  
 As each triangle in $T'$ occurs in exactly two tetrahedra and all morphism spaces are zero- or one-dimensional,  the product of all generalised 6j symbols in $T'$ is one if the oriented product of the group elements on each triangle $\Delta$ in $T''$ vanishes and zero else. It follows that the state sum is  proportional to  the number of labelings $l$ that satisfy this condition
\begin{align}
\label{eq:statesumknotcomp}
\mathcal Z(M)=\frac { |\{l: E''\to G\mid \Pi_{e\in f}l(e)=1\;\forall f\in \Delta'' \}|} {|G|^{|V''|}},
\end{align}
 where $V''$, $E''$ and $\Delta''$ are the sets of vertices, edges and triangles of $T''$.
 
 Every such  labeling $l$  defines a group homomorphism $\rho_l:\pi_1(S^3\setminus K)\to G$, and every group homomorphism $\rho:\pi_1(S^3\setminus K)\to G$ is obtained from such a labeling.  The group homomorphisms $\rho_l,\rho_{l'}$ for two labelings $l,l': E''\to G$ are conjugate  iff there is a map $f: V''\to G$, $v\mapsto g_v$ such that
$l'(e)=g_{t(e)}\cdot l(e)\cdot g_{s(e)}$  for all $e\in E''$, where $s(e)$ and $t(e)$ denote the starting and target vertex of  $e$. 
This shows that $\mathcal Z(M)$ is the number of conjugacy classes of group homomorphisms $\rho:\pi_1(S^3\setminus K)\to G$. 
 \end{proof}
 
 Although the data in Example \ref{ex:knotcomplement} is quite trivial, it shows that the state sum with defects can distinguish non-isotopic embeddings of surfaces. It would be interesting to  investigate this further and to see if it can distinguish  knotted embeddings of  spheres.

\subsection{Ribbon invariants from line defects}

In this section we show how defect lines on trivial defect surfaces give rise to ribbon invariants.  These surfaces are labeled by a spherical fusion category $\mac$ as a $(\mac,\mac)$-bimodule category.   As $\mac$ is a spherical fusion category, its center $\mathcal Z(\mac)$ is a  ribbon category, see for instance \cite[Lemma 10.1]{TVr}. By Example \ref{ex:modulefuncs}, the monoidal category $\End_\mac(\mac)$  is  monoidally equivalent to $\mathcal Z(\mac)$. 
Under this monoidal equivalence, an object $(a, \sigma_{-,a})$ of $\mathcal Z(\mac)$  corresponds to the $(\mac,\mac)$-bimodule functor $F=a\oo -:\mac\to\mac$ with coherence data
\begin{align}\label{eq:coeherencecenter}
&s^F_{c,c'}=a_{c,a,c'}\circ (\sigma_{c,a}^\inv\oo 1_{c'})\circ a_{a,c,c'}^\inv: a\oo ( c\oo c')\to (a\oo c)\oo c'\to (c\oo a)\oo c'\to c\oo (a\oo c'),\\ 
&t^F_{c,c'}=a_{a,c,c'}: (a\oo c)\oo c'\to a\oo (c\oo c')\nonumber
\end{align}
and a morphism $\alpha: a\to b$ in $\mathcal Z(\mac)$  to the $(\mac,\mac)$-bimodule natural transformation $\nu=\alpha\oo -: a\oo -\Rightarrow b\oo -$ from Example \ref{ex:modulefuncs}. 

Any diagram representing a ribbon tangle for $\mathcal Z(\mac)$ can be realised as a defect graph on a trivial defect disc. This amounts to replacing
over- and undercrossings in the diagram  with defect vertices labeled by (inverse) braidings
and maxima and minima  by defect vertices labeled with the (co)evaluations in $\mathcal Z(\mac)$.
Every  diagram $D$ for a $(0,0)$-ribbon tangle  then defines a morphism $\alpha\in \End_{\mathcal Z(\mac)}(e)$ and
a  bimodule natural transformation $\nu=\alpha\oo -:\id_{\mac}\Rightarrow\id_\mac$  with
 $\nu_x=\tr(\alpha) 1_x=\mathrm{ev}(D)1_x: x\to x$ for every  $x\in I_\mac$.

\begin{example}\label{ex:ribbonplane} Let $\mac$ be a spherical fusion category and $t$ a tetrahedron as in Figure \ref{fig:dectet} (b) with a trivial defect surface that contains a diagram $D$ for a  $(0,0)$-ribbon tangle  labeled with  $\mathcal Z(\mac)$. 
Then one has
$$
\mathcal Z'(t, l_{\partial M}, b_{\partial M})=\mathrm{ev}(D) \cdot \mathrm{6j}(t', l_{\partial M}, b_{\partial M}),
$$
where $t'$ is the tetrahedron with the same labels and without the defects and $\mathrm{ev}(D)\in\C$ the evaluation of $D$.
\end{example}

\begin{proof} The 6j symbol of the defect tetrahedron $t$ is obtained from the associated polygon diagram
\begin{align*}
\begin{tikzpicture}[scale=.3, baseline=(current bounding box.center)]
\draw[line width=1pt, dotted] (-4,4)--(4,4)node[anchor=west]{$q$} --(4,-4) node[anchor=west]{$n$}--(-4,-4)node[anchor=east]{$m$}--(-4,4) node[anchor=east]{$p$};
\draw[line width=.5pt, color=black] (0,4)--(0,-4) node[pos=.2, anchor=west]{$j$}node[sloped, pos=.2, allow upside down]{\arrowOut};
\draw[color=white, fill=white] (0,0) circle (1.8); 
\draw[line width=1pt, color=cyan, dashed] (0,0) node[anchor=south] {$D$} circle (1.5); 
\draw[fill=white, color=white] (1.5,0) circle (.4);
\draw[fill=white, color=white] (-1.5,0) circle (.4);
\draw[line width=.5pt, color=black] (4,0)--(-4,0) node[pos=.8, anchor=north]{$i$}node[sloped, pos=.8, allow upside down]{\arrowOut};
\draw[color=black, fill=black] (0,4) circle (.2) node[anchor=south]{$\alpha$};
\draw[color=black, fill=black] (4,0) circle (.2) node[anchor=west]{$\beta$};
\draw[color=black, fill=black] (0,-4) circle (.2) node[anchor=north]{$\gamma$};
\draw[color=black, fill=black] (-4,0) circle (.2) node[anchor=east]{$\delta$};
\end{tikzpicture}
\end{align*}
where the dashed circle labeled $D$ contains the diagram $D$. 
This diagram superimposes the 6j symbol for  $t'$  and the  diagram  $D$.
By \eqref{eq:coeherencecenter},  overcrossings of the line labeled $i$ over lines in $D$ correspond to half-braidings $\sigma_{i,a}^{\pm 1}$, while overcrossings of lines in $D$ over $j$ correspond to associators  in $\mac$. The line labeled $i$ describes the functor $i\oo -: \mac\to\mac$ and the line labeled $j$ the functor $-\oo j:\mac\to\mac$. Identities \eqref{pic:rm2} to \eqref{pic:bimodulenat}  allow one to separate the two diagrams and yield the result. 
\end{proof}

Example \ref{ex:ribbonplane} shows that ribbon invariants can be realised as line and point defects on a trivial defect plane. However, 
this involves ribbon diagrams  rather than ribbons in three-dimensional space, and the braidings are external input, namely  defect data labeling vertices. We now show how ribbon invariants are obtained from line defects in a 3-ball that are not confined to a plane and where the braidings  arise from the   state sum.

\begin{example} \label{ex:braiding}Let $\mac$ be spherical fusion category  as a $(\mac,\mac)$-bimodule category and  $F=a\oo -:\mac\to\mac$ and $G=b\oo -:\mac\to \mac$  the bimodule functors  for  $(a, \sigma_{-,a}),(b, \sigma_{-,b})\in\mathcal \Ob \mathcal Z(\mac)$ with coherence data as in
\eqref{eq:coeherencecenter}.
We consider the following three tetrahedra labeled with objects  in $I_\mac$: 
\begin{align*}
\begin{tikzpicture}[scale=.5]
\draw[line width=1pt, color=black] (-5,0)--(0,4)  node[sloped, pos=0.4, allow upside down]{\arrowIn} node[pos=.4, anchor=south east]{$p$};
\draw[line width=1pt, color=black] (5,0)--(0,4)  node[sloped, pos=0.4, allow upside down]{\arrowIn} node[pos=.4, anchor=south west]{$q$};
\draw[line width=1pt, color=black] (-5,0)--(0,-4)  node[sloped, pos=0.4, allow upside down]{\arrowIn} node[pos=.4, anchor=north east]{$m$};
\draw[line width=1pt, color=black] (5,0)--(0,-4)  node[sloped, pos=0.4, allow upside down]{\arrowIn} node[pos=.4, anchor=north west]{$n$};
\draw[line width=1pt, color=black] (-5,0)--(-.5,0)  node[sloped, pos=0.5, allow upside down]{\arrowIn} node[pos=.5, anchor=south]{$j$};
\draw[line width=1pt, color=black] (.5,0)--(5,0)  ;
\draw[draw=none, fill=white] (-2,0) circle(.3);
\draw[draw=none, fill=white] (2,0) circle(.3);
\draw[draw=none, fill=white] (0,1) circle(.3);
\draw[line width=1pt, dotted] (-2,2.4)--(-2,-2.4)--(2,-2.4)--(2,2.4)--(-2,2.4);
\draw[draw=none, fill=white] (0,2.4) circle(.5);
\draw[draw=none, fill=white] (0,-2.4) circle(.5);
\draw[color=cyan, line width=1pt] (2, 1)--(-2,1) node[sloped, pos=0.7, allow upside down]{\arrowIn} node[pos=.7, anchor=south]{$F$};
\draw[draw=none, fill=white] (0,1) circle(.3);
\draw[line width=1pt, color=black] (0,-4)--(0,4)  node[sloped, pos=0.3, allow upside down]{\arrowIn} node[pos=.3, anchor=east]{$i$};
\draw[line width=.5pt, -stealth] (2,-1)--(2.5,-1.5);
\end{tikzpicture}
\qquad
\begin{tikzpicture}[scale=.5]
\draw[line width=1pt, color=black] (0,-4)--(0,4)  node[sloped, pos=0.3, allow upside down]{\arrowIn} node[pos=.3, anchor=east]{$k$};
\draw[draw=none, fill=white] (0,2.4) circle(.3);
\draw[draw=none, fill=white] (0,-2.4) circle(.3);
\draw[line width=1pt, dotted] (-2,2.4)--(-2,-2.4)--(2,-2.4)--(2,2.4)--(-2,2.4);
\draw[draw=none, fill=white] (-2,0) circle(.3);
\draw[draw=none, fill=white] (2,0) circle(.3);
\draw[line width=1pt, color=black] (-5,0)--(0,4)  node[sloped, pos=0.4, allow upside down]{\arrowIn} node[pos=.4, anchor=south east]{$p$};
\draw[line width=1pt, color=black] (5,0)--(0,4)  node[sloped, pos=0.4, allow upside down]{\arrowIn} node[pos=.4, anchor=south west]{$q$};
\draw[line width=1pt, color=black] (-5,0)--(0,-4)  node[sloped, pos=0.4, allow upside down]{\arrowIn} node[pos=.4, anchor=north east]{$m$};
\draw[line width=1pt, color=black] (5,0)--(0,-4)  node[sloped, pos=0.4, allow upside down]{\arrowIn} node[pos=.4, anchor=north west]{$n$};
\draw[draw=none, fill=white] (0,0) circle(.3);
\draw[line width=1pt, color=black] (-5,0)--(5,0)  node[sloped, pos=0.2, allow upside down]{\arrowIn} node[pos=.2, anchor=south]{$j$};
\draw[line width=.5pt, -stealth] (2,1)--(2.5,1.5);
\end{tikzpicture}
\qquad
\begin{tikzpicture}[scale=.5]
\draw[line width=1pt, color=black] (-5,0) --(0,4)  node[sloped, pos=0.4, allow upside down]{\arrowIn} node[pos=.4, anchor=south east]{$p$};
\draw[line width=1pt, color=black] (5,0)--(0,4)  node[sloped, pos=0.4, allow upside down]{\arrowIn} node[pos=.4, anchor=south west]{$q$};
\draw[line width=1pt, color=black] (-5,0)--(0,-4)   node[sloped, pos=0.4, allow upside down]{\arrowIn} node[pos=.4, anchor=north east]{$m$};
\draw[line width=1pt, color=black] (5,0)--(0,-4)  node[sloped, pos=0.4, allow upside down]{\arrowIn} node[pos=.4, anchor=north west]{$n$};
\draw[line width=1pt, color=black] (-5,0)--(-.5,0)  node[sloped, pos=0.5, allow upside down]{\arrowIn} node[pos=.5, anchor=south]{$l$};
\draw[line width=1pt, color=black] (.5,0)--(5,0)  ;
\draw[draw=none, fill=white] (-2,0) circle(.3);
\draw[draw=none, fill=white] (2,0) circle(.3);
\draw[draw=none, fill=white] (1,0) circle(.3);
\draw[line width=1pt, dotted] (-2,2.4)--(-2,-2.4)--(2,-2.4)--(2,2.4)--(-2,2.4);
\draw[draw=none, fill=white] (0,2.4) circle(.5);
\draw[draw=none, fill=white] (0,-2.4) circle(.5);
\draw[color=magenta, line width=1pt] (1,2.4)--(1,-2.4) node[pos=.7, anchor=west]{$G$} node[sloped, pos=0.7, allow upside down]{\arrowIn} ;
\draw[line width=1pt, color=black] (0,-4)--(0,4)  node[sloped, pos=0.3, allow upside down]{\arrowIn} node[pos=.3, anchor=east]{$k$};
\draw[line width=.5pt, -stealth] (2,-1)--(2.5,-1.5);
\end{tikzpicture}
\end{align*}
Glue the first tetrahedron  to the second along the two  faces labeled with  $j,m,n$ and $j, p, q$ and the second to the third along the two  faces labeled with  $k,m,p$ and $k,n,q$.  This yields a triangulated 3-ball $M$ with  a one-holed torus as a trivial defect surface. 
Then the rescaled state sum  $\mathcal Z'(M, l_{\partial M}, b_{\partial M})$ is the 6j symbol of the  tetrahedron
\begin{align*}
\begin{tikzpicture}[scale=.5]
\draw[line width=1pt, color=black] (-5,0)--(0,4)  node[sloped, pos=0.4, allow upside down]{\arrowIn} node[pos=.4, anchor=south east]{$p$};
\draw[line width=1pt, color=black] (5,0)--(0,4)  node[sloped, pos=0.4, allow upside down]{\arrowIn} node[pos=.4, anchor=south west]{$q$};
\draw[line width=1pt, color=black] (-5,0)--(0,-4)  node[sloped, pos=0.4, allow upside down]{\arrowIn} node[pos=.4, anchor=north east]{$m$};
\draw[line width=1pt, color=black] (5,0)--(0,-4)  node[sloped, pos=0.4, allow upside down]{\arrowIn} node[pos=.4, anchor=north west]{$n$};
\draw[line width=1pt, color=black] (-5,0)--(-.5,0)  node[sloped, pos=0.5, allow upside down]{\arrowIn} node[pos=.5, anchor=south]{$l$};
\draw[line width=1pt, color=black] (.5,0)--(5,0)  ;
\draw[draw=none, fill=white] (-2,0) circle(.3);
\draw[draw=none, fill=white] (2,0) circle(.3);
\draw[draw=none, fill=white] (0,1) circle(.3);
\draw[draw=none, fill=white] (1,0) circle(.3);
\draw[line width=1pt, dotted] (-2,2.4)--(-2,-2.4)--(2,-2.4)--(2,2.4)--(-2,2.4);
\draw[draw=none, fill=white] (0,2.4) circle(.5);
\draw[draw=none, fill=white] (0,-2.4) circle(.5);
\draw[color=magenta, line width=1pt] (1,2.4)--(1,-2.4) node[pos=.7, anchor=west]{$G$} node[sloped, pos=0.7, allow upside down]{\arrowIn} ;
\draw[draw=none, fill=white] (1,1) circle(.3);
\draw[color=cyan, line width=1pt] (2, 1)--(-2,1) node[pos=.7, anchor=south]{$F$}node[sloped, pos=.7, allow upside down]{\arrowIn} ;
\draw[draw=none, fill=white] (0,1) circle(.3);
\draw[line width=1pt, color=black] (0,-4)--(0,4)  node[sloped, pos=.3, allow upside down]{\arrowIn} node[pos=.3, anchor=east]{$i$};
\draw[line width=.5pt, -stealth] (2,-1)--(2.5,-1.5);
\begin{scope}[shift={(12,0)}, scale=.8]
\node at (-6,0)[anchor=east]{$=$};
\draw[line width=1pt, color=magenta] (0,3) ..controls (-6,2) and (-6,0) .. (2,-1) node[midway, anchor=east]{$b$};
\draw[color=white, fill=white]  (-1.4,-.4) circle (.3);
\draw[color=white, fill=white]  (-2.4,-.2) circle (.3);
\draw[line width=.5pt, color=black] (-1,5)--(1,5) node[anchor=west]{$p$};
\draw[line width=.5pt] (0,5)--(0,3); 
\draw[line width=.5pt, color=black] (-1,-5)--(1,-5) node[anchor=west]{$p$};
\draw[line width=.5pt] (0,-5)--(0,-3); 
\draw[line width=.5pt] (0,3)--(2,-1) node[midway, anchor=south west] {$l$};
\draw[line width=.5pt] (0,3)--(-2,1) node[midway, anchor=east] {$q$};
\draw[line width=.5pt] (-2,1)--(2,-1) node[midway, anchor=south] {$n$};
\draw[line width=.5pt] (0,-3)--(2,-1) node[midway, anchor=north west] {$m$};
\draw[line width=.5pt] (0,-3)--(-2,1) node[midway, anchor=north east] {$i$};
\draw[line width=1pt, color=cyan] (-2,1) ..controls (-3,-2).. (0,-3) node[midway, anchor=east]{$a$};
\draw[color=black, fill=black] (0,3) circle (.2) node[anchor=west]{$\;\nu$};
\draw[color=black, fill=black] (2,-1) circle (.2) node[anchor=west]{$\;\mu$};
\draw[color=black, fill=black] (-2,1) circle (.2) node[anchor=east]{$\beta\;$};
\draw[color=black, fill=black] (0,-3) circle (.2) node[anchor=west]{$\;\delta$};
\end{scope}
\end{tikzpicture}
\end{align*}
where the crossing of the lines labeled $F$ and $G$ on the defect surface is assigned the $(\mac,\mac)$-bimodule natural transformation $\nu=\sigma_{a,b}^\inv\oo -: (b\oo a)\rhd -\Rightarrow (a\oo b)\rhd -$.
\end{example}

\begin{proof} The polygon diagram and the associated 6j symbols for the first tetrahedron are given by
\begin{align*}
\begin{tikzpicture}[scale=.4, baseline=(current bounding box.center)]
\draw[line width=1pt, dotted] (-4,4)--(4,4)node[anchor=south west]{$q$} --(4,-4) node[anchor=north west]{$n$}--(-4,-4)node[anchor=north east]{$m$}--(-4,4) node[anchor=south east]{$p$};
\draw[line width=.5pt, color=gray] (0,4)--(0,-4) node[pos=.2, anchor=west]{$j$}node[sloped, pos=.2, allow upside down]{\arrowOut};
\draw[fill=white, color=white] (0,-.5) circle (1.2);
\draw[line width=.5pt, color=black] (4,0) .. controls (0,-1.5) ..  (-4,0) node[pos=.8, anchor=north]{$i$}node[sloped, pos=.8, allow upside down]{\arrowOut};
\draw[line width=1pt, color=cyan, dashed] (4,0)--(-4,0) node[pos=.8, anchor=south]{$F$}node[sloped, pos=.8, allow upside down]{\arrowIn};
\draw[color=black, fill=black] (0,4) circle (.2) node[anchor=south]{$\alpha$};
\draw[color=black, fill=black] (4,0) circle (.2) node[anchor=west]{$\beta$};
\draw[color=black, fill=black] (0,-4) circle (.2) node[anchor=north]{$\gamma$};
\draw[color=black, fill=black] (-4,0) circle (.2) node[anchor=east]{$\delta$};
\end{tikzpicture}
%%%%%%
\qquad
%%%%%%
\begin{tikzpicture}[scale=.4, baseline=(current bounding box.center)]
\draw[line width=1.5pt, color=black] (-1,5)--(1,5) node[anchor=west]{$p$};
\draw[line width=1.5pt, color=black] (-1,-5)--(1,-5) node[anchor=west]{$p$};
\draw[line width=1.5pt, color=black] (0,5)--(0,-5);
\draw[line width=.5pt, color=gray] (0,3)  .. controls (-3,2) and (-3,0) .. (0,-1)  node[midway, anchor=east]{$j$};
\draw[color=white, fill=white] (-2,0) circle (.4);
\draw[color=white, fill=white] (-2.2,.2) circle (.4);
\draw[line width=.5pt, color=black] (0,1)  .. controls (-3,0) and (-3,-2) .. (0,-3)  node[midway, anchor= east]{$i$};
\draw[line width=1pt, color=cyan, dashed] (0,1)  .. controls (-5,0) and (-5,-2) .. (0,-3)  node[midway, anchor=south east]{$F$};
\node at (0,2)[anchor=west, color=black] {$q$};
\node at (0,0)[anchor=west, color=black] {$n$};
\node at (0,-2)[anchor=west, color=black] {$m$};
\draw[color=black, fill=black] (0,3) circle (.2) node[anchor=west]{$\;\alpha$};
\draw[color=black, fill=black] (0,1) circle (.2) node[anchor=west]{$\;\beta$};
\draw[color=black, fill=black] (0,-1) circle (.2) node[anchor=west]{$\;\gamma$};
\draw[color=black, fill=black] (0,-3) circle (.2) node[anchor=west]{$\;\delta$};
\end{tikzpicture}
%%%%%%
\qquad
%%%%%%
\begin{tikzpicture}[scale=.4, baseline=(current bounding box.center)]
\draw[line width=.5pt, color=black] (-1,5)--(1,5) node[anchor=west]{$p$};
\draw[line width=.5pt] (0,5)--(0,3); 
\draw[line width=.5pt] (0,-5)--(0,-3); 
\draw[line width=.5pt, color=black] (-1,-5)--(1,-5) node[anchor=west]{$p$};
\draw[line width=.5pt] (0,3)--(-2,1) node[midway, anchor=south east]{$q$};
\draw[line width=.5pt] (0,3)--(2,-1) node[midway, anchor=south west] {$j$};
\draw[line width=.5pt] (-2,1)--(2,-1) node[midway, anchor=south]{$n$};
\draw[line width=.5pt] (-2,1)--(0,-3) node[midway, anchor=north east]{$i$};  
\draw[line width=1pt, color=cyan] (-2,1) ..controls (-3,-2).. (0,-3) node[midway, anchor=east]{$a$};
\draw[line width=.5pt] (2,-1)--(0,-3) node[midway, anchor=north west]{$m$};    
\draw[color=black, fill=black] (0,3) circle (.2) node[anchor=west]{$\;\alpha$};
\draw[color=black, fill=black] (-2,1) circle (.2) node[anchor=east]{$\beta\;$};
\draw[color=black, fill=black] (2,-1) circle (.2) node[anchor=west]{$\;\gamma$};
\draw[color=black, fill=black] (0,-3) circle (.2) node[anchor=west]{$\;\delta$};
\end{tikzpicture}
\end{align*}
where the first diagram for the 6j symbol is  in the diagrammatic notation of the previous sections and the second the usual diagram for a spherical fusion category as in \cite{TV, BW}, with the objects from $\mathcal Z(\mac)$ highlighted in colour. It is obtained from the second diagram by inserting the functor $-\oo j:\mac\to\mac$ for the  line labeled $j$ and the functor  $F(i\oo-)=a\oo (i\oo -):\mac\to\mac$ for the lines with $F$ and $i$. 

The polygon diagram and 6j symbols for the second tetrahedron are given by
\begin{align*}
\begin{tikzpicture}[scale=.4, baseline=(current bounding box.center)]
\draw[line width=1pt, dotted] (-4,4)--(4,4)node[anchor=south west]{$m$} --(4,-4) node[anchor=north west]{$n$}--(-4,-4)node[anchor=north east]{$q$}--(-4,4) node[anchor=south east]{$p$};
\draw[line width=.5pt, color=gray] (4,0)--(-4,0) node[pos=.8, anchor=north]{$j$}node[sloped, pos=.8, allow upside down]{\arrowOut};
\draw[fill=white, color=white] (0,0) circle (.4);
\draw[line width=.5pt, color=black] (0,4)--(0,-4) node[pos=.8, anchor=west]{$k$}node[sloped, pos=.8, allow upside down]{\arrowOut};
\draw[color=black, fill=black] (0,4) circle (.2) node[anchor=south]{$\rho$};
\draw[color=black, fill=black] (4,0) circle (.2) node[anchor=west]{$\gamma$};
\draw[color=black, fill=black] (0,-4) circle (.2) node[anchor=north]{$\tau$};
\draw[color=black, fill=black] (-4,0) circle (.2) node[anchor=east]{$\alpha$};
\end{tikzpicture}
%%%%%%
\qquad
%%%%%%
\begin{tikzpicture}[scale=.4, baseline=(current bounding box.center)]
\draw[line width=1.5pt, color=black] (-1,5)--(1,5) node[anchor=west]{$p$};
\draw[line width=1.5pt, color=black] (-1,-5)--(1,-5) node[anchor=west]{$p$};
\draw[line width=1.5pt, color=black] (0,5)--(0,-5);
\draw[line width=.5pt, color=gray] (0,1)  .. controls (-3,0) and (-3,-2) .. (0,-3)  node[midway, anchor= east]{$j$};
\draw[color=white, fill=white] (-2,0) circle (.4);
\draw[line width=.5pt, color=black] (0,3)  .. controls (-3,2) and (-3,0) .. (0,-1)  node[midway, anchor=east]{$k$};
\node at (0,2)[anchor=west, color=black] {$m$};
\node at (0,0)[anchor=west, color=black] {$n$};
\node at (0,-2)[anchor=west, color=black] {$q$};
\draw[color=black, fill=black] (0,3) circle (.2) node[anchor=west]{$\;\rho$};
\draw[color=black, fill=black] (0,1) circle (.2) node[anchor=west]{$\;\gamma$};
\draw[color=black, fill=black] (0,-1) circle (.2) node[anchor=west]{$\;\tau$};
\draw[color=black, fill=black] (0,-3) circle (.2) node[anchor=west]{$\;\alpha$};
\end{tikzpicture}
%%%%%%
\qquad
%%%%%%
\begin{tikzpicture}[scale=.4, baseline=(current bounding box.center)]
\draw[line width=.5pt, color=black] (-1,5)--(1,5) node[anchor=west]{$p$};
\draw[line width=.5pt] (0,5)--(0,3); 
\draw[line width=.5pt] (0,-5)--(0,-3); 
\draw[line width=.5pt, color=black] (-1,-5)--(1,-5) node[anchor=west]{$p$};
\draw[line width=.5pt] (0,3)--(2,1) node[midway, anchor=south west]{$m$};
\draw[line width=.5pt] (0,3)--(-2,-1) node[midway, anchor=south east] {$k$};
\draw[line width=.5pt] (-2,-1)--(2,1) node[midway, anchor=south]{$n$};
\draw[line width=.5pt] (2,1)--(0,-3) node[midway, anchor=north west]{$j$};  
\draw[line width=.5pt] (-2,-1)--(0,-3) node[midway, anchor=north east]{$q$};    
\draw[color=black, fill=black] (0,3) circle (.2) node[anchor=west]{$\;\rho$};
\draw[color=black, fill=black] (2,1) circle (.2) node[anchor=west]{$\gamma\;$};
\draw[color=black, fill=black] (-2,-1) circle (.2) node[anchor=east]{$\;\tau$};
\draw[color=black, fill=black] (0,-3) circle (.2) node[anchor=west]{$\;\alpha$};
\end{tikzpicture}
\end{align*}
Note that here the polygon diagram is viewed from the back, not the front, since the trivial defect plane is oriented in the opposite direction. Again, the second diagram for the 6j symbol is the standard diagram for a spherical fusion category obtained by inserting the functors $k\oo-: \mac\to\mac$ and $-\oo j:\mac\to\mac$ for the black line labeled $k$ and the grey line labeled $j$ in the second diagram.

The polygon diagram and the 6j symbols for the third tetrahedron are given by
\begin{align*}
\begin{tikzpicture}[scale=.4, baseline=(current bounding box.center)]
\draw[line width=1pt, dotted] (-4,4)--(4,4)node[anchor=south west]{$q$} --(4,-4) node[anchor=north west]{$n$}--(-4,-4)node[anchor=north east]{$m$}--(-4,4) node[anchor=south east]{$p$};
\draw[line width=.5pt, color=gray] (0,4).. controls (1.5,0)..(0,-4) node[pos=.8, anchor=west]{$l$}node[sloped, pos=.8, allow upside down]{\arrowOut};
\draw[line width=1pt, color=magenta, dashed] (0,4)--(0,-4) node[pos=.8, anchor=east]{$G$}node[sloped, pos=.8, allow upside down]{\arrowIn};
\draw[fill=white, color=white] (.5,0) circle (.9);
\draw[line width=.5pt, color=black] (4,0)--(-4,0) node[pos=.8, anchor=south]{$k$}node[sloped, pos=.8, allow upside down]{\arrowOut};
\draw[color=black, fill=black] (0,4) circle (.2) node[anchor=south]{$\nu$};
\draw[color=black, fill=black] (4,0) circle (.2) node[anchor=west]{$\tau$};
\draw[color=black, fill=black] (0,-4) circle (.2) node[anchor=north]{$\mu$};
\draw[color=black, fill=black] (-4,0) circle (.2) node[anchor=east]{$\rho$};
\end{tikzpicture}
%%%%%%
\qquad
%%%%%%
\begin{tikzpicture}[scale=.4, baseline=(current bounding box.center)]
\draw[line width=1.5pt, color=black] (-1,5)--(1,5) node[anchor=west]{$p$};
\draw[line width=1.5pt, color=black] (-1,-5)--(1,-5) node[anchor=west]{$p$};
\draw[line width=1.5pt, color=black] (0,5)--(0,-5);
\draw[line width=1pt, color=magenta, dashed] (0,3)  .. controls (-5,2) and (-5,0) .. (0,-1)  node[midway, anchor= east]{$G$};
\draw[line width=.5pt, color=gray] (0,3)  .. controls (-3,2) and (-3,0) .. (0,-1)  node[midway, anchor= east]{$l$};
\draw[color=white, fill=white] (-2,0) circle (.4);
\draw[color=white, fill=white] (-2.2,-.5) circle (.4);
\draw[line width=.5pt, color=black] (0,1)  .. controls (-3,0) and (-3,-2) .. (0,-3)  node[midway, anchor= east]{$k$};
\node at (0,2)[anchor=west, color=black] {$q$};
\node at (0,0)[anchor=west, color=black] {$n$};
\node at (0,-2)[anchor=west, color=black] {$m$};
\draw[color=black, fill=black] (0,3) circle (.2) node[anchor=west]{$\;\nu$};
\draw[color=black, fill=black] (0,1) circle (.2) node[anchor=west]{$\;\tau$};
\draw[color=black, fill=black] (0,-1) circle (.2) node[anchor=west]{$\;\mu$};
\draw[color=black, fill=black] (0,-3) circle (.2) node[anchor=west]{$\;\rho$};
\end{tikzpicture}
%%%%%%
\qquad
%%%%%%
\begin{tikzpicture}[scale=.4, baseline=(current bounding box.center)]
\draw[line width=1pt, color=magenta] (0,3) ..controls (-6,2) and (-6,0) .. (2,-1) node[midway, anchor=east]{$b$};
\draw[color=white, fill=white] (-1.3,-.5) circle (.4);
\draw[line width=.5pt, color=black] (-1,5)--(1,5) node[anchor=west]{$p$};
\draw[line width=.5pt] (0,5)--(0,3); 
\draw[line width=.5pt] (0,-5)--(0,-3); 
\draw[line width=.5pt, color=black] (-1,-5)--(1,-5) node[anchor=west]{$p$};
\draw[line width=.5pt] (0,3)--(-2,1) node[midway, anchor=east]{$q$};
\draw[line width=.5pt] (0,3)--(2,-1) node[midway, anchor=south west] {$l$};
\draw[line width=.5pt] (-2,1)--(2,-1) node[midway, anchor=south]{$n$};
\draw[line width=.5pt] (-2,1)--(0,-3) node[midway, anchor=north east]{$k$};  
\draw[line width=.5pt] (2,-1)--(0,-3) node[midway, anchor=north west]{$m$};    
\draw[color=black, fill=black] (0,3) circle (.2) node[anchor=west]{$\;\nu$};
\draw[color=black, fill=black] (-2,1) circle (.2) node[anchor=east]{$\tau\;$};
\draw[color=black, fill=black] (2,-1) circle (.2) node[anchor=west]{$\;\mu$};
\draw[color=black, fill=black] (0,-3) circle (.2) node[anchor=west]{$\;\rho$};
\end{tikzpicture}
\end{align*}
Here, the second 6j symbol is obtained from the first by inserting the functor $k\oo -:\mac\to \mac$ for the black line labeled $k$ and the functor $G(-\oo l)=b\oo (-\oo l):\mac\to\mac$ for the
lines labeled $G$ and $l$. The crossing in the second 6j symbol stands for the morphism $\sigma^\inv_{k,b}: b\oo k\to k\oo b$ associated with $b$.  

The rescaled state sum is obtained by multiplying the three 6j symbols and the dimensions of the internal edges, summing over the labels $j,k\in I_\mac$  and over the morphisms $\alpha,\gamma,\rho,\tau$ assigned to the glued faces
\begin{align*}
&\begin{tikzpicture}[scale=.4, baseline=(current bounding box.center)]
\node at (-5,0) [anchor=east]{$\mathcal Z'(M, l_{\partial M}, b_{\partial M})=\sum_{\substack{{j,k}\\{\alpha,\gamma,\rho,\tau}}} \dim(j)\dim(k)$};
\begin{scope}[shift={(0,0)}]
\draw[line width=.5pt, color=black] (-1,5)--(1,5) node[anchor=west]{$p$};
\draw[line width=.5pt] (0,5)--(0,3); 
\draw[line width=.5pt] (0,-5)--(0,-3); 
\draw[line width=.5pt, color=black] (-1,-5)--(1,-5) node[anchor=west]{$p$};
\draw[line width=.5pt] (0,3)--(-2,1) node[midway, anchor=south east]{$q$};
\draw[line width=.5pt] (0,3)--(2,-1) node[midway, anchor=south west] {$j$};
\draw[line width=.5pt] (-2,1)--(2,-1) node[midway, anchor=south]{$n$};
\draw[line width=.5pt] (-2,1)--(0,-3) node[midway, anchor=north east]{$i$};  
\draw[line width=1pt, color=cyan] (-2,1) ..controls (-3,-2).. (0,-3) node[midway, anchor=east]{$a$};
\draw[line width=.5pt] (2,-1)--(0,-3) node[midway, anchor=north west]{$m$};    
\draw[color=black, fill=black] (0,3) circle (.2) node[anchor=west]{$\;\alpha$};
\draw[color=black, fill=black] (-2,1) circle (.2) node[anchor=east]{$\beta\;$};
\draw[color=black, fill=black] (2,-1) circle (.2) node[anchor=west]{$\;\gamma$};
\draw[color=black, fill=black] (0,-3) circle (.2) node[anchor=west]{$\;\delta$};
\end{scope}
%%%%%%%
\begin{scope}[shift={(8,0)}]
\draw[line width=.5pt, color=black] (-1,5)--(1,5) node[anchor=west]{$p$};
\draw[line width=.5pt] (0,5)--(0,3); 
\draw[line width=.5pt] (0,-5)--(0,-3); 
\draw[line width=.5pt, color=black] (-1,-5)--(1,-5) node[anchor=west]{$p$};
\draw[line width=.5pt] (0,3)--(2,1) node[midway, anchor=south west]{$m$};
\draw[line width=.5pt] (0,3)--(-2,-1) node[midway, anchor=south east] {$k$};
\draw[line width=.5pt] (-2,-1)--(2,1) node[midway, anchor=south]{$n$};
\draw[line width=.5pt] (2,1)--(0,-3) node[midway, anchor=north west]{$j$};  
\draw[line width=.5pt] (-2,-1)--(0,-3) node[midway, anchor=north east]{$q$};    
\draw[color=black, fill=black] (0,3) circle (.2) node[anchor=west]{$\;\rho$};
\draw[color=black, fill=black] (2,1) circle (.2) node[anchor=west]{$\gamma\;$};
\draw[color=black, fill=black] (-2,-1) circle (.2) node[anchor=east]{$\;\tau$};
\draw[color=black, fill=black] (0,-3) circle (.2) node[anchor=west]{$\;\alpha$};
\end{scope}
%%%%%%%
\begin{scope}[shift={(17,0)}]
\draw[line width=1pt, color=magenta] (0,3) ..controls (-6,2) and (-6,0) .. (2,-1) node[midway, anchor=east]{$b$};
\draw[color=white, fill=white] (-1.3,-.5) circle (.4);
\draw[line width=.5pt, color=black] (-1,5)--(1,5) node[anchor=west]{$p$};
\draw[line width=.5pt] (0,5)--(0,3); 
\draw[line width=.5pt] (0,-5)--(0,-3); 
\draw[line width=.5pt, color=black] (-1,-5)--(1,-5) node[anchor=west]{$p$};
\draw[line width=.5pt] (0,3)--(-2,1) node[midway, anchor=east]{$q$};
\draw[line width=.5pt] (0,3)--(2,-1) node[midway, anchor=south west] {$l$};
\draw[line width=.5pt] (-2,1)--(2,-1) node[midway, anchor=south]{$n$};
\draw[line width=.5pt] (-2,1)--(0,-3) node[midway, anchor=north east]{$k$};  
\draw[line width=.5pt] (2,-1)--(0,-3) node[midway, anchor=north west]{$m$};    
\draw[color=black, fill=black] (0,3) circle (.2) node[anchor=west]{$\;\nu$};
\draw[color=black, fill=black] (-2,1) circle (.2) node[anchor=east]{$\tau\;$};
\draw[color=black, fill=black] (2,-1) circle (.2) node[anchor=west]{$\;\mu$};
\draw[color=black, fill=black] (0,-3) circle (.2) node[anchor=west]{$\;\rho$};
\end{scope}
\end{tikzpicture}
\end{align*}
To  compute this state sum we first use   \eqref{pic:gluepoly} to glue the polygon diagrams along the faces labeled by the objects $k,m,p$ and $q,j,p$  and the morphisms $\rho$ and $\alpha$. 
 We then  simplify the resulting expressions with  \eqref{eq:traceout}  and  apply  \eqref{eq:snake} together with the naturality of the half-braiding. This gives
\begin{align*}
%%%%%%%%%%
&\begin{tikzpicture}[scale=.4, baseline=(current bounding box.center)]
\draw[line width=1pt, color=magenta] (0,3) ..controls (-6,2) and (-6,0) .. (2,-1) node[midway, anchor=east]{$b$};
\draw[color=white, fill=white] (-2,-.2) circle(.4);
\node at (-3,-3) [anchor=east]{$\mathcal Z'(M, l_{\partial M}, b_{\partial M})\stackrel{\eqref{pic:gluepoly}}=\sum_{\substack{{j,k}\\{\gamma,\tau}}} \dim(j)$};
\node at (-3,-4)[anchor=east]{$\dim(k)$};
\draw[line width=.5pt, color=black] (-1,5)--(1,5) node[anchor=west]{$p$};
\draw[line width=.5pt] (0,5)--(0,3); 
\draw[line width=.5pt] (0,-11)--(0,-13); 
\draw[line width=.5pt, color=black] (-1,-13)--(1,-13) node[anchor=west]{$p$};
\draw[line width=.5pt] (0,3)--(-2,1) node[midway, anchor=east]{$q$};
\draw[line width=.5pt] (0,3)--(2,-1) node[midway, anchor=south west] {$l$};
\draw[line width=.5pt] (-2,1)--(2,-1) node[midway, anchor=south]{$n$};
\draw[line width=.5pt] (-2,1)--(-2,-5) node[midway, anchor=north east]{$k$};  
\draw[line width=.5pt] (2,-1)--(2,-3) node[midway, anchor=west]{$m$};  
\draw[line width=.5pt] (2,-3)--(-2,-5) node[midway, anchor=south]{$n$};    
\draw[line width=.5pt] (-2,-5)--(-2,-7) node[midway, anchor=east]{$q$};  
\draw[line width=.5pt] (2,-3)--(2,-9) node[midway, anchor=west]{$j$};  
\draw[line width=.5pt] (-2,-7)--(2,-9) node[midway, anchor=south]{$n$};  
\draw[line width=.5pt] (0,-11)--(2,-9) node[midway, anchor=north west]{$m$};  
\draw[line width=.5pt] (0,-11)--(-2,-7) node[midway, anchor=north east]{$i$};  
\draw[line width=1pt, color=cyan] (-2,-7) ..controls (-3,-10).. (0,-11) node[midway, anchor=east]{$a$};
\draw[color=black, fill=black] (0,3) circle (.2) node[anchor=west]{$\;\nu$};
\draw[color=black, fill=black] (-2,1) circle (.2) node[anchor=east]{$\tau\;$};
\draw[color=black, fill=black] (2,-1) circle (.2) node[anchor=west]{$\;\mu$};
\draw[color=black, fill=black] (-2,-5) circle (.2) node[anchor=east]{$\tau\;$};
\draw[color=black, fill=black] (2,-3) circle (.2) node[anchor=west]{$\;\gamma$};
\draw[color=black, fill=black] (-2,-7) circle (.2) node[anchor=east]{$\;\beta$};
\draw[color=black, fill=black] (2,-9) circle (.2) node[anchor=west]{$\;\gamma$};
\draw[color=black, fill=black] (0,-11) circle (.2) node[anchor=west]{$\;\delta$};
%%%%%%%%%%%%%
\begin{scope}[shift={(13,0)}]
\draw[line width=1pt, color=magenta] (0,3) ..controls (-6,2) and (-6,0) .. (2,-1) node[midway, anchor=east]{$b$};
\draw[color=white, fill=white] (-2,-.2) circle(.4);
\node at (-3,-3) [anchor=east]{$\stackrel{\eqref{eq:traceout}}=\sum_{k,\tau} \dim(k)$};
\draw[line width=.5pt, color=black] (-1,5)--(1,5) node[anchor=west]{$p$};
\draw[line width=.5pt] (0,5)--(0,3); 
\draw[line width=.5pt] (0,-11)--(0,-13); 
\draw[line width=.5pt, color=black] (-1,-13)--(1,-13) node[anchor=west]{$p$};
\draw[line width=.5pt] (0,3)--(-2,1) node[midway, anchor=east]{$q$};
\draw[line width=.5pt] (0,3)--(2,-1) node[midway, anchor=south west] {$l$};
\draw[line width=.5pt] (-2,1)--(2,-1) node[midway, anchor=south]{$n$};
\draw[line width=.5pt] (-2,1)--(-2,-5) node[midway, anchor=north east]{$k$};  
\draw[line width=.5pt] (2,-1)--(0,-11) node[midway, anchor=west]{$m$};  
\draw[line width=.5pt] (-2,-5)--(-2,-7) node[midway, anchor=east]{$q$};   
\draw[line width=.5pt] (0,-11)--(-2,-7) node[midway, anchor=north east]{$i$}; 
 \draw[line width=.5pt] (-2,-7).. controls (-1,-8) and (0,-8).. (0,-7);
 \draw[line width=.5pt] (-2,-5).. controls (-1,-4) and (0,-4).. (0,-5); 
 \draw[line width=.5pt] (0,-7)--(0,-5) node[midway, anchor=east]{$n$} node[sloped, pos=.5, allow upside down]{\arrowOut};
\draw[line width=1pt, color=cyan] (-2,-7) ..controls (-3,-10).. (0,-11) node[midway, anchor=east]{$a$};
\draw[color=black, fill=black] (0,3) circle (.2) node[anchor=west]{$\;\nu$};
\draw[color=black, fill=black] (-2,1) circle (.2) node[anchor=east]{$\tau\;$};
\draw[color=black, fill=black] (2,-1) circle (.2) node[anchor=west]{$\;\mu$};
\draw[color=black, fill=black] (-2,-5) circle (.2) node[anchor=east]{$\tau\;$};
\draw[color=black, fill=black] (-2,-7) circle (.2) node[anchor=east]{$\;\beta$};
\draw[color=black, fill=black] (0,-11) circle (.2) node[anchor=west]{$\;\delta$};
\end{scope}
%%%%%%%%%%%%%
\begin{scope}[shift={(24.5,0)}]
\draw[line width=1pt, color=magenta] (0,3) ..controls (-6,2) and (-6,0) .. (2,-1) node[midway, anchor=east]{$b$};
\draw[color=white, fill=white] (-1.4,-.4) circle(.4);
\draw[color=white, fill=white] (.4,-.8) circle(.4);
\node at (-3,-3) [anchor=east]{$\stackrel{\eqref{eq:snake}}=\sum_{k,\tau} \dim(k)$};
\draw[line width=.5pt, color=black] (-1,5)--(1,5) node[anchor=west]{$p$};
\draw[line width=.5pt] (0,5)--(0,3); 
\draw[line width=.5pt] (0,-11)--(0,-13); 
\draw[line width=.5pt, color=black] (-1,-13)--(1,-13) node[anchor=west]{$p$};
\draw[line width=.5pt] (0,3)--(-2,-2) node[midway, anchor=east]{$q$};
\draw[line width=.5pt] (0,3)--(2,-1) node[midway, anchor=south west] {$l$};
\draw[line width=.5pt] (-2,-2).. controls (-.2,-5) and (.2,2).. (2,-1) node[midway, anchor=north west]{$n$} node[sloped, pos=.5, allow upside down]{\arrowOut};
\draw[line width=.5pt] (-2,-2)--(-2,-5) node[pos=.3, anchor=north east]{$k$};  
\draw[line width=.5pt] (2,-1)--(0,-11) node[midway, anchor=west]{$m$};  
\draw[line width=.5pt] (-2,-5)--(-2,-7) node[midway, anchor=east]{$q$};   
\draw[line width=.5pt] (0,-11)--(-2,-7) node[midway, anchor=north east]{$i$}; 
 \draw[line width=.5pt] (-2,-7).. controls (-1,-8) and (0,-8).. (0,-7);
 \draw[line width=.5pt] (-2,-5).. controls (-1,-4) and (0,-4).. (0,-5); 
 \draw[line width=.5pt] (0,-7)--(0,-5) node[midway, anchor=east]{$n$} node[sloped, pos=.5, allow upside down]{\arrowOut};
\draw[line width=1pt, color=cyan] (-2,-7) ..controls (-3,-10).. (0,-11) node[midway, anchor=east]{$a$};
\draw[color=black, fill=black] (0,3) circle (.2) node[anchor=west]{$\;\nu$};
\draw[color=black, fill=black] (-2,-2) circle (.2) node[anchor=east]{$\tau\;$};
\draw[color=black, fill=black] (2,-1) circle (.2) node[anchor=west]{$\;\mu$};
\draw[color=black, fill=black] (-2,-5) circle (.2) node[anchor=east]{$\tau\;$};
\draw[color=black, fill=black] (-2,-7) circle (.2) node[anchor=east]{$\;\beta$};
\draw[color=black, fill=black] (0,-11) circle (.2) node[anchor=west]{$\;\delta$};
\end{scope}
\end{tikzpicture}
\end{align*}
Applying again identities \eqref{eq:traceout}  and  \eqref{eq:snake} together with the naturality of the half-braiding   yields
\begin{align*}
\begin{tikzpicture}[scale=.4, baseline=(current bounding box.center)]
\node at (-5,0) [anchor=east]{$\mathcal Z'(M, l_{\partial M}, b_{\partial M})\stackrel{\eqref{eq:traceout}}=$};
\begin{scope}[scale=.7, shift={(0,2)}]
\draw[line width=1pt, color=magenta] (0,3) ..controls (-6,2) and (-6,0) .. (2,-1) node[midway, anchor=east]{$b$};
\draw[color=white, fill=white] (-1.7,-.3) circle(.4);
\draw[color=white, fill=white] (.2,-.7) circle(.4);
\draw[line width=.5pt, color=black] (-1.5,5)--(1.5,5) node[anchor=west]{$p$};
\draw[line width=.5pt] (0,5)--(0,3); 
\draw[line width=.5pt, color=black] (-1.5,-9)--(1.5,-9) node[anchor=west]{$p$};
\draw[line width=.5pt] (0,-7)--(0,-9); 
\draw[line width=.5pt] (0,3)--(2,-1) node[midway, anchor=south west] {$l$};
\draw[line width=.5pt] (0,3)--(-3,-3) node[pos=.3, anchor=east]{$q$};
\draw[line width=.5pt] (0,-7)--(2,-1) node[midway, anchor=north west]{$m$};
\draw[line width=.5pt]  (-3,-3).. controls (-1,-6.5) and  (0,2.5) ..  (2,-1) node[pos=.5, anchor=east]{$n$} node[sloped, pos=.5, allow upside down]{\arrowOut};
\draw[line width=.5pt] (0,-7)--(-3,-3) node[pos=.3, anchor=south west]{$i$}; 
\draw[line width=1pt, color=cyan] (-3,-3) ..controls (-4,-6).. (0,-7) node[midway, anchor=east]{$a$};
\draw[color=black, fill=black] (0,3) circle (.28) node[anchor=west]{$\;\nu$};
\draw[color=black, fill=black] (2,-1) circle (.28) node[anchor=west]{$\;\mu$};
\draw[color=black, fill=black] (-3,-3) circle (.28) node[anchor=east]{$\beta\;$};
\draw[color=black, fill=black] (0,-7) circle (.28) node[anchor=west]{$\;\delta$};
\end{scope}
%%%%%%%%%
\begin{scope}[shift={(11,0)}]
\node at (-6,0)[anchor=east]{$\stackrel{\eqref{eq:snake}}=$};
\draw[line width=1pt, color=magenta] (0,3) ..controls (-6,2) and (-6,0) .. (2,-1) node[midway, anchor=east]{$b$};
\draw[color=white, fill=white]  (-1.4,-.4) circle (.3);
\draw[color=white, fill=white]  (-2.4,-.2) circle (.3);
\draw[line width=.5pt, color=black] (-1,5)--(1,5) node[anchor=west]{$p$};
\draw[line width=.5pt] (0,5)--(0,3); 
\draw[line width=.5pt, color=black] (-1,-5)--(1,-5) node[anchor=west]{$p$};
\draw[line width=.5pt] (0,-5)--(0,-3); 
\draw[line width=.5pt] (0,3)--(2,-1) node[midway, anchor=south west] {$l$};
\draw[line width=.5pt] (0,3)--(-2,1) node[midway, anchor=east] {$q$};
\draw[line width=.5pt] (-2,1)--(2,-1) node[midway, anchor=south] {$n$};
\draw[line width=.5pt] (0,-3)--(2,-1) node[midway, anchor=north west] {$m$};
\draw[line width=.5pt] (0,-3)--(-2,1) node[midway, anchor=north east] {$i$};
\draw[line width=1pt, color=cyan] (-2,1) ..controls (-3,-2).. (0,-3) node[midway, anchor=east]{$a$};
\draw[color=black, fill=black] (0,3) circle (.2) node[anchor=west]{$\;\nu$};
\draw[color=black, fill=black] (2,-1) circle (.2) node[anchor=west]{$\;\mu$};
\draw[color=black, fill=black] (-2,1) circle (.2) node[anchor=east]{$\beta\;$};
\draw[color=black, fill=black] (0,-3) circle (.2) node[anchor=west]{$\;\delta$};
\end{scope}
\end{tikzpicture}
\end{align*}
\end{proof}

Example \ref{ex:braiding} allows one to construct any  ribbon link invariant for the centre $\mathcal Z(\mac)$ of a spherical fusion category $\mac$ from state sums that involve only defect surfaces and defect lines. This is achieved by considering a  ribbon diagram and realising each crossing as in Example \ref{ex:braiding}.

\begin{example}\label{ex:tangleex} Let $\mac$ be a spherical fusion category with center $\mathcal Z(\mac)$ and $L$ a ribbon link whose connected components are labeled by objects of $\mathcal Z(\mac)$.  Let $M$ be the following triangulated 3-manifold with defect data:
\begin{enumerate}
\item  Realise a  ribbon diagram $D_L$ for $L$  as a  
PL diagram  that is the horizontal and vertical composite of the following squares  and the ones obtained by rotating them, reversing orientation in the third square and replacing an over- by an undercrossing in the last square 
\begin{align*}
\begin{tikzpicture}[scale=.3]
\draw[line width=1pt, dotted] (-4,4)--(4,4) --(4,-4)--(-4,-4)--(-4,4);
\end{tikzpicture}
\qquad
\begin{tikzpicture}[scale=.3]
\draw[line width=1pt, dotted] (-4,4)--(4,4) --(4,-4)--(-4,-4)--(-4,4);
\draw[line width=1pt, color=cyan, style=dashed] (4,0)--(-4,0) node[pos=.5, anchor=north] {$a$} node[sloped, pos=.5, allow upside down]{\arrowIn};
\end{tikzpicture}
\qquad
\begin{tikzpicture}[scale=.3]
\draw[line width=1pt, dotted] (-4,4)--(4,4) --(4,-4)--(-4,-4)--(-4,4);
\draw[line width=1pt, color=cyan, style=dashed] (0,4)--(0,0)--(-4,0) node[pos=.6, anchor=north] {$a$} node[sloped, pos=.6, allow upside down]{\arrowIn};
\end{tikzpicture}
\qquad
\begin{tikzpicture}[scale=.3]
\draw[line width=1pt, dotted] (-4,4)--(4,4) --(4,-4)--(-4,-4)--(-4,4);
\draw[line width=1pt, color=magenta, style=dashed] (0,4)--(0,-4) node[pos=.8, anchor=west] {$b$} node[sloped, pos=.8, allow upside down]{\arrowIn};
\draw[fill=white, color=white] (0,0) circle (1.2);
\draw[line width=1pt, color=cyan, style=dashed] (4,0)--(-4,0) node[pos=.7, anchor=north] {$a$} node[sloped, pos=.7, allow upside down]{\arrowIn};
\end{tikzpicture}
 \end{align*}

\item  Assign to a  dashed line with label $(a,\sigma_{-,a})$ in $\mathcal Z(\mac)$   the $(\mac,\mac)$-bimodule functor $F_a=a\oo -: \mac\to\mac$  with coherence data as in \eqref{eq:coeherencecenter}, to a crossing with labels $a,b$  the bimodule natural transformation given by  $\sigma_{a,b}^{\pm 1}$, as in Example \ref{ex:braiding}, and to the midpoint of the second  and  third diagram either the identity or  the (co)units  of the  adjunction $F_a^l\dashv F_a$. 

\item Draw solid vertical and horizontal lines in each square that connect the midpoints of the boundary edges and end to the left of the dashed lines, viewed in the direction of their orientation. Let the vertical  line cross under and the horizontal  line over all other lines of the diagram. 

\item 
Label each vertex of the square and each solid line  by a simple object of $\mac$, such that the labels match at the vertices and edges where the squares are glued to form the  diagram $D_L$. Label each  endpoint of lines at the boundary  by a morphism in the appropriate morphism space in $\mac$
\begin{align*}
&\begin{tikzpicture}[scale=.3, baseline=(current bounding box.center)]
\draw[line width=1pt, dotted] (-4,4)--(4,4)node[anchor=south west]{$q$} --(4,-4) node[anchor=north west]{$n$}--(-4,-4)node[anchor=north east]{$m$}--(-4,4) node[anchor=south east]{$p$};
\draw[line width=.5pt, color=black] (0,4)--(0,-4) node[pos=.2, anchor=west]{$j$}node[sloped, pos=.2, allow upside down]{\arrowOut};
\draw[fill=white, color=white] (0,0) circle (1.2);
\draw[line width=.5pt, color=black] (4,0) --(-4,0) node[pos=.8, anchor=north]{$i$}node[sloped, pos=.8, allow upside down]{\arrowOut};
\draw[color=black, fill=black] (0,4) circle (.2) node[anchor=south]{$\alpha$};
\draw[color=black, fill=black] (4,0) circle (.2) node[anchor=west]{$\beta$};
\draw[color=black, fill=black] (0,-4) circle (.2) node[anchor=north]{$\gamma$};
\draw[color=black, fill=black] (-4,0) circle (.2) node[anchor=east]{$\delta$};
\end{tikzpicture}
&
&\begin{tikzpicture}[scale=.3, baseline=(current bounding box.center)]
\draw[line width=1pt, dotted] (-4,4)--(4,4)node[anchor=south west]{$q$} --(4,-4) node[anchor=north west]{$n$}--(-4,-4)node[anchor=north east]{$m$}--(-4,4) node[anchor=south east]{$p$};
\draw[line width=.5pt, color=black] (0,4)--(0,-4) node[pos=.2, anchor=west]{$j$}node[sloped, pos=.2, allow upside down]{\arrowOut};
\draw[fill=white, color=white] (0,-.5) circle (1.2);
\draw[line width=.5pt, color=black] (4,0) .. controls (0,-1.5) ..  (-4,0) node[pos=.8, anchor=north]{$i$}node[sloped, pos=.8, allow upside down]{\arrowOut};
\draw[line width=1pt, color=cyan, dashed] (4,0)--(-4,0) node[pos=.8, anchor=south]{$a$}node[sloped, pos=.8, allow upside down]{\arrowIn};
\draw[color=black, fill=black] (0,4) circle (.2) node[anchor=south]{$\alpha$};
\draw[color=black, fill=black] (4,0) circle (.2) node[anchor=west]{$\beta$};
\draw[color=black, fill=black] (0,-4) circle (.2) node[anchor=north]{$\gamma$};
\draw[color=black, fill=black] (-4,0) circle (.2) node[anchor=east]{$\delta$};
\end{tikzpicture}
%%%
&
&\begin{tikzpicture}[scale=.3, baseline=(current bounding box.center)]
\draw[line width=1pt, dotted] (-4,4)--(4,4)node[anchor=south west]{$q$} --(4,-4) node[anchor=north west]{$n$}--(-4,-4)node[anchor=north east]{$m$}--(-4,4) node[anchor=south east]{$p$};
\draw[line width=.5pt, color=black] (0,4)--(0,-4) node[pos=.2, anchor=west]{$j$}node[sloped, pos=.2, allow upside down]{\arrowOut};
\draw[fill=white, color=white] (0,0) circle (1.2);
\draw[line width=.5pt, color=black] (4,0) --(-4,0) node[pos=.8, anchor=north]{$i$}node[sloped, pos=.8, allow upside down]{\arrowOut};
\draw[line width=1pt, color=cyan, dashed] (0,4).. controls (-1,1).. (-4,0) node[pos=.8, anchor=south]{$a$}node[sloped, pos=.8, allow upside down]{\arrowIn};
\draw[color=black, fill=black] (0,4) circle (.2) node[anchor=south]{$\alpha$};
\draw[color=black, fill=black] (4,0) circle (.2) node[anchor=west]{$\beta$};
\draw[color=black, fill=black] (0,-4) circle (.2) node[anchor=north]{$\gamma$};
\draw[color=black, fill=black] (-4,0) circle (.2) node[anchor=east]{$\delta$};
\end{tikzpicture}
%%%
&
&\begin{tikzpicture}[scale=.3, baseline=(current bounding box.center)]
\draw[line width=1pt, color=magenta, dashed] (0,4)--(0,-4) node[pos=.8, anchor=east]{$b$}node[sloped, pos=.8, allow upside down]{\arrowIn};
\draw[line width=1pt, dotted] (-4,4)--(4,4)node[anchor=south west]{$q$} --(4,-4) node[anchor=north west]{$n$}--(-4,-4)node[anchor=north east]{$m$}--(-4,4) node[anchor=south east]{$p$};
\draw[line width=.5pt, color=black] (0,4).. controls (1.5,0) ..(0,-4) node[pos=.2, anchor=west]{$j$}node[sloped, pos=.2, allow upside down]{\arrowOut};
\draw[fill=white, color=white] (.5,-.5) circle (1.2);
\draw[line width=.5pt, color=black] (4,0) .. controls (0,-1.5) ..  (-4,0) node[pos=.8, anchor=north]{$i$}node[sloped, pos=.8, allow upside down]{\arrowOut};
\draw[line width=1pt, color=cyan, dashed] (4,0)--(-4,0) node[pos=.8, anchor=south]{$a$}node[sloped, pos=.8, allow upside down]{\arrowIn};
\draw[color=black, fill=black] (0,4) circle (.2) node[anchor=south]{$\alpha$};
\draw[color=black, fill=black] (4,0) circle (.2) node[anchor=west]{$\beta$};
\draw[color=black, fill=black] (0,-4) circle (.2) node[anchor=north]{$\gamma$};
\draw[color=black, fill=black] (-4,0) circle (.2) node[anchor=east]{$\delta$};
\end{tikzpicture}                                             
\end{align*}

\item Assign to  each  square in 4.~a defect tetrahedron for the $(\mac,\mac)$-bimodule category $\mac$ whose 6j symbol is the polygon diagram in 4.~,  as in Figure \ref{fig:basictets}. Construct the last defect tetrahedron in Figure \ref{fig:basictets} as in Example \ref{ex:braiding} by gluing three tetrahedra with defect planes without crossings.

\item The polygon diagrams glue to a diagram that superimposes $D_L$ with additional edges labeled by objects of $\mac$, whose endpoints  are labeled by morphisms in $\mac$. This defines a gluing pattern for the associated tetrahedra, as   in Figure \ref{fig:braidingtet}, that yields  a triangulated 3-ball $M$ with a ribbon link $L'$ in the interior.
\end{enumerate}
The state sum of $M$ is  
\begin{align*}
\mathcal Z(M, l_{\partial M}, b_{\partial M})=\mathrm{ev}(D_L)\cdot\mathcal Z(M', l_{\partial M}, b_{\partial M}) ,
\end{align*}
where $\mathcal Z(M', l_{\partial M}, b_{\partial M})$ is the  state sum of the 3-ball $M'$ with the same labels, but without  defects.
\end{example}

\begin{figure}
\begin{align*}
&\begin{tikzpicture}[scale=.5]
\draw[line width=1pt, color=black] (-5,0)--(0,4)  node[sloped, pos=0.4, allow upside down]{\arrowIn} node[pos=.4, anchor=south east]{$p$};
\draw[line width=1pt, color=black] (5,0)--(0,4)  node[sloped, pos=0.4, allow upside down]{\arrowIn} node[pos=.4, anchor=south west]{$q$};
\draw[line width=1pt, color=black] (-5,0)--(0,-4)  node[sloped, pos=0.4, allow upside down]{\arrowIn} node[pos=.4, anchor=north east]{$m$};
\draw[line width=1pt, color=black] (5,0)--(0,-4)  node[sloped, pos=0.4, allow upside down]{\arrowIn} node[pos=.4, anchor=north west]{$n$};
\draw[line width=1pt, color=black] (-5,0)--(-.5,0)  node[sloped, pos=0.5, allow upside down]{\arrowIn} node[pos=.5, anchor=south]{$j$};
\draw[line width=1pt, color=black] (.5,0)--(5,0)  ;
\draw[draw=none, fill=white] (-2,0) circle(.3);
\draw[draw=none, fill=white] (2,0) circle(.3);
\draw[draw=none, fill=white] (0,1) circle(.3);
\draw[line width=1pt, dotted] (-2,2.4)--(-2,-2.4)--(2,-2.4)--(2,2.4)--(-2,2.4);
\draw[draw=none, fill=white] (0,2.4) circle(.5);
\draw[draw=none, fill=white] (0,-2.4) circle(.5);
\draw[draw=none, fill=white] (0,1) circle(.3);
\draw[line width=1pt, color=black] (0,-4)--(0,4)  node[sloped, pos=0.3, allow upside down]{\arrowIn} node[pos=.3, anchor=east]{$i$};
\draw[line width=.5pt, -stealth] (2,-1)--(2.5,-1.5);
\begin{scope}[shift={(12,0)}, scale=.7]
\draw[line width=1pt, dotted] (-4,4)--(4,4)node[anchor=south west]{$q$} --(4,-4) node[anchor=north west]{$n$}--(-4,-4)node[anchor=north east]{$m$}--(-4,4) node[anchor=south east]{$p$};
\draw[line width=.5pt, color=black] (0,4)--(0,-4) node[pos=.2, anchor=west]{$j$}node[sloped, pos=.2, allow upside down]{\arrowOut};
\draw[fill=white, color=white] (0,0) circle (1.2);
\draw[line width=.5pt, color=black] (4,0) --(-4,0) node[pos=.8, anchor=north]{$i$}node[sloped, pos=.8, allow upside down]{\arrowOut};
\draw[color=black, fill=black] (0,4) circle (.2) node[anchor=south]{$\alpha$};
\draw[color=black, fill=black] (4,0) circle (.2) node[anchor=west]{$\beta$};
\draw[color=black, fill=black] (0,-4) circle (.2) node[anchor=north]{$\gamma$};
\draw[color=black, fill=black] (-4,0) circle (.2) node[anchor=east]{$\delta$};
\end{scope}
\end{tikzpicture}
\\
&\begin{tikzpicture}[scale=.5]
\draw[line width=1pt, color=black] (-5,0)--(0,4)  node[sloped, pos=0.4, allow upside down]{\arrowIn} node[pos=.4, anchor=south east]{$p$};
\draw[line width=1pt, color=black] (5,0)--(0,4)  node[sloped, pos=0.4, allow upside down]{\arrowIn} node[pos=.4, anchor=south west]{$q$};
\draw[line width=1pt, color=black] (-5,0)--(0,-4)  node[sloped, pos=0.4, allow upside down]{\arrowIn} node[pos=.4, anchor=north east]{$m$};
\draw[line width=1pt, color=black] (5,0)--(0,-4)  node[sloped, pos=0.4, allow upside down]{\arrowIn} node[pos=.4, anchor=north west]{$n$};
\draw[line width=1pt, color=black] (-5,0)--(-.5,0)  node[sloped, pos=0.5, allow upside down]{\arrowIn} node[pos=.5, anchor=south]{$j$};
\draw[line width=1pt, color=black] (.5,0)--(5,0)  ;
\draw[draw=none, fill=white] (-2,0) circle(.3);
\draw[draw=none, fill=white] (2,0) circle(.3);
\draw[draw=none, fill=white] (0,1) circle(.3);
\draw[line width=1pt, dotted] (-2,2.4)--(-2,-2.4)--(2,-2.4)--(2,2.4)--(-2,2.4);
\draw[draw=none, fill=white] (0,2.4) circle(.5);
\draw[draw=none, fill=white] (0,-2.4) circle(.5);
\draw[color=cyan, line width=1pt] (2, 1)--(-2,1) node[sloped, pos=0.7, allow upside down]{\arrowIn} node[pos=.7, anchor=south]{$F_a$};
\draw[draw=none, fill=white] (0,1) circle(.3);
\draw[line width=1pt, color=black] (0,-4)--(0,4)  node[sloped, pos=0.3, allow upside down]{\arrowIn} node[pos=.3, anchor=east]{$i$};
\draw[line width=.5pt, -stealth] (2,-1)--(2.5,-1.5);
\begin{scope}[shift={(12,0)}, scale=.7]
\draw[line width=1pt, dotted] (-4,4)--(4,4)node[anchor=south west]{$q$} --(4,-4) node[anchor=north west]{$n$}--(-4,-4)node[anchor=north east]{$m$}--(-4,4) node[anchor=south east]{$p$};
\draw[line width=.5pt, color=black] (0,4)--(0,-4) node[pos=.2, anchor=west]{$j$}node[sloped, pos=.2, allow upside down]{\arrowOut};
\draw[fill=white, color=white] (0,-.5) circle (1.2);
\draw[line width=.5pt, color=black] (4,0) .. controls (0,-1.5) ..  (-4,0) node[pos=.8, anchor=north]{$i$}node[sloped, pos=.8, allow upside down]{\arrowOut};
\draw[line width=1pt, color=cyan, dashed] (4,0)--(-4,0) node[pos=.8, anchor=south]{$a$}node[sloped, pos=.8, allow upside down]{\arrowIn};
\draw[color=black, fill=black] (0,4) circle (.2) node[anchor=south]{$\alpha$};
\draw[color=black, fill=black] (4,0) circle (.2) node[anchor=west]{$\beta$};
\draw[color=black, fill=black] (0,-4) circle (.2) node[anchor=north]{$\gamma$};
\draw[color=black, fill=black] (-4,0) circle (.2) node[anchor=east]{$\delta$};
\end{scope}
\end{tikzpicture}
\\
&\begin{tikzpicture}[scale=.5]
\draw[line width=1pt, color=black] (-5,0)--(0,4)  node[sloped, pos=0.4, allow upside down]{\arrowIn} node[pos=.4, anchor=south east]{$p$};
\draw[line width=1pt, color=black] (5,0)--(0,4)  node[sloped, pos=0.4, allow upside down]{\arrowIn} node[pos=.4, anchor=south west]{$q$};
\draw[line width=1pt, color=black] (-5,0)--(0,-4)  node[sloped, pos=0.4, allow upside down]{\arrowIn} node[pos=.4, anchor=north east]{$m$};
\draw[line width=1pt, color=black] (5,0)--(0,-4)  node[sloped, pos=0.4, allow upside down]{\arrowIn} node[pos=.4, anchor=north west]{$n$};
\draw[line width=1pt, color=black] (-5,0)--(-.5,0)  node[sloped, pos=0.5, allow upside down]{\arrowIn} node[pos=.5, anchor=south]{$j$};
\draw[line width=1pt, color=black] (.5,0)--(5,0)  ;
\draw[draw=none, fill=white] (-2,0) circle(.3);
\draw[draw=none, fill=white] (2,0) circle(.3);
\draw[draw=none, fill=white] (0,1) circle(.3);
\draw[line width=1pt, dotted] (-2,2.4)--(-2,-2.4)--(2,-2.4)--(2,2.4)--(-2,2.4);
\draw[draw=none, fill=white] (0,2.4) circle(.5);
\draw[draw=none, fill=white] (0,-2.4) circle(.5);
\draw[color=cyan, line width=1pt] (1, 1)--(-2,1) node[sloped, pos=0.7, allow upside down]{\arrowIn} node[pos=.7, anchor=south]{$F_a$};
\draw[color=cyan, line width=1pt] (1, 2.4)--(1,1) node[sloped, pos=0.7, allow upside down]{\arrowIn};
\draw[draw=none, fill=white] (0,1) circle(.3);
\draw[line width=1pt, color=black] (0,-4)--(0,4)  node[sloped, pos=0.3, allow upside down]{\arrowIn} node[pos=.3, anchor=east]{$i$};
\draw[line width=.5pt, -stealth] (2,-1)--(2.5,-1.5);
\begin{scope}[shift={(12,0)}, scale=.7]
 \draw[line width=1pt, dotted] (-4,4)--(4,4)node[anchor=south west]{$q$} --(4,-4) node[anchor=north west]{$n$}--(-4,-4)node[anchor=north east]{$m$}--(-4,4) node[anchor=south east]{$p$};
\draw[line width=.5pt, color=black] (0,4)--(0,-4) node[pos=.2, anchor=west]{$j$}node[sloped, pos=.2, allow upside down]{\arrowOut};
\draw[fill=white, color=white] (0,0) circle (1.2);
\draw[line width=.5pt, color=black] (4,0) --(-4,0) node[pos=.8, anchor=north]{$i$}node[sloped, pos=.8, allow upside down]{\arrowOut};
\draw[line width=1pt, color=cyan, dashed] (0,4).. controls (-1,1).. (-4,0) node[pos=.8, anchor=south]{$a$}node[sloped, pos=.8, allow upside down]{\arrowIn};
\draw[color=black, fill=black] (0,4) circle (.2) node[anchor=south]{$\alpha$};
\draw[color=black, fill=black] (4,0) circle (.2) node[anchor=west]{$\beta$};
\draw[color=black, fill=black] (0,-4) circle (.2) node[anchor=north]{$\gamma$};
\draw[color=black, fill=black] (-4,0) circle (.2) node[anchor=east]{$\delta$};
\end{scope}
\end{tikzpicture}
\\
&\begin{tikzpicture}[scale=.5]
\draw[line width=1pt, color=black] (-5,0)--(0,4)  node[sloped, pos=0.4, allow upside down]{\arrowIn} node[pos=.4, anchor=south east]{$p$};
\draw[line width=1pt, color=black] (5,0)--(0,4)  node[sloped, pos=0.4, allow upside down]{\arrowIn} node[pos=.4, anchor=south west]{$q$};
\draw[line width=1pt, color=black] (-5,0)--(0,-4)  node[sloped, pos=0.4, allow upside down]{\arrowIn} node[pos=.4, anchor=north east]{$m$};
\draw[line width=1pt, color=black] (5,0)--(0,-4)  node[sloped, pos=0.4, allow upside down]{\arrowIn} node[pos=.4, anchor=north west]{$n$};
\draw[line width=1pt, color=black] (-5,0)--(-.5,0)  node[sloped, pos=0.5, allow upside down]{\arrowIn} node[pos=.5, anchor=south]{$j$};
\draw[line width=1pt, color=black] (.5,0)--(5,0)  ;
\draw[draw=none, fill=white] (-2,0) circle(.3);
\draw[draw=none, fill=white] (2,0) circle(.3);
\draw[draw=none, fill=white] (0,1) circle(.3);
\draw[draw=none, fill=white] (1,0) circle(.3);
\draw[line width=1pt, dotted] (-2,2.4)--(-2,-2.4)--(2,-2.4)--(2,2.4)--(-2,2.4);
\draw[draw=none, fill=white] (0,2.4) circle(.5);
\draw[draw=none, fill=white] (0,-2.4) circle(.5);
\draw[color=magenta, line width=1pt] (1,2.4)--(1,-2.4) node[pos=.7, anchor=west]{$F_b$} node[sloped, pos=0.7, allow upside down]{\arrowIn} ;
\draw[draw=none, fill=white] (1,1) circle(.3);
\draw[color=cyan, line width=1pt] (2, 1)--(-2,1) node[pos=.7, anchor=south]{$F_a$}node[sloped, pos=.7, allow upside down]{\arrowIn} ;
\draw[draw=none, fill=white] (0,1) circle(.3);
\draw[line width=1pt, color=black] (0,-4)--(0,4)  node[sloped, pos=.3, allow upside down]{\arrowIn} node[pos=.3, anchor=east]{$i$};
\draw[line width=.5pt, -stealth] (2,-1)--(2.5,-1.5);
\begin{scope}[shift={(12,0)}, scale=.7]
\draw[line width=1pt, color=magenta, dashed] (0,4)--(0,-4) node[pos=.8, anchor=east]{$b$}node[sloped, pos=.8, allow upside down]{\arrowIn};
\draw[line width=1pt, dotted] (-4,4)--(4,4)node[anchor=south west]{$q$} --(4,-4) node[anchor=north west]{$n$}--(-4,-4)node[anchor=north east]{$m$}--(-4,4) node[anchor=south east]{$p$};
\draw[line width=.5pt, color=black] (0,4).. controls (1.5,0) ..(0,-4) node[pos=.2, anchor=west]{$j$}node[sloped, pos=.2, allow upside down]{\arrowOut};
\draw[fill=white, color=white] (.5,-.5) circle (1.2);
\draw[line width=.5pt, color=black] (4,0) .. controls (0,-1.5) ..  (-4,0) node[pos=.8, anchor=north]{$i$}node[sloped, pos=.8, allow upside down]{\arrowOut};
\draw[line width=1pt, color=cyan, dashed] (4,0)--(-4,0) node[pos=.8, anchor=south]{$a$}node[sloped, pos=.8, allow upside down]{\arrowIn};
\draw[color=black, fill=black] (0,4) circle (.2) node[anchor=south]{$\alpha$};
\draw[color=black, fill=black] (4,0) circle (.2) node[anchor=west]{$\beta$};
\draw[color=black, fill=black] (0,-4) circle (.2) node[anchor=north]{$\gamma$};
\draw[color=black, fill=black] (-4,0) circle (.2) node[anchor=east]{$\delta$};
\end{scope}
\end{tikzpicture}
\end{align*}
\caption{Defect tetrahedra assigned to the square diagrams from Example \ref{ex:tangleex}. The last tetrahedron is obtained by gluing three tetrahedra without crossings as in Example \ref{ex:braiding}.}
 \label{fig:basictets}
\end{figure}
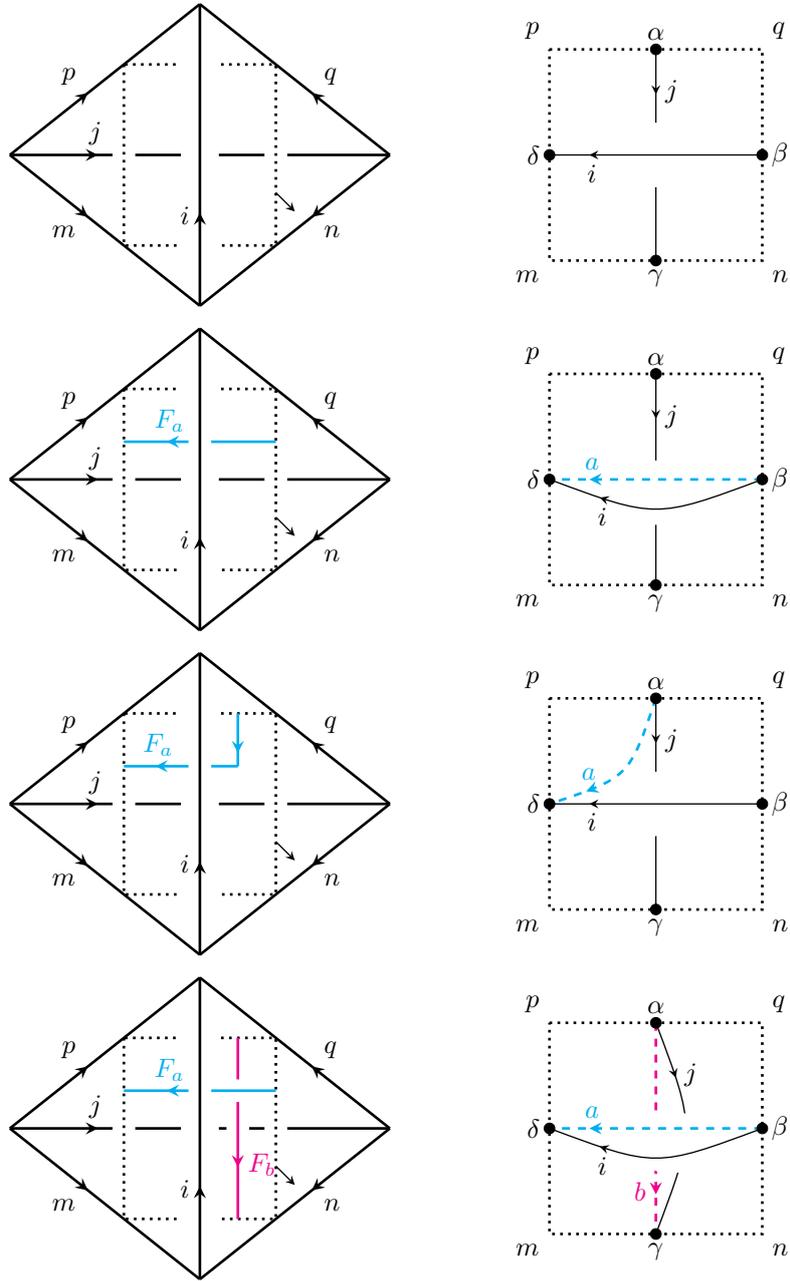

\begin{figure}
\begin{center}
\def\svgwidth{.8\columnwidth}
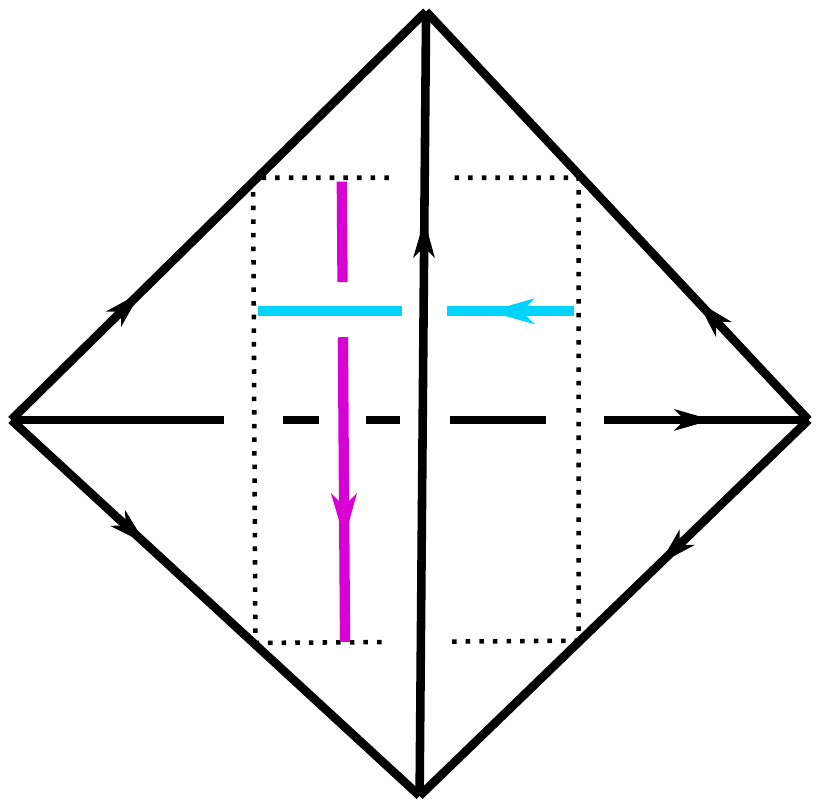 
 \end{center}
\caption{Gluing four defect tetrahedra and the associated polygon diagram.}
 \label{fig:braidingtet}
\end{figure}

\begin{proof} As the dimension factors for boundary edges and vertices coincide for $M$ and $M'$, it is sufficient to consider the rescaled state sums. 
By Proposition \ref{th:topological invariance} and Example \ref{ex:braiding}, the rescaled state sum for $M$ is the evaluation of the polygon diagram obtained by gluing all defect squares. This is the ribbon diagram $D_L$, superimposed with a graph labeled by  $\mac$,  which  is the dual of the triangulation on $\partial M$. 

The evaluation of this polygon diagram is computed as in Example \ref{ex:ribbonplane}. By \eqref{eq:coeherencecenter}   overcrossings of solid lines over lines in $T$ correspond to half-braidings of objects of $\mathcal Z(\mac)$ with objects in $\mac$, while undercrossings correspond to associators  in $\mac$.  As all vertices involve only solid lines,  one can use  identities \eqref{pic:rm2} to \eqref{pic:bimodulenat} to separate the graph labeled by $\mac$ and the 
 diagram  $D_L$. This yields the product of  the rescaled state sum $\mathcal Z'(M',l_{\partial M}, b_{\partial M})$ and the  evaluation of $D_L$.
\end{proof}

Example \ref{ex:tangleex} constructs a specific example of a triangulated  3-ball with a ribbon link as a defect. An analogous factorisation of the state sum is obtained for any  ribbon link labeled by $\mathcal Z(\mac)$   in any triangulated 3-ball labeled by $\mac$. This follows from the triangulation independence of the state sum and the fact that it is invariant under $\Delta$-moves, which encode ambient isotopies of ribbon links.

\begin{proposition}\label{prop:ribbonball} Let $M$ be a triangulated 3-ball labeled by a spherical fusion category $\mac$.  Let $L$ be a ribbon link in the interior of $M$ that is blackboard-framed with respect to a disc bisecting $M$ and labeled by $\mathcal Z(\mac)$. Then the state sum is given by
$$
\mathcal Z(M,l_{\partial M}, b_{\partial M})=\mathcal Z(M', l_{\partial M}, b_{\partial M})\cdot \mathrm{ev}(D_L),
$$
where $\mathcal Z(M', l_{\partial M}, b_{\partial M})$ is the state sum of the 3-ball $M'$  with the same labels, but  without the link, and $D_L$ a ribbon diagram obtained by projecting $L$ on the disc.
\end{proposition}

\begin{proof} 1.~We first show that the state sum is invariant under PL ambient isotopy of the ribbon link. By a result of Graeub \cite{Gr}, see also Burde and Zieschang \cite[Prop.~1.10]{BZ}
and Yetter \cite[Sec.~2.2]{Y}, two PL links in a 3-ball are ambient isotopic iff they are related by  a finite sequence of $\Delta$-moves. 
A $\Delta$-move 
places a triangle on the link that intersects it only  in one side and replaces this side by the other two sides of the triangle. 
\begin{center}
\begin{tikzpicture}[scale=.4]
\draw[line width=1pt] (0,1)--(5,1);
\draw[line width=1pt] (0,1)--(-1,.5);
\draw[line width=1pt] (5,1)--(6,.5);
\draw[line width=1pt, style=dashed] (0,1)--(2.5,3.5);
\draw[line width=1pt, style=dashed] (5,1)--(2.5,3.5);
\draw[line width=1pt, -stealth] (7.5,2)--(9.5,2) node[midway, anchor=south]{$\Delta$};
\begin{scope}[shift={(12,0)}]
\draw[line width=1pt] (0,1)--(-1,.5);
\draw[line width=1pt] (5,1)--(6,.5);
\draw[line width=1pt, ] (0,1)--(2.5,3.5);
\draw[line width=1pt, ] (5,1)--(2.5,3.5);
\draw[line width=1pt, style=dashed] (0,1)--(5,1);
\end{scope}
\end{tikzpicture}
\end{center}

An analogous result holds for ambient isotopies of blackboard-framed  ribbon links if one restricts attention to $\Delta$-moves that respect the framing, that is, keep the links blackboard framed  and can be realised without introducing kinks, see for instance \cite[Sec.~2.2]{Y}. Thus it is sufficient to show that the state sum is invariant under $\Delta$-moves that respect the blackboard framing of $L$.

Consider a blackboard framed $\Delta$-move with a triangle $\Delta$.  By Corollary \ref{cor:pachner} and Lemma \ref{lem:finetriang}, suppose without loss of generality  that the tetrahedra intersecting $\Delta$ form a fine neighbourhood of $\Delta$ and that  gluing all tetrahedra that intersect $\Delta$ yields a triangulated 3-ball $B$ bisected by a defect plane that extends $\Delta$.

By Proposition \ref{th:topological invariance} the state sums for $B$ before and after the $\Delta$-move are obtained by projecting the dual of its boundary triangulation on the defect plane and evaluating the resulting polygon diagram. This yields ribbon diagrams that superimpose a ribbon segment labeled by $\mathcal Z(\mac)$ with the boundary diagrams labeled by $\mac$. The polygon diagrams before and after the $\Delta$-move are related by a $\Delta$-move on the ribbon segment in the polygon. Identities \eqref{pic:rm2} to \eqref{pic:natfuncconst}  imply that the evaluations of these polygon diagrams are equal.  As $M\setminus B$ is not affected by the $\Delta$-move,  the state sum of $M$ is invariant under the $\Delta$-move by Corollary \ref{cor:gluing}. 

2.~By applying 2d isotopies to the ribbon diagram $D_L$, one can transform it into a diagram $D'$ as in Example \ref{ex:tangleex} and construct a triangulated 3-ball $N$ with a ribbon link $L'$ in the interior that projects to $D'=D_{L'}$.  By Example \ref{ex:tangleex} one has
$\mathcal Z(N, l_{\partial N}, b_{\partial N})=\mathcal Z(N', l_{\partial N}, b_{\partial N})\cdot \mathrm{ev}(D_{L'})$, where $N'$ is the 3-ball  without $L'$.  

As the boundary triangulations of $M$ and $N$ are related by elementary shellings and inverse shellings with tetrahedra that do not intersect $L$ and $L'$, it is sufficient to show that  $\mathcal Z(M, l_{\partial M}, b_{\partial M})=\mathcal Z(N, l_{\partial N}, b_{\partial N})$, if the labeled boundary triangulations for $M$ and $N$ agree. 
As the link diagrams $D_L$ and $D_{L'}$ are isotopic, one has $\mathrm{ev}(D_{L})=\mathrm{ev}(D_{L'})$ and the links $L$ and $L'$ that project to $D_L$ and $D_{L'}$ are ambient isotopic. This allows one to transform $L$ into $L'$ by a finite sequence of $\Delta$-moves. The first part of the proof implies that this does not affect the state sum. By the triangulation independence of the usual Turaev-Viro-Barrett-Westbury state sum, one has $\mathcal Z(M', l_{\partial M}, b_{\partial M})=\mathcal Z(N', l_{\partial N}, b_{\partial N})$, and with
Theorem \ref{th:fulltopinv} this yields
$$\mathcal Z(M, l_{\partial M}, b_{\partial M})\stackrel{1.}=\mathcal Z(N, l_{\partial N}, b_{\partial N})\stackrel{\ref{ex:tangleex}}=\mathrm{ev}(D_{L'})\cdot \mathcal Z(N', l_{\partial N}, b_{\partial N})\stackrel{\ref{th:fulltopinv}}=\mathrm{ev}(D_{L})\cdot \mathcal Z(M', l_{\partial M}, b_{\partial M}).$$
\end{proof}

Proposition \ref{prop:ribbonball} can be generalised  to  ribbon tangles realised as defect graphs on trivial defect surfaces. One can also generalise it to ribbon links embedded  in more general 3-manifolds. Embedding a ribbon link $L$ labeled by $\mathcal Z(\mac)$ in a single tetrahedron labeled by $\mac$ and gluing it to the  tetrahedron with the opposite orientation  along all four faces yields the state sum for a 3-sphere  $L$  in the interior
$$\mathcal Z(S^3,L)=\frac{\mathrm{ev}(D_L)} {\dim \mac}.$$
This coincides with Turaev and Virelizier's  state sum graph TQFT  for a 3-sphere containing a ribbon link in  \cite[Th.~16.1]{TVbook}.  In fact, we expect that our state sums with defect lines and vertices on trivial defect surfaces always coincide with the corresponding  graph TQFTs with ribbons from \cite[Ch.~16]{TVbook} and \cite{TVr}.  It would  be nice to show this in full generality and  to compare with the results by Balsam and Kirillov in  \cite{KB,B} that also involve ribbons labeled by $\mathcal Z(\mac)$.

We also expect that our state sum is  related  to the modular functor constructed from bimodule categories over finite tensor categories by Fuchs et al.~\cite{FSS} if one restricts the latter to our categorical data. However,  is not clear to us  how our construction is related to the work \cite{LY} by Lee and Yetter, which also constructs triangulation independent state sums with defect surfaces. However, the defect surfaces in \cite{LY} are labeled with different categorical data, namely  bicategories equipped with spherical duals and certain 2-functors. They  should be thought of as defects with additional structure, according to the authors. We believe that this is a different type of defect that cannot be related directly to our construction.

The examples in this section show that the  state sum models with defects give relevant information about defect surfaces and are easy to compute for simple categorical data. It would be feasible to investigate further examples that involve defect lines and points on non-trivial defect surfaces,  at least for spherical fusion categories $\mac=\mathrm{Vec}_G$ and $\mad=\mathrm{Vec}_{G'}$ and bimodule categories defined by finite transitive $G\times G'^{op}$-sets. In this case, the categorical defect data for defect lines and defect points  can be worked out explicitly to describe defects of all codimensions for Dijkgraaf-Witten theory \cite{DW}. It would also be interesting to investigate knotted defect surfaces beyond the torus and with less trivial categorical data.

With the results by Balsam and Kirillov \cite{BK,Kr}  on the relation between  Turaev-Viro-Barrett-Westbury state sums, Levin-Wen models \cite{LW} and Kitaev's quantum double models \cite{Kit}, it should also be possible to  apply our state sums with defects to those models and to make contact 
with the  work \cite{KK}  by Kitaev and Kong.
In the long run it would be desirable to determine if  our state sum model with defects defines a defect TQFT in the sense of  the works \cite{CMS,CRS2} by Carqueville et al.~and to relate it to the works \cite{CRS,domwall} by Carqueville et al.~and by Koppen et al.~on defects  and domain walls in the context of Reshetikhin-Turaev TQFTs. We leave these questions for future work.

\subsection*{Acknowledgements} 
This article originated from discussions with John W.~Barrett. We had originally planned to work on this project together, but this became impossible due to the Covid-19 pandemic and its impact on  research and teaching. I am extremely grateful to him for helpful discussions, encouragement and comments on drafts of this article. 
I am also  grateful to Gregor Schaumann for  remarks and discussions on bimodule categories and bimodule traces,  to Andreas Knauf for discussions  and  comments on drafts of this article and to the referee for helpful comments and suggestions.

\end{document}

%% file: 3int2.pdf_tex
%% Creator: Inkscape 1.0.2 (e86c8708, 2021-01-15), www.inkscape.org
%% Accompanies image file '3int2.pdf' (pdf, eps, ps)
\begingroup%
  \makeatletter%
  \providecommand\color[2][]{%
    \errmessage{(Inkscape) Color is used for the text in Inkscape, but the package 'color.sty' is not loaded}%
    \renewcommand\color[2][]{}%
  }%
  \providecommand\transparent[1]{%
    \errmessage{(Inkscape) Transparency is used (non-zero) for the text in Inkscape, but the package 'transparent.sty' is not loaded}%
    \renewcommand\transparent[1]{}%
  }%
  \providecommand\rotatebox[2]{#2}%
  \newcommand*\fsize{\dimexpr\f@size pt\relax}%
  \newcommand*\lineheight[1]{\fontsize{\fsize}{#1\fsize}\selectfont}%
  \ifx\svgwidth\undefined%
    \setlength{\unitlength}{841.88976378bp}%
    \ifx\svgscale\undefined%
      \relax%
    \else%
      \setlength{\unitlength}{\unitlength * \real{\svgscale}}%
    \fi%
  \else%
    \setlength{\unitlength}{\svgwidth}%
  \fi%
  \global\let\svgwidth\undefined%
  \global\let\svgscale\undefined%
  \makeatother%
  \begin{picture}(1,0.70707071)%
    \lineheight{1}%
    \setlength\tabcolsep{0pt}%
    \put(0,0){\includegraphics[width=\unitlength,page=1]{3int2.pdf}}%
    \put(0.21688627,0.56190663){\makebox(0,0)[lt]{\lineheight{1.25}\smash{\begin{tabular}[t]{l}{\color{blue}$\mam$}\end{tabular}}}}%
    \put(0.23094249,0.48841835){\makebox(0,0)[lt]{\lineheight{1.25}\smash{\begin{tabular}[t]{l}{\color{red}$\man$}\end{tabular}}}}%
    \put(0.32578021,0.56077513){\makebox(0,0)[lt]{\lineheight{1.25}\smash{\begin{tabular}[t]{l}{\color{violet}$\map$}\end{tabular}}}}%
    \put(0.3935224,0.41797967){\makebox(0,0)[lt]{\lineheight{1.25}\smash{\begin{tabular}[t]{l}{\color{violet}$p$}\end{tabular}}}}%
    \put(0.24140517,0.37340037){\makebox(0,0)[lt]{\lineheight{1.25}\smash{\begin{tabular}[t]{l}{\color{red}$n$}\end{tabular}}}}%
    \put(0.11091949,0.41031701){\makebox(0,0)[lt]{\lineheight{1.25}\smash{\begin{tabular}[t]{l}{\color{blue}$m$}\end{tabular}}}}%
    \put(0.23453653,0.63893075){\makebox(0,0)[lt]{\lineheight{1.25}\smash{\begin{tabular}[t]{l}$\mac$\end{tabular}}}}%
    \put(0.29492909,0.69088385){\makebox(0,0)[lt]{\lineheight{1.25}\smash{\begin{tabular}[t]{l}$i$\end{tabular}}}}%
    \put(0.12372914,0.62821736){\makebox(0,0)[lt]{\lineheight{1.25}\smash{\begin{tabular}[t]{l}$j$\end{tabular}}}}%
    \put(0.43900836,0.62783462){\makebox(0,0)[lt]{\lineheight{1.25}\smash{\begin{tabular}[t]{l}$k$\end{tabular}}}}%
    \put(0.28232983,0.60320948){\makebox(0,0)[lt]{\lineheight{1.25}\smash{\begin{tabular}[t]{l}$F$\end{tabular}}}}%
    \put(0.13236772,0.50794026){\makebox(0,0)[lt]{\lineheight{1.25}\smash{\begin{tabular}[t]{l}$G$\end{tabular}}}}%
    \put(0.39667213,0.50671507){\makebox(0,0)[lt]{\lineheight{1.25}\smash{\begin{tabular}[t]{l}$H$\end{tabular}}}}%
    \put(0.26239353,0.54112973){\makebox(0,0)[lt]{\lineheight{1.25}\smash{\begin{tabular}[t]{l}$\nu$\end{tabular}}}}%
    \put(0.22402412,0.42085387){\makebox(0,0)[lt]{\lineheight{1.25}\smash{\begin{tabular}[t]{l}{\color{gray}$\mad$}\end{tabular}}}}%
    \put(0,0){\includegraphics[width=\unitlength,page=2]{3int2.pdf}}%
    \put(0.6196044,0.43506938){\makebox(0,0)[lt]{\lineheight{1.25}\smash{\begin{tabular}[t]{l}$c$\end{tabular}}}}%
    \put(0.45250704,0.27225887){\makebox(0,0)[lt]{\lineheight{1.25}\smash{\begin{tabular}[t]{l}{\color{gray}$d$}\end{tabular}}}}%
    \put(0.38233051,0.18284227){\makebox(0,0)[lt]{\lineheight{1.25}\smash{\begin{tabular}[t]{l}{\color{blue}$m$}\end{tabular}}}}%
    \put(0.39582247,0.33694946){\makebox(0,0)[lt]{\lineheight{1.25}\smash{\begin{tabular}[t]{l}{\color{violet}$p$}\end{tabular}}}}%
    \put(0.89223329,0.17439563){\makebox(0,0)[lt]{\lineheight{1.25}\smash{\begin{tabular}[t]{l}{\color{red}$n$}\end{tabular}}}}%
    \put(0.88998269,0.33773262){\makebox(0,0)[lt]{\lineheight{1.25}\smash{\begin{tabular}[t]{l}{\color{cyan}$q$}\end{tabular}}}}%
    \put(0.66941387,0.41360684){\makebox(0,0)[lt]{\lineheight{1.25}\smash{\begin{tabular}[t]{l}$F$\end{tabular}}}}%
    \put(0.66900892,0.09207451){\makebox(0,0)[lt]{\lineheight{1.25}\smash{\begin{tabular}[t]{l}$K$\end{tabular}}}}%
    \put(0.79940303,0.21963496){\makebox(0,0)[lt]{\lineheight{1.25}\smash{\begin{tabular}[t]{l}$G$\end{tabular}}}}%
    \put(0.49177418,0.22211885){\makebox(0,0)[lt]{\lineheight{1.25}\smash{\begin{tabular}[t]{l}$H$\end{tabular}}}}%
    \put(0.54640629,0.36137591){\makebox(0,0)[lt]{\lineheight{1.25}\smash{\begin{tabular}[t]{l}{\color{violet}$\map$}\end{tabular}}}}%
    \put(0.73867587,0.36080907){\makebox(0,0)[lt]{\lineheight{1.25}\smash{\begin{tabular}[t]{l}{\color{cyan}$\mathcal{Q}$}\end{tabular}}}}%
    \put(0.5431667,0.13825043){\makebox(0,0)[lt]{\lineheight{1.25}\smash{\begin{tabular}[t]{l}{\color{blue}$\mam$}\end{tabular}}}}%
    \put(0.73697485,0.13703559){\makebox(0,0)[lt]{\lineheight{1.25}\smash{\begin{tabular}[t]{l}{\color{red}$\man$}\end{tabular}}}}%
    \put(0,0){\includegraphics[width=\unitlength,page=3]{3int2.pdf}}%
    \put(0.59449547,0.07559497){\makebox(0,0)[lt]{\lineheight{1.25}\smash{\begin{tabular}[t]{l}$\mac$\end{tabular}}}}%
    \put(0.47329568,0.16618291){\makebox(0,0)[lt]{\lineheight{1.25}\smash{\begin{tabular}[t]{l}{\color{gray}$\mad$}\end{tabular}}}}%
    \put(0.68575901,0.23377852){\makebox(0,0)[lt]{\lineheight{1.25}\smash{\begin{tabular}[t]{l}$\nu$\end{tabular}}}}%
    \put(0.63715242,0.54485687){\makebox(0,0)[lt]{\lineheight{1.25}\smash{\begin{tabular}[t]{l}(b)\end{tabular}}}}%
    \put(0.23423233,0.16433551){\makebox(0,0)[lt]{\lineheight{1.25}\smash{\begin{tabular}[t]{l}(a)\end{tabular}}}}%
  \end{picture}%
\endgroup%

%% file: drawduals3.pdf_tex
%% Creator: Inkscape 1.0.2 (e86c8708, 2021-01-15), www.inkscape.org
%% Accompanies image file 'drawduals3.pdf' (pdf, eps, ps)
\begingroup%
  \makeatletter%
  \providecommand\color[2][]{%
    \errmessage{(Inkscape) Color is used for the text in Inkscape, but the package 'color.sty' is not loaded}%
    \renewcommand\color[2][]{}%
  }%
  \providecommand\transparent[1]{%
    \errmessage{(Inkscape) Transparency is used (non-zero) for the text in Inkscape, but the package 'transparent.sty' is not loaded}%
    \renewcommand\transparent[1]{}%
  }%
  \providecommand\rotatebox[2]{#2}%
  \newcommand*\fsize{\dimexpr\f@size pt\relax}%
  \newcommand*\lineheight[1]{\fontsize{\fsize}{#1\fsize}\selectfont}%
  \ifx\svgwidth\undefined%
    \setlength{\unitlength}{841.88976378bp}%
    \ifx\svgscale\undefined%
      \relax%
    \else%
      \setlength{\unitlength}{\unitlength * \real{\svgscale}}%
    \fi%
  \else%
    \setlength{\unitlength}{\svgwidth}%
  \fi%
  \global\let\svgwidth\undefined%
  \global\let\svgscale\undefined%
  \makeatother%
  \begin{picture}(1,0.70707071)%
    \lineheight{1}%
    \setlength\tabcolsep{0pt}%
    \put(0,0){\includegraphics[width=\unitlength,page=1]{drawduals3.pdf}}%
    \put(0.21688627,0.56190663){\makebox(0,0)[lt]{\lineheight{1.25}\smash{\begin{tabular}[t]{l}{\color{blue}$\mam$}\end{tabular}}}}%
    \put(0.23094249,0.48841835){\makebox(0,0)[lt]{\lineheight{1.25}\smash{\begin{tabular}[t]{l}{\color{red}$\man$}\end{tabular}}}}%
    \put(0.32578021,0.56077513){\makebox(0,0)[lt]{\lineheight{1.25}\smash{\begin{tabular}[t]{l}{\color{violet}$\map$}\end{tabular}}}}%
    \put(0.3935224,0.41797967){\makebox(0,0)[lt]{\lineheight{1.25}\smash{\begin{tabular}[t]{l}{\color{violet}$p$}\end{tabular}}}}%
    \put(0.24140517,0.37340037){\makebox(0,0)[lt]{\lineheight{1.25}\smash{\begin{tabular}[t]{l}{\color{red}$n$}\end{tabular}}}}%
    \put(0.11091949,0.41031701){\makebox(0,0)[lt]{\lineheight{1.25}\smash{\begin{tabular}[t]{l}{\color{blue}$m$}\end{tabular}}}}%
    \put(0.18821213,0.62824051){\makebox(0,0)[lt]{\lineheight{1.25}\smash{\begin{tabular}[t]{l}$\mac$\end{tabular}}}}%
    \put(0.29492909,0.69088385){\makebox(0,0)[lt]{\lineheight{1.25}\smash{\begin{tabular}[t]{l}$i$\end{tabular}}}}%
    \put(0.12372914,0.62821736){\makebox(0,0)[lt]{\lineheight{1.25}\smash{\begin{tabular}[t]{l}$j$\end{tabular}}}}%
    \put(0.43900836,0.62783462){\makebox(0,0)[lt]{\lineheight{1.25}\smash{\begin{tabular}[t]{l}$k$\end{tabular}}}}%
    \put(0.28232983,0.59786436){\makebox(0,0)[lt]{\lineheight{1.25}\smash{\begin{tabular}[t]{l}$F$\end{tabular}}}}%
    \put(0.13236772,0.50794026){\makebox(0,0)[lt]{\lineheight{1.25}\smash{\begin{tabular}[t]{l}$G$\end{tabular}}}}%
    \put(0.39667213,0.50671507){\makebox(0,0)[lt]{\lineheight{1.25}\smash{\begin{tabular}[t]{l}$H$\end{tabular}}}}%
    \put(0.26239353,0.54112973){\makebox(0,0)[lt]{\lineheight{1.25}\smash{\begin{tabular}[t]{l}$\nu$\end{tabular}}}}%
    \put(0.17591801,0.42085387){\makebox(0,0)[lt]{\lineheight{1.25}\smash{\begin{tabular}[t]{l}{\color{gray}$\mad$}\end{tabular}}}}%
    \put(0,0){\includegraphics[width=\unitlength,page=2]{drawduals3.pdf}}%
    \put(0.34096666,0.63573841){\makebox(0,0)[lt]{\lineheight{1.25}\smash{\begin{tabular}[t]{l}$\alpha$\end{tabular}}}}%
    \put(0,0){\includegraphics[width=\unitlength,page=3]{drawduals3.pdf}}%
    \put(0.6196044,0.43506938){\makebox(0,0)[lt]{\lineheight{1.25}\smash{\begin{tabular}[t]{l}$c$\end{tabular}}}}%
    \put(0.45250704,0.27225887){\makebox(0,0)[lt]{\lineheight{1.25}\smash{\begin{tabular}[t]{l}{\color{gray}$d$}\end{tabular}}}}%
    \put(0.38233051,0.18284227){\makebox(0,0)[lt]{\lineheight{1.25}\smash{\begin{tabular}[t]{l}{\color{blue}$m$}\end{tabular}}}}%
    \put(0.39582247,0.33694946){\makebox(0,0)[lt]{\lineheight{1.25}\smash{\begin{tabular}[t]{l}{\color{violet}$p$}\end{tabular}}}}%
    \put(0.89223329,0.17439563){\makebox(0,0)[lt]{\lineheight{1.25}\smash{\begin{tabular}[t]{l}{\color{red}$n$}\end{tabular}}}}%
    \put(0.88998269,0.33773262){\makebox(0,0)[lt]{\lineheight{1.25}\smash{\begin{tabular}[t]{l}{\color{cyan}$q$}\end{tabular}}}}%
    \put(0.66941387,0.41360684){\makebox(0,0)[lt]{\lineheight{1.25}\smash{\begin{tabular}[t]{l}$F$\end{tabular}}}}%
    \put(0.66900892,0.09207451){\makebox(0,0)[lt]{\lineheight{1.25}\smash{\begin{tabular}[t]{l}$K$\end{tabular}}}}%
    \put(0.79940303,0.21963496){\makebox(0,0)[lt]{\lineheight{1.25}\smash{\begin{tabular}[t]{l}$G$\end{tabular}}}}%
    \put(0.49177418,0.22211885){\makebox(0,0)[lt]{\lineheight{1.25}\smash{\begin{tabular}[t]{l}$H$\end{tabular}}}}%
    \put(0.54640629,0.36137591){\makebox(0,0)[lt]{\lineheight{1.25}\smash{\begin{tabular}[t]{l}{\color{violet}$\map$}\end{tabular}}}}%
    \put(0.73867587,0.36080907){\makebox(0,0)[lt]{\lineheight{1.25}\smash{\begin{tabular}[t]{l}{\color{cyan}$\mathcal{Q}$}\end{tabular}}}}%
    \put(0.5431667,0.13825043){\makebox(0,0)[lt]{\lineheight{1.25}\smash{\begin{tabular}[t]{l}{\color{blue}$\mam$}\end{tabular}}}}%
    \put(0.73697485,0.13703559){\makebox(0,0)[lt]{\lineheight{1.25}\smash{\begin{tabular}[t]{l}{\color{red}$\man$}\end{tabular}}}}%
    \put(0,0){\includegraphics[width=\unitlength,page=4]{drawduals3.pdf}}%
    \put(0.59449547,0.07559497){\makebox(0,0)[lt]{\lineheight{1.25}\smash{\begin{tabular}[t]{l}$\mac$\end{tabular}}}}%
    \put(0.47329568,0.16618291){\makebox(0,0)[lt]{\lineheight{1.25}\smash{\begin{tabular}[t]{l}{\color{gray}$\mad$}\end{tabular}}}}%
    \put(0,0){\includegraphics[width=\unitlength,page=5]{drawduals3.pdf}}%
    \put(0.68575901,0.23377852){\makebox(0,0)[lt]{\lineheight{1.25}\smash{\begin{tabular}[t]{l}$\nu$\end{tabular}}}}%
    \put(0.63760237,0.54935635){\makebox(0,0)[lt]{\lineheight{1.25}\smash{\begin{tabular}[t]{l}(b)\end{tabular}}}}%
    \put(0.23894783,0.17184941){\makebox(0,0)[lt]{\lineheight{1.25}\smash{\begin{tabular}[t]{l}(a)\end{tabular}}}}%
  \end{picture}%
\endgroup%

%% file: bistellar.pdf_tex
%% Creator: Inkscape inkscape 0.91, www.inkscape.org
%% Accompanies image file 'bistellar.pdf' (pdf, eps, ps)
\begingroup%
  \makeatletter%
  \providecommand\color[2][]{%
    \errmessage{(Inkscape) Color is used for the text in Inkscape, but the package 'color.sty' is not loaded}%
    \renewcommand\color[2][]{}%
  }%
  \providecommand\transparent[1]{%
    \errmessage{(Inkscape) Transparency is used (non-zero) for the text in Inkscape, but the package 'transparent.sty' is not loaded}%
    \renewcommand\transparent[1]{}%
  }%
  \providecommand\rotatebox[2]{#2}%
  \ifx\svgwidth\undefined%
    \setlength{\unitlength}{841.88976378bp}%
    \ifx\svgscale\undefined%
      \relax%
    \else%
      \setlength{\unitlength}{\unitlength * \real{\svgscale}}%
    \fi%
  \else%
    \setlength{\unitlength}{\svgwidth}%
  \fi%
  \global\let\svgwidth\undefined%
  \global\let\svgscale\undefined%
  \makeatother%
  \begin{picture}(1,0.70707071)%
    \put(0,0){\includegraphics[width=\unitlength,page=1]{bistellar.pdf}}%
    \put(0.24977821,0.34480252){\color[rgb]{0,0,0}\makebox(0,0)[lb]{\smash{1-4}}}%
    \put(0.8090642,0.34751751){\color[rgb]{0,0,0}\makebox(0,0)[lb]{\smash{2-3}}}%
  \end{picture}%
\endgroup%

%% file: stellar.pdf_tex
%% Creator: Inkscape 1.0.2 (e86c8708, 2021-01-15), www.inkscape.org
%% Accompanies image file 'stellar.pdf' (pdf, eps, ps)
\begingroup%
  \makeatletter%
  \providecommand\color[2][]{%
    \errmessage{(Inkscape) Color is used for the text in Inkscape, but the package 'color.sty' is not loaded}%
    \renewcommand\color[2][]{}%
  }%
  \providecommand\transparent[1]{%
    \errmessage{(Inkscape) Transparency is used (non-zero) for the text in Inkscape, but the package 'transparent.sty' is not loaded}%
    \renewcommand\transparent[1]{}%
  }%
  \providecommand\rotatebox[2]{#2}%
  \newcommand*\fsize{\dimexpr\f@size pt\relax}%
  \newcommand*\lineheight[1]{\fontsize{\fsize}{#1\fsize}\selectfont}%
  \ifx\svgwidth\undefined%
    \setlength{\unitlength}{841.88976378bp}%
    \ifx\svgscale\undefined%
      \relax%
    \else%
      \setlength{\unitlength}{\unitlength * \real{\svgscale}}%
    \fi%
  \else%
    \setlength{\unitlength}{\svgwidth}%
  \fi%
  \global\let\svgwidth\undefined%
  \global\let\svgscale\undefined%
  \makeatother%
  \begin{picture}(1,0.70707071)%
    \lineheight{1}%
    \setlength\tabcolsep{0pt}%
    \put(0,0){\includegraphics[width=\unitlength,page=1]{stellar.pdf}}%
  \end{picture}%
\endgroup%

%% file: shellingdef.pdf_tex
%% Creator: Inkscape 1.1.2 (b8e25be8, 2022-02-05), www.inkscape.org
%% Accompanies image file 'shellingdef.pdf' (pdf, eps, ps)
\begingroup%
  \makeatletter%
  \providecommand\color[2][]{%
    \errmessage{(Inkscape) Color is used for the text in Inkscape, but the package 'color.sty' is not loaded}%
    \renewcommand\color[2][]{}%
  }%
  \providecommand\transparent[1]{%
    \errmessage{(Inkscape) Transparency is used (non-zero) for the text in Inkscape, but the package 'transparent.sty' is not loaded}%
    \renewcommand\transparent[1]{}%
  }%
  \providecommand\rotatebox[2]{#2}%
  \newcommand*\fsize{\dimexpr\f@size pt\relax}%
  \newcommand*\lineheight[1]{\fontsize{\fsize}{#1\fsize}\selectfont}%
  \ifx\svgwidth\undefined%
    \setlength{\unitlength}{841.88976378bp}%
    \ifx\svgscale\undefined%
      \relax%
    \else%
      \setlength{\unitlength}{\unitlength * \real{\svgscale}}%
    \fi%
  \else%
    \setlength{\unitlength}{\svgwidth}%
  \fi%
  \global\let\svgwidth\undefined%
  \global\let\svgscale\undefined%
  \makeatother%
  \begin{picture}(1,0.70707071)%
    \lineheight{1}%
    \setlength\tabcolsep{0pt}%
    \put(0,0){\includegraphics[width=\unitlength,page=1]{shellingdef.pdf}}%
  \end{picture}%
\endgroup%

%% file: barycentric3.pdf_tex
%% Creator: Inkscape 1.0.2 (e86c8708, 2021-01-15), www.inkscape.org
%% Accompanies image file 'barycentric3.pdf' (pdf, eps, ps)
\begingroup%
  \makeatletter%
  \providecommand\color[2][]{%
    \errmessage{(Inkscape) Color is used for the text in Inkscape, but the package 'color.sty' is not loaded}%
    \renewcommand\color[2][]{}%
  }%
  \providecommand\transparent[1]{%
    \errmessage{(Inkscape) Transparency is used (non-zero) for the text in Inkscape, but the package 'transparent.sty' is not loaded}%
    \renewcommand\transparent[1]{}%
  }%
  \providecommand\rotatebox[2]{#2}%
  \newcommand*\fsize{\dimexpr\f@size pt\relax}%
  \newcommand*\lineheight[1]{\fontsize{\fsize}{#1\fsize}\selectfont}%
  \ifx\svgwidth\undefined%
    \setlength{\unitlength}{841.88976378bp}%
    \ifx\svgscale\undefined%
      \relax%
    \else%
      \setlength{\unitlength}{\unitlength * \real{\svgscale}}%
    \fi%
  \else%
    \setlength{\unitlength}{\svgwidth}%
  \fi%
  \global\let\svgwidth\undefined%
  \global\let\svgscale\undefined%
  \makeatother%
  \begin{picture}(1,0.70707071)%
    \lineheight{1}%
    \setlength\tabcolsep{0pt}%
    \put(0,0){\includegraphics[width=\unitlength,page=1]{barycentric3.pdf}}%
    \put(0.77979086,0.38235433){\makebox(0,0)[lt]{\lineheight{1.25}\smash{\begin{tabular}[t]{l}{\color{blue}$D$}\end{tabular}}}}%
    \put(0.71181095,0.41963738){\makebox(0,0)[lt]{\lineheight{1.25}\smash{\begin{tabular}[t]{l}{\color{black}$E_f$}\end{tabular}}}}%
    \put(0.54162048,0.40194951){\makebox(0,0)[lt]{\lineheight{1.25}\smash{\begin{tabular}[t]{l}{\color{black}$e_f$}\end{tabular}}}}%
    \put(0,0){\includegraphics[width=\unitlength,page=2]{barycentric3.pdf}}%
    \put(0.84049628,0.35034529){\makebox(0,0)[lt]{\lineheight{1.25}\smash{\begin{tabular}[t]{l}{\color{gray}$E_i$}\end{tabular}}}}%
    \put(0.45732281,0.29646479){\makebox(0,0)[lt]{\lineheight{1.25}\smash{\begin{tabular}[t]{l}{\color{gray}$e_i$}\end{tabular}}}}%
    \put(0,0){\includegraphics[width=\unitlength,page=3]{barycentric3.pdf}}%
    \put(0.42726762,0.55988859){\makebox(0,0)[lt]{\lineheight{1.25}\smash{\begin{tabular}[t]{l}{\color{blue}$D$}\end{tabular}}}}%
    \put(0.38537698,0.58953576){\makebox(0,0)[lt]{\lineheight{1.25}\smash{\begin{tabular}[t]{l}{\color{black}$E_f$}\end{tabular}}}}%
    \put(0.47270126,0.53487913){\makebox(0,0)[lt]{\lineheight{1.25}\smash{\begin{tabular}[t]{l}{\color{gray}$E_i$}\end{tabular}}}}%
    \put(0.10815687,0.53366849){\makebox(0,0)[lt]{\lineheight{1.25}\smash{\begin{tabular}[t]{l}{\color{black}$e_f$}\end{tabular}}}}%
    \put(0.05682086,0.46127258){\makebox(0,0)[lt]{\lineheight{1.25}\smash{\begin{tabular}[t]{l}{\color{gray}$e_i$}\end{tabular}}}}%
    \put(0.24431263,0.51316611){\makebox(0,0)[lt]{\lineheight{1.25}\smash{\begin{tabular}[t]{l}$t_f$\end{tabular}}}}%
    \put(0.33683407,0.51252968){\makebox(0,0)[lt]{\lineheight{1.25}\smash{\begin{tabular}[t]{l}$\Delta_f$\end{tabular}}}}%
    \put(0.6389604,0.31514213){\makebox(0,0)[lt]{\lineheight{1.25}\smash{\begin{tabular}[t]{l}$t_f$\end{tabular}}}}%
    \put(0.61592528,0.25596423){\makebox(0,0)[lt]{\lineheight{1.25}\smash{\begin{tabular}[t]{l}$\Delta_f$\end{tabular}}}}%
  \end{picture}%
\endgroup%

%% file: Pachner2.pdf_tex
%% Creator: Inkscape 1.0.2 (e86c8708, 2021-01-15), www.inkscape.org
%% Accompanies image file 'Pachner2.pdf' (pdf, eps, ps)
\begingroup%
  \makeatletter%
  \providecommand\color[2][]{%
    \errmessage{(Inkscape) Color is used for the text in Inkscape, but the package 'color.sty' is not loaded}%
    \renewcommand\color[2][]{}%
  }%
  \providecommand\transparent[1]{%
    \errmessage{(Inkscape) Transparency is used (non-zero) for the text in Inkscape, but the package 'transparent.sty' is not loaded}%
    \renewcommand\transparent[1]{}%
  }%
  \providecommand\rotatebox[2]{#2}%
  \newcommand*\fsize{\dimexpr\f@size pt\relax}%
  \newcommand*\lineheight[1]{\fontsize{\fsize}{#1\fsize}\selectfont}%
  \ifx\svgwidth\undefined%
    \setlength{\unitlength}{595.27559055bp}%
    \ifx\svgscale\undefined%
      \relax%
    \else%
      \setlength{\unitlength}{\unitlength * \real{\svgscale}}%
    \fi%
  \else%
    \setlength{\unitlength}{\svgwidth}%
  \fi%
  \global\let\svgwidth\undefined%
  \global\let\svgscale\undefined%
  \makeatother%
  \begin{picture}(1,1.41428571)%
    \lineheight{1}%
    \setlength\tabcolsep{0pt}%
    \put(0,0){\includegraphics[width=\unitlength,page=1]{Pachner2.pdf}}%
    \put(0.01911499,1.06599382){\color[rgb]{0,0,0}\makebox(0,0)[lt]{\lineheight{0}\smash{\begin{tabular}[t]{l}(c)\end{tabular}}}}%
    \put(0.01102421,0.7823142){\color[rgb]{0,0,0}\makebox(0,0)[lt]{\lineheight{0}\smash{\begin{tabular}[t]{l}(e)\end{tabular}}}}%
    \put(0.00718493,0.48669802){\color[rgb]{0,0,0}\makebox(0,0)[lt]{\lineheight{0}\smash{\begin{tabular}[t]{l}(g)\end{tabular}}}}%
    \put(0.0125813,0.21162446){\color[rgb]{0,0,0}\makebox(0,0)[lt]{\lineheight{0}\smash{\begin{tabular}[t]{l}(h)\end{tabular}}}}%
    \put(0.53893322,0.22035564){\color[rgb]{0,0,0}\makebox(0,0)[lt]{\lineheight{0}\smash{\begin{tabular}[t]{l}(i)\end{tabular}}}}%
    \put(0.53537817,0.77824156){\color[rgb]{0,0,0}\makebox(0,0)[lt]{\lineheight{0}\smash{\begin{tabular}[t]{l}(f)\end{tabular}}}}%
    \put(0.54449053,1.06366928){\color[rgb]{0,0,0}\makebox(0,0)[lt]{\lineheight{0}\smash{\begin{tabular}[t]{l}(d)\end{tabular}}}}%
    \put(0,0){\includegraphics[width=\unitlength,page=2]{Pachner2.pdf}}%
    \put(0.02631453,1.35001595){\color[rgb]{0,0,0}\makebox(0,0)[lt]{\lineheight{0}\smash{\begin{tabular}[t]{l}(a)\end{tabular}}}}%
    \put(0,0){\includegraphics[width=\unitlength,page=3]{Pachner2.pdf}}%
    \put(0.53873091,1.34913129){\color[rgb]{0,0,0}\makebox(0,0)[lt]{\lineheight{0}\smash{\begin{tabular}[t]{l}(b)\end{tabular}}}}%
    \put(0,0){\includegraphics[width=\unitlength,page=4]{Pachner2.pdf}}%
  \end{picture}%
\endgroup%

%% file: bubblenolabel1.pdf_tex
%% Creator: Inkscape 1.1.2 (b8e25be8, 2022-02-05), www.inkscape.org
%% Accompanies image file 'bubblenolabel1.pdf' (pdf, eps, ps)
\begingroup%
  \makeatletter%
  \providecommand\color[2][]{%
    \errmessage{(Inkscape) Color is used for the text in Inkscape, but the package 'color.sty' is not loaded}%
    \renewcommand\color[2][]{}%
  }%
  \providecommand\transparent[1]{%
    \errmessage{(Inkscape) Transparency is used (non-zero) for the text in Inkscape, but the package 'transparent.sty' is not loaded}%
    \renewcommand\transparent[1]{}%
  }%
  \providecommand\rotatebox[2]{#2}%
  \newcommand*\fsize{\dimexpr\f@size pt\relax}%
  \newcommand*\lineheight[1]{\fontsize{\fsize}{#1\fsize}\selectfont}%
  \ifx\svgwidth\undefined%
    \setlength{\unitlength}{841.88976378bp}%
    \ifx\svgscale\undefined%
      \relax%
    \else%
      \setlength{\unitlength}{\unitlength * \real{\svgscale}}%
    \fi%
  \else%
    \setlength{\unitlength}{\svgwidth}%
  \fi%
  \global\let\svgwidth\undefined%
  \global\let\svgscale\undefined%
  \makeatother%
  \begin{picture}(1,0.70707071)%
    \lineheight{1}%
    \setlength\tabcolsep{0pt}%
    \put(0,0){\includegraphics[width=\unitlength,page=1]{bubblenolabel1.pdf}}%
    \put(0.26367634,0.50648594){\color[rgb]{0,0,0}\makebox(0,0)[lt]{\lineheight{1.25}\smash{\begin{tabular}[t]{l}$j$\end{tabular}}}}%
    \put(0.59486256,0.23648755){\color[rgb]{0,0,0}\makebox(0,0)[lt]{\lineheight{1.25}\smash{\begin{tabular}[t]{l}{\color{blue}$\mam$}\end{tabular}}}}%
    \put(0,0){\includegraphics[width=\unitlength,page=2]{bubblenolabel1.pdf}}%
    \put(0.29108658,0.15472135){\color[rgb]{0,0,0}\makebox(0,0)[lt]{\lineheight{1.25}\smash{\begin{tabular}[t]{l}$i$\end{tabular}}}}%
    \put(0.76638249,0.14320713){\color[rgb]{0,0,0}\makebox(0,0)[lt]{\lineheight{1.25}\smash{\begin{tabular}[t]{l}$k$\end{tabular}}}}%
    \put(0.64125719,0.50877013){\color[rgb]{0,0,0}\makebox(0,0)[lt]{\lineheight{1.25}\smash{\begin{tabular}[t]{l}$l$\end{tabular}}}}%
    \put(0.5232638,0.42722071){\color[rgb]{0,0,0}\makebox(0,0)[lt]{\lineheight{1.25}\smash{\begin{tabular}[t]{l}$h$\end{tabular}}}}%
    \put(0.26005202,0.25307339){\color[rgb]{0,0,0}\makebox(0,0)[lt]{\lineheight{1.25}\smash{\begin{tabular}[t]{l}$g$\end{tabular}}}}%
    \put(0.39230115,0.56379015){\color[rgb]{0,0,0}\makebox(0,0)[lt]{\lineheight{1.25}\smash{\begin{tabular}[t]{l}{\color{blue}$q$}\end{tabular}}}}%
    \put(0.48652017,0.17697801){\color[rgb]{0,0,0}\makebox(0,0)[lt]{\lineheight{1.25}\smash{\begin{tabular}[t]{l}{\color{blue}$m$}\end{tabular}}}}%
    \put(0.57197427,0.38526146){\color[rgb]{0,0,0}\makebox(0,0)[lt]{\lineheight{1.25}\smash{\begin{tabular}[t]{l}{\color{blue}$n$}\end{tabular}}}}%
    \put(0.34652456,0.39670563){\color[rgb]{0,0,0}\makebox(0,0)[lt]{\lineheight{1.25}\smash{\begin{tabular}[t]{l}{\color{blue}$p$}\end{tabular}}}}%
  \end{picture}%
\endgroup%

%% file: polygonsurface2.pdf_tex
%% Creator: Inkscape 1.0.2 (e86c8708, 2021-01-15), www.inkscape.org
%% Accompanies image file 'polygonsurface2.pdf' (pdf, eps, ps)
\begingroup%
  \makeatletter%
  \providecommand\color[2][]{%
    \errmessage{(Inkscape) Color is used for the text in Inkscape, but the package 'color.sty' is not loaded}%
    \renewcommand\color[2][]{}%
  }%
  \providecommand\transparent[1]{%
    \errmessage{(Inkscape) Transparency is used (non-zero) for the text in Inkscape, but the package 'transparent.sty' is not loaded}%
    \renewcommand\transparent[1]{}%
  }%
  \providecommand\rotatebox[2]{#2}%
  \newcommand*\fsize{\dimexpr\f@size pt\relax}%
  \newcommand*\lineheight[1]{\fontsize{\fsize}{#1\fsize}\selectfont}%
  \ifx\svgwidth\undefined%
    \setlength{\unitlength}{841.88976378bp}%
    \ifx\svgscale\undefined%
      \relax%
    \else%
      \setlength{\unitlength}{\unitlength * \real{\svgscale}}%
    \fi%
  \else%
    \setlength{\unitlength}{\svgwidth}%
  \fi%
  \global\let\svgwidth\undefined%
  \global\let\svgscale\undefined%
  \makeatother%
  \begin{picture}(1,0.70707071)%
    \lineheight{1}%
    \setlength\tabcolsep{0pt}%
    \put(0,0){\includegraphics[width=\unitlength,page=1]{polygonsurface2.pdf}}%
    \put(0.42570046,0.01335077){\makebox(0,0)[lt]{\lineheight{1.25}\smash{\begin{tabular}[t]{l}$a_1$\end{tabular}}}}%
    \put(0.68328424,0.07583775){\makebox(0,0)[lt]{\lineheight{1.25}\smash{\begin{tabular}[t]{l}$b_1$\end{tabular}}}}%
    \put(0.82671159,0.4078714){\makebox(0,0)[lt]{\lineheight{1.25}\smash{\begin{tabular}[t]{l}$a_1$\end{tabular}}}}%
    \put(0.68073894,0.63427534){\makebox(0,0)[lt]{\lineheight{1.25}\smash{\begin{tabular}[t]{l}$b_1$\end{tabular}}}}%
    \put(0.53832972,0.68925365){\makebox(0,0)[lt]{\lineheight{1.25}\smash{\begin{tabular}[t]{l}$a_2$\end{tabular}}}}%
    \put(0.25962001,0.62562129){\makebox(0,0)[lt]{\lineheight{1.25}\smash{\begin{tabular}[t]{l}$b_2$\end{tabular}}}}%
    \put(0.1215378,0.28264289){\makebox(0,0)[lt]{\lineheight{1.25}\smash{\begin{tabular}[t]{l}$a_2$\end{tabular}}}}%
    \put(0.25580206,0.08792789){\makebox(0,0)[lt]{\lineheight{1.25}\smash{\begin{tabular}[t]{l}$b_2$\end{tabular}}}}%
    \put(0.17626162,0.16174142){\makebox(0,0)[lt]{\lineheight{1.25}\smash{\begin{tabular}[t]{l}{\color{gray}$b'_2$}\end{tabular}}}}%
    \put(0.11771985,0.34945687){\makebox(0,0)[lt]{\lineheight{1.25}\smash{\begin{tabular}[t]{l}{\color{gray}$a'_2$}\end{tabular}}}}%
    \put(0.41411936,0.69001722){\makebox(0,0)[lt]{\lineheight{1.25}\smash{\begin{tabular}[t]{l}{\color{gray}$a'_2$}\end{tabular}}}}%
    \put(0.17944324,0.56007998){\makebox(0,0)[lt]{\lineheight{1.25}\smash{\begin{tabular}[t]{l}{\color{gray}$b'_2$}\end{tabular}}}}%
    \put(0.7324084,0.5818422){\makebox(0,0)[lt]{\lineheight{1.25}\smash{\begin{tabular}[t]{l}{\color{gray}$b'_1$}\end{tabular}}}}%
    \put(0.73139029,0.12521642){\makebox(0,0)[lt]{\lineheight{1.25}\smash{\begin{tabular}[t]{l}{\color{gray}$b'_1$}\end{tabular}}}}%
    \put(0.54214766,0.01118724){\makebox(0,0)[lt]{\lineheight{1.25}\smash{\begin{tabular}[t]{l}{\color{gray}$a'_1$}\end{tabular}}}}%
    \put(0.82276634,0.30440513){\makebox(0,0)[lt]{\lineheight{1.25}\smash{\begin{tabular}[t]{l}{\color{gray}$a'_1$}\end{tabular}}}}%
    \put(0.61214325,0.01411433){\makebox(0,0)[lt]{\lineheight{1.25}\smash{\begin{tabular}[t]{l}{\color{blue}$m$}\end{tabular}}}}%
    \put(0.83040223,0.2241011){\makebox(0,0)[lt]{\lineheight{1.25}\smash{\begin{tabular}[t]{l}{\color{blue}$m$}\end{tabular}}}}%
    \put(0.82594795,0.49810197){\makebox(0,0)[lt]{\lineheight{1.25}\smash{\begin{tabular}[t]{l}{\color{blue}$m$}\end{tabular}}}}%
    \put(0.61150692,0.68988991){\makebox(0,0)[lt]{\lineheight{1.25}\smash{\begin{tabular}[t]{l}{\color{blue}$m$}\end{tabular}}}}%
    \put(0.32707029,0.69052624){\makebox(0,0)[lt]{\lineheight{1.25}\smash{\begin{tabular}[t]{l}{\color{blue}$m$}\end{tabular}}}}%
    \put(0.11262925,0.48690271){\makebox(0,0)[lt]{\lineheight{1.25}\smash{\begin{tabular}[t]{l}{\color{blue}$m$}\end{tabular}}}}%
    \put(0.11670172,0.22155579){\makebox(0,0)[lt]{\lineheight{1.25}\smash{\begin{tabular}[t]{l}{\color{blue}$m$}\end{tabular}}}}%
    \put(0.33470617,0.01309619){\makebox(0,0)[lt]{\lineheight{1.25}\smash{\begin{tabular}[t]{l}{\color{blue}$m$}\end{tabular}}}}%
    \put(0,0){\includegraphics[width=\unitlength,page=2]{polygonsurface2.pdf}}%
  \end{picture}%
\endgroup%

%% file: handle.pdf_tex
%% Creator: Inkscape 1.0.2 (e86c8708, 2021-01-15), www.inkscape.org
%% Accompanies image file 'handle.pdf' (pdf, eps, ps)
\begingroup%
  \makeatletter%
  \providecommand\color[2][]{%
    \errmessage{(Inkscape) Color is used for the text in Inkscape, but the package 'color.sty' is not loaded}%
    \renewcommand\color[2][]{}%
  }%
  \providecommand\transparent[1]{%
    \errmessage{(Inkscape) Transparency is used (non-zero) for the text in Inkscape, but the package 'transparent.sty' is not loaded}%
    \renewcommand\transparent[1]{}%
  }%
  \providecommand\rotatebox[2]{#2}%
  \newcommand*\fsize{\dimexpr\f@size pt\relax}%
  \newcommand*\lineheight[1]{\fontsize{\fsize}{#1\fsize}\selectfont}%
  \ifx\svgwidth\undefined%
    \setlength{\unitlength}{595.27559055bp}%
    \ifx\svgscale\undefined%
      \relax%
    \else%
      \setlength{\unitlength}{\unitlength * \real{\svgscale}}%
    \fi%
  \else%
    \setlength{\unitlength}{\svgwidth}%
  \fi%
  \global\let\svgwidth\undefined%
  \global\let\svgscale\undefined%
  \makeatother%
  \begin{picture}(1,1.41428571)%
    \lineheight{1}%
    \setlength\tabcolsep{0pt}%
    \put(0,0){\includegraphics[width=\unitlength,page=1]{handle.pdf}}%
    \put(0.28701212,1.09033439){\makebox(0,0)[lt]{\lineheight{1.25}\smash{\begin{tabular}[t]{l}{\color{blue}$p_1$}\end{tabular}}}}%
    \put(0.21185975,1.02128659){\makebox(0,0)[lt]{\lineheight{1.25}\smash{\begin{tabular}[t]{l}{\color{blue}$n_1$}\end{tabular}}}}%
    \put(0.03890166,1.09126216){\makebox(0,0)[lt]{\lineheight{1.25}\smash{\begin{tabular}[t]{l}{\color{blue}$m_1$}\end{tabular}}}}%
    \put(0.02986854,0.87693892){\makebox(0,0)[lt]{\lineheight{1.25}\smash{\begin{tabular}[t]{l}{\color{blue}$m_2$}\end{tabular}}}}%
    \put(0.21668777,0.81477592){\makebox(0,0)[lt]{\lineheight{1.25}\smash{\begin{tabular}[t]{l}{\color{blue}$n_2$}\end{tabular}}}}%
    \put(0.3177637,0.87810473){\makebox(0,0)[lt]{\lineheight{1.25}\smash{\begin{tabular}[t]{l}{\color{blue}$p_2$}\end{tabular}}}}%
    \put(0,0){\includegraphics[width=\unitlength,page=2]{handle.pdf}}%
    \put(0.30494476,0.92372774){\makebox(0,0)[lt]{\lineheight{1.25}\smash{\begin{tabular}[t]{l}$c_1$\end{tabular}}}}%
    \put(0.28221828,0.96162109){\makebox(0,0)[lt]{\lineheight{1.25}\smash{\begin{tabular}[t]{l}$j$\end{tabular}}}}%
    \put(0.06430224,0.94530007){\makebox(0,0)[lt]{\lineheight{1.25}\smash{\begin{tabular}[t]{l}$i$\end{tabular}}}}%
    \put(0.1256217,1.02341639){\makebox(0,0)[lt]{\lineheight{1.25}\smash{\begin{tabular}[t]{l}{\color{gray}$d_1$}\end{tabular}}}}%
    \put(0,0){\includegraphics[width=\unitlength,page=3]{handle.pdf}}%
    \put(0.63238496,0.88327433){\makebox(0,0)[lt]{\lineheight{1.25}\smash{\begin{tabular}[t]{l}{\color{blue}$p_2$}\end{tabular}}}}%
    \put(0.79036945,0.84052767){\makebox(0,0)[lt]{\lineheight{1.25}\smash{\begin{tabular}[t]{l}{\color{blue}$n_2$}\end{tabular}}}}%
    \put(0.90412765,0.88891049){\makebox(0,0)[lt]{\lineheight{1.25}\smash{\begin{tabular}[t]{l}{\color{blue}$q_2$}\end{tabular}}}}%
    \put(0.90332598,1.05620524){\makebox(0,0)[lt]{\lineheight{1.25}\smash{\begin{tabular}[t]{l}{\color{blue}$q_1$}\end{tabular}}}}%
    \put(0.6333988,1.06868446){\makebox(0,0)[lt]{\lineheight{1.25}\smash{\begin{tabular}[t]{l}{\color{blue}$p_1$}\end{tabular}}}}%
    \put(0.80504826,1.0481821){\makebox(0,0)[lt]{\lineheight{1.25}\smash{\begin{tabular}[t]{l}{\color{blue}$n_1$}\end{tabular}}}}%
    \put(0.67743523,0.94133686){\makebox(0,0)[lt]{\lineheight{1.25}\smash{\begin{tabular}[t]{l}$c_1$\end{tabular}}}}%
    \put(0.89262245,0.93422463){\makebox(0,0)[lt]{\lineheight{1.25}\smash{\begin{tabular}[t]{l}$k$\end{tabular}}}}%
    \put(0.85129075,1.00749443){\makebox(0,0)[lt]{\lineheight{1.25}\smash{\begin{tabular}[t]{l}$l$\end{tabular}}}}%
    \put(0.7308673,1.04018235){\makebox(0,0)[lt]{\lineheight{1.25}\smash{\begin{tabular}[t]{l}{\color{gray}$d'_1$}\end{tabular}}}}%
    \put(0,0){\includegraphics[width=\unitlength,page=4]{handle.pdf}}%
    \put(0.32495695,0.55429208){\makebox(0,0)[lt]{\lineheight{1.25}\smash{\begin{tabular}[t]{l}{\color{blue}$p_2$}\end{tabular}}}}%
    \put(0.22715788,0.49428115){\makebox(0,0)[lt]{\lineheight{1.25}\smash{\begin{tabular}[t]{l}{\color{blue}$n_2$}\end{tabular}}}}%
    \put(0.06175929,0.57433617){\makebox(0,0)[lt]{\lineheight{1.25}\smash{\begin{tabular}[t]{l}{\color{blue}$m_2$}\end{tabular}}}}%
    \put(0.05272618,0.36001289){\makebox(0,0)[lt]{\lineheight{1.25}\smash{\begin{tabular}[t]{l}{\color{blue}$m_3$}\end{tabular}}}}%
    \put(0.23450575,0.29533008){\makebox(0,0)[lt]{\lineheight{1.25}\smash{\begin{tabular}[t]{l}{\color{blue}$n_3$}\end{tabular}}}}%
    \put(0.3355561,0.36117866){\makebox(0,0)[lt]{\lineheight{1.25}\smash{\begin{tabular}[t]{l}{\color{blue}$p_3$}\end{tabular}}}}%
    \put(0,0){\includegraphics[width=\unitlength,page=5]{handle.pdf}}%
    \put(0.32276273,0.40932159){\makebox(0,0)[lt]{\lineheight{1.25}\smash{\begin{tabular}[t]{l}$c_2$\end{tabular}}}}%
    \put(0.30003626,0.4446951){\makebox(0,0)[lt]{\lineheight{1.25}\smash{\begin{tabular}[t]{l}$j$\end{tabular}}}}%
    \put(0.08212017,0.42837408){\makebox(0,0)[lt]{\lineheight{1.25}\smash{\begin{tabular}[t]{l}$i$\end{tabular}}}}%
    \put(0.13336029,0.50649036){\makebox(0,0)[lt]{\lineheight{1.25}\smash{\begin{tabular}[t]{l}{\color{gray}$d_2$}\end{tabular}}}}%
    \put(0,0){\includegraphics[width=\unitlength,page=6]{handle.pdf}}%
    \put(0.64516325,0.36634834){\makebox(0,0)[lt]{\lineheight{1.25}\smash{\begin{tabular}[t]{l}{\color{blue}$p_3$}\end{tabular}}}}%
    \put(0.80818743,0.3236015){\makebox(0,0)[lt]{\lineheight{1.25}\smash{\begin{tabular}[t]{l}{\color{blue}$n_3$}\end{tabular}}}}%
    \put(0.92698524,0.3719845){\makebox(0,0)[lt]{\lineheight{1.25}\smash{\begin{tabular}[t]{l}{\color{blue}$q_3$}\end{tabular}}}}%
    \put(0.92366377,0.54431886){\makebox(0,0)[lt]{\lineheight{1.25}\smash{\begin{tabular}[t]{l}{\color{blue}$q_2$}\end{tabular}}}}%
    \put(0.65121656,0.55175847){\makebox(0,0)[lt]{\lineheight{1.25}\smash{\begin{tabular}[t]{l}{\color{blue}$p_2$}\end{tabular}}}}%
    \put(0.81782663,0.54385514){\makebox(0,0)[lt]{\lineheight{1.25}\smash{\begin{tabular}[t]{l}{\color{blue}$n_2$}\end{tabular}}}}%
    \put(0.69525321,0.42189106){\makebox(0,0)[lt]{\lineheight{1.25}\smash{\begin{tabular}[t]{l}$c_2$\end{tabular}}}}%
    \put(0.91044043,0.4172986){\makebox(0,0)[lt]{\lineheight{1.25}\smash{\begin{tabular}[t]{l}$k$\end{tabular}}}}%
    \put(0.87414834,0.4905684){\makebox(0,0)[lt]{\lineheight{1.25}\smash{\begin{tabular}[t]{l}$l$\end{tabular}}}}%
    \put(0.74868527,0.52325636){\makebox(0,0)[lt]{\lineheight{1.25}\smash{\begin{tabular}[t]{l}{\color{gray}$d'_2$}\end{tabular}}}}%
    \put(0,0){\includegraphics[width=\unitlength,page=7]{handle.pdf}}%
    \put(0.56482965,0.71978651){\makebox(0,0)[lt]{\lineheight{1.25}\smash{\begin{tabular}[t]{l}$c_1$\end{tabular}}}}%
    \put(0.39021646,0.67407132){\makebox(0,0)[lt]{\lineheight{1.25}\smash{\begin{tabular}[t]{l}$c_2$\end{tabular}}}}%
    \put(0.59860044,0.68713917){\makebox(0,0)[lt]{\lineheight{1.25}\smash{\begin{tabular}[t]{l}$k$\end{tabular}}}}%
    \put(0.6069207,0.75378927){\makebox(0,0)[lt]{\lineheight{1.25}\smash{\begin{tabular}[t]{l}$l$\end{tabular}}}}%
    \put(0.39655773,0.77352527){\makebox(0,0)[lt]{\lineheight{1.25}\smash{\begin{tabular}[t]{l}$j$\end{tabular}}}}%
    \put(0.37479277,0.72873412){\makebox(0,0)[lt]{\lineheight{1.25}\smash{\begin{tabular}[t]{l}$i$\end{tabular}}}}%
    \put(0,0){\includegraphics[width=\unitlength,page=8]{handle.pdf}}%
    \put(0.43237537,0.26955719){\makebox(0,0)[lt]{\lineheight{1.25}\smash{\begin{tabular}[t]{l}$j$\end{tabular}}}}%
    \put(0.37659265,0.17473057){\makebox(0,0)[lt]{\lineheight{1.25}\smash{\begin{tabular}[t]{l}$i$\end{tabular}}}}%
    \put(0.56990048,0.15727161){\makebox(0,0)[lt]{\lineheight{1.25}\smash{\begin{tabular}[t]{l}$k$\end{tabular}}}}%
    \put(0.60013853,0.26562442){\makebox(0,0)[lt]{\lineheight{1.25}\smash{\begin{tabular}[t]{l}$l$\end{tabular}}}}%
    \put(0.47137084,0.07394409){\makebox(0,0)[lt]{\lineheight{1.25}\smash{\begin{tabular}[t]{l}{\color{blue}$b$}\end{tabular}}}}%
    \put(0.532027,1.29102741){\makebox(0,0)[lt]{\lineheight{1.25}\smash{\begin{tabular}[t]{l}{\color{blue}$t$}\end{tabular}}}}%
    \put(0.49664594,0.26056504){\makebox(0,0)[lt]{\lineheight{1.25}\smash{\begin{tabular}[t]{l}{\color{blue}$p_3$}\end{tabular}}}}%
    \put(0.40801482,0.14791928){\makebox(0,0)[lt]{\lineheight{1.25}\smash{\begin{tabular}[t]{l}{\color{blue}$n_3$}\end{tabular}}}}%
    \put(0.34700764,0.19622323){\makebox(0,0)[lt]{\lineheight{1.25}\smash{\begin{tabular}[t]{l}{\color{blue}$m_3$}\end{tabular}}}}%
    \put(0.58950646,0.22529377){\makebox(0,0)[lt]{\lineheight{1.25}\smash{\begin{tabular}[t]{l}{\color{blue}$q_3$}\end{tabular}}}}%
    \put(0.38558201,1.15048716){\makebox(0,0)[lt]{\lineheight{1.25}\smash{\begin{tabular}[t]{l}$i$\end{tabular}}}}%
    \put(0.39057156,1.24636343){\makebox(0,0)[lt]{\lineheight{1.25}\smash{\begin{tabular}[t]{l}$j$\end{tabular}}}}%
    \put(0.58664173,1.15273105){\makebox(0,0)[lt]{\lineheight{1.25}\smash{\begin{tabular}[t]{l}$k$\end{tabular}}}}%
    \put(0.60470628,1.24831943){\makebox(0,0)[lt]{\lineheight{1.25}\smash{\begin{tabular}[t]{l}$l$\end{tabular}}}}%
    \put(0.6225437,1.22539459){\makebox(0,0)[lt]{\lineheight{1.25}\smash{\begin{tabular}[t]{l}{\color{blue}$q_1$}\end{tabular}}}}%
    \put(0.57292259,1.2632522){\makebox(0,0)[lt]{\lineheight{1.25}\smash{\begin{tabular}[t]{l}{\color{blue}$p_1$}\end{tabular}}}}%
    \put(0.48917087,1.15707525){\makebox(0,0)[lt]{\lineheight{1.25}\smash{\begin{tabular}[t]{l}{\color{blue}$n_1$}\end{tabular}}}}%
    \put(0.38357994,1.19565559){\makebox(0,0)[lt]{\lineheight{1.25}\smash{\begin{tabular}[t]{l}{\color{blue}$m_1$}\end{tabular}}}}%
    \put(0.3930914,1.30683968){\makebox(0,0)[lt]{\lineheight{1.25}\smash{\begin{tabular}[t]{l}$x_1$\end{tabular}}}}%
    \put(0.46238701,1.29514042){\makebox(0,0)[lt]{\lineheight{1.25}\smash{\begin{tabular}[t]{l}$y_1$\end{tabular}}}}%
    \put(0.52682292,1.32861831){\makebox(0,0)[lt]{\lineheight{1.25}\smash{\begin{tabular}[t]{l}$w_1$\end{tabular}}}}%
    \put(0.61753707,1.31097941){\makebox(0,0)[lt]{\lineheight{1.25}\smash{\begin{tabular}[t]{l}$z_1$\end{tabular}}}}%
    \put(0.36087341,0.10721523){\makebox(0,0)[lt]{\lineheight{1.25}\smash{\begin{tabular}[t]{l}$x_2$\end{tabular}}}}%
    \put(0.44654799,0.10055571){\makebox(0,0)[lt]{\lineheight{1.25}\smash{\begin{tabular}[t]{l}$y_2$\end{tabular}}}}%
    \put(0.59323876,0.11063494){\makebox(0,0)[lt]{\lineheight{1.25}\smash{\begin{tabular}[t]{l}$z_2$\end{tabular}}}}%
    \put(0.52124332,0.09713585){\makebox(0,0)[lt]{\lineheight{1.25}\smash{\begin{tabular}[t]{l}$w_2$\end{tabular}}}}%
    \put(0,0){\includegraphics[width=\unitlength,page=9]{handle.pdf}}%
    \put(0.41271021,0.87036714){\makebox(0,0)[lt]{\lineheight{1.25}\smash{\begin{tabular}[t]{l}$x_1$\end{tabular}}}}%
    \put(0.50738413,0.87414623){\makebox(0,0)[lt]{\lineheight{1.25}\smash{\begin{tabular}[t]{l}$y_1$\end{tabular}}}}%
    \put(0.4634669,0.87072692){\makebox(0,0)[lt]{\lineheight{1.25}\smash{\begin{tabular}[t]{l}$w_1$\end{tabular}}}}%
    \put(0.55004134,0.87486681){\makebox(0,0)[lt]{\lineheight{1.25}\smash{\begin{tabular}[t]{l}$z_1$\end{tabular}}}}%
    \put(0.54752154,0.58256821){\makebox(0,0)[lt]{\lineheight{1.25}\smash{\begin{tabular}[t]{l}$z_2$\end{tabular}}}}%
    \put(0.4971247,0.58256821){\makebox(0,0)[lt]{\lineheight{1.25}\smash{\begin{tabular}[t]{l}$y_2$\end{tabular}}}}%
    \put(0.39885083,0.58760783){\makebox(0,0)[lt]{\lineheight{1.25}\smash{\begin{tabular}[t]{l}$x_2$\end{tabular}}}}%
    \put(0.44420802,0.58508802){\makebox(0,0)[lt]{\lineheight{1.25}\smash{\begin{tabular}[t]{l}$w_2$\end{tabular}}}}%
  \end{picture}%
\endgroup%

%% file: knotcomplement.pdf_tex
%% Creator: Inkscape 1.0.2 (e86c8708, 2021-01-15), www.inkscape.org
%% Accompanies image file 'knotcomplement.pdf' (pdf, eps, ps)
\begingroup%
  \makeatletter%
  \providecommand\color[2][]{%
    \errmessage{(Inkscape) Color is used for the text in Inkscape, but the package 'color.sty' is not loaded}%
    \renewcommand\color[2][]{}%
  }%
  \providecommand\transparent[1]{%
    \errmessage{(Inkscape) Transparency is used (non-zero) for the text in Inkscape, but the package 'transparent.sty' is not loaded}%
    \renewcommand\transparent[1]{}%
  }%
  \providecommand\rotatebox[2]{#2}%
  \newcommand*\fsize{\dimexpr\f@size pt\relax}%
  \newcommand*\lineheight[1]{\fontsize{\fsize}{#1\fsize}\selectfont}%
  \ifx\svgwidth\undefined%
    \setlength{\unitlength}{841.88976378bp}%
    \ifx\svgscale\undefined%
      \relax%
    \else%
      \setlength{\unitlength}{\unitlength * \real{\svgscale}}%
    \fi%
  \else%
    \setlength{\unitlength}{\svgwidth}%
  \fi%
  \global\let\svgwidth\undefined%
  \global\let\svgscale\undefined%
  \makeatother%
  \begin{picture}(1,0.70707071)%
    \lineheight{1}%
    \setlength\tabcolsep{0pt}%
    \put(0,0){\includegraphics[width=\unitlength,page=1]{knotcomplement.pdf}}%
    \put(0.22731229,0.41739418){\makebox(0,0)[lt]{\lineheight{1.25}\smash{\begin{tabular}[t]{l}$g_1$\end{tabular}}}}%
    \put(0,0){\includegraphics[width=\unitlength,page=2]{knotcomplement.pdf}}%
    \put(0.27176903,0.21752652){\makebox(0,0)[lt]{\lineheight{1.25}\smash{\begin{tabular}[t]{l}$g_2$\end{tabular}}}}%
    \put(0.66851509,0.53513605){\makebox(0,0)[lt]{\lineheight{1.25}\smash{\begin{tabular}[t]{l}$g_3$\end{tabular}}}}%
    \put(0.5489374,0.32551421){\makebox(0,0)[lt]{\lineheight{1.25}\smash{\begin{tabular}[t]{l}$g_5$\end{tabular}}}}%
    \put(0.38245638,0.04474622){\makebox(0,0)[lt]{\lineheight{1.25}\smash{\begin{tabular}[t]{l}$g_4$\end{tabular}}}}%
    \put(0.77301185,0.21392694){\makebox(0,0)[lt]{\lineheight{1.25}\smash{\begin{tabular}[t]{l}$g_6$\end{tabular}}}}%
    \put(0,0){\includegraphics[width=\unitlength,page=3]{knotcomplement.pdf}}%
  \end{picture}%
\endgroup%

%% file: braidingtet.pdf_tex
%% Creator: Inkscape 1.0.2 (e86c8708, 2021-01-15), www.inkscape.org
%% Accompanies image file 'braidingtet.pdf' (pdf, eps, ps)
\begingroup%
  \makeatletter%
  \providecommand\color[2][]{%
    \errmessage{(Inkscape) Color is used for the text in Inkscape, but the package 'color.sty' is not loaded}%
    \renewcommand\color[2][]{}%
  }%
  \providecommand\transparent[1]{%
    \errmessage{(Inkscape) Transparency is used (non-zero) for the text in Inkscape, but the package 'transparent.sty' is not loaded}%
    \renewcommand\transparent[1]{}%
  }%
  \providecommand\rotatebox[2]{#2}%
  \newcommand*\fsize{\dimexpr\f@size pt\relax}%
  \newcommand*\lineheight[1]{\fontsize{\fsize}{#1\fsize}\selectfont}%
  \ifx\svgwidth\undefined%
    \setlength{\unitlength}{841.88976378bp}%
    \ifx\svgscale\undefined%
      \relax%
    \else%
      \setlength{\unitlength}{\unitlength * \real{\svgscale}}%
    \fi%
  \else%
    \setlength{\unitlength}{\svgwidth}%
  \fi%
  \global\let\svgwidth\undefined%
  \global\let\svgscale\undefined%
  \makeatother%
  \begin{picture}(1,0.70707071)%
    \lineheight{1}%
    \setlength\tabcolsep{0pt}%
    \put(0,0){\includegraphics[width=\unitlength,page=1]{braidingtet.pdf}}%
    \put(0.01828534,0.60139521){\makebox(0,0)[lt]{\lineheight{1.25}\smash{\begin{tabular}[t]{l}$p$\end{tabular}}}}%
    \put(0.00586755,0.5051573){\makebox(0,0)[lt]{\lineheight{1.25}\smash{\begin{tabular}[t]{l}$m$\end{tabular}}}}%
    \put(0.25260399,0.50747775){\makebox(0,0)[lt]{\lineheight{1.25}\smash{\begin{tabular}[t]{l}$n$\end{tabular}}}}%
    \put(0.25809192,0.60152517){\makebox(0,0)[lt]{\lineheight{1.25}\smash{\begin{tabular}[t]{l}$q$\end{tabular}}}}%
    \put(0.09287346,0.52632513){\makebox(0,0)[lt]{\lineheight{1.25}\smash{\begin{tabular}[t]{l}{\color{magenta}$b$}\end{tabular}}}}%
    \put(0.16542475,0.57179654){\makebox(0,0)[lt]{\lineheight{1.25}\smash{\begin{tabular}[t]{l}{\color{cyan}$a$}\end{tabular}}}}%
    \put(0.16131111,0.61446174){\makebox(0,0)[lt]{\lineheight{1.25}\smash{\begin{tabular}[t]{l}$i$\end{tabular}}}}%
    \put(0.21958839,0.52770073){\makebox(0,0)[lt]{\lineheight{1.25}\smash{\begin{tabular}[t]{l}$j$\end{tabular}}}}%
    \put(0,0){\includegraphics[width=\unitlength,page=2]{braidingtet.pdf}}%
    \put(0.72741369,0.59572139){\makebox(0,0)[lt]{\lineheight{1.25}\smash{\begin{tabular}[t]{l}$q$\end{tabular}}}}%
    \put(0.72568623,0.50839203){\makebox(0,0)[lt]{\lineheight{1.25}\smash{\begin{tabular}[t]{l}$n$\end{tabular}}}}%
    \put(0.96173225,0.50180391){\makebox(0,0)[lt]{\lineheight{1.25}\smash{\begin{tabular}[t]{l}$s$\end{tabular}}}}%
    \put(0.96722022,0.59585132){\makebox(0,0)[lt]{\lineheight{1.25}\smash{\begin{tabular}[t]{l}$r$\end{tabular}}}}%
    \put(0.80631344,0.52065129){\makebox(0,0)[lt]{\lineheight{1.25}\smash{\begin{tabular}[t]{l}{\color{orange}$c$}\end{tabular}}}}%
    \put(0.87455306,0.56612269){\makebox(0,0)[lt]{\lineheight{1.25}\smash{\begin{tabular}[t]{l}{\color{cyan}$a$}\end{tabular}}}}%
    \put(0.87043945,0.60878788){\makebox(0,0)[lt]{\lineheight{1.25}\smash{\begin{tabular}[t]{l}$i$\end{tabular}}}}%
    \put(0.92871668,0.52202691){\makebox(0,0)[lt]{\lineheight{1.25}\smash{\begin{tabular}[t]{l}$j'$\end{tabular}}}}%
    \put(0,0){\includegraphics[width=\unitlength,page=3]{braidingtet.pdf}}%
    \put(0.01524256,0.17825881){\makebox(0,0)[lt]{\lineheight{1.25}\smash{\begin{tabular}[t]{l}$m$\end{tabular}}}}%
    \put(0.01575995,0.09064428){\makebox(0,0)[lt]{\lineheight{1.25}\smash{\begin{tabular}[t]{l}$k$\end{tabular}}}}%
    \put(0.2560288,0.08434138){\makebox(0,0)[lt]{\lineheight{1.25}\smash{\begin{tabular}[t]{l}$l$\end{tabular}}}}%
    \put(0.26151673,0.17838874){\makebox(0,0)[lt]{\lineheight{1.25}\smash{\begin{tabular}[t]{l}$n$\end{tabular}}}}%
    \put(0.09629826,0.1031887){\makebox(0,0)[lt]{\lineheight{1.25}\smash{\begin{tabular}[t]{l}{\color{magenta}$b$}\end{tabular}}}}%
    \put(0.16884956,0.14866011){\makebox(0,0)[lt]{\lineheight{1.25}\smash{\begin{tabular}[t]{l}{\color{blue}$d$}\end{tabular}}}}%
    \put(0.16473591,0.1913253){\makebox(0,0)[lt]{\lineheight{1.25}\smash{\begin{tabular}[t]{l}$i'$\end{tabular}}}}%
    \put(0.2230132,0.10456436){\makebox(0,0)[lt]{\lineheight{1.25}\smash{\begin{tabular}[t]{l}$j$\end{tabular}}}}%
    \put(0,0){\includegraphics[width=\unitlength,page=4]{braidingtet.pdf}}%
    \put(0.72577605,0.17980081){\makebox(0,0)[lt]{\lineheight{1.25}\smash{\begin{tabular}[t]{l}$n$\end{tabular}}}}%
    \put(0.72198179,0.09218634){\makebox(0,0)[lt]{\lineheight{1.25}\smash{\begin{tabular}[t]{l}$l$\end{tabular}}}}%
    \put(0.96009474,0.08588345){\makebox(0,0)[lt]{\lineheight{1.25}\smash{\begin{tabular}[t]{l}$t$\end{tabular}}}}%
    \put(0.96558271,0.1799308){\makebox(0,0)[lt]{\lineheight{1.25}\smash{\begin{tabular}[t]{l}$s$\end{tabular}}}}%
    \put(0.80467587,0.10473076){\makebox(0,0)[lt]{\lineheight{1.25}\smash{\begin{tabular}[t]{l}{\color{orange}$c$}\end{tabular}}}}%
    \put(0.87291543,0.15020217){\makebox(0,0)[lt]{\lineheight{1.25}\smash{\begin{tabular}[t]{l}{\color{blue}$d$}\end{tabular}}}}%
    \put(0.86880182,0.19286737){\makebox(0,0)[lt]{\lineheight{1.25}\smash{\begin{tabular}[t]{l}$i'$\end{tabular}}}}%
    \put(0.92707918,0.10610636){\makebox(0,0)[lt]{\lineheight{1.25}\smash{\begin{tabular}[t]{l}$j'$\end{tabular}}}}%
    \put(0,0){\includegraphics[width=\unitlength,page=5]{braidingtet.pdf}}%
    \put(0.30836531,0.5414418){\makebox(0,0)[lt]{\lineheight{1.25}\smash{\begin{tabular}[t]{l}$p$\end{tabular}}}}%
    \put(0.30988494,0.15593733){\makebox(0,0)[lt]{\lineheight{1.25}\smash{\begin{tabular}[t]{l}$k$\end{tabular}}}}%
    \put(0.67846422,0.157578){\makebox(0,0)[lt]{\lineheight{1.25}\smash{\begin{tabular}[t]{l}$t$\end{tabular}}}}%
    \put(0.68166922,0.53924355){\makebox(0,0)[lt]{\lineheight{1.25}\smash{\begin{tabular}[t]{l}$r$\end{tabular}}}}%
    \put(0.48672376,0.54172776){\makebox(0,0)[lt]{\lineheight{1.25}\smash{\begin{tabular}[t]{l}$q$\end{tabular}}}}%
    \put(0.70148588,0.36001486){\makebox(0,0)[lt]{\lineheight{1.25}\smash{\begin{tabular}[t]{l}$s$\end{tabular}}}}%
    \put(0.48843946,0.15480641){\makebox(0,0)[lt]{\lineheight{1.25}\smash{\begin{tabular}[t]{l}$l$\end{tabular}}}}%
    \put(0.27177582,0.35946837){\makebox(0,0)[lt]{\lineheight{1.25}\smash{\begin{tabular}[t]{l}$m$\end{tabular}}}}%
    \put(0.44546586,0.28901311){\makebox(0,0)[lt]{\lineheight{1.25}\smash{\begin{tabular}[t]{l}{\color{blue}$d$}\end{tabular}}}}%
    \put(0.44487411,0.457685){\makebox(0,0)[lt]{\lineheight{1.25}\smash{\begin{tabular}[t]{l}{\color{cyan}$a$}\end{tabular}}}}%
    \put(0.55895881,0.33562509){\makebox(0,0)[lt]{\lineheight{1.25}\smash{\begin{tabular}[t]{l}{\color{orange}$c$}\end{tabular}}}}%
    \put(0.36951689,0.33092685){\makebox(0,0)[lt]{\lineheight{1.25}\smash{\begin{tabular}[t]{l}{\color{magenta}$b$}\end{tabular}}}}%
    \put(0.51732538,0.38174995){\makebox(0,0)[lt]{\lineheight{1.25}\smash{\begin{tabular}[t]{l}$i$\end{tabular}}}}%
    \put(0.52020444,0.21580469){\makebox(0,0)[lt]{\lineheight{1.25}\smash{\begin{tabular}[t]{l}$i'$\end{tabular}}}}%
    \put(0.43974148,0.33759185){\makebox(0,0)[lt]{\lineheight{1.25}\smash{\begin{tabular}[t]{l}$j$\end{tabular}}}}%
    \put(0.63249381,0.33506393){\makebox(0,0)[lt]{\lineheight{1.25}\smash{\begin{tabular}[t]{l}$j'$\end{tabular}}}}%
  \end{picture}%
\endgroup%